\documentclass[11pt]{article}
\usepackage[utf8]{inputenc}

\usepackage{amsmath}
\usepackage{amsthm}
\usepackage{mathrsfs,bbm}
\usepackage{amssymb}
\usepackage{paralist}
\usepackage{appendix}
\usepackage{color}
\usepackage[top=1in,bottom=1.2in,left=1in,right=1in,xetex]{geometry}

\usepackage[unicode=true,pdfusetitle,
bookmarks=true,bookmarksnumbered=true,bookmarksopen=true,bookmarksopenlevel=3,
breaklinks=false,pdfborder={0 0 1},backref=false,colorlinks=false]
{hyperref}

\numberwithin{equation}{section}

\newtheorem{theorem}{Theorem}[section]

\newtheorem{lemma}[theorem]{Lemma}
\newtheorem{proposition}[theorem]{Proposition}

\newcommand{\br}{\mathbb{R}}
\newcommand{\brn}{{\mathbb{R}^n}}
\newcommand{\brnn}{{\mathbb{R}^{n\times n}}}
\newcommand{\brd}{{\mathbb{R}^d}}
\newcommand{\dd}{\partial}
\newcommand{\de}{\Delta}

\newcommand{\e}{\mathbb{E}}
\newcommand{\he}{\widehat{\mathbb{E}}}
\newcommand{\f}{\mathcal{F}}
\newcommand{\lr}{\mathcal{L}}
\newcommand{\sr}{\mathcal{S}}
\newcommand{\pr}{\mathcal{P}}

\newcommand{\hv}{\widehat{v}}
\newcommand{\hr}{\mathcal{H}}
\newcommand{\x}{\mathscr{X}}
\newcommand{\y}{\mathscr{Y}}
\newcommand{\p}{\mathscr{P}}
\newcommand{\R}{\mathscr{R}}
\newcommand{\U}{\mathscr{U}}
\newcommand{\ir}{\mathcal{I}}

\newcommand{\argmin}{\mathop{\rm argmin}}

\allowdisplaybreaks[3]

\title{Optimal Control of Mean-field FBSDEs and Classical Solutions of the Related Generalized HJB Equations}

\usepackage{authblk}

\author{Ziyu Huang\footnote{E-mail: ziyu.huang@hu-berlin.de}}

\affil{\small\it Department of Mathematics, Humboldt University Berlin}

\begin{document}

\maketitle

\begin{abstract}

In this article, we study optimal control problems driven by mean-field forward-backward stochastic differential equations (FBSDEs). The aim is two-fold: (i) to provide the unique optimal control; (ii) to establish classical solutions of the associated generalized Hamilton-Jacobi-Bellman (HJB) equations on the space of probability measures. For (i), we study the coupled system arising from the maximum principle, which involves both the state FBSDEs and the adjoint FBSDEs. This system is non-standard and more complicated than that in standard mean-field type control (MFTC) problems. By extending the method of continuation in the coefficients to our framework, we establish the well-posedness result, and subsequently obtain the differentiability of the value functional with respect to the distribution argument --- thereby extending existing results for MFTC problems. For (ii), since the flow property does not generally hold for control problems of FBSDEs, standard dynamic-programming arguments cannot be applied directly. We therefore develop probabilistic approaches to the classical solvability of generalized HJB equations under three distinct sets of conditions. The auxiliary processes arising in these arguments differ significantly from those in standard MFTC problems. We establish new connections between the value function and the adjoint processes, derive equations for the relevant decoupling field, and obtain first- and second-order generalized HJB equations, including a coupled system on the space of probability measures. The standard HJB equations for MFTC problems are recovered as special cases. Our approach yields classical solutions without requiring non-degeneracy of the diffusion. \\



\end{abstract}

\tableofcontents

\section{Introduction}\label{sec:intro}

% State the control problem (open-loop) for MF-FBSDEs, aim: 1. optilam control (SMP+FBSDEs), and study the regularity of the value function 2. HJB equation. We give both answers. 
In this article, we study optimal control problems driven by distribution-dependent forward-backward stochastic differential equations (FBSDEs), together with the regularity of their value functions and the classical solvability of the associated generalized Hamilton--Jacobi--Bellman (HJB) equations.
%and also the classical solutions to their related Hamilton–Jacobi–Bellman (HJB) equations. 
Let $(\Omega,\f,\{\f_t\}_{0\le t\le T},\mathbb{P})$ be a complete filtered probability space on which an $n$-dimensional Brownian motion $\{B_t\}_{0\le t\le T}$ is defined. For an initial time $t\in[0,T]$ and initial condition $\xi\in \lr_{\f_t}^2$, the state process associated with a control $v\in \lr_\f^2(t,T)$ is given by the following mean-field FBSDE:
\begin{align}\label{FBSDE^v(Introduction)}
    \left\{
    \begin{aligned}
        X^v_s=\ & \xi+\int_t^s b \left(r,X^v_r, \lr\left(X^v_r\right), v_r \right)dr + \int_t^s \sigma\left(r,X^v_r, \lr\left(X^v_r\right), v_r \right)dB_r,\\
        Y^v_s=\ & g\left(X^v_T, \lr\left(X^v_T\right) \right) + \int_s^T f\left(r,X^v_r, Y^v_r,Z^v_r, \lr\left(X^v_r, Y^v_r,Z^v_r\right), v_r \right)dr -\int_s^T Z^v_r dB_r,
    \end{aligned}
    \right.
\end{align}
where $\lr(\cdot)$ denotes the law of a random variable. The associated cost functional is
\begin{equation}\label{cost^v}
\begin{aligned}
    J^{t,\xi}(v):= \e\bigg[&\int_t^T h\left(s,X^v_s, Y^v_s,Z^v_s, \lr\left(X^v_s, Y^v_s,Z^v_s\right), v_s \right) ds + \Phi\left(X_T^v,\lr\left(X_T^v\right)\right)+\varphi\left(Y^v_t\right)\bigg].
\end{aligned}
\end{equation}
%where we use the superscript $t,\xi$ to denote the dependence of the cost on the initial condition. 
We denote by $(\mathsf{P}_{t,\xi})$ the (open-loop) optimal control problem with state dynamics \eqref{FBSDE^v(Introduction)} and cost \eqref{cost^v}. Its value function $V:[0,T]\times\pr_2(\brn)\to\br$ is defined by
\begin{equation}\label{problem}
    V(t,\mu):= \inf_{v\in\lr_\f^2(t,T)} J^{t,\xi}(v),
\end{equation}
which depends on $\xi$ only through its law $\mu$ (see Section~\ref{sec:value}). The main objectives of this paper are twofold: \vspace{-0.5em}
\begin{enumerate}[(I)]
    \item Establish the well-posedness of the coupled forward-backward system (comprising both state and adjoint FBSDEs) arising from the maximum principle, thereby obtaining the unique optimal control for $(\mathsf{P}_{t,\xi})$ and deriving the differentiability of the value function $V$; \vspace{-0.5em}
    \item Derive the associated generalized HJB equation under three distinct sets of conditions, and prove the existence of classical solutions in each case.
\end{enumerate}

%standard MFTC, novelty, and also an extension of my precious work. FBSDEs as state processes, motivation for mean-field FBSDEs, examples, literature. Mean-field FBSDEs as state process
In standard stochastic control problems for stochastic differential equations (SDEs), backward stochastic differential equations (BSDEs) naturally arise as adjoint processes via the stochastic maximum principle, a line of work initiated by Bismut \cite{MR453161,MR469466} and subsequently developed by Pardoux and Peng \cite{MR1037747,MR1202528}. In mean-field type control (MFTC), where the state process is a McKean-Vlasov SDE (also called mean-field SDE), the adjoint BSDEs involve derivatives with respect to the distribution of the state; see, e.g., \cite{AB10,CR} and the references therein. Our problem $(\mathsf{P}_{t,\xi})$ differs from standard MFTC in that the controlled state includes both a McKean-Vlasov SDE and a dependent BSDE (not as an adjoint, but as part of the state), with an additional $\varphi(Y^{t,\xi}_t)$ term in the cost. This structure precludes the direct application or straightforward extension of standard MFTC methods, and makes the problem significantly more challenging. Examples from financial mathematics that motivate optimal control problems of FBSDEs include the non-Markovian Merton problem \cite[Section 3]{Oksendal2016} and social planner problems with heterogeneous Epstein--Zin preferences \cite[Section 5]{MR4781585}, highlighting the importance of adopting FBSDEs as the underlying state processes. Optimal control problems for FBSDEs have been studied under various settings: existence of optimal controls via approximation \cite{Bahlali18082018,Buckdahn01062010}, dynamic programming principle (DPP) and viscosity solutions to HJB equations \cite{MR3205734}, maximum principles for jump-diffusion cases \cite{oksendal2010maximum}, and problems with partially observable information \cite{wang2018introduction}. With the development of mean-field theory, mean-field FBSDEs as state processes have found applications in financial mathematics; see \cite[Section 6]{MR4199898} for mean-variance portfolio selection with mean-field recursive utility, and \cite[Section 4]{Agram04032021} for optimal portfolio with mean-field risk minimization. On the theoretical side, maximum principles for optimal control problems driven by mean-field FBSDEs have been established under various settings: fully coupled systems with applications to Stackelberg games \cite{MR4587511}, regime-switching systems with singular controls \cite{SHI2025106216}, models involving Teugels martingales and L\'evy processes \cite{MR3952459}, and cases with heterogeneous mean-field interactions \cite{SZ}. The existing literature has primarily concentrated on maximum principles under diverse settings, while the well-posedness and regularity of the resulting coupled state-adjoint system, as well as the classical solvability of the associated HJB equations, remain largely unexplored. This paper fills this gap by providing a complete well-posedness result and establishing classical solvability of the HJB equation in distinct cases.

%For Aim 1, SMP, adjoint process, the system contains four equations. well-posedness (new), extend the continuation in method method. we allow diffusion to be control-dependent. our linear setting can be further relaxed by using the cone property. the value function 
For Aim (I), we apply the stochastic maximum principle to derive the optimality condition for $(\mathsf{P}_{t,\xi})$ via the adjoint processes associated with the state FBSDEs \eqref{FBSDE^v(Introduction)}. In standard stochastic control problems for SDEs, the adjoint process is characterized as the solution of a BSDE \cite{MR1696772}. In our setting, the adjoint process $P^v$ for $X^v$ satisfies a BSDE, while the adjoint process $R^v$ for $Y^v$ satisfies an SDE, and together they form a coupled FBSDE system (see \eqref{adjoint^v(Introduction)}). This system, together with state equations and the optimality condition, constitutes a forward-backward system consisting of two SDEs and two BSDEs (see System \eqref{FBSDE:MFTC}), which is significantly more complex than coupled systems arising in standard MFTC problems. We establish the well-posedness of this coupled optimality system, whose solution is denoted by
\begin{align*}
    \Pi^{t,\xi}_s:=\left(X^{t,\xi}_s,Y^{t,\xi}_s,Z^{t,\xi}_s,P^{t,\xi}_s,Q^{t,\xi}_s,R^{t,\xi}_s,u^{t,\xi}_s\right),\quad s\in[t,T],
\end{align*}
under a linear setting for $(b,\sigma,f)$ by extending the continuation method in coefficients \cite{YH2} to our framework. For this part, our assumptions allow the diffusion to be control-dependent and unbounded, and require no non-degeneracy condition. To the best of our knowledge, this well-posedness result is new. Our linearity assumptions on the state coefficients are likely to be further relaxed by following the approaches developed in \cite{AB14}. The well-posedness and stability estimates for $\Pi^{t,\xi}$ further allow us to establish the differentiability of the value functional $V$ defined in \eqref{problem} with respect to the distribution argument, extending the corresponding regularity results for standard MFTC problems.

%For Aim 2, we still use the approach for the value function to give the HJB euqtion (cite our previous works). But here, due to the lack of flow property, the problem is time-inconsistent, and the DPP does not generally hold. Also literature for the time-consistent problem. We still focus on the time-inconsistent case, and propose new approaches to the calssical solutions to the HJB equations under assumptions. Our results for HJB divided into three parts, (1) HJB equation for the first-order problem, (2) for the second-order HJB, (2a) consider the feedback cases (special structrue of $g$), and (2b) study the special conditions to ensure the flow property (special structrue of $h$ in $(y,z)$).
For Aim (II), a natural approach is to derive the HJB equation by studying the time regularity of $V$ via $\Pi^{t,\xi}$. This probabilistic viewpoint together with the associated DPP, has been widely used in standard MFTC problems, where the flow property plays a central role; see \cite{AB10,BR,DMF,PH}. However, for control problems of mean-field FBSDEs, the flow property generally fails (see Section~\ref{sec:flow}), so the same approach does not directly apply. Indeed, the open-loop problem driven by mean-field FBSDEs \eqref{FBSDE^v(Introduction)} is time-inconsistent, and the DPP does not generally hold; see \cite{MR4781585} for a study of equilibrium strategies and related equilibrium HJB equations for control problems driven by FBSDEs, and \cite{MR3639569,MR4805756,MR3646768} for related linear-quadratic problems. In this paper, we focus on the original open-loop problem $(\mathsf{P}_{t,\xi})$ as defined in \eqref{FBSDE^v(Introduction)}-\eqref{problem}, rather than reformulating it in terms of equilibrium solutions. For this time-inconsistent problem, we develop new probabilistic approaches to establish classical solutions to the associated HJB equations under the following three distinct scenarios: \vspace{-0.5em}
\begin{enumerate}[(1)]
    \item the first-order case (i.e., without Brownian motion); \vspace{-0.5em}
    \item the second-order case (with Brownian motion) and a feedback structure on $Y^{t,\xi}$; \vspace{-0.5em}
    \item the second-order case under conditions ensuring the flow property.
\end{enumerate} \vspace{-0.6em}
To the best of our knowledge, these results are the first in the literature --- even in the first-order setting --- concerning classical solutions to HJB equations corresponding to optimal control problems of mean-field FBSDEs.

%HJB equation for the first-order problem. We skip the DPP
For (1), to circumvent the lack of the flow property, we propose a direct study of the time regularity of the map $t\mapsto \Pi^{t,\xi}$ under $C^2$ assumptions on the cost coefficients. The limit of the difference quotient $\frac{1}{\epsilon}\left(\Pi^{t+\epsilon,\xi}_s-\Pi^{t,\xi}_s\right),\ s\in[t+\epsilon,T]$ as $\epsilon\to0$ is characterized as the solution of a new forward-backward system (see \eqref{FBSDEs:high}), whose initial conditions make it significantly different from FBSDEs arising from the standard MFTC problems. With the regularity of these auxiliary processes, we show that $V$ is the classical solution of the following HJB equation:
\begin{equation}\label{HJB(Introduction)}
    \left\{
    \begin{aligned}
        &\dd_t V(t,\mu)+\int_\brn \inf_{v\in\brd}\bigg\{ h\left(t,x,Y^{t,x,\mu}_t,\left(I_n\otimes Y^{t,\cdot,\mu}_t\right) \#\mu,v\right) + b(t,x,\mu,v)^\top D\frac{dV}{d\nu} (t,\mu)(x) \\
        &\quad\qquad\qquad\qquad\qquad +f\left(t,x,Y^{t,x,\mu}_t,\left(I_n\otimes Y^{t,\cdot,\mu}_t\right) \#\mu,v\right)^\top  \varphi_y \left(Y^{t,x,\mu}_t\right) \bigg\}\mu(dx) = 0,\\
        &V(T,\mu)= \int_\brn  \left[\Phi(x,\mu)+\varphi\left(g(x,\mu)\right)\right] \mu(dx),
    \end{aligned}
    \right.
\end{equation}
where $\left(I_n\otimes Y^{t,\cdot,\mu}_t\right) \#\mu$ denotes the push-forward of $\mu$ by the map $x\mapsto (x,Y^{t,x,\mu}_t)$. Here, $Y^{t,x,\mu}$ is the component of the solution of a forward-backward system with deterministic initial condition $x$, while still depending on the law of the original system (see \eqref{FBSDEs:x}). The integral term in the first equation of \eqref{HJB(Introduction)} can be expressed in terms of the Hamiltonian function (see Equation \eqref{HJB}). This HJB equation generalizes the standard first-order mean-field type HJB equation (recovered when $f$, $g$, and $\varphi$ vanish); and extends the corresponding results for controlled FBSDEs to the mean-field setting, with the equation now defined on the infinite-dimensional space $\pr_2(\brn)$. Furthermore, we also introduce the map 
\begin{align}\label{def:U(introduction)}
    U(t,x,\mu):=Y^{t,x,\mu}_t,
\end{align}
and establish its regularity by studying higher-order derivatives of $Y^{t,x,\lr(\xi)}_s$ with respect to $t$, $x$, and $\xi$ via the lifting method. These derivatives are characterized as components of solutions of several auxiliary FBSDE systems, whose well-posedness yields the differentiability of $U$ in $x,\mu$ and $t$. 

%the feedback section, we transform the problem into MFTC, and see the relations 
For (2), we return to the second-order problems and consider the case where the BSDE component admits a feedback representation $Y^v_s= F\left(s,X^v_s\right)$ for some suitable map $F$. This requires a structural assumption on $f$ (Assumption (C) in Section~\ref{sec:feedback}), under which $(\mathsf{P}_{t,\xi})$ can be transformed into a standard MFTC problem with a generalized cost. We establish a novel connection between the adjoint processes of the transformed and original problems (see \eqref{relation_PQR-PQ}), which yields the corresponding HJB equation. When the diffusion is independent of the control, $V$ is shown to be the classical solution to the second-order HJB equation \eqref{HJB:second-order}, which generalizes the standard MFTC HJB equation (recovered when $F\equiv0$ and $\varphi\equiv0$) and provides a characterization of $Z^{t,\xi}_t$ and $Q^{t,\xi}_t$ (see \eqref{feedback:ZQ}). We note that \eqref{HJB:second-order} and \eqref{HJB(Introduction)} are derived under different assumptions: the former relies on the feedback structure and allows $\sigma\not\equiv 0$, while the latter requires no structural assumption on the backward state equation but is obtained in the first-order setting. Thus, neither equation is a special case of the other.

%flow property, HJB coupled with decoupling field
For (3), we establish a necessary and sufficient condition for the flow property (see \eqref{flow_property}), and also provide workable conditions on $h$ in $(y,z)$ (Assumption (D) in Section~\ref{sec:flow}) that ensure it holds. This characterization also expresses $R^{t,\xi}_s$ as a function of $Y^{t,\xi}_s$ and $P^{t,\xi}_s$ as a function of $X^{t,\xi}_s$, which facilitates the DPP and the time expansion of the value function. The latter involves $(X^{t,\xi},Y^{t,\xi},Z^{t,\xi})$ and motivates us to establish the continuity of $Z^{t,\xi}_s$ in $s$, which in turn yields the differentiability of $V$ in time and the associated HJB equation. To this end, we study the map $U$ in \eqref{def:U(introduction)} as a decoupling field for the forward-backward state processes (when the flow property holds), and characterize $(Y^{t,\xi}_s,Z^{t,\xi}_s)$ as functions of $X^{t,\xi}_s$ when the diffusion is control-independent. Under $C^3$ regularity assumptions on the cost coefficients, we prove that the pair
\begin{align*}
    \Big(V(t,\mu),\  U(t,x,\mu)\Big),\qquad t\in[0,T],\ x\in\brn,\ \mu\in\pr_2(\brn)
\end{align*}
forms a classical solution to a coupled system of two nonlinear second-order backward PDEs on the measure space (see \eqref{discuss:Bell_main}). This system is analytically challenging, but our probabilistic approach yields its classical solvability without requiring non-degeneracy of $\left(\sigma\sigma^\top\right)(t,x,\mu)$. To the best of our knowledge, the classical solvability results for the HJB equations in all three cases above are new to the literature.

The rest of the paper is organized as follows. In Section~\ref{sec:notation}, we introduce the notations, problem formulation and assumptions. In Section~\ref{sec:MFTC}, we introduce the adjoint FBSDEs \eqref{adjoint^v(Introduction)}, prove the maximum principle, and give the well-posedness of the resulting coupled forward-backward system. In Section~\ref{sec:value}, we give the regularity of the value functional $V$ defined in \eqref{problem}. In Section~\ref{sec:HJB:1st}, we establish the classical solution to the HJB equation \eqref{HJB(Introduction)} corresponding to the first-order problem. In Section~\ref{sec:feedback}, we study the feedback case of the second-order problem, and establish the classical solution to the second-order HJB equation. In Section~\ref{sec:flow}, we provide conditions to ensure the flow property, and eventually establish the classical solution of the coupled PDE system for $(V,U)$. Some statements in Sections~\ref{sec:MFTC}, \ref{sec:HJB:1st}, \ref{sec:feedback} and \ref{sec:flow} are proven in Appendices~\ref{App:sec:MFTC}, \ref{App:sec:HJB:1st}, \ref{App:sec:feedback} and \ref{App:sec:flow} respectively.

\section{Preliminaries}\label{sec:notation}

\subsection{Probability spaces, Wasserstein spaces and notations}

Let $(\Omega,\f,\{\f_t,0\le t\le T\},\mathbb{P})$ be a completed filtered probability space (with the filtration being  augmented by all the $\mathbb{P}$-null sets) on which an $n$-dimensional Brownian motion $\{B_t,\ 0\le t\le T\}$ is defined. For any $\xi\in L^2(\Omega,\f,\mathbb{P};\brn)$, we denote by $\lr(\xi)$ its law (we also use the notation $\xi\sim\mu$); and denote by $\|\xi\|_2$ its $L^2$-norm. For every $t\in[0,T]$, we denote by $L^2_{\f_t}$ the set of all $\f_t$-measurable square-integrable $\brn$-valued random vectors. For $0\le t\le T$, we denote by $\lr^2_{\f}(t,T)$ the family of all $\f$-progressively measurable processes $\alpha_\cdot=\{\alpha_s,\ t\le s\le T\}$ such that $\|\alpha\|_{\lr^2_\f}^2:=\e\left[\int_t^T |\alpha_s|^2ds\right]<\infty$; and we denote by $\sr^2_{\f}(t,T)$ the family of all $\f$-adapted c\`adl\`ag processes $\beta_\cdot=\{\beta_s,\ t\le s\le T\}$ such that $\|\beta\|_{\sr^2_{\f}}^2:=\e\left[\sup_{t\le s\le T} |\beta_s|^2\right]<\infty$. We refer to \cite{MR1696772} for more details for these spaces. 

We denote by $\pr_{2}(\brn)$ the space of all probability measures on $\brn$ with finite second-order moments, equipped with the following Wasserstein metric: 
\begin{align*}
    W_2\left(m,m'\right):=\inf_{\pi\in\Pi\left(m,m'\right)}\left(\int_{\brn\times\brn}\left|x-x'\right|^2\pi\left(dx,dx'\right)\right)^{\frac{1}{2}},
\end{align*}
where $\Pi\left(m,m'\right)$ contains all joint probability measures with respective marginals $m$ and $m'$. For $m\in\pr_2(\brn)$, we denote by $\overline{m}:=\int_\brn xm(dx)$ and $|m|_2:=\sqrt{\int_\brn |x|^2m(dx)}$. See \cite{AL} for more results about Wasserstein spaces. For a map $\phi\in L_m^2(\brn;\brn)$, we denote by $\phi\#m$ the push-forward of $m$ by the map $\phi$. For a test function $f$, we have
\begin{align*}
    \int_\brn f(x') (\phi\#m) (dx') = \int_\brn f(\phi(x')) m (dx').
\end{align*}
Specifically, for a random variable $\xi\sim m$, we know that $\lr(\phi(\xi))=\phi\# m$, and that $\lr(\xi,\phi(\xi))=(I_n\otimes\phi)\# m$, where $I_n$
is the $n$-order identity matrix. See \cite{bogachev2007measure} for detailed discussions on the push-forward measures. 

In this article, we use the concept of the linear functional-derivative. Consider a functional $k:\pr_{2}(\brn)\to\br$, its linear functional-derivative at $m$ can be viewed as another continuous functional $\pr_{2}(\brn)\times \brn\ni(m,x')\mapsto\dfrac{dk}{d\nu}(m)(x')$ satisfying  $\Big|\dfrac{dk}{d\nu}(m)(x')\Big|\leq c(m)(1+|x'|^2)$ for some positive constant $c(m)$, which is bounded on any bounded subsets of $\pr_2(\brn)$ such that 
\begin{equation*}
    \lim_{\epsilon\to0}\dfrac{k((1-\epsilon)m+\epsilon m')-k(m)}{\epsilon}=\int_\brn\dfrac{dk}{d\nu}(m)(x')\left(m'(dx')-m(dx')\right), \quad \forall m'\in\pr_{2}(\brn).
\end{equation*}
See \cite{AB,book_mfg} for more details about the linear functional-derivatives. Particularly, the linear functional-derivatives are connected with the G\^ateaux derivatives in $L^2(\Omega,\f,\mathbb{P};\brn)$ in the following manner. For a functional $k:\pr_2(\brn)\to\br$ such that the derivative $D_{x'}\frac{d k}{d\nu}(m)(x')$ is continuous in $(m,x')$ and satisfies 
\begin{align*}
    D_{x'}\frac{d k}{d\nu}(m)(x')\le c(m)(1+|x'|),\quad (m,x')\in\pr_2(\brn)\times\brn, 
\end{align*}
then, the functional $K(\xi):=k(\lr(\xi)), \  \xi\in L^2(\Omega,\f,\mathbb{P};\brn)$ has the following G\^ateaux derivative:
\begin{align}\label{lem01_1}
	D_\xi K(\xi)(\omega)=D_{x'}\frac{d k}{d\nu}(\lr(\xi))(\xi(\omega)).
\end{align}
We also refer to \cite{book_mfg} for the relations of the linear functional-derivatives and the Lions derivatives (or $L$-derivative, see Lions \cite{CP} and Cardaliaguet \cite{cardaliaguet2015weak}).

We also introduce the following notations in this article. For a distribution $\varrho\in \pr_2(\brn\times\brn\times\brnn)$, we denote by its marginals $\varrho_x$, $\varrho_y$ and $\varrho_{z^i}$ as 
\begin{equation}\label{marginals}
\begin{aligned}
    &\varrho_x \left(dx'\right) := \int_{y'\in\brn,z'\in\brnn} d\varrho\left(x',y',z'\right),\quad \varrho_y\left(dy'\right) := \int_{x'\in\brn,z'\in\brnn}  d\varrho\left(x',y',z'\right),\\
    &\varrho_{z^i} \left(d {z'}^i\right) := \int_{x',y',{z'}^j\in\brn,j\neq i} d\varrho\left(x',y',z'\right),\quad 1\le i\le n. 
\end{aligned}
\end{equation}
Here, the subscripts $x$ in $\varrho_x$ (and $y$ in $\varrho_y$ and $z^i$ in $\varrho_{z^i}$) are only notations, not variables. For a linear functional differentiable functional $k:\pr_2(\brn\times\brn\times\brnn)\to\br$, we denote its linear functional-derivative at $\rho$ as 
\begin{align*}
    \frac{dk}{d\nu}\left(\rho\right)(x',y',z'),\quad x',y',{z'}^j\in\brn,\quad 1\le j\le n, 
\end{align*}
where ${z'}^j$ is the $j$-th column vector of $z'$; and the derivatives $D_{x'}\frac{dk}{d\nu},\ D_{y'}\frac{dk}{d\nu},\ D_{{z'}^j}\frac{dk}{d\nu}$ can be defined accordingly. Here and in the rest of this article, we write $D_{x'}\frac{dk}{d\nu}$ rather than $D_{x}\frac{dk}{d\nu}$ to emphasize that the differentiation is performed with respect to the variable that originates from the derivative in distribution. For a random variable $\zeta\in L^2(\Omega,\f,\mathbb{P};\brn\times\brn\times\brnn)$, we denote it by $\zeta=\left(\zeta_x,\zeta_y,\zeta_z\right)$, with its components $\zeta_x,\zeta_y,\zeta_{z^j}\in L^2(\Omega,\f,\mathbb{P};\brn)$ for $1\le j\le n$. The subscripts $x$, $y$ and $z^j$ are only notations, not variables. 

From here onward, for any random variable $\xi$, we use $\widetilde{\xi}$ to denote an independent copy of $\xi$, and $\widetilde{\e}\left[\widetilde{\xi}\right]$ for the corresponding expectation; we also denote by $\widehat{\xi}$ another independent copy, and by $\he$ the expectation with respect to it. For the sake of convenience, we use the notation $f|_{a}^b:=f(b)-f(a)$ for the difference of a  function $f$ between two points $b$ and $a$. For any symmetric matrix $Q$ and vector $x$, we denote by $Qx^{\otimes 2}:=x^\top Qx$. In all the following statements and proofs, we mean by $C(\alpha_1,\dots, \alpha_k)$  a constant depending only on parameters $(\alpha_1,\dots, \alpha_k)$, which might be different in different lines.

\subsection{Problem formulation}

Let coefficients for the McKean-Vlasov SDE in \eqref{FBSDE^v(Introduction)} 
\begin{align*}
    &b:[0,T]\times\brn\times\pr_2(\brn)\times\brd\to\brn,\qquad \sigma:[0,T]\times\brn\times\pr_2(\brn)\times\brd\to\brnn;
\end{align*}
and the coefficients for the BSDE in \eqref{FBSDE^v(Introduction)} 
\begin{align*}
    &f:[0,T]\times\brn\times\brn\times\brnn\times\pr_2( \brn\times\brn\times\brnn)\times\brd\to\brn,\qquad g:\brn\times\pr_2(\brn)\to\brn;
\end{align*}
and the coefficients for the cost functional defined in \eqref{cost^v}
\begin{align*}
    &h:[0,T]\times\brn\times\brn\times\brnn\times\pr_2( \brn\times\brn\times\brnn)\times\brd\to\br,\\
    &\Phi:\brn\times\pr_2(\brn)\to\br,\qquad \varphi:\brn\to\br.
\end{align*}
The assumptions on these coefficients are as follows. 

{\bf (A1)} The coefficients $b$, $\sigma$, $f$ and $g$ are linear in the following way: for $s\in[0,T]$, $v\in\brd$, $x,y\in\brn$, $z\in\brnn$, $m\in\pr_2(\brn)$ and $\varrho\in\pr_2(\brn\times\brn\times\brnn)$ (with its marginals $\varrho_x$, $\varrho_y$ and $\varrho_z$),
\begin{align*}
    b(s,x,m,v)=\ & b_0(s)+b_1(s)x+b_2(s)\overline{m}+b_3(s)v,\\
    \sigma^j(s,x,m,v)=\ & \sigma^j_0(s)+\sigma^j_1(s)x+\sigma^j_2(s)\overline{m}+\sigma^j_3(s)v,\quad 1\le j\le n,\\
    f(s,x,y,z,{ \varrho},v)=\ & f_0(s)+f_1(s)x+f_2(s)y+\sum_{j=1}^n f^j_3(s)z^j{ +f_4(s)\overline{\varrho_x}+ f_5(s)\overline{\varrho_y}+\sum_{j=1}^n f_6^j(s)\overline{\varrho_{z^j}}}+f_7(s)v,\\
    g(x,m)=\ & g_1x+g_2 \overline{m}.
\end{align*}
The functions $b_0,\sigma^j_0,f_0:[0,T]\to\brn$ and $b_1,b_2,\sigma^j_1,\sigma^j_2,f_1,f_2,f^j_3,f_4,f_5,f^j_6:[0,T]\to\brnn$ and  $b_3,\sigma^j_3,f_7:[0,T]\to\br^{n\times d}$ are continuous and are all bounded in norms by $L$; and the matrices $g_1,g_2\in\brnn$ are bounded in norms by $L$.

\textbf{(A2)} The function $h$ is continuous and has a quadratic growth in { $(x,y,z,\varrho,v)$}, and is differentiable in $(x,y,z,v)$ and also functional differentiable in { $\varrho\in\pr_2(\brn\times\brn\times\brnn)$}, with the derivatives satisfying: 
\begin{align*}
    &\left|(h_x,h_v,h_y,h_z)(s,x,y,z,{ \varrho},v)\right|\le L\left(1+|x|+|y|+|z|+{ |\varrho|_2}+|v|\right),\\
    & \left|\left(D_{x'}\frac{dh}{d\nu},\ D_{y'}\frac{dh}{d\nu},\ D_{{z'}^j}\frac{dh}{d\nu}\right)(s,x,y,z,{\varrho},v)(x',y',z')\right|\\
    &\qquad \le L\left(1+|x|+|y|+|z|+{ |\varrho|_2}+|v|{ +|x'|+|y'|+|z'|}\right), \quad 1\le j\le n,
\end{align*}
and the derivatives are $L$-Lipschitz continuous in $(x,y,z,{ \varrho},v,{ x',y',z'})$. The function $\varphi$ has a quadratic growth and is differentiable in $y$; the function $\Phi$ has a quadratic growth in $(x,m)$ and is differentiable in $x$ and also functional differentiable in $m\in\pr_2(\brn)$; and the derivatives satisfy: 
\begin{align*}
    &|\Phi_x(x,m)|\le L\left(1+|x|+|m|_2\right),\qquad \left|D_{x'}\frac{d\Phi}{d\nu}(x,m)(x')\right|\le L\left(1+|x|+|m|_2+|x'|\right),\\
    &|\varphi_y(y)|\le L(1+|y|)
\end{align*}
and the derivatives are $L$-Lipschitz continuous in $(x,m,v,y,z,x')$. 

{\bf (A3)} The running cost function $h$ is convex in the following way: there exists $\lambda>0$, such that for any $s\in[0,T]$, $x^1,x^2,y^1,y^2\in\brn$, $v^1,v^2\in\brd$, $z^1,z^2\in\brnn$ and $ \zeta^1,\zeta^2\in L^2(\Omega,\f,\mathbb{P};\brn\times\brn\times\brnn)$, 
\begin{align*}
    &h\left(s,x^2,y^2,z^2,{ \lr\left(\zeta^2\right)},v^2\right)-h\left(s,x^1,y^1,z^1,{ \lr\left(\zeta^1\right)},v^1\right)\\
    \geq\ & h_x\left(s,x^1,y^1,z^1,{ \lr\left(\zeta^1\right)},v^1\right)^\top \left(x^2-x^1\right) +h_y\left(s,x^1,y^1,z^1,{ \lr\left(\zeta^1\right)},v^1\right)^\top \left(y^2-y^1\right) \\
    &+\sum_{j=1}^n h_{z^j}\left(s,x^1,y^1,z^1,{ \lr\left(\zeta^1\right)},v^1\right)^\top \left(z^{2,j}-z^{1,j}\right) +h_v\left(s,x^1,y^1,z^1,{ \lr\left(\zeta^1\right)},v^1\right)^\top \left(v^2-v^1\right)\\
    &+\e\bigg[\left(D_{x'}\frac{dh}{d\nu}\left(s,x^1,y^1,z^1,{ \lr\left(\zeta^1\right)},v^1\right)\left({ \zeta^1}\right)\right)^\top \left({ \zeta_x^2-\zeta_x^1 }\right) \\
    &\quad\qquad + \left(D_{y'}\frac{dh}{d\nu}\left(s,x^1,y^1,z^1,{ \lr\left(\zeta^1\right)},v^1\right)\left({ \zeta^1}\right)\right)^\top \left({ \zeta_y^2-\zeta_y^1 }\right)\\
    &\quad\qquad + \sum_{j=1}^n \left(D_{{z'}^j}\frac{dh}{d\nu}\left(s,x^1,y^1,z^1,{ \lr\left(\zeta^1\right)},v^1\right)\left({ \zeta^1}\right)\right)^\top \left({ \zeta_{z^j}^2-\zeta_{z^j}^1 }\right) \bigg] +\lambda\left|v^2-v^1\right|^2;
\end{align*}
and the function $\frac{dh}{d\nu}(s,x,y,z,{ \varrho},v)(\cdot):{ \brn\times\brn\times\brnn}\to\br$ is convex. The terminal cost function $\Phi$ is convex in $(x,m)$ in a similar way, and the function $\varphi$ is convex in $y$. 

We define the Lagrangian $L:[0,T]\times\brn\times\brn\times\brnn\times\pr_2({ \brn\times\brn\times\brnn})\times\brd\times\brn\times\brnn\times\brn\to\br$ as 
\begin{equation}\label{Lagrangian}
\begin{aligned}
    L(s,x,y,z,{ \varrho},v,p,q,r):= \ & b\left(s,x,{ \varrho_x},v\right)^\top p + \sum_{j=1}^n\sigma^j\left(s,x,{ \varrho_x},v\right)^\top q^j  \\
    & -f(s,x,y,z,{ \varrho},v)^\top r + h(s,x,y,z,{ \varrho},v), 
\end{aligned}
\end{equation}
and the minimizer in $v$ of $L$ as 
\begin{align}\label{minimizer}
    \hv(s,x,y,z,{ \varrho},p,q,r):=\argmin_{v\in\brd} L(s,x,y,z,{ \varrho},v,p,q,r),
\end{align}
and then define the Hamiltonian $H$ as
\begin{align}\label{def:Ham}
    H(s,x,y,z,{ \varrho},p,q,r):=\inf_{v\in\brd} L(s,x,y,z,{ \varrho},v,p,q,r)=L\left(s,x,y,z,{ \varrho},\hv(s,x,y,z,{ \varrho},p,q,r),p,q,r\right).
\end{align}
From the convexity assumption on $h$ in (A3), we know that \eqref{minimizer} is well-defined. In the rest of this article, for notational convenience, we simply denote by 
\begin{align*}
    \pi:=(x,y,z,{ \varrho},p,q,r)\in \brn\times\brn\times\brnn\times\pr_2(\brn\times\brn\times\brnn)\times\brn\times\brnn\times\brn. 
\end{align*}
By simple computation, we know that 
\begin{align*}
    %&H_x(s,\pi)=b_1(s)^\top p + \sum_{j=1}^n \sigma^j_1(s)^\top q^j-f_1(s)^\top r +h_x\left(s,x,y,z,{ \varrho},\hv(s,\pi)\right),\\
    %&H_y(s,\pi)=-f_2(s)^\top r +h_y\left(s,x,y,z,{ \varrho},\hv(s,\pi)\right),\quad H_{z^j}(s,\pi)=-f^j_3(s)^\top r +h_{z^j}\left(s,x,y,z,{ \varrho},\hv(s,\pi)\right),\\
    &H_p(s,\pi)=b\left(s,x,{ \varrho_x},\hv(s,\pi)\right),\  H_{q^j}(s,\pi)=\sigma^j\left(s,x,{ \varrho_x},\hv(s,\pi)\right),\  H_r(s,\pi)=-f\left(s,x,y,z,{ \varrho},\hv(s,\pi)\right),\\
    &D_{x'}\frac{dH}{d\nu}(s,\pi)({ x',y',z'})=b_2(s)^\top p + \sum_{j=1}^n \sigma^j_2(s)^\top q^j-f_4(s)^\top r +D_{x'}\frac{dh}{d\nu}\left(s,x,y,z,{ \varrho},\hv(s,\pi)\right)(x',y',z'),\\
    & D_{y'}\frac{dH}{d\nu}(s,\pi)({ x',y',z'})=-f_5(s)^\top r +D_{y'}\frac{dh}{d\nu}\left(s,x,y,z,{ \varrho},\hv(s,\pi)\right)(x',y',z'),\\
    & D_{{z'}^j}\frac{dH}{d\nu}(s,\pi)({ x',y',z'})=-f^j_6(s)^\top r +D_{{z'}^j}\frac{dh}{d\nu}\left(s,x,y,z,{ \varrho},\hv(s,\pi)\right)(x',y',z').
\end{align*}

For any initial time $t\in[0,T]$ and initial condition $\xi\in \lr_{\f_t}^2$ with its law $\lr(\xi)=\mu\in\pr_2(\brn)$, we denote by $\left(X^{t,\xi,v}, Y^{t,\xi,v}, Z^{t,\xi,v}\right)$ the state processes corresponding to a control $v\in \lr_\f^2(t,T)$, which satisfy the distribution-dependent FBSDEs \eqref{FBSDE^v(Introduction)}. Here, we use the superscript $t,\xi,v$ to emphasize their dependence on the initial condition and the control. Before studying the optimal control problem, we first give the following estimates for the controlled state processes. From \cite{MR1696772}, the forward process $X^{t,\xi,v}$ of \eqref{FBSDE^v(Introduction)} is well-defined under the conditions on the coefficients $b$ and $\sigma$ in Assumption (A1), and it is easy to check that it satisfies 
\begin{align}\label{SDE:bound}
    \left\|X^{t,\xi,v}\right\|_{\sr_\f^2}^2\le C(L,T) \left(1+\|\xi\|_2^2+ \|v\|_{\lr_\f^2}^2 \right).
\end{align}
Then, given $v_s$ and $X^{t,\xi,v}_s$ for $s\in[t,T]$, the backward processes $\left(Y^{t,\xi,v},Z^{t,\xi,v}\right)$ are well-defined under the condition on the coefficients $f$ and $g$ in Assumption (A2), and by using \eqref{SDE:bound}, we have 
\begin{align}
    \left\|Y^{t,\xi,v}\right\|_{\sr_\f^2}^2 + \left\|Z^{t,\xi,v}\right\|_{\lr_\f^2}^2 \le\ & C(L,T) \e\left[1+|\xi|^2 +\left|X^{t,\xi,v}_T\right|^2 +\int_t^T \left(|v_s|^2 + \left|X^{t,\xi,v}_s\right|^2\right)ds \right] \notag\\
    \le\ & C(L,T) \left(1+\|\xi\|_2^2+ \|v\|_{\lr_\f^2}^2 \right). \label{BSDE:bound}
\end{align}
In a similar way, we also have the following continuity of the processes $\left(X^{t,\xi,v},Y^{t,\xi,v},Z^{t,\xi,v}\right)$ in $v$:
\begin{align}
    &\left\|X^{t,\xi,v'}-X^{t,\xi,v}\right\|_{\sr_\f^2}^2+\left\|Y^{t,\xi,v'}-Y^{t,\xi,v}\right\|_{\sr_\f^2}^2 + \left\|Z^{t,\xi,v'}-Z^{t,\xi,v}\right\|_{\lr_\f^2}^2 \le C(L,T) \left\|v'-v\right\|_{\lr_\f^2}^2. \label{FBSDE:conti_v}
\end{align}
For standard estimates on SDEs and BSDEs, see \cite{MR1696772}; for the regularity of mean-field SDEs, see \cite{BR}; and for the study on mean-field BSDEs and related PDEs, see \cite{BUCKDAHN20093133}. We refer the readers to \cite{MR4124528,HZ,TSE2021134} for the study on the Fokker–Planck equations corresponding to McKean-Vlasov SDEs. 

\section{Optimal controls for mean-field FBSDEs}\label{sec:MFTC}

In this section, we first introduce the adjoint FBSDEs corresponding to Problem $\left(\mathsf{P}_{t,\xi}\right)$, and then give the optimality condition for the optimal control in Subsection~\ref{subsec:maxi}, which derives the associated coupled system of FBSDEs for the control problem. Then, we prove the well-posedness of this system in Subsection~\ref{subsec:FBSDE}, which together with the sufficient maximum principle gives the unique optimal control for Problem $\left(\mathsf{P}_{t,\xi}\right)$. For notational convenience, in this section, we temporarily omit the superscript $t,\xi$ in the processes. 

\subsection{Maximum principle and adjoint FBSDEs}\label{subsec:maxi}

Under Assumptions (A1) and (A2), suppose that $u\in\lr^2_\f(t,T)$ is an optimal control for $\left(\mathsf{P}_{t,\xi}\right)$. We first introduce the following adjoint distribution-dependent FBSDEs: 
\begin{align}\label{adjoint^v(Introduction)}
    \left\{
    \begin{aligned}
        P^u_s=\ & - g_1^\top  R^u_T - g_2^\top \e\left[R^u_T\right] +\Phi_x\left(X^u_T, \lr\left(X^u_T\right) \right) +\widetilde{\e}\left[D_{x'}\frac{d\Phi}{d\nu} \left(\widetilde{X^u_T}, \lr\left(X^u_T\right)\right)\left(X^u_T\right) \right] \\
        &+ \int_s^T \bigg\{L_x\left(r,\pi^u_r, u_r\right)  + \widetilde{\e} \left[D_{x'}\frac{dL}{d\nu} \left(r,\widetilde{\pi^u_r}, \widetilde{u_r} \right)\left(\theta^u_r\right) \right] \bigg\} dr -\int_s^T Q^u_r dB_r, \\
        R^u_s=\ & -\varphi_y \left(Y^u_t\right) -\int_t^s \bigg\{ L_y \left(r,\pi^u_r, u_r \right) { + \widetilde{\e} \left[D_{y'}\frac{dL}{d\nu} \left(r,\widetilde{\pi^u_r}, \widetilde{u_r} \right)\left(\theta^u_r\right) \right]} \bigg\} dr \\
        &-\sum_{j=1}^n \int_t^s \bigg\{ L_{z^j}\left(r,\pi^u_r, u_r \right) { + \widetilde{\e} \left[D_{{z'}^j}\frac{dL}{d\nu} \left(r,\widetilde{\pi^u_r}, \widetilde{u_r} \right)\left(\theta^u_r\right) \right]} \bigg\} dB^j_r,
    \end{aligned}
    \right.
\end{align}
where $\theta^u_r:=\left(X^u_r, Y^u_r,Z^u_r\right)$ and $\pi^u_r:= \left(\theta^u_r, \lr\left(\theta^u_r\right),P^u_r,Q^u_r,R^u_r\right)$. Given the control $u$ together with the state processes $\theta^{u}$, System \eqref{adjoint^v(Introduction)} is a system of  distribution-dependent FBSDEs for $\left(P^{u},Q^{u},R^{u}\right)$. We have the following well-posedness result, whose proof is given in Appendix~\ref{App:lem:wellposedness_adjoint}; and we also refer the readers to \cite[Equation (10)]{MR4587511} and \cite[Equation (5.8)]{MR4199898} for the adjoint mean-field FBSDEs under different settings. 

\begin{lemma}\label{lem:wellposedness_adjoint}
    Under Assumptions (A1) and (A2), for the control $u$, the adjoint system \eqref{adjoint^v(Introduction)} has a unique solution $\left(P^{u},Q^{u},R^{u}\right)\in\left(\sr_\f^2\times\lr_\f^2\times\sr_\f^2\right)(t,T)$, and it satisfies 
    \begin{equation}
        \left\|P^{u}\right\|_{\sr^2_\f}^2+\left\|Q^{u}\right\|_{\lr^2_\f}^2+\left\|R^{u}\right\|_{\sr^2_\f}^2\le C(L,T) \left(1+\|\xi\|_2^2+ \|u\|_{\lr^2_\f}^2 \right). \label{lem:wellposedness_adjoint_0}
    \end{equation}
\end{lemma}

We now give the necessary part of the maximum principle for Problem $\left(\mathsf{P}_{t,\xi}\right)$, using the adjoint FBSDE system.

\begin{theorem}
    Under Assumptions (A1) and (A2), suppose that $u\in\lr_\f^2(t,T)$ is an optimal control for Problem $\left(\mathsf{P}_{t,\xi}\right)$. Let $\theta^u$ be the associated state processes and $\left(P^u,Q^u,R^u\right)$ be the corresponding adjoint processes. Then, the following condition holds:
    \begin{equation}\label{optimal_condition}
        L_v\left(s,\theta^u_s,\lr(\theta^u_s),u_s,P^u_s,Q^u_s,R^u_s\right)=0,\quad  s\in[t,T],\ a.e., \quad \mathbb{P}-a.s..
    \end{equation}
\end{theorem}

\begin{proof}
For any other control $v\in \lr_\f^2(t,T)$ for Problem $\left(\mathsf{P}_{t,\xi}\right)$, we denote by $\de v:=v-u\in\lr_\f^2(t,T)$. For any $\epsilon\in(0,1)$, we introduce $v^\epsilon:=u+\epsilon\de v$ which is also an admissible control. We define by the process $\de X^\epsilon:=\frac{1}{\epsilon}\left(X^{v^\epsilon}-X^{u}\right)$ (and also $Y^\epsilon$ and $Z^\epsilon$ in a similar way), and denote by $\de \theta^\epsilon:=\left(\de X^\epsilon,\de Y^\epsilon,\de Z^\epsilon \right)$. From Assumption (A1), we know that the processes $\de \theta^\epsilon$ satisfy a linear system of FBSDEs. By using the regularity conditions on $h$ in Assumption (A2), we can compute that 
\begin{align}
    &\frac{d}{d\epsilon}J^{t,\xi}(u+\epsilon\de v)\bigg|_{\epsilon=0} \notag \\
    =\ & \e\bigg\{\int_t^T \bigg[ h_x\left(s,\theta^u_s,{ \lr\left(\theta^u_s\right)},u_s\right)^\top \de X^\epsilon_s+h_y\left(s,\theta^u_s,{ \lr\left(\theta^u_s\right)},u_s\right)^\top \de Y^\epsilon_s \notag \\
    &\qquad\qquad +\sum_{j=1}^n h_{z^j}\left(s,\theta^u_s,{ \lr\left(\theta^u_s\right)}, u_s\right)^\top \de Z^{\epsilon,j}_s +h_v\left(s,\theta^u_s,{ \lr\left(\theta^u_s\right)},u_s\right)^\top \de v_s \notag \\
    &\qquad\qquad +\widetilde{\e}\bigg[\left(D_{x'} \frac{dh}{d\nu}\left(s,\theta^u_s,{ \lr\left(\theta^u_s\right)},u_s\right)\left({ \widetilde{\theta^u_s}}\right)\right)^\top \widetilde{\de X^\epsilon_s}  \bigg] \notag \\
    &\qquad\qquad +\widetilde{\e}\bigg[\left(D_{y'} \frac{dh}{d\nu}\left(s,\theta^u_s,\lr\left(\theta^u_s\right),u_s\right)\left(\widetilde{\theta^u_s}\right)\right)^\top \widetilde{\de Y^\epsilon_s}  \bigg] \notag\\
    &\qquad\qquad +\sum_{j=1}^n \widetilde{\e}\bigg[\left(D_{{z'}^j} \frac{dh}{d\nu}\left(s,\theta^u_s, \lr\left(\theta^u_s\right),u_s\right)\left( \widetilde{\theta^u_s}\right)\right)^\top \widetilde{\de Z^{\epsilon,j}_s} \bigg]  \bigg] ds \notag \\
    &\qquad +\Phi_x \left(X^u_T,\lr\left(X^u_T\right)\right)^\top \de X^\epsilon_T +\widetilde{\e}\bigg[\left(D_{x'} \frac{d\Phi}{d\nu}\left(X^u_T,\lr\left(X^u_T\right)\right)\left(\widetilde{X^u_T}\right)\right)^\top \widetilde{\de X^\epsilon_T} \bigg]  + \varphi_y \left(Y^u_t\right)^\top \de Y^\epsilon_t  \bigg\}. \label{lem:dJ_1}
\end{align}
By applying It\^o's formula for $\left(P^u_s\right)^\top \de X^\epsilon_s$, taking expectation and using Fubini's theorem, we have
\begin{align}
    &\e\left[\left(P^u_T\right)^\top \de X^\epsilon_T \right] \notag \\
    =\ & \e \int_t^T \bigg[ b_3(s)^\top P^u_s+ \sum_{j=1}^n\sigma^j_3(s)^\top Q^{u,j}_s \bigg]^\top \de v_s \ ds \notag \\
    & + \e\int_t^T \bigg\{ f_1(s)^\top R^u_s + f_4(s)^\top\e\left[ R^u_s \right] - h_x\left(s,\theta^u_s,{ \lr\left(\theta^u_s\right)}, u_s\right)  \notag \\
    &\qquad\qquad - \widetilde{\e} \left[D_{x'}\frac{dh}{d\nu} \left(s,\widetilde{\theta^u_s}, { \lr\left(\theta^u_s\right)}, \widetilde{u_s} \right)\left({ \theta^u_s}\right) \right] \bigg\}^\top \de X^\epsilon_s \ ds  \notag \\
    =\ & \e \int_t^T \bigg\{ \bigg[ b_3(s)^\top P^u_s+ \sum_{j=1}^n\sigma^j_3(s)^\top Q^{u,j}_s \bigg]^\top \de v_s  +\left(R^u_s\right)^\top f_1(s) \de X^\epsilon_s  \notag \\
    &\quad\qquad + \left(R^u_s \right)^\top f_4(s) \e\left[\de X^\epsilon_s\right] - h_x\left(s,\theta^u_s,{ \lr\left(\theta^u_s\right)}, u_s\right)^\top \de X^\epsilon_s   \notag \\
    &\quad\qquad - \widetilde{\e} \bigg[\left(D_{x'}\frac{dh}{d\nu} \left(s,\theta^u_s,{ \lr\left(\theta^u_s\right)}, u_s \right)\left({ \widetilde{\theta^u_s}}\right)\right)^\top  \widetilde{\de X^\epsilon_s}\bigg] \bigg\}   ds. \label{lem:dJ_2-1}
\end{align}
Similarly, by applying It\^o's formula for $\left(\de Y^\epsilon_s\right)^\top R^u_s$ and taking expectation and using Fubini's theorem, we have 
\begin{align}
    &\e\left[\left(\de Y^\epsilon_T\right)^\top  R^u_T-\left(\de Y^\epsilon_t\right)^\top  R^u_t\right] \notag \\
    =\ & -\e\int_t^T \bigg\{ \left(R^u_s\right)^\top f_1(s) \de X^\epsilon_s + \left(R^u_s\right)^\top f_4(s)\e\left[\de X^\epsilon_s\right] + \left(R^u_s\right)^\top f_7(s) \de v_s \notag\\
    &\qquad\qquad + h_y \left(s,\theta^u_s,{ \lr\left(\theta^u_s\right)}, u_s \right)^\top  \de Y^\epsilon_s +\sum_{j=1}^n h_{z^j}\left(s,\theta^u_s,{ \lr\left(\theta^u_s\right)}, u_s \right)^\top \de Z^{\epsilon,j}_s \notag \\
    & \qquad\qquad + \widetilde{\e} \bigg[\left(D_{y'}\frac{dh}{d\nu} \left(s,\theta^u_s, \lr\left(\theta^u_s\right), u_s \right)\left( \widetilde{\theta^u_s}\right)\right)^\top  \widetilde{\de Y^\epsilon_s}\bigg] \notag\\
    & \qquad\qquad +\sum_{j=1}^n \widetilde{\e} \bigg[\left(D_{{z'}^j}\frac{dh}{d\nu} \left(s,\theta^u_s, \lr\left(\theta^u_s\right), u_s \right)\left(\widetilde{\theta^u_s}\right)\right)^\top  \widetilde{\de Z^{\epsilon,j}_s}\bigg] \bigg\} ds. \label{lem:dJ_2-2}
\end{align}
By using the terminal conditions for $P^u_T$ and $\de Y^\epsilon_T$, and the initial condition for $R^u_t$ as well as Fubini's theorem (similar as in \eqref{lem:dJ_2-1}), we have
\begin{align}
    &\e\left[\left(P^u_T\right)^\top \de X^\epsilon_T + \left(\de Y^\epsilon_T\right)^\top R^u_T-\left(\de Y^\epsilon_t\right)^\top R^u_t \right] \notag \\
    =\ & \e\bigg\{\Phi_x\left(X^u_T, \lr\left(X^u_T\right) \right)^\top \de X^\epsilon_T +\widetilde{\e}\left[D_{x'}\frac{d\Phi}{d\nu} \left({X^u_T}, \lr\left(X^u_T\right)\right)\left(\widetilde{X^u_T}\right)^\top \widetilde{\de X^\epsilon_T} \right]  + \varphi_y \left(Y^u_t\right)^\top \de Y^\epsilon_t \bigg\}. \notag
\end{align}
By substituting \eqref{lem:dJ_2-1}, \eqref{lem:dJ_2-2} and the last equality into \eqref{lem:dJ_1}, we have  
\begin{align*}
        &\frac{d}{d\epsilon}J^{t,\xi}(u+\epsilon\de v)\bigg|_{\epsilon=0} \notag \\
        =\ & \e\int_t^T \left[ b_3(s)^\top P^u_s+ \sigma_3(s)^\top Q^u_s -f_7(s)^\top R^u_s + h_v\left(s,\theta^u_s,{ \lr\left(\theta^u_s\right)},u_s\right) \right]^\top \de v_s ds, 
\end{align*}
from which we obtain \eqref{optimal_condition}. 
\end{proof}

For the study of necessary maximum principles in the literature, we also refer the readers to \cite{MR4587511} for fully coupled mean-field controlled FBSDEs, to \cite{SHI2025106216} for mean-field forward-backward regime-switching systems with singular controls, and to \cite{MR3952459} for problems involving Teugels martingales and associated Lévy processes. The controlled FBSDEs \eqref{FBSDE^v(Introduction)} for $(X,Y,Z)$ associated with the optimal control $u$, the corresponding FBSDEs \eqref{adjoint^v(Introduction)} for $(P,Q,R)$, and the optimality condition \eqref{optimal_condition} altogether give the following coupled system:
\begin{equation}\label{FBSDE:MFTC}
    \left\{
        \begin{aligned}
            &X_s= \xi+\int_t^s b \left(r,X_r, \lr\left(X_r\right), u_r \right)dr + \int_t^s \sigma\left(r,X_r, \lr\left(X_r\right), u_r \right)dB_r,\\
            &Y_s= g\left(X_T, \lr\left(X_T\right) \right) + \int_s^T f\left(r,\theta_r,\lr\left({ \theta_r}\right), u_r \right)dr -\int_s^T Z_r dB_r,\\
            &P_s= - g_1^\top  R_T - g_2^\top \e\left[R_T\right]  +\Phi_x\left(X_T, \lr\left(X_T\right) \right) +\widetilde{\e}\left[D_{x'}\frac{d\Phi}{d\nu} \left(\widetilde{X_T}, \lr\left(X_T\right)\right)\left(X_T\right) \right] \\
            &\quad\qquad + \int_s^T \bigg\{L_x\left(r,\pi_r, u_r\right)  + \widetilde{\e} \left[D_{x'}\frac{dL}{d\nu} \left(r,\widetilde{\pi_r}, \widetilde{u_r} \right)\left({ \theta_r}\right) \right] \bigg\} dr -\int_s^T Q_r dB_r, \\
            &R_s= -\varphi_y \left(Y_t\right) -\int_t^s \bigg\{ L_y \left(r,\pi_r, u_r \right) { + \widetilde{\e} \left[D_{y'}\frac{dL}{d\nu} \left(r,\widetilde{\pi_r}, \widetilde{u_r} \right)\left(\theta_r\right) \right]} \bigg\} dr \\
            &\quad\qquad  -\sum_{j=1}^n \int_t^s \bigg\{ L_{z^j}\left(r,\pi_r, u_r \right) { + \widetilde{\e} \left[D_{{z'}^j}\frac{dL}{d\nu} \left(r,\widetilde{\pi_r}, \widetilde{u_r} \right)\left(\theta_r\right) \right]} \bigg\} dB^j_r,\\
            &b_3(s)^\top P_s+ \sum_{j=1}^n \sigma^j_3(s)^\top Q^j_s -f_7(s)^\top R_s + h_v\left(s,\theta_s,\lr\left({ \theta_s}\right),u_s\right) =0,\quad s\in[t,T],
        \end{aligned}
    \right.
\end{equation}
where $\theta_r:=\left(X_r,Y_r,Z_r\right)$ and $\pi_r:= \left(\theta_r, \lr\left({ \theta_r}\right), P_r,Q_r,R_r\right)$. This system generalizes the forward-backward systems arising from the maximum principle associated with MFTC problems; see \cite[Equation (6.5)]{AB10} for instance. We next give the sufficient condition for Problem $\left(\mathsf{P}_{t,\xi}\right)$ under the additional convexity assumption (A3).

\begin{theorem}\label{thm:suff}
    Under Assumptions (A1)-(A3), 
    if System \eqref{FBSDE:MFTC} has a unique solution 
    \begin{align*}
        (X,Y,Z,P,Q,R,u)\in \left(\sr_\f^2\times\sr_\f^2\times\lr_\f^2\times\sr_\f^2\times\lr_\f^2\times\sr_\f^2\times\lr_\f^2\right)(t,T),
    \end{align*}
    then, $u$ is the unique optimal control for Problem $\left(\mathsf{P}_{t,\xi}\right)$, and for any control $v\in\lr_\f^2(t,T)$, we have
    \begin{align}\label{thm:suff_0}
        J^{t,\xi}(v)-J^{t,\xi}(u) \geq \lambda \int_t^T \left\|v_s-u_s\right\|_2^2 ds.
    \end{align}
\end{theorem}

\begin{proof}
We denote by $\de X_s:=X^{v}_s-X_s$, and define $\de Y_s$ and $\de Z_s$ in a similar way. From \eqref{cost^v} and the convexity conditions in Assumption (A3), we know that 
\begin{align}
    &J^{t,\xi}(v)-J^{t,\xi}(u) \notag \\
    \geq\ & \e\bigg\{\int_t^T \bigg[ h_x\left(s,\theta_s,\lr\left({ \theta_s}\right),u_s\right)^\top \de X_s+h_y\left(s,\theta_s,\lr\left({ \theta_s}\right),u_s\right)^\top \de Y_s\notag\\
    &\quad\qquad +\sum_{j=1}^n h_{z^j}\left(s,\theta_s,\lr\left({ \theta_s}\right),u_s\right)^\top \de Z^j_s +h_v\left(s,\theta_s,\lr\left({ \theta_s}\right),u_s\right)^\top \de v_s \notag \\
    &\quad\qquad +\widetilde{\e}\bigg[\left(D_{x'} \frac{dh}{d\nu}\left(s,\theta_s,\lr\left({ \theta_s}\right),u_s\right)\left(\widetilde{ \theta_s}\right)\right)^\top \widetilde{\de X_s} { +\left(D_{y'} \frac{dh}{d\nu}\left(s,\theta_s,\lr\left(\theta_s\right),u_s\right)\left(\widetilde{\theta_s}\right)\right)^\top \widetilde{\de Y_s} } \notag \\
    &\quad\qquad\qquad { +\sum_{j=1}^n \left(D_{{z'}^j} \frac{dh}{d\nu}\left(s,\theta_s,\lr\left(\theta_s\right),u_s\right)\left(\widetilde{\theta_s}\right)\right)^\top \widetilde{\de Z^j_s} }\bigg] + \lambda |v_s-u_s|^2\bigg] ds \notag \\
    &\qquad +\Phi_x \left(X_T,\lr\left(X_T\right)\right)^\top \de X_T +\widetilde{\e}\bigg[\left(D_{x'} \frac{d\Phi}{d\nu}\left(X_T,\lr\left(X_T\right)\right)\left(\widetilde{X_T}\right)\right)^\top \widetilde{\de X_T} \bigg]  + \varphi_y \left(Y_t\right)^\top \de Y_t  \bigg\} \notag\\
    :=\ & \mathcal{K}(u,v) + \e \int_t^T \left(h_v\left(s,\theta_s,\lr\left({ \theta_s}\right),u_s\right)^\top \de v_s + \lambda |v_s-u_s|^2\right) ds. \label{thm:suff_1}
\end{align}
As in \eqref{lem:dJ_2-1}-\eqref{lem:dJ_2-2}, by applying It\^o's formula for $P_s^\top \de X^\epsilon_s$ and $\de Y_s^\top R_s$, taking expectation and using Fubini's theorem, and then by using the terminal conditions for $P_T$ and $\de Y_T$ and the initial condition for $R_t$, we have 
\begin{align}
    \mathcal{K}(u,v) =\ & \e \int_t^T \bigg[ b_3(s)^\top P_s+ \sum_{j=1}^n \sigma^j_3(s)^\top Q^j_s -f_7(s)^\top R_s \bigg]^\top \de v_s ds. \label{thm:suff_4}
\end{align}
From \eqref{thm:suff_1}, \eqref{thm:suff_4} and the optimality condition in \eqref{FBSDE:MFTC}, we obtain \eqref{thm:suff_0}. 
\end{proof}

In view of Theorem~\ref{thm:suff}, we know that when System \eqref{FBSDE:MFTC} has a solution $\pi_s$, the optimal control for the (open-loop) problem $\left(\mathsf{P}_{t,\xi}\right)$  exists and is of the feedback form: $u_s=\hv(s,\pi_s)$. Therefore, the next step is to study the well-posedness of the system \eqref{FBSDE:MFTC}.

\subsection{Well-posedness of the system}\label{subsec:FBSDE}

We next give the well-posedness of System \eqref{FBSDE:MFTC} and also the regularity of the solutions with respect to the initial condition $\xi$. In this section, we denote by $\mathbb{S}$ the space of processes 
\begin{align*}
    \Pi_s:=(\pi_s,u_s):= \left(\theta_s, \lr\left({ \theta_s}\right),P_s,Q_s,R_s,u_s\right):= \left(X_s,Y_s,Z_s, \lr\left({ X_s,Y_s,Z_s}\right),P_s,Q_s,R_s,u_s\right),\  s\in[t,T],
\end{align*}
with the norm
\begin{equation*}
    \|\Pi\|_{\mathbb{S}}^2:=\|X\|_{\sr_\f^2}^2+\|Y\|_{\sr_\f^2}^2+\|Z\|_{\lr_\f^2}^2 +\|P\|_{\sr_\f^2}^2+ \|Q\|_{\lr_\f^2}^2 +\|R\|_{\sr_\f^2}^2+\|u\|_{\lr_\f^2}^2   <\infty.
\end{equation*}
The strategy is to prove that the solvability of System \eqref{FBSDE:MFTC} is preserved when the coefficients are slightly perturbed. We denote by 
\begin{align*}
    \ir=\left(\ir^b,\ir^\sigma,\ir^f,\ir^g,\ir^x,\ir^y,\ir^z,\ir^\Phi,\ir^\varphi\right)
\end{align*}
an input for System \eqref{FBSDE:MFTC} (designed to be injected into the system); and denote by $\mathbb{I}$ the space of all inputs, endowed with the squared norm
\begin{equation*}
	\begin{split}
		\|\ir\|_{\mathbb{I}}^2:=\e\bigg[&\left|\ir^g\right|^2 + \left|\ir^\Phi\right|^2 + \left|\ir^\varphi\right|^2 +\int_t^T\left(\left|\ir^b_s\right|^2+\left|\ir^\sigma_s\right|^2 + \left|\ir^f_s\right|^2 + \left|\ir^x_s\right|^2 + \left|\ir^y_s\right|^2 + \left|\ir^z_s\right|^2\right)ds\bigg]<\infty.
	\end{split}
\end{equation*}
For any $(\gamma,\xi,\ir)\in [0,1]\times\lr_{\f_t}^2\times\mathbb{I}$, we denote by $\mathcal{E}(\gamma,\xi,\ir)$ the following forward-backward equations:
\begin{equation}\label{FBSDE:MFTC=gamma}
    \left\{
        \begin{aligned}
            &X_s= \xi+\int_t^s \left[\gamma b \left(r,X_r, \lr\left(X_r\right), u_r \right) + \ir^b_r\right] dr + \int_t^s \left[\gamma\sigma\left(r,X_r, \lr\left(X_r\right), u_r \right) + \ir^\sigma_r\right] dB_r,\\
            &Y_s= \gamma g\left(X_T, \lr\left(X_T\right) \right) +\ir^g + \int_s^T \left[\gamma f\left(r,\theta_r,\lr({ \theta_r}), u_r \right) + \ir^f_r\right]dr -\int_s^T Z_r dB_r,\\
            &P_s= \gamma\left\{- g_1^\top  R_T - g_2^\top \e\left[R_T\right]  +\Phi_x\left(X_T, \lr\left(X_T\right) \right) +\widetilde{\e}\left[D_{x'}\frac{d\Phi}{d\nu} \left(\widetilde{X_T}, \lr\left(X_T\right)\right)\left(X_T\right) \right] \right\} + \ir^\Phi \\
            &\quad\qquad + \int_s^T \bigg\{\gamma L_x\left(r,\Pi_r\right)  + \gamma \widetilde{\e} \left[D_{x'}\frac{dL}{d\nu} \left(r,\widetilde{\Pi_r} \right)\left({ \theta_r}\right) \right] + \ir^x_r \bigg\} dr -\int_s^T Q_r dB_r, \\
            &R_s= -\gamma \varphi_y \left(Y_t\right) - \ir^\varphi -\int_t^s \left\{\gamma L_y \left(r,\Pi_r\right) { + \gamma \widetilde{\e} \left[D_{y'}\frac{dL}{d\nu} \left(r,\widetilde{\Pi_r} \right)\left(\theta_r\right) \right]} +\ir^y_r\right\} dr \\
            &\quad\qquad -\sum_{j=1}^n \int_t^s \left\{ \gamma L_{z^j}\left(r,\Pi_r\right)  { + \gamma \widetilde{\e} \left[D_{{z'}^j}\frac{dL}{d\nu} \left(r,\widetilde{\Pi_r} \right)\left(\theta_r\right) \right]}+ \ir^{z^j}_r\right\} dB^j_r,\\
            &b_3(s)^\top P_s+ \sum_{j=1}^n \sigma^j_3(s)^\top Q^j_s -f_7(s)^\top R_s + h_v\left(s,\theta_s,\lr({ \theta_s}),u_s\right) =0,\quad s\in[t,T].
        \end{aligned}
    \right.
\end{equation}
Note that $\mathcal{E}(1,\xi,0)$ is exactly System \eqref{FBSDE:MFTC}. For $\gamma\in[0,1]$, we say that property $(S_\gamma)$ holds true if, for any $(\xi,\ir)\in L^2_{\f_t}\times\mathbb{I}$, System $\mathcal{E}(\gamma,\xi,\ir)$ has a unique solution in $\mathbb{S}$. Note the fact that $(S_0)$ holds true; indeed, there is on coupling in the system \eqref{FBSDE:MFTC=gamma} when $\gamma=0$, so the processes $\pi$ are defined directly by $\ir$, and then $u$ is defined by $\pi$ through the last equation in \eqref{FBSDE:MFTC=gamma} under the convexity assumption on $h$. Our aim is to show $(S_1)$ holds true. We first have the following estimate, whose proof is given in Appendix~\ref{App:lem:estimate}. 

\begin{lemma}\label{lem:estimate}
    Under Assumptions (A1)-(A3), suppose that $(S_\gamma)$ holds true for some $\gamma\in[0,1]$. Then, there exists a positive constant depending only on $(L, T,\lambda)$, such that for any $\xi^1,\xi^2\in L^2_{\f_t}$ and $\ir^1,\ir^2\in\mathbb{I}$, the corresponding solutions $\Pi^1,\Pi^2\in\mathbb{S}$ of $\mathcal{E}\left(\gamma,\xi^1,\ir^1\right)$ and $\mathcal{E}\left(\gamma,\xi^2,\ir^2\right)$ satisfy
    \begin{align}\label{lem:estimate_0}
        \left\|\Pi^2-\Pi^1\right\|^2_{\mathbb{S}}\le C(L, T,\lambda) \left(\left\|\xi^2-\xi^1\right\|_2^2 + \left\|\ir^2-\ir^1\right\|^2_{\mathbb{I}}\right).
    \end{align}
\end{lemma}

The next result is crucial in the proof of our main result.
\begin{lemma}\label{lem:extend}
    Under Assumptions (A1)-(A3), suppose that $(S_\gamma)$ holds true for some $\gamma\in[0,1)$. Then, there exists a positive constant $\eta_0$ depending only on $(L, T,\lambda)$, such that $(S_{\gamma+\eta})$ holds true for any $\eta\in [0, \eta_0 \wedge (1-\gamma )]$.
\end{lemma}

\begin{proof}
For any $(\xi,\ir)\in L^2_{\f_t}\times\mathbb{I}$ and $\eta>0$, we need to show that the system $\mathcal{E}(\gamma+\eta,\xi,\ir)$ has a unique solution in $\mathbb{S}$. To this end, we define a map $\Psi:\mathbb{S}\to \mathbb{S}$ as follows. For any $\Pi\in\mathbb{S}$, let $\Psi(\Pi)$ be the unique solution of $\mathcal{E}(\gamma,\xi,\ir')$, where 
\begin{align*}
    &{\ir'_s}^b:= \eta b \left(s,X_s, \lr\left(X_s\right), u_s \right) + \ir^b_s,\qquad {\ir'_s}^\sigma:= \eta \sigma \left(s,X_s, \lr\left(X_s\right), u_s \right) + \ir^\sigma_s,\\
    &{\ir'_s}^f:= \eta f\left(s,\theta_s,\lr({ \theta_s}), u_s \right) + \ir^f_s,\qquad {\ir'}^g:= \eta g\left(X_T, \lr\left(X_T\right) \right) +\ir^g,\\
    &{\ir'}^\Phi:= \eta\left\{- g_1^\top  R_T - g_2^\top \e\left[R_T\right]  +\Phi_x\left(X_T, \lr\left(X_T\right) \right) +\widetilde{\e}\left[D_{x'}\frac{d\Phi}{d\nu} \left(\widetilde{X_T}, \lr\left(X_T\right)\right)\left(X_T\right) \right] \right\} + \ir^\Phi,\\
    &{\ir'_s}^x:=\eta L_x\left(s,\Pi_s\right)  + \eta \widetilde{\e} \left[D_{x'}\frac{dL}{d\nu} \left(s,\widetilde{\Pi_s}\right)\left({ \theta_s}\right) \right] + \ir^x_s,\\
    &{\ir'}^\varphi:= \eta \varphi_y \left(Y_t\right) + \ir^\varphi, \qquad {\ir'_s}^y:= \eta L_y \left(s,\Pi_s\right) { + \eta \widetilde{\e} \left[D_{y'}\frac{dL}{d\nu} \left(s,\widetilde{\Pi_s}\right)\left({ \theta_s}\right) \right]} +\ir^y_s,\\
    &{\ir'_s}^{z^j}:= \eta L_{z^j}\left(s,\Pi_s \right) { + \eta \widetilde{\e} \left[D_{{z'}^j}\frac{dL}{d\nu} \left(s,\widetilde{\Pi_s}\right)\left({ \theta_s}\right) \right]} + \ir^{z^j}_s.
\end{align*}
Here, $\Psi(\Pi)$ is well-defined due to the fact that $(S_\gamma)$ holds true. It is obvious that $\Pi$ is a fixed point of the map $\Psi$ if and only if it is a solution
of $\mathcal{E}(\gamma+\eta,\xi,\ir)$. Therefore, we only need to show that $\Psi$ is a contraction when
$\eta$ is small enough. In fact, we know from Lemma~\ref{lem:estimate} that, there exists a positive constant depending only on $(L,T,\lambda)$ such that
\begin{align*}
    \left\| \Psi\left(\Pi^2\right)- \Psi\left(\Pi^1\right)\right\|_{\mathbb{S}} \le C(L,T,\lambda) \left\| {\ir'}^2 -{\ir'}^1 \right\|_{\mathbb{I}}.
\end{align*}
Besides, from the Lipschitz-continuous conditions of the coefficients in Assumptions (A1) and (A2), we also see that 
\begin{align*}
    \left\| {\ir'}^2 -{\ir'}^1 \right\|_{\mathbb{I}} \le C(L,T)\eta \left\|\Pi^2-\Pi^1\right\|_{\mathbb{S}}.
\end{align*}
From the last two inequalities, we know that when $\eta$ is small enough, $\Psi$ is indeed a contraction, which concludes the proof.
\end{proof}

We now give the well-posedness result of System \eqref{FBSDE:MFTC}, and the boundedness and continuity of the solution with respect to the initial condition.

\begin{theorem}\label{main2_thm}
	Under Assumptions (A1)-(A3), System \eqref{FBSDE:MFTC} has a unique solution $\Pi\in\mathbb{S}$. There exists a positive constant $C$ depending only on $(L,T,\lambda)$, such that for any initial conditions $\xi^1,\xi^2\in \lr_{\f_t}^2$, the corresponding solutions $\Pi^1$ and $\Pi^2$ of System \eqref{FBSDE:MFTC} satisfy
    \begin{align}
        \left\|\Pi^1\right\|_{\mathbb{S}}\le C \left( 1+\left\|\xi^1\right\|_2\right), \quad \left\|\Pi^2-\Pi^1 \right\|_{\mathbb{S}} \le C\left\|\xi^2-\xi^1\right\|_2.\label{lem:kappa_0}
    \end{align}   
\end{theorem}

\begin{proof}
    The well-posedness result is a direct consequence of Lemma~\ref{lem:extend} and the fact that $(S_0)$ holds true. The proof of the estimates in \eqref{lem:kappa_0} is similar to that for \eqref{lem:estimate_0}, and is omitted here.
\end{proof}

As a direct consequence of the well-posedness of System \eqref{FBSDE:MFTC} in Theorem \ref{main2_thm} and the sufficiency maximum principle in Theorem~\ref{thm:suff}, we can now obtain the optimal control of Problem $\left(\mathsf{P}_{t,\xi}\right)$.

\begin{theorem}\label{thm:solvability}
    Under Assumptions (A1)-(A3), $u\in \lr_\f^2(t,T)$ (the solution of System \eqref{FBSDE:MFTC}) is the unique optimal control for Problem $\left(\mathsf{P}_{t,\xi}\right)$.
\end{theorem}

Theorem~\ref{thm:solvability} establishes the solvability of the optimal control problem of mean-field FBSDEs, which generalizes the standard MFTC problems (see \cite{AB10} for instance), and also extends the maximum principle results for controlled distribution-dependent FBSDEs in the literature (see \cite{MR4587511} for instance). Our approach allows the diffusion to depend on distribution and control, and does not require nondegeneracy. Moreover, the linearity conditions in Assumption (A1) can likely be relaxed using the cone-property-based approach developed in \cite{AB14}. Furthermore, our results are expected to be extended naturally to the jump-diffusion setting, as pursued in \cite{AB14}.

\section{Regularity of the value function}\label{sec:value}

In this section, we consider the value function $V$ defined in \eqref{problem} for Problem $\left(\mathsf{P}_{t,\xi}\right)$. From Theorem~\ref{thm:solvability}, we see that $V$ satisfies
\begin{align}\label{value_equation}
    V(t,\lr(\xi))=J^{t,\xi}\left(u^{t,\xi}\right)=\e\left[\int_t^T h\left(s,\theta_s^{t,\xi},\lr\left({ \theta_s^{t,\xi}}\right),u^{t,\xi}_s\right) ds + \Phi\left(X_T^{t,\xi},\lr\left(X_T^{t,\xi}\right)\right)+\varphi\left(Y_t^{t,\xi}\right)\right].
\end{align}
Here, $\theta^{t,\xi}_s:= \left(X_s^{t,\xi},Y^{t,\xi}_s,Z^{t,\xi}_s\right)$, and  $\left(\theta^{t,\xi},P^{t,\xi},Q^{t,\xi},R^{t,\xi},u^{t,\xi}\right)$ is the unique solution of System \eqref{FBSDE:MFTC} in $\mathbb{S}$. From \eqref{value_equation}, the optimality condition for $u^{t,\xi}$ and the independence of $\xi$ and the Brownian motion, we know that $V$ depends on $\xi$ only through its law $\lr(\xi)$, therefore, the value function $V:[0,T]\times\pr_2(\brn)\to\br$ is well-defined. We also refer the readers to \cite{AB10} for more discussions. We now give the regularity of the map $\xi\mapsto V(t,\lr(\xi))$.

\begin{theorem}\label{thm:value_regu_mu}
    Under Assumptions (A1)-(A3), the map $\xi\mapsto V (t,\lr(\xi))$ is G\^ateaux differentiable with the derivative
	\begin{align}
		D_\xi  V(t,\lr(\xi))=P^{t,\xi}_t, \label{DV_P}
	\end{align}
	and we have the following estimates
	\begin{align}
		|V(t,\lr(\xi))|\le C(L,T,\lambda)\left(1+\|\xi\|_2^2\right);\label{V:growth}
	\end{align}
    and 
    \begin{equation}\label{DV:growth}
    \begin{aligned}
        \left\|D_\xi  V(t,\lr(\xi))\right\|_2\le\ & C(L,T,\lambda)(1+\|\xi\|_2),\\
        \left\|D_\xi  V(t,\lr(\xi'))-D_\xi  V(t,\lr(\xi))\right\|_2\le\ & C(L,T,\lambda)\|\xi'-\xi\|_2. 
    \end{aligned}
	\end{equation}
    Moreover, the map $\pr_2(\brn)\ni\mu\mapsto V(t,\mu)$ is linearly functional-differentiable, and the linear functional-derivative satisfies the relation
    \begin{equation}\label{relation:DV-dV}
    \begin{aligned}
        D_\xi  V(t,\lr(\xi))&=D_x \frac{dV}{d\nu}(t,\lr(\xi))(\xi).
    \end{aligned}
    \end{equation}   
\end{theorem}

\begin{proof}
The growth estimate \eqref{V:growth} is a direct consequence of \eqref{value_equation}, the quadratic growth conditions on the coefficients $(h,\Phi,\varphi)$ in Assumption (A2), and the first estimate in \eqref{lem:kappa_0}. We now prove \eqref{DV_P}. For $\xi,\xi'\in L^2_{\f_t}$, from the definition of the value function, we have
\begin{equation}\label{thm:value_regu_mu_1}
	  J^{t,\xi'}\left(u^{t,\xi'}\right)-J^{t,\xi}\left(u^{t,\xi'}\right)\le V\left(t,\lr(\xi')\right)-V(t,\lr(\xi))\le J^{t,\xi'}\left(u^{t,\xi}\right)-J^{t,\xi}\left(u^{t,\xi}\right).
\end{equation}	
For the upper bound, from Assumption (A2), we have
\begin{align}
	&J^{t,\xi'}\left(u^{t,\xi}\right)-J^{t,\xi}\left(u^{t,\xi}\right) \notag \\
	=\ &\e\bigg\{\int_t^T\left[h\left(s,\theta^{t,\xi',u^{t,\xi}}_s,\lr\left({ \theta^{t,\xi',u^{t,\xi}}_s}\right),u^{t,\xi}_s\right)-h\left(s,\theta^{t,\xi}_s,\lr\left({ \theta^{t,\xi}_s}\right),u^{t,\xi}_s\right)\right] ds \notag \\
    &\quad + \Phi\left(X^{t,\xi',u^{t,\xi}}_T,\lr\left(X^{t,\xi',u^{t,\xi}}_T\right)\right)-\Phi\left(X^{t,\xi}_T,\lr\left(X^{t,\xi}_T\right)\right) + \varphi\left(Y^{t,\xi',u^{t,\xi}}_t\right)-\varphi\left(Y^{t,\xi}_t\right) \bigg\} \notag \\
	\le\ & \e\Bigg\{\int_t^T  \bigg\{ h_x\left(s,\theta^{t,\xi}_s,\lr\left({ \theta^{t,\xi}_s}\right),u^{t,\xi}_s\right)^\top \left(X^{t,\xi',u^{t,\xi}}_s-X^{t,\xi}_s\right)\notag\\
    &\qquad\qquad + h_y\left(s,\theta^{t,\xi}_s,\lr\left({ \theta^{t,\xi}_s}\right),u^{t,\xi}_s\right)^\top \left(Y^{t,\xi',u^{t,\xi}}_s-Y^{t,\xi}_s\right)  \notag\\
    &\qquad\qquad +\sum_{j=1}^n h_{z^j}\left(s,\theta^{t,\xi}_s,\lr\left({ \theta^{t,\xi}_s}\right),u^{t,\xi}_s\right)^\top \left(Z^{t,\xi',u^{t,\xi},j}_s-Z^{t,\xi,j}_s\right)  \notag \\
    &\qquad\qquad +\widetilde{\e}\bigg[\left(D_{x'}\frac{dh}{d\nu}\left(s,\theta_s^{t,\xi},\lr\left({ \theta^{t,\xi}_s}\right),u_s^{t,\xi}\right)\left(\widetilde{ \theta_s^{t,\xi}}\right)\right)^\top \left(\widetilde{X^{t,\xi',u^{t,\xi}}_s}-\widetilde{X^{t,\xi}_s}\right) \bigg] \notag \\
    & \qquad\qquad +\widetilde{\e}\bigg[\left(D_{y'}\frac{dh}{d\nu}\left(s,\theta_s^{t,\xi},\lr\left({ \theta^{t,\xi}_s}\right),u_s^{t,\xi}\right)\left(\widetilde{ \theta_s^{t,\xi}}\right)\right)^\top \left(\widetilde{Y^{t,\xi',u^{t,\xi}}_s}-\widetilde{Y^{t,\xi}_s}\right) \bigg] \notag \\
    & \qquad\qquad +\sum_{j=1}^n \widetilde{\e}\bigg[\left(D_{{z'}^j}\frac{dh}{d\nu}\left(s,\theta_s^{t,\xi},\lr\left({ \theta^{t,\xi}_s}\right),u_s^{t,\xi}\right)\left(\widetilde{ \theta_s^{t,\xi}}\right)\right)^\top \left(\widetilde{Z^{t,\xi',u^{t,\xi},j}_s}-\widetilde{Z^{t,\xi,j}_s}\right) \bigg] \bigg\} ds \notag \\
	&\quad + \varphi_y\left(Y^{t,\xi}_t\right)^\top\left(Y^{t,\xi',u^{t,\xi}}_t-Y^{t,\xi}_t\right)+ \Phi_x\left(X^{t,\xi}_T,\lr\left(X^{t,\xi}_T\right)\right)^\top\left(X^{t,\xi',u^{t,\xi}}_T-X^{t,\xi}_T\right) \notag \\
    &\quad +\widetilde{\e}\bigg[\left(D_{x'}\frac{d\Phi}{d\nu}\left(X^{t,\xi}_T,\lr\left(X^{t,\xi}_T\right)\right)\left(\widetilde{X_T^{t,\xi}}\right)\right)^\top \left(\widetilde{X^{t,\xi',u^{t,\xi}}_T}-\widetilde{X^{t,\xi}_T}\right)\bigg]\Bigg\} \notag \\
	&+C(L,T)\bigg(\left\|X^{t,\xi',u^{t,\xi}}-X^{t,\xi}\right\|_{\sr_\f^2}^2+\left\|Y^{t,\xi',u^{t,\xi}}-Y^{t,\xi}\right\|_{\sr_\f^2}^2+ \left\|Z^{t,\xi',u^{t,\xi}}-Z^{t,\xi}\right\|_{\lr_\f^2}^2 \bigg), \label{thm:value_regu_mu_2}
\end{align}
where $\theta^{t,\xi',u^{t,\xi}}$ denotes the state processes corresponding to the initial condition $\xi'$ and the control $u^{t,\xi}$. By applying It\^o's formula on $\left(P^{t,\xi}_s\right)^\top \left(X^{t,\xi',u^{t,\xi}}_s-X^{t,\xi}_s\right)$ and taking expectation, and then using Fubini's theorem (similar as in \eqref{lem:dJ_2-1}), we know that 
\begin{align}
    &\e\left[\left(P^{t,\xi}_T\right)^\top \left(X^{t,\xi',u^{t,\xi}}_T-X^{t,\xi}_T\right)-\left(P^{t,\xi}_t\right)^\top \left(\xi'-\xi\right)\right] \notag \\
    =\ & \e\int_t^T \bigg\{ \left(R^{t,\xi}_s\right)^\top  f_1(s) \left(X^{t,\xi',u^{t,\xi}}_s-X^{t,\xi}_s\right) + \e\left[\left(R^{t,\xi}_s\right)\right]^\top f_4(s)\left(X^{t,\xi',u^{t,\xi}}_s-X^{t,\xi}_s\right)  \notag \\
    &\quad\qquad - h_x\left(s,\theta^{t,\xi}_s,\lr\left({ \theta^{t,\xi}_s}\right),u^{t,\xi}_s\right)^\top \left(X^{t,\xi',u^{t,\xi}}_s-X^{t,\xi}_s\right) \notag\\
    &\quad\qquad  - \widetilde{\e} \bigg[D_{x'}\frac{dh}{d\nu} \left(s,\widetilde{\theta^{t,\xi}_s},\lr\left({ \theta^{t,\xi}_s}\right),\widetilde{u^{t,\xi}_s}\right)\left({ \theta^{t,\xi}_s}\right) \bigg]^\top \left(X^{t,\xi',u^{t,\xi}}_s-X^{t,\xi}_s\right) \bigg\} ds \notag \\
    =\ & \e\int_t^T \bigg\{ \left(R^{t,\xi}_s\right)^\top  \Big[f_1(s) \left(X^{t,\xi',u^{t,\xi}}_s-X^{t,\xi}_s\right) + f_4(s)\e\left[X^{t,\xi',u^{t,\xi}}_s-X^{t,\xi}_s\right] \Big]  \notag \\
    &\quad\qquad - h_x\left(s,\theta^{t,\xi}_s,\lr\left({ \theta^{t,\xi}_s}\right),u^{t,\xi}_s\right)^\top \left(X^{t,\xi',u^{t,\xi}}_s-X^{t,\xi}_s\right) \notag\\
    &\quad\qquad  - \widetilde{\e} \bigg[D_{x'}\frac{dh}{d\nu} \left(s,\theta^{t,\xi}_s,\lr\left({ \theta^{t,\xi}_s}\right),u^{t,\xi}_s\right)\left(\widetilde{ \theta^{t,\xi}_s}\right)^\top \left(\widetilde{X^{t,\xi',u^{t,\xi}}_s}-\widetilde{X^{t,\xi}_s}\right) \bigg]   \bigg\} ds. \label{thm:value_regu_mu_3}
\end{align}
Similarly, by using It\^o's formula and Fubini's theorem for $\left(R^{t,\xi}_s\right)^\top \left(Y^{t,\xi',u^{t,\xi}}_s-Y^{t,\xi}_s\right)$, we have 
\begin{align*}
    &\e\left[\left(R^{t,\xi}_T\right)^\top \left(Y^{t,\xi',u^{t,\xi}}_T-Y^{t,\xi}_T\right) - \left(R^{t,\xi}_t\right)^\top \left(Y^{t,\xi',u^{t,\xi}}_t-Y^{t,\xi}_t \right)\right] \notag \\
    =\ & \e\int_t^T \bigg\{ - \left(R^{t,\xi}_s\right)^\top \Big[f_1(s) \left(X^{t,\xi',u^{t,\xi}}_s-X^{t,\xi}_s\right) + f_4(s)\e\left[X^{t,\xi',u^{t,\xi}}_s-X^{t,\xi}_s\right] \Big] \notag\\
    &\quad\qquad  - h_y \left(s,\theta^{t,\xi}_s,\lr\left({ \theta^{t,\xi}_s}\right),u^{t,\xi}_s\right)^\top \left(Y^{t,\xi',u^{t,\xi}}_s-Y^{t,\xi}_s\right)\notag\\
    &\quad\qquad - \sum_{j=1}^n   h_{z^j}\left(s,\theta^{t,\xi}_s,\lr\left({ \theta^{t,\xi}_s}\right),u^{t,\xi}_s\right)^\top  \left(Z^{t,\xi',u^{t,\xi},j}_s-Z^{t,\xi,j}_s\right) \notag\\
    & \quad\qquad - \widetilde{\e} \bigg[D_{y'}\frac{dh}{d\nu} \left(s,\theta^{t,\xi}_s,\lr\left({ \theta^{t,\xi}_s}\right),u^{t,\xi}_s\right)\left(\widetilde{ \theta^{t,\xi}_s}\right)^\top \left(\widetilde{Y^{t,\xi',u^{t,\xi}}_s}-\widetilde{Y^{t,\xi}_s}\right) \bigg] \notag\\
    & \quad\qquad - \widetilde{\e} \bigg[\sum_{j=1}^n D_{{z'}^j}\frac{dh}{d\nu} \left(s,\theta^{t,\xi}_s,\lr\left({ \theta^{t,\xi}_s}\right),u^{t,\xi}_s\right)\left(\widetilde{ \theta^{t,\xi}_s}\right)^\top \left(\widetilde{Z^{t,\xi',u^{t,\xi},j}_s}-\widetilde{Z^{t,\xi,j}_s}\right) \bigg] \bigg\} ds. 
\end{align*}
Plugging the last equality with \eqref{thm:value_regu_mu_3}, we deduce that
\begin{align}
    &\e\bigg[\left(P^{t,\xi}_T\right)^\top \left(X^{t,\xi',u^{t,\xi}}_T-X^{t,\xi}_T\right) + \left(R^{t,\xi}_T\right)^\top \left(Y^{t,\xi',u^{t,\xi}}_T-Y^{t,\xi}_T\right) \notag \\
    &\qquad - \left(R^{t,\xi}_t\right)^\top \left(Y^{t,\xi',u^{t,\xi}}_t-Y^{t,\xi}_t \right)-\left(P^{t,\xi}_t\right)^\top \left(\xi'-\xi\right)\bigg] \notag \\
    =\ & \e\int_t^T \bigg\{ - h_x\left(s,\theta^{t,\xi}_s,\lr\left({ \theta^{t,\xi}_s}\right),u^{t,\xi}_s\right)^\top \left(X^{t,\xi',u^{t,\xi}}_s-X^{t,\xi}_s\right) \notag\\
    &\quad\qquad - h_y \left(s,\theta^{t,\xi}_s,\lr\left({ \theta^{t,\xi}_s}\right),u^{t,\xi}_s\right)^\top \left(Y^{t,\xi',u^{t,\xi}}_s-Y^{t,\xi}_s\right) \notag\\
    &\quad\qquad - \sum_{j=1}^n   h_{z^j}\left(s,\theta^{t,\xi}_s,\lr\left({ \theta^{t,\xi}_s}\right),u^{t,\xi}_s\right)^\top  \left(Z^{t,\xi',u^{t,\xi},j}_s-Z^{t,\xi,j}_s\right) \notag\\
    &\quad\qquad  - \widetilde{\e} \bigg[D_{x'}\frac{dh}{d\nu} \left(s,\theta^{t,\xi}_s,\lr\left({ \theta^{t,\xi}_s}\right),u^{t,\xi}_s\right)\left(\widetilde{ \theta^{t,\xi}_s}\right)^\top \left(\widetilde{X^{t,\xi',u^{t,\xi}}_s}-\widetilde{X^{t,\xi}_s}\right) \bigg]  \notag\\
    & \quad\qquad - \widetilde{\e} \bigg[D_{y'}\frac{dh}{d\nu} \left(s,\theta^{t,\xi}_s,\lr\left({ \theta^{t,\xi}_s}\right),u^{t,\xi}_s\right)\left(\widetilde{ \theta^{t,\xi}_s}\right)^\top \left(\widetilde{Y^{t,\xi',u^{t,\xi}}_s}-\widetilde{Y^{t,\xi}_s}\right) \bigg] \notag\\
    & \quad\qquad - \widetilde{\e} \bigg[\sum_{j=1}^n D_{{z'}^j}\frac{dh}{d\nu} \left(s,\theta^{t,\xi}_s,\lr\left({ \theta^{t,\xi}_s}\right),u^{t,\xi}_s\right)\left(\widetilde{ \theta^{t,\xi}_s}\right)^\top \left(\widetilde{Z^{t,\xi',u^{t,\xi},j}_s}-\widetilde{Z^{t,\xi,j}_s}\right) \bigg] \bigg\} ds. \notag
\end{align}
By using the terminal conditions for $P^{t,\xi}_T$ and $Y^{t,\xi}_T$, the initial condition for $R^{t,\xi}_t$, and then applying Fubini's theorem, we have
\begin{align}
    &\e\bigg[\left(P^{t,\xi}_T\right)^\top \left(X^{t,\xi',u^{t,\xi}}_T-X^{t,\xi}_T\right) + \left(R^{t,\xi}_T\right)^\top \left(Y^{t,\xi',u^{t,\xi}}_T-Y^{t,\xi}_T\right) - \left(R^{t,\xi}_t\right)^\top \left(Y^{t,\xi',u^{t,\xi}}_t-Y^{t,\xi}_t \right)\bigg] \notag \\
    =\ & \e\bigg\{ \varphi_y \left(Y^{t,\xi}_t\right)^\top \left(Y^{t,\xi',u^{t,\xi}}_t-Y^{t,\xi}_t\right) + \Phi_x\left(X^{t,\xi}_T, \lr\left(X^{t,\xi}_T\right) \right)^\top \left(X^{t,\xi',u^{t,\xi}}_T-X^{t,\xi}_T\right)\notag \\
    &\quad +\widetilde{\e}\bigg[ D_{x'}\frac{d\Phi}{d\nu} \left(X^{t,\xi}_T, \lr\left(X^{t,\xi}_T\right)\right)\left(\widetilde{X^{t,\xi}_T}\right)^\top \left(\widetilde{X^{t,\xi',u^{t,\xi}}_T} - \widetilde{X^{t,\xi}_T} \right) \bigg]  \bigg\}. \notag
\end{align}
Substituting the last two equalities back into \eqref{thm:value_regu_mu_2}, we have
\begin{align}
	&J^{t,\xi'}\left(u^{t,\xi}\right)-J^{t,\xi}\left(u^{t,\xi}\right) \notag \\
	\le\ & \e \left[ \left(P^{t,\xi}_t\right)^\top \left(\xi'-\xi\right) \right] +C(L,T)\bigg(\left\|X^{t,\xi',u^{t,\xi}}-X^{t,\xi}\right\|^2_{\sr_\f^2} + \left\|Y^{t,\xi',u^{t,\xi}}-Y^{t,\xi}\right\|^2_{\sr_\f^2} \notag\\
    &\qquad\qquad\qquad\qquad\qquad\qquad\qquad + \left\|Z^{t,\xi',u^{t,\xi}}-Z^{t,\xi}\right\|^2_{\lr_\f^2}  \bigg). \label{thm:value_regu_mu_7}
\end{align}
From Assumption (A1) and the Gr\"onwall's inequality, by using standard estimate for the SDE of $X^{t,\xi',u^{t,\xi}}_s-X^{t,\xi}$, we have the following estimate
\begin{equation}\label{thm:value_regu_mu_8}
	\begin{split}
		\left\|X^{t,\xi',u^{t,\xi}}-X^{t,\xi}\right\|^2_{\sr_\f^2} \le C(L,T)\left\|\xi'-\xi\right\|^2_2;
	\end{split}
\end{equation}
then, by using standard estimate for the BSDE of $(\de Y,\de Z)$ and using \eqref{thm:value_regu_mu_8}, we have
\begin{align}
    &\left\|Y^{t,\xi',u^{t,\xi}}-Y^{t,\xi}\right\|^2_{\sr_\f^2} +  \left\|Z^{t,\xi',u^{t,\xi}}-Z^{t,\xi}\right\|^2_{\lr_\f^2}  \le C(L,T)\left\|X^{t,\xi',u^{t,\xi}}-X^{t,\xi}\right\|^2_{\sr_\f^2} \le C(L,T)\left\|\xi'-\xi\right\|^2_2. \label{thm:value_regu_mu_9}
\end{align}
Substituting \eqref{thm:value_regu_mu_8} and \eqref{thm:value_regu_mu_9} into \eqref{thm:value_regu_mu_7}, we deduce that
\begin{align}
	&J^{t,\xi'}\left(u^{t,\xi}\right)-J^{t,\xi}\left(u^{t,\xi}\right)\le \e \left[ \left(P^{t,\xi}_t\right)^\top \left(\xi'-\xi\right) \right] +C(L,T)\left\|\xi'-\xi\right\|_2^2. \label{thm:value_regu_mu_10}
\end{align}
In a similar way, we can also obtain the following lower bound:
\begin{equation}\label{thm:value_regu_mu_11}
	\begin{split}
		&J^{t,\xi'}\left(u^{t,\xi'}\right)-J^{t,\xi}\left(u^{t,\xi'}\right)\geq \e\left[\left(P^{t,\xi'}_t\right)^\top \left(\xi'-\xi\right)\right] -C(L,T)\left\|\xi'-\xi\right\|_2^2.
	\end{split}
\end{equation}
From the second estimate of \eqref{lem:kappa_0}, we know that 
\begin{align}\label{thm:value_regu_mu_12}
    \left\|P^{t,\xi'}_t-P^{t,\xi}_t \right\|\le C(L,T,\lambda)\left\|\xi'-\xi\right\|_2.
\end{align}
From \eqref{thm:value_regu_mu_10}-\eqref{thm:value_regu_mu_12} and \eqref{thm:value_regu_mu_1}, we conclude that 
\begin{align*}
    &\left|V\left(t,\lr(\xi')\right)-V(t,\lr(\xi))-\e\left[\left(P^{t,\xi}_t\right)^\top \left(\xi'-\xi\right)\right]\right|\le C(L,T,\lambda) \left\|\xi'-\xi\right\|_2^2,
\end{align*}
which implies that $D_\xi  V(t,\lr(\xi))=P^{t,\xi}_t$. Then, the estimate \eqref{DV:growth} is a direct consequence of Estimate \eqref{lem:kappa_0}. Besides, from the G\^ateaux differentiability of $V$ in $\xi$ and Estimates \eqref{V:growth}-\eqref{DV:growth}, following similar arguments as in \cite[Theorem 7.1]{AB10}, we can show that $V$ is linearly functional-differentiable; and from standard relations between the G\^ateaux derivatives and linear functional-derivatives as in \eqref{lem01_1} (also see \cite{AB10,AB11}), we obtain \eqref{relation:DV-dV}.
\end{proof}

Theorem~\ref{thm:value_regu_mu} provides the differentiability of the value function with respect to the distribution variable. In the following sections, we shall study the regularity of the value function in the time variable, and also establish the classical solvability of the HJB equations, under three different kinds of settings. 

\section{HJB equation for the first-order problem}\label{sec:HJB:1st}

In this section, we aim to give the related HJB equation for the value function $V$ corresponding to the first-order optimal control problem for mean-field type forward-backward equations. That is, the system does not include the Brownian motion (i.e., $\sigma \equiv 0$), and all the randomness comes from the initial condition. For initial time $t$ and initial condition $\xi\in L_{\f_t}^2$, the admissible control set for Problem $\left(\mathsf{P}_{t,\xi}\right)$ is defined as $\left\{(v_s)_{t\le s\le T},\ :v_s\in L_{\f_t}^2,\ \e\int_t^T |v_s|^2 ds <\infty\right\}$. Since no additional Brownian randomness is injected into the system beyond the initial state, the backward equation for $Y$ is well-posed without the stochastic integral term; we have $Z \equiv 0$. For this case, the maps $(f,h)$ are functions for $(s,x,y,\varrho,v)$ (without the $z$ variable), and $\varrho\in \pr_2(\brn\times\brn)$ with its marginals $\varrho_x$ and $\varrho_y$. In this degenerate case, in the adjoint FBSDEs, the absence of Brownian randomness also implies $Q \equiv 0$. 

In this section, we still assume Assumptions (A1)-(A3) hold, with the only difference being that all coefficients involving $z$ or $\varrho_z$ are set to zero. We denote by $\Theta^{t,\xi}:=\left(X^{t,\xi},Y^{t,\xi},P^{t,\xi},R^{t,\xi},u^{t,\xi}\right)$ the solution of System \eqref{FBSDE:MFTC} corresponding to the first-order case, and use the notation $\theta_s^{t,\xi}:=\left(X_s^{t,\xi},Y_s^{t,\xi}\right)$. The optimality condition for the first-order problem reads
\begin{equation}\label{FBSDEs:first}
    b_3(s)^\top  P^{t,\xi}_s -f_7(s)^\top  R^{t,\xi}_s + h_v\left(s,\theta^{t,\xi}_s,\lr\left(\theta^{t,\xi}_s\right),u^{t,\xi}_s\right) =0,\quad s\in[t,T],\ a.e., \quad \mathbb{P}-a.s..
\end{equation}
Since our aim is to derive the HJB equation, we also need the following assumption on the continuity of coefficients in time.

{\bf (B)} The coefficient $\sigma\equiv0$. The function $h$ is continuous in the time variable, and we have the following continuity conditions in the time variable:
\begin{align*}
    &\left|b_k(s')-b_k(s)\right| \le L |s'-s|^{\frac{1}{2}},\quad 0\le k\le 3;\qquad \left|f_k(s')-f_k(s)\right| \le L |s'-s|^{\frac{1}{2}},\quad 0\le k\le 7;\\
    & \left|(h_y,h_v)(\cdot,x,y,\varrho,v)\big|^{s'}_s \right| \le L\left(1+|x|+|y|+|\varrho|_2+|v|\right) |s'-s|^{\frac{1}{2}};\\
    & \left|D_{y'}\frac{dh}{d\nu}(\cdot,x,y,\varrho,v)(x',y')\Big|^{s'}_s \right|\le L\left(1+|x|+|y|+|\varrho|_2+|v|+|x'|+|y'|\right) |s'-s|^{\frac{1}{2}}.
\end{align*}

%In this section, for any $t\in[0,T]$, we denote by $S^2_{\f_t}(t,T)$ the family of all processes $\beta_\cdot=\{\beta_s,\ t\le s\le T\}$ such that $\beta_s$ is $\f_t$-measurable and $\|\beta\|_{S^2_{\f_t}}^2:=\e\left[\sup_{t\le s\le T} |\beta_s|^2\right]<\infty$. 

\subsection{Auxiliary processes}\label{subsec:auxiliary}

We begin by giving the following boundedness and H\"older continuity results of $\Theta_s^{t,\xi}$ in $s$ under Assumption (B), whose proof is provided in Appendix~\ref{App:lem:estimate_s}. 

\begin{lemma}\label{lem:estimate_s}
    Under Assumptions (A1)-(A3) and (B), we have
    \begin{align}
        &\e\bigg[\sup_{t\le s\le T}\left|X^{t,\xi}_s\right|^2+\sup_{t\le s\le T}\left|Y^{t,\xi}_s\right|^2+\sup_{t\le s\le T}\left|P^{t,\xi}_s\right|^2+\sup_{t\le s\le T}\left|R^{t,\xi}_s\right|^2+\sup_{t\le s\le T}\left|u^{t,\xi}_s\right|^2\bigg] \notag\\
        \le\ & C(L,T,\lambda) \left( 1+\left\|\xi\right\|^2_2\right); \label{1order-FBSDE-0}
    \end{align}
    and for $t\le s\le s'\le T$, we have
    \begin{align}
        &\e\left[\left|X_{s'}-X_s\right|^2 + \left|Y_{s'}-Y_s\right|^2 + \left|P_{s'}-P_s\right|^2 + \left|R_{s'}-R_s\right|^2\right]\le C(L,T,\lambda) (s'-s)^2, \label{lem:estimate_s_1}\\
        &\e\left[\left|u_{s'}-u_s\right|^2\right]\le C(L,T,\lambda)\left(1+\|\xi\|_2^2\right)(s'-s). \label{lem:estimate_s_2}
    \end{align}
\end{lemma}

The aim of this section is to study the differentiability of the value function $V$ in $t$, and this will also require the following assumption, which is the regularity-enhanced version of Assumption (A2). Here, we emphasize that Assumption (A2') does not require the first-order assumption $\sigma\equiv 0$; it includes the full set of conditions with respect to the $z$ variable, as it will also be used later for the second-order problems in the following sections.

\textbf{(A2')} The functionals $h$, $\Phi$ and $\varphi$ satisfy (A2). Moreover, the following derivatives exist, and they are uniformly continuous in $\left(x,y,z,{ \varrho},v,{ x',y',z',x'',y'',z''}\right)$, and are globally bounded by $L$ in norm:
\begin{align*}
    &h_{xx},\ h_{xy}, \ h_{xz},\  h_{xv},\ h_{yy},\ h_{yz},\ h_{yv},\ h_{zz},\ h_{zv},\ h_{vv},\\
    & \left(D_{x'},D_{y'},D_{z'}\right) \frac{d\left(h_x,h_y,h_z,h_v\right)}{d\nu},\quad \left(D_{x''},D_{y''},D_{z''}\right)\left(D_{x'},D_{y'},D_{z'}\right) \frac{d^2h}{d\nu^2},\quad \left(D_{x'},D_{y'},D_{z'}\right)^2 \frac{dh}{d\nu};\\
    &\Phi_{xx},\  D_{x'} \frac{d\Phi_x}{d\nu},\ D_{x''}D_{x'} \frac{d^2\Phi}{d\nu^2},\ D_{x'}^2 \frac{d\Phi}{d\nu};\quad \varphi_{yy}.
\end{align*}

Under Assumptions (A2') and (B), the first convexity condition of $h$ in (A3) can also be written as, for any $s\in[0,T]$ and $(x,y,v,{ \zeta}),\ (\de x,\de y,\de v,\de { \zeta})\in\brn\times \brn\times \brd\times L^2(\Omega,\f,\mathbb{P};\brn{ \times\brn})$,
\begin{align}
    & \left[\left(
		\begin{array}{ccc}
			h_{xx} & h_{xy} & h_{xv}\\
			h_{yx} & h_{yy} & h_{yv}\\
            h_{vx} & h_{vy} & h_{vv}\\
		\end{array}
	\right)(s,x,y,\lr({ \zeta}),v) \right] \left(
		\begin{array}{c}
			\de x\\
			\de y\\
            \de v
		\end{array}
	\right)^{\otimes 2} \notag \\
    &+2\e \left\{\left(
		\begin{array}{c}
			\de x\\
			\de y\\
            \de v
		\end{array}
	\right)^\top \left[\left(
		\begin{array}{cc}
			D_{x'}\frac{dh_x}{d\nu} & D_{y'}\frac{dh_x}{d\nu} \\
			D_{x'}\frac{dh_y}{d\nu} & D_{y'}\frac{dh_y}{d\nu}\\
            D_{x'}\frac{dh_v}{d\nu} & D_{y'}\frac{dh_v}{d\nu} 
		\end{array}
	\right)(s,x,y,\lr({ \zeta}),v)({ \zeta})\right] \left(
		\begin{array}{c}
			\de { \zeta_x}\\
			\de { \zeta_y}
		\end{array}
	\right)\right\} \notag \\
    &+ \e \left\{\left[\left(
		\begin{array}{cc}
			D^2_{x'}\frac{dh}{d\nu} & D_{y'}D_{x'}\frac{dh}{d\nu} \\
			D_{x'}D_{y'}\frac{dh}{d\nu} & D^2_{y'}\frac{dh}{d\nu} \\
		\end{array}
	\right)\left(s,x,y,\lr({ \zeta}),v\right)({ \zeta}) \right] \left(
		\begin{array}{c}
			\de { \zeta_x}\\
			\de { \zeta_y}
		\end{array}
	\right)^{\otimes 2} \right\} \notag \\
    &+ \e\widetilde{\e} \left\{ \left(
		\begin{array}{c}
			\de { \zeta_x}\\
			\de { \zeta_y}
		\end{array}
	\right)^\top \left[\left(
		\begin{array}{cc}
			D_{x''}D_{x'}\frac{d^2h}{d\nu^2} & D_{y''}D_{x'}\frac{d^2h}{d\nu^2} \\
			D_{x''}D_{y'}\frac{d^2h}{d\nu^2} & D_{y''}D_{y'}\frac{d^2h}{d\nu^2} \\
		\end{array}
	\right)\left(s,x,y,\lr({ \zeta}),v\right)\left({ \zeta},\widetilde{{ \zeta}}\right) \right] \left(
		\begin{array}{c}
			\widetilde{\de{ \zeta_x}} \\
			\widetilde{\de{ \zeta_y}} 
		\end{array}
	\right) \right\} \notag \\
    \geq\ & 2\lambda |\de v|^2. \label{def:convex_h}
\end{align}
Similarly, the convexity of $\Phi$ and $\varphi$ in (A3) also gives, for any $(x,\xi),\ (\de x,\de \xi)\in\brn\times L^2(\Omega,\f,\mathbb{P};\brn)$,
\begin{align}
    & \de x^\top \Phi_{xx}(x,\lr(\xi)) \de x +2\e\bigg[\de \xi^\top D_{x'}\frac{d\Phi_x}{d\nu}(x,\lr(\xi))(\xi)\bigg] \de x\notag \\
    &+\e\left[\de \xi^\top D_{x'}^2\frac{d\Phi}{d\nu}(x,\lr(\xi))\left(\xi\right) \de \xi\right] +\e\widetilde{\e}\left[\de \xi^\top D_{x''}D_{x'}\frac{d^2\Phi}{d\nu^2}(x,\lr(\xi))\left(\xi,\widetilde{\xi}\right) \widetilde{\de \xi}\right] \geq  0; \label{def:convex_Phi} \\
    & \de y^\top \varphi_{yy}(y) \de y \geq 0. \label{def:convex_varphi}
\end{align}

The proof of the differentiability of $V$ in time will be based on the study of the differentiability of $\Theta_s^{t,\xi}$ in $t$. For notational convenience, we drop the superscript $t,\xi$ in $\Theta^{t,\xi}_s$, and simply use the notation $\Theta_s^{\epsilon}$ to denote by $\Theta_s^{t+\epsilon,\xi}$ with $\epsilon\in(0,1)$; we also define 
\begin{align}\label{def:x_s^epsilon}
    &\x_s^\epsilon:= \frac{X_s^\epsilon -X_s}{\epsilon},\quad \y_s^\epsilon:= \frac{Y_s^\epsilon -Y_s}{\epsilon},\quad \p_s^\epsilon:= \frac{P_s^\epsilon -P_s}{\epsilon},\quad \R_s^\epsilon:= \frac{R_s^\epsilon -R_s}{\epsilon},\quad \U_s^\epsilon:= \frac{u_s^\epsilon -u_s}{\epsilon}. 
\end{align}
The following result gives the estimate for the uniform boundedness of the processes defined in \eqref{def:x_s^epsilon}, whose proof is given in Appendix~\ref{App:lem:de_bound}.

\begin{lemma}\label{lem:de_bound}
    Under Assumptions (A1), (A2'), (A3) and (B), there exists a positive constant $C$ depending only on $(L,T,\lambda)$, such that 
    \begin{align}
        &\e\bigg[\sup_{t+\epsilon\le s\le T}\left|\x^\epsilon_s\right|^2+\sup_{t+\epsilon\le s\le T}\left|\y^\epsilon_s\right|^2+\sup_{t+\epsilon\le s\le T}\left|\p^\epsilon_s\right|^2+\sup_{t+\epsilon\le s\le T}\left|\R^\epsilon_s\right|^2+\sup_{t+\epsilon\le s\le T}\left|\U^\epsilon_s\right|^2\bigg] \notag\\
        \le\ & C \left( 1+\left\|\xi\right\|_2^2\right). \label{lem:xypru_0'}
    \end{align}       
\end{lemma}

To give the limit of the processes $\frac{\Theta_s^{t+\epsilon,\xi}-\Theta_s^{t,\xi}}{\epsilon}$ for $s\in(t+\epsilon,T)$ as $\epsilon\to 0$, we  consider the following system of FBSDEs for $\left(\x^{t,\xi}_s,\y^{t,\xi}_s,\p^{t,\xi}_s,\R^{t,\xi}_s,\U^{t,\xi}_s\right)$ with $s\in[t,T]$:
\begin{align}
    &\x^{t,\xi}_s=-\left\{b_0(t)+b_1(t) \xi + b_2(t)\e\left[\xi\right]+b_3(t)u^{t,\xi}_t\right\} \notag \\
    &\quad\qquad +\int_t^s \left\{ b_1(r) \x^{t,\xi}_r + b_2(r)\e\left[\x^{t,\xi}_r\right]+b_3(r)\U^{t,\xi}_r\right\}   dr, \notag \\
    &\y^{t,\xi}_s= g_1 \x^{t,\xi}_T+ g_2 \e\left[ \x^{t,\xi}_T\right]  + \int_s^T  \Big[f_1(r) \x^{t,\xi}_r + f_2(r) \y^{t,\xi}_r + f_4(r) \e\left[ \x^{t,\xi}_r\right] \notag\\
    &\qquad\qquad\qquad\qquad\qquad\qquad\qquad\qquad + f_5(r) \e\left[ \y^{t,\xi}_r\right] + f_7(r) \U^{t,\xi}_r  \Big] dr, \notag \\
    & \p^{t,\xi}_s= - g_1^\top \R^{t,\xi}_T - g_2^\top \e\left[ \R^{t,\xi}_T\right]  +\Phi_{xx}\left(X^{t,\xi}_T, \lr\left(X^{t,\xi}_T\right) \right) \x^{t,\xi}_T \notag \\
    &\quad\qquad +\widetilde{\e}\bigg[D_{x'}\frac{d\Phi_x}{d\nu}\left(X^{t,\xi}_T, \lr\left(X^{t,\xi}_T\right) \right)\left(\widetilde{X^{t,\xi}_T}\right) \widetilde{\x^{t,\xi}_T}+\left(D_{x'}\frac{d\Phi_x}{d\nu} \left(\widetilde{X^{t,\xi}_T}, \lr\left(X^{t,\xi}_T\right)\right)\left(X^{t,\xi}_T\right)\right)^\top \widetilde{\x^{t,\xi}_T}  \bigg] \notag \\
    &\quad\qquad +\he\widetilde{\e}\left[D_{x''}D_{x'}\frac{d^2\Phi}{d\nu^2} \left(\widetilde{X^{t,\xi}_T}, \lr\left(X^{t,\xi}_T\right)\right)\left(X^{t,\xi}_T,\widehat{X^{t,\xi}_T}\right) \widehat{\x^{t,\xi}_T} \right] \notag \\
    &\quad\qquad +\widetilde{\e}\left[D_{x'}^2\frac{d\Phi}{d\nu} \left(\widetilde{X^{t,\xi}_T}, \lr\left(X^{t,\xi}_T\right)\right)\left(X^{t,\xi}_T\right) \x^{t,\xi}_T \right] \notag \\
    &\quad\qquad + \int_s^T \Bigg\{  b_1(r)^\top  \p^{t,\xi}_r + b_2(r)^\top \e\left[\p^{t,\xi}_r\right] - f_1(r)^\top  \R^{t,\xi}_r - f_4(r)^\top \e\left[\R^{t,\xi}_r\right] \notag \\
    &\qquad\qquad +\left[\left(\begin{array}{ccc}
         h_{xx} & h_{xy} & h_{xv}
    \end{array} \right)\left(r,\theta^{t,\xi}_r,\lr\left(\theta^{t,\xi}_r\right), u^{t,\xi}_r\right) \right] \left(\text{\small$\begin{array}{c}
         \x^{t,\xi}_r  \\
         \y^{t,\xi}_r  \\
         \U^{t,\xi}_r
    \end{array}$}\right) \notag \\
    & \qquad\qquad +\widetilde{\e}\left\{\left[\left(\begin{array}{cc}
         D_{x'}\frac{dh_x}{d\nu} &  D_{y'}\frac{dh_x}{d\nu}
    \end{array}\right)\left(r,\theta^{t,\xi}_r,\lr\left(\theta^{t,\xi}_r\right),u^{t,\xi}_r\right)\left(\widetilde{\theta^{t,\xi}_r}\right) \right] \left(\text{\small $\begin{array}{cc}
         \widetilde{\x^{t,\xi}_r}  \\
         \widetilde{\y^{t,\xi}_r} 
    \end{array}$}\right) \right\} \notag \\
    & \qquad\qquad +\widetilde{\e}\left\{\left[\left(\begin{array}{c}
         D_{x'}\frac{dh_x}{d\nu} \\
         D_{x'}\frac{dh_y}{d\nu} \\
         D_{x'}\frac{dh_v}{d\nu}
    \end{array}\right)\left(r,\widetilde{\theta^{t,\xi}_r},\lr\left(\theta^{t,\xi}_r\right),\widetilde{u^{t,\xi}_r}\right)\left(\theta^{t,\xi}_r\right) \right] \left(\text{\small $\begin{array}{cc}
         \widetilde{\x^{t,\xi}_r}  \\
         \widetilde{\y^{t,\xi}_r}  \\
         \widetilde{\U^{t,\xi}_r}
    \end{array}$}\right) \right\} \notag \\
    & \qquad\qquad  +  \he\widetilde{\e} \left\{ \left[\left(\begin{array}{cc}
         D_{x''}D_{x'}\frac{d^2h}{d\nu^2} & D_{y''}D_{x'}\frac{d^2h}{d\nu^2}  
    \end{array}\right) \left(r,\widetilde{\theta^{t,\xi}_r},\lr\left(\theta^{t,\xi}_r\right),\widetilde{u^{t,\xi}_r}\right)\left(\theta^{t,\xi}_r,\widehat{\theta^{t,\xi}_r}\right) \right]  \left(\text{\small $\begin{array}{cc}
         \widehat{\x^{t,\xi}_r}   \\
         \widehat{\y^{t,\xi}_r}  
    \end{array}$} \right) \right\} \notag \\
    & \qquad\qquad  +  \widetilde{\e} \left\{ \left[\left(\begin{array}{cc}
         D^2_{x'}\frac{dh}{d\nu} & D_{y'}D_{x'}\frac{dh}{d\nu}  
    \end{array}\right) \left(r,\widetilde{\theta^{t,\xi}_r},\lr\left(\theta^{t,\xi}_r\right),\widetilde{u^{t,\xi}_r}\right)\left(\theta^{t,\xi}_r\right) \right]  \left(\text{\small $\begin{array}{cc}
         \x^{t,\xi}_r   \\
         \y^{t,\xi}_r  
    \end{array}$} \right) \right\} \Bigg\} dr, \notag \\
    & \R^{t,\xi}_s= - \varphi_{yy} \left(Y^{t,\xi}_t\right) \left(\y^{t,\xi}_t - f\left(t,\xi,Y^{t,\xi}_t,\lr\left(\xi,Y^{t,\xi}_t\right),u^{t,\xi}_t\right) \right) - f_2(t)^\top  R^{t,\xi}_t  { - f_5(t)^\top  \e\left[R^{t,\xi}_t\right]}\notag \\
    &\quad\qquad + h_y \left(t,\xi,Y^{t,\xi}_t,\lr\left({ \xi, Y^{t,\xi}_t}\right),u^{t,\xi}_t\right) { +\widetilde{\e}\left[D_{y'} \frac{dh}{d\nu}\left(t,\widetilde{\xi},\widetilde{Y^{t,\xi}_t},\lr\left(\xi,Y^{t,\xi}_t\right),\widetilde{u^{t,\xi}_t} \right)\left(\xi,Y^{t,\xi}_t\right) \right]} \notag \\
    &\quad\qquad +\int_t^s \Bigg\{ f_2(r)^\top  \R^{t,\xi}_r { + f_5(r)^\top  \e\left[\R^{t,\xi}_r\right]} - \left[\left( \begin{array}{ccc}
         h_{yx} & h_{yy} & h_{yv}
    \end{array}\right)\left(r,\theta^{t,\xi}_r,\lr\left({ \theta^{t,\xi}_r}\right),u^{t,\xi}_r\right)\right] \left(\text{\small $\begin{array}{c}
         \x^{t,\xi}_r \\
         \y^{t,\xi}_r \\
         \U^{t,\xi}_r
    \end{array}$}\right) \notag \\
    & \qquad\qquad -\widetilde{\e}\left\{ \left[\left(\begin{array}{cc}
         D_{x'}\frac{dh_{y}}{d\nu} & D_{y'}\frac{dh_{y}}{d\nu} 
    \end{array} \right) \left(r,\theta^{t,\xi}_r,\lr\left(\theta^{t,\xi}_r\right),u^{t,\xi}_r\right)\left(\widetilde{\theta^{t,\xi}_r}\right) \right] \left( \text{\small $\begin{array}{cc}
         \widetilde{\x^{t,\xi}_r}  \\
         \widetilde{\y^{t,\xi}_r} 
    \end{array}$}\right) \right\}  \notag\\
    & \qquad\qquad - \widetilde{\e}\left\{\left[\left(\begin{array}{ccc}
        D_{y'}\frac{dh_x}{d\nu} \\ 
        D_{y'}\frac{dh_y}{d\nu} \\
        D_{y'}\frac{dh_v}{d\nu}
    \end{array}\right)\left(r,\widetilde{\theta^{t,\xi}_r},\lr\left({ \theta^{t,\xi}_r}\right),\widetilde{u^{t,\xi}_r}\right)\left({ \theta^{t,\xi}_r}\right)\right]^\top  \left(\text{\small $\begin{array}{c}
        \widetilde{\x^{t,\xi}_r}  \\
        \widetilde{\y^{t,\xi}_r}  \\
        \widetilde{\U^{t,\xi}_r} 
    \end{array}$}
    \right) \right\}  \notag\\
    & \qquad\qquad  - \he\widetilde{\e} \left\{ \left[\left( \begin{array}{cc}
         D_{x''}D_{y'}\frac{d^2h}{d\nu^2} & D_{y''}D_{y'}\frac{d^2h}{d\nu^2} \\
    \end{array}\right) \left(r,\widetilde{\theta^{t,\xi}_r},\lr\left({ \theta^{t,\xi}_r}\right),\widetilde{u^{t,\xi}_r}\right)\left({ \theta^{t,\xi}_r},\widehat{{ \theta^{t,\xi}_r}}\right) \right] \left(\text{\small $\begin{array}{c}
         \widehat{\x^{t,\xi}_r}   \\
         \widehat{\y^{t,\xi}_r} 
    \end{array}$}\right) \right\}  \notag\\
    & \qquad\qquad  - \widetilde{\e} \left\{ \left[\left( \begin{array}{cc}
         D_{x'}D_{y'}\frac{dh}{d\nu} & D_{y'}^2\frac{dh}{d\nu} \\
    \end{array}\right) \left(r,\widetilde{\theta^{t,\xi}_r},\lr\left({ \theta^{t,\xi}_r}\right),\widetilde{u^{t,\xi}_r}\right)\left({ \theta^{t,\xi}_r}\right) \right] \left(\text{\small $\begin{array}{c}
         \x^{t,\xi}_r   \\
         \y^{t,\xi}_r 
    \end{array}$}\right) \right\} \Bigg\} dr, \notag \\
    &0=b_3(s)^\top  \p^{t,\xi}_s -f_7(s)^\top  \R^{t,\xi}_s + \left[\left(\begin{array}{ccc}
         h_{vx} & h_{vy} & h_{vv} 
    \end{array} \right)\left(s,\theta^{t,\xi}_s,\lr\left({ \theta^{t,\xi}_s}\right),u^{t,\xi}_s\right) \right] \left(\text{\small $\begin{array}{c}
         \x^{t,\xi}_s \\
         \y^{t,\xi}_s \\
         \U^{t,\xi}_s 
    \end{array}$}\right) \notag \\
    & \qquad +\widetilde{\e}\left\{ \left[\left(\begin{array}{cc}
         D_{x'}\frac{dh_{v}}{d\nu} & D_{y'}\frac{dh_{v}}{d\nu} 
    \end{array} \right) \left(s,\theta^{t,\xi}_s,\lr\left(\theta^{t,\xi}_s\right),u^{t,\xi}_s\right)\left(\widetilde{\theta^{t,\xi}_s}\right) \right] \left(\text{\small $\begin{array}{cc}
         \widetilde{\x^{t,\xi}_s}  \\
         \widetilde{\y^{t,\xi}_s} 
    \end{array}$}\right) \right\}. \label{FBSDEs:high}
\end{align}
System \eqref{FBSDEs:high} seems to be complicated, but with $\Theta^{t,\xi}$ being given, it is indeed a linear system for $\left(\x^{t,\xi}_s,\y^{t,\xi}_s,\p^{t,\xi}_s,\R^{t,\xi}_s,\U^{t,\xi}_s\right)$. The equations for $\left(\x^{t,\xi}_s,\y^{t,\xi}_s\right)$ and equations for $\left(\p^{t,\xi}_s,\R^{t,\xi}_s\right)$ are coupled through $\U^{t,\xi}_s$, which is a feedback function of $\left(\x^{t,\xi}_s,\y^{t,\xi}_s,\p^{t,\xi}_s,\R^{t,\xi}_s,\lr\left(\x^{t,\xi}_s\right),\lr\left(\y^{t,\xi}_s\right)\right)$ in view of the last equality of \eqref{FBSDEs:high} and the strong convexity of $h$ in $v$. 

The following theorem provides the well-posedness of System \eqref{FBSDEs:high} and establishes the differentiability of $\Theta_s^{t,\xi}$ with respect to $t$. Its proof is given in Appendix~\ref{App:thm:convergence}. 

\begin{theorem}\label{thm:convergence}
    Under Assumptions (A1), (A2'), (A3) and (B), System \eqref{FBSDEs:high} has a unique $\f_t$-measurable solution $\left(\x^{t,\xi},\y^{t,\xi},\p^{t,\xi},\R^{t,\xi},\U^{t,\xi}\right)$, such that 
    \begin{align}
        &\e\bigg[\sup_{t\le s\le T}\left|\x^{t,\xi}_s\right|^2+\sup_{t\le s\le T}\left|\y^{t,\xi}_s\right|^2+\sup_{t\le s\le T}\left|\p^{t,\xi}_s\right|^2+\sup_{t\le s\le T}\left|\R^{t,\xi}_s\right|^2+\sup_{t\le s\le T}\left|\U^{t,\xi}_s\right|^2\bigg]\notag\\
        \le\ & C(L,T,\lambda) \left( 1+\left\|\xi\right\|_2^2\right). \label{lem:xypru_0}
    \end{align}   
    We define 
    \begin{equation}\label{def:de x^epsilon}
    \begin{aligned}
        &\de \x^\epsilon_s:=\x^\epsilon_s-\x^{t,\xi}_s,\quad \de \y^\epsilon_s:=\y^\epsilon_s-\y^{t,\xi}_s,\\
        &\de \p^\epsilon_s:=\p^\epsilon_s-\p^{t,\xi}_s,\quad \de \R^\epsilon_s:=\R^\epsilon_s-\R^{t,\xi}_s,\quad \de \U^\epsilon_s:=\U^\epsilon_s-\U^{t,\xi}_s,\quad s\in[t+\epsilon,T],
    \end{aligned}
    \end{equation}
    and we have
    \begin{equation}\label{derivative_t}
    \begin{aligned}
        \lim_{\epsilon\to 0}\ \bigg[&\sup_{t+\epsilon\le s\le T}\left\|\de\x_s^\epsilon\right\|_2^2+\sup_{t+\epsilon\le s\le T}\left\|\de\y_s^\epsilon\right\|_2^2\\
        &+\sup_{t+\epsilon\le s\le T}\left\|\de\p_s^\epsilon\right\|_2^2+\sup_{t+\epsilon\le s\le T}\left\|\de\R_s^\epsilon\right\|_2^2+\sup_{t+\epsilon\le s\le T}\left\|\de\U_s^\epsilon\right\|_2^2 \bigg] =0. 
    \end{aligned}
    \end{equation}
    That is, the solution of System \eqref{FBSDEs:high} gives the derivative of the map $t\mapsto\Theta^{t,\xi}$. 
\end{theorem}

\subsection{The HJB equation}

With the formula \eqref{value_equation} and the differentiability of the map $t\mapsto\Theta_s^{t,\xi}$, we now establish the differentiability of $V$ in $t$, which yields the HJB equation for Problem $\left(\mathsf{P}_{t,\xi}\right)$. Here, for the first-order problem, the associated Lagrangian functional $L$ (defined in \eqref{Lagrangian}) and the Hamiltonian function $H$ (defined in \eqref{def:Ham}) do not depend on the $z$ or $q$ arguments. For notational convenience, from now on, in this section, we drop the superscript $t,\xi$ in $\left(\x^{t,\xi},\y^{t,\xi},\p^{t,\xi},\R^{t,\xi},\U^{t,\xi}\right)$.

\begin{theorem}\label{thm:HJB}
    Under Assumptions (A1), (A2'), (A3) and (B), the value function $V$ is differentiable in time, and it satisfies the following HJB equation:
    \begin{equation}\label{HJB'}
    \left\{
    \begin{aligned}
        &\dd_t V(t,\lr(\xi))+\e \left[H\left(t,\xi,Y^{t,\xi}_t,\lr\left({ \xi,Y^{t,\xi}_t}\right),D_\xi V(t,\lr(\xi)),- \varphi_y \left(Y^{t,\xi}_t\right) \right) \right]=0, \quad t\in[0,T),\\
        &V(T,\lr(\xi))= \e\left[\Phi(\xi,\lr(\xi))+\varphi\left(g_1\xi+g_2\e[\xi]\right)\right].
    \end{aligned}
    \right.
    \end{equation}
\end{theorem}

\begin{proof}
In the proof, we simply drop the superscript $t,\xi$ of $\Theta^{t,\xi}$ whenever no ambiguity arises. For $\epsilon\in(0,1)$, from \eqref{value_equation} and Assumption (A2), with the notations in \eqref{def:x_s^epsilon}, we can write that
\begin{align}
    &\frac{V(t+\epsilon,\lr(\xi))-V(t,\lr(\xi))}{\epsilon} \notag\\
    =\ & \frac{1}{\epsilon}\ \e\bigg\{\int_{t+\epsilon}^T \left[h\left(s,\cdot\right)\Big|^{\left(\theta_s^{\epsilon},\lr\left( \theta_s^{\epsilon}\right),u^{\epsilon}_s\right)}_{\left(\theta_s,\lr\left( \theta_s\right),u_s\right)}\right] ds + \Phi \big|^{\left(X_T^{\epsilon},\lr\left(X_T^{\epsilon}\right)\right)}_{\left(X_T,\lr\left(X_T\right)\right)} +\varphi\big|^{Y_{t+\epsilon}^{\epsilon}}_{Y_{t+\epsilon}} + \varphi\left(Y_{t+\epsilon}\right)- \varphi\left(Y_t\right)  \bigg\}  \notag \\
    & -\frac{1}{\epsilon}\e\left[\int_t^{t+\epsilon} h\left(s,\theta_s,\lr\left({ \theta_s}\right),u_s\right) ds \right] \notag \\
    =\ & \e \int_{t+\epsilon}^T \int_0^1 \bigg\{ h_x\left(s,\theta_s^{\epsilon,\delta},\lr\left({ \theta_s^{\epsilon,\delta}}\right),u_s^{\epsilon,\delta}\right)^\top \x^\epsilon_s +h_y\left(s,\theta_s^{\epsilon,\delta},\lr\left({ \theta_s^{\epsilon,\delta}}\right),u_s^{\epsilon,\delta}\right)^\top \y^\epsilon_s \notag\\
    &\quad  +h_v\left(s,\theta_s^{\epsilon,\delta},\lr\left({ \theta_s^{\epsilon,\delta}}\right),u_s^{\epsilon,\delta}\right)^\top \U^\epsilon_s + \widetilde{\e}\bigg[D_{x'}\frac{dh}{d\nu}\left(s,\theta_s^{\epsilon,\delta},\lr\left({ \theta_s^{\epsilon,\delta}}\right),u_s^{\epsilon,\delta}\right)\left(\widetilde{ \theta_s^{\epsilon,\delta}}\right)^\top \widetilde{\x^\epsilon_s} \bigg] \notag\\
    &\quad + \widetilde{\e}\bigg[D_{y'}\frac{dh}{d\nu}\left(s,\theta_s^{\epsilon,\delta},\lr\left( \theta_s^{\epsilon,\delta}\right),u_s^{\epsilon,\delta}\right)\left(\widetilde{\theta_s^{\epsilon,\delta}}\right)^\top \widetilde{\y^\epsilon_s} \bigg]\bigg\} d\delta ds  \notag \\
    &+\e \int_0^1 \bigg\{ \Phi_x\left(X_T^{\epsilon,\delta},\lr\left(X_T^{\epsilon,\delta}\right)\right)^\top \x^\epsilon_T + \widetilde{\e}\bigg[D_{x'}\frac{d\Phi}{d\nu}\left(X_T^{\epsilon,\delta},\lr\left(X_T^{\epsilon,\delta}\right)\right)\left(\widetilde{X_T^{\epsilon,\delta}}\right)^\top \widetilde{\x^\epsilon_T} \bigg] \bigg\} d\delta \notag \\
    &- \frac{1}{\epsilon}\e\int_t^{t+\epsilon}  \varphi_y(Y_r)^\top f\left(r,\theta_r,\lr({ \theta_r}),u_r\right) dr  \notag\\
    &+ \e \int_0^1 \varphi_y\left(Y_{t+\epsilon}^{\epsilon,\delta}\right)^\top \y^\epsilon_{t+\epsilon} d\delta -\frac{1}{\epsilon}\e\left[\int_t^{t+\epsilon} h\left(s,\theta_s,\lr({ \theta_s}),u_s\right) ds \right], \notag 
\end{align}
where 
\begin{align*}
    &X^{\epsilon,\delta}_s:=X_s + \delta \left(X^{\epsilon}_s-X_s\right),\quad Y^{\epsilon,\delta}_s:=Y_s + \delta \left(Y^{\epsilon}_s-Y_s\right),\quad u^{\epsilon,\delta}_s:=u_s + \delta \left(u^{\epsilon}_s-u_s\right),
\end{align*}
and $\theta^{\epsilon,\delta}_s:=\left( X^{\epsilon,\delta}_s,Y^{\epsilon,\delta}_s\right)$. From Lemma~\ref{lem:de_bound},  Theorem~\ref{thm:convergence} and Estimates \eqref{lem:estimate_s_1}, \eqref{lem:estimate_s_2} and \eqref{zeta_R-2}, we know that the right-hand side of the last equality has a limit as $\epsilon$ goes to $0$, and it is 
\begin{align}
    &\lim_{\epsilon\to 0}\frac{V(t+\epsilon,\lr(\xi))-V(t,\lr(\xi))}{\epsilon} \notag\\
    =\ & \e \int_t^T \bigg\{ h_x\left(s,\theta_s,\lr({ \theta_s}),u_s\right)^\top \x_s  +h_y\left(s,\theta_s,\lr({ \theta_s}),u_s\right)^\top \y_s +h_v\left(s,\theta_s,\lr({ \theta_s}),u_s\right)^\top \U_s  \notag\\
    &\quad\qquad + \widetilde{\e}\bigg[D_{x'}\frac{dh}{d\nu}\left(s,\theta_s,\lr({ \theta_s}),u_s\right)\left(\widetilde{ \theta_s}\right)^\top \widetilde{\x_s} { + D_{y'}\frac{dh}{d\nu}\left(s,\theta_s,\lr( \theta_s),u_s\right)\left(\widetilde{\theta_s}\right)^\top \widetilde{\y_s}} \bigg] \bigg\}  ds \notag \\
    &+\e \bigg\{ \Phi_x\left(X_T,\lr\left(X_T\right)\right)^\top \x_T + \widetilde{\e}\bigg[D_{x'}\frac{d\Phi}{d\nu}\left(X_T,\lr\left(X_T\right)\right)\left(\widetilde{X_T}\right)^\top \widetilde{\x_T} \bigg] \bigg\} \notag \\
    &+\e \big[ \varphi_y(Y_t)^\top \left(\y_t - f\left(t,\xi,Y_t,\lr({ \xi,Y_t}),u_t\right) \right) - h\left(t,\xi,Y_t,\lr({ \xi,Y_t}),u_t \right)\big]. \label{D_tV-6}
\end{align}
Now we apply It\^o's formula on $P_s^\top \x_s$, take expectation and use Fubini's theorem to obtain
\begin{align}
    &\e\left[P_T^\top \x_T-P_t^\top \x_t\right]\notag\\
    =\ & \e\int_t^T \bigg\{ P_s^\top b_3(s)\U_s +\x_s^\top f_1(s)^\top  R_s + \x_s^\top f_4(s)^\top\e\left[R_s\right] -h_x\left(s,\theta_s,\lr\left({ \theta_s}\right),u_s\right)^\top \x_s \notag \\
    &\qquad\qquad  -\widetilde{\e} \bigg[D_{x'}\frac{dh}{d\nu} \left(s,\widetilde{\theta_s},\lr\left({ \theta_s}\right),\widetilde{u_s}\right)\left({ \theta_s}\right)^\top \x_s \bigg]  \bigg\} ds \notag\\
    =\ & \e\int_t^T \bigg\{ P_s^\top b_3(s)\U_s +R_s^\top f_1(s)\x_s + R_s^\top f_4(s)\e\left[\x_s\right] -h_x\left(s,\theta_s,\lr\left({ \theta_s}\right),u_s\right)^\top \x_s \notag \\
    &\qquad\qquad  -\widetilde{\e} \bigg[D_{x'}\frac{dh}{d\nu} \left(s,\theta_s,\lr\left({ \theta_s}\right),u_s\right)\left(\widetilde{ \theta_s}\right)^\top \widetilde{\x_s} \bigg]  \bigg\} ds. \label{D_tV-1}
\end{align}
Similarly, by again using It\^o's formula and Fubini's theorem on the term $R_s^\top \y_s$, we have
\begin{align}
    &\e\left[R_T^\top \y_T-R_t^\top \y_t\right]\notag\\
    =\ & \e\int_t^T \bigg\{ - R_s^\top f_1(s) \x_s - R_s^\top f_4(s) \e\left[ \x_s\right] - R_s^\top f_7(s) \U_s -h_y\left(s,\theta_s,\lr\left({ \theta_s}\right),u_s\right)^\top \y_s \notag\\
    & \qquad\qquad  -\widetilde{\e} \bigg[D_{y'}\frac{dh}{d\nu} \left(s,\theta_s,\lr\left( \theta_s\right),u_s\right)\left(\widetilde{ \theta_s}\right)^\top \widetilde{\y_s} \bigg] \bigg\} ds.  \label{D_tV-2}
\end{align}
Plugging \eqref{D_tV-1} with \eqref{D_tV-2} and using \eqref{FBSDEs:first}, we know that 
\begin{align}
    &\e\left[P_T^\top \x_T-P_t^\top \x_t + R_T^\top \y_T-R_t^\top \y_t \right]\notag\\
    =\ & \e\int_t^T \bigg\{ \left(P_s^\top b_3(s) - R_s^\top f_7(s)\right) \U_s -h_x\left(s,\theta_s,\lr\left({ \theta_s}\right),u_s\right)^\top \x_s -h_y \left(s,\theta_s,\lr\left({ \theta_s}\right),u_s\right)^\top \y_s \notag \\
    &\qquad\qquad  -\widetilde{\e} \bigg[D_{x'}\frac{dh}{d\nu} \left(s,\theta_s,\lr\left({ \theta_s}\right),u_s\right)\left(\widetilde{ \theta_s}\right)^\top \widetilde{\x_s} { + D_{y'}\frac{dh}{d\nu} \left(s,\theta_s,\lr\left( \theta_s\right),u_s\right)\left(\widetilde{ \theta_s}\right)^\top \widetilde{\y_s}}  \bigg]  \bigg\} ds \notag\\
    =\ & \e\int_t^T \bigg\{ -h_v\left(s,\theta_s,\lr\left({ \theta_s}\right),u_s\right)^\top \U_s-h_x\left(s,\theta_s,\lr\left({ \theta_s}\right),u_s\right)^\top \x_s -h_y \left(s,\theta_s,\lr\left({ \theta_s}\right),u_s\right)^\top \y_s \notag \\
    &\qquad\qquad   -\widetilde{\e} \bigg[D_{x'}\frac{dh}{d\nu} \left(s,\theta_s,\lr\left({ \theta_s}\right),u_s\right)\left(\widetilde{ \theta_s}\right)^\top \widetilde{\x_s} { + D_{y'}\frac{dh}{d\nu} \left(s,\theta_s,\lr\left( \theta_s\right),u_s\right)\left(\widetilde{ \theta_s}\right)^\top \widetilde{\y_s}}  \bigg]  \bigg\} ds. \label{D_tV-3}
\end{align}
On the other hand, by using the terminal conditions for $P_T$ and $\y_T$ and the initial conditions for $\x_t$ and $R_t$, and then using Fubini's theorem again, we know that 
\begin{align}
    &\e\left[P_T^\top \x_T-P_t^\top \x_t + R_T^\top \y_T-R_t^\top \y_t \right]\notag\\
    =\ &\e\bigg\{ \Phi_x\left(X_T, \lr\left(X_T\right) \right)^\top \x_T +\widetilde{\e}\left[D_{x'}\frac{d\Phi}{d\nu} \left(\widetilde{X_T}, \lr\left(X_T\right)\right)\left(X_T\right)^\top \x_T \right] \notag \\
    &\quad +P_t^\top \left[b_0(t)+b_1(t) \xi + b_2(t)\e\left[\xi\right]+b_3(t)u_t\right] +\y_t^\top  \varphi_y \left(Y_t\right)\bigg\}\notag\\
    =\ & \e\bigg\{ \Phi_x\left(X_T, \lr\left(X_T\right) \right)^\top \x_T +\widetilde{\e}\left[D_{x'}\frac{d\Phi}{d\nu} \left(X_T, \lr\left(X_T\right)\right)\left(\widetilde{X_T}\right)^\top \widetilde{\x_T} \right] \notag \\
    &\quad +P_t^\top \left[b_0(t)+b_1(t) \xi + b_2(t)\e\left[\xi\right]+b_3(t)u_t\right] + \varphi_y \left(Y_t\right)^\top \y_t \bigg\}. \label{D_tV-4}
\end{align}
Substituting \eqref{D_tV-3} and \eqref{D_tV-4} into \eqref{D_tV-6}, we have
\begin{align*}
    &\lim_{\epsilon\to 0}\frac{V(t+\epsilon,\lr(\xi))-V(t,\lr(\xi))}{\epsilon} \notag\\
    =\ & \e \big[ -P_t^\top b\left(t,\xi,\lr(\xi),u_t\right) - \varphi_y(Y_t)^\top f\left(t,\xi,Y_t,\lr({ \xi,Y_t}),u_t\right) - h\left(t,\xi,Y_t,\lr({ \xi,Y_t}),u_t\right)\big] \\
    =\ & \e \big[ -P_t^\top b\left(t,\xi,\lr(\xi),u_t\right) + R_t^\top f\left(t,\xi,Y_t,\lr({ \xi,Y_t}), u_t\right) - h\left(t,\xi,Y_t,\lr({ \xi,Y_t}),u_t\right)\big].
\end{align*}
Then, by the definition of the Lagrangian functional $L$ in \eqref{Lagrangian}, the definition of the Hamiltonian function $H$, and the optimality condition of $u_t$, we know that 
\begin{align*}
    \dd_t V(t,\lr(\xi))= -\e \big[L\left(t,\xi,Y_t,\lr({ \xi,Y_t}),u_t,P_t,R_t\right) \big]=-\e \big[H\left(t,\xi,Y_t,\lr({ \xi,Y_t}),P_t,R_t\right) \big].
\end{align*}
Using the initial condition of $R_t$ and the relation \eqref{DV_P}, we can also write 
\begin{align}
    \dd_t V(t,\lr(\xi))=-\e \left[H\left(t,\xi,Y_t,\lr({ \xi,Y_t}),D_\xi V(t,\lr(\xi)),- \varphi_y \left(Y_t\right) \right) \right]. \label{D_tV-9}
\end{align}
For the terminal condition for $V$, from \eqref{value_equation} and the terminal condition for $Y^{T,\xi}_T$, we see that 
\begin{align*}
    V(T,\lr(\xi))=\ & \e\left[\Phi(\xi,\lr(\xi))+\varphi\left(Y_T^{T,\xi}\right)\right] =\e\left[\Phi(\xi,\lr(\xi))+\varphi\left(g_1\xi+g_2\e[\xi]\right)\right],
\end{align*}
with which we obtain \eqref{HJB'}. 
\end{proof}

Now we denote by $\mu:=\lr(\xi)$, and want to write the HJB equation \eqref{HJB'} with $\mu$ rather that $\xi$. With the first-order setting (B), in the system for $\Theta^{t,\xi}_s$, all the randomness comes from the initial condition $\xi$. Therefore, the processes $\Theta^{t,\xi}_s$ can be viewed as some deterministic functions of $\xi$, which may depend on the law $\mu$. We denote by these deterministic functions as 
\begin{align*}
    \Theta^{t,\cdot,\mu}_s:x\mapsto \Theta^{t,x,\mu}_s,\quad x\in\brn,
\end{align*}
which should satisfy $\Theta^{t,x,\mu}_s \big|_{x=\xi}=\Theta^{t,\xi}_s$. Indeed, from the forward-backward equation for $\Theta^{t,\xi}$, the deterministic processes $\Theta^{t,x,\mu}_s=\left(X^{t,x,\mu}_s,Y^{t,x,\mu}_s,P^{t,x,\mu}_s,R^{t,x,\mu}_s,u^{t,x,\mu}_s\right)$ can be characterized as the solutions of the following forward-backward deterministic ODEs:
\begin{equation}\label{FBSDEs:x}
    \left\{
        \begin{aligned}
            &X^{t,x,\mu}_s= x+\int_t^s \left[b_0(r)+ b_1(r) X^{t,x,\mu}_r + b_2(r)\e\left[X^{t,\xi}_r\right]+b_3(r)u^{t,x,\mu}_r\right]   dr,\\
            &Y^{t,x,\mu}_s= g_1 X^{t,x,\mu}_T+ g_2 \e\left[ X^{t,\xi}_T\right]  + \int_s^T  f\left(r, X^{t,x,\mu}_r, Y^{t,x,\mu}_r , \lr\left({ X^{t,\xi}_r,Y^{t,\xi}_r}\right),  u^{t,x,\mu}_r \right) dr,\\
            & P^{t,x,\mu}_s= - g_1^\top R^{t,x,\mu}_T - g_2^\top \e\left[ R^{t,\xi}_T\right]  +\Phi_x\left(X^{t,x,\mu}_T, \lr\left(X^{t,\xi}_T\right) \right) +{\e}\left[D_{x'}\frac{d\Phi}{d\nu} \left({X^{t,\xi}_T}, \lr\left(X^{t,\xi}_T\right)\right)\left(X^{t,x,\mu}_T\right)  \right]  \\
            &\quad\qquad + \int_s^T \bigg\{  b_1(r)^\top  P^{t,x,\mu}_r - f_1(r)^\top  R^{t,x,\mu}_r +h_x\left(r, \theta^{t,x,\mu}_r, \lr\left( \theta^{t,\xi}_r\right),  u^{t,x,\mu}_r  \right) \\
            &\qquad\qquad\qquad  + {\e} \bigg[b_2(r)^\top {P^{t,\xi}_r} - f_4(r)^\top {R^{t,\xi}_r} +D_{x'}\frac{dh}{d\nu} \left(r, \theta^{t,\xi}_r, \lr\left(\theta^{t,\xi}_r\right),  u^{t,\xi}_r  \right)\left({ \theta^{t,x,\mu}_r}\right) \bigg]  \bigg\} dr, \\
            & R^{t,x,\mu}_s= - \varphi_y \left(Y^{t,x,\mu}_t\right) +\int_t^s \bigg\{ f_2(r)^\top  R^{t,x,\mu}_r  -h_y \left(r, \theta^{t,x,\mu}_r, \lr\left( \theta^{t,\xi}_r\right),  u^{t,x,\mu}_r  \right) \\
            &\qquad\qquad\qquad\qquad\qquad\qquad  +\e\left[f_5(r)^\top R^{t,\xi}_r-D_{y'}\frac{dh}{d\nu} \left(r, \theta^{t,\xi}_r, \lr\left(\theta^{t,\xi}_r\right),  u^{t,\xi}_r  \right)\left(\theta^{t,x,\mu}_r\right)\right]\bigg\} dr, \\
            &b_3(s)^\top  P^{t,x,\mu}_s -f_7(s)^\top  R^{t,x,\mu}_s + h_v\left(s, \theta^{t,x,\mu}_s, \lr\left(\theta^{t,\xi}_s\right),  u^{t,x,\mu}_s \right) =0,\quad s\in[t,T],
        \end{aligned}
    \right.
\end{equation}
where $\theta^{t,x,\mu}_s:=\left(X^{t,x,\mu}_s,Y^{t,x,\mu}_s\right)$. System \eqref{FBSDEs:x} is a deterministic system with initial condition $(t,x)$, and its coefficients also depend on $\Theta^{t,\xi}$; moreover, we see that \eqref{FBSDEs:x} depends on $\xi$ only through the law $\mu$. The following result shows that the processes $\Theta^{t,x,\mu}$ are well-defined, whose proof is given in Appendix~\ref{App:lem:tmux}.

\begin{lemma}\label{lem:tmux}
    Under Assumptions (A1), (A2'), (A3) and (B), there is a unique solution $\Theta^{t,x,\mu}$ of System \eqref{FBSDEs:x} satisfying 
    \begin{align}\label{tmux-0'}
        &\sup_{t\le s\le T} \left|\left(X^{t,x,\mu}_s,Y^{t,x,\mu}_s,P^{t,x,\mu}_s,R^{t,x,\mu}_s,u^{t,x,\mu}_s\right)\right|^2 \le  C(L,T,\lambda)\left(1+|x|^2+|\mu|_2^2\right); 
    \end{align}    
    and for any $x_1,x_2\in\brn$,
    \begin{align}\label{tmux-0}
        &\sup_{t\le s\le T}\left|X^{t,x_2,\mu}_s-X^{t,x_1,\mu}_s\right|^2 + \sup_{t\le s\le T}\left|Y^{t,x_2,\mu}_s-Y^{t,x_1,\mu}_s\right|^2 + \sup_{t\le s\le T}\left|P^{t,x_2,\mu}_s-P^{t,x_1,\mu}_s\right|^2 \notag \\
        &+ \sup_{t\le s\le T}\left|R^{t,x_2,\mu}_s-R^{t,x_1,\mu}_s\right|^2 + \sup_{t\le s\le T}\left|u^{t,x_2,\mu}_s-u^{t,x_1,\mu}_s\right|^2 \le  C(L,T,\lambda)  |x_2-x_1|^2. 
    \end{align}
    And we have $\Theta^{t,x,\mu}\Big|_{x=\xi}=\Theta^{t,\xi}$. As a corollary, we have $D_x\frac{dV}{d\nu}(t,\mu)(\cdot)=P^{t,\cdot,\mu}_t,\ \mu-a.e.$.
\end{lemma}

With Lemma~\ref{lem:tmux}, we know that $Y^{t,\xi}_t=Y^{t,x,\mu}_t\big|_{x=\xi}$, and therefore
\begin{align*}
    \lr\left(\xi, Y^{t,\xi}_t\right)= \left(I_n\otimes Y^{t, \cdot,\mu}_t\right) \#\mu,
\end{align*}
where $\left(I_n\otimes Y^{t, \cdot,\mu}_t\right) \#\mu$ represents the push-forward of $\mu$ by the map $x\mapsto (x,Y^{t,x,\mu}_t)$. As a direct consequence of Theorem~\ref{thm:HJB}, Lemma~\ref{lem:tmux} and the relation \eqref{relation:DV-dV}, we have the following result. 

\begin{theorem}\label{thm:HJB'}
    Under Assumptions (A1), (A2'), (A3) and (B), the value function $V$ is a classical solution of the following HJB equation on $[0,T]\times\pr_2(\brn)$: 
    \begin{equation}\label{HJB}
    \left\{
    \begin{aligned}
        &\dd_t V(t,\mu)+\int_\brn H\left(t,x,Y^{t,x,\mu}_t,\left(I_n\otimes Y^{t, \cdot,\mu}_t\right) \#\mu,D_{x}\frac{dV}{d\nu} (t,\mu)(x),- \varphi_y \left(Y^{t,x,\mu}_t\right) \right)\mu(dx)=0, \\
        &V(T,\mu)= \int_\brn  \left[\Phi(x,\mu)+\varphi\left(g(x,\mu)\right)\right] \mu(dx),
    \end{aligned}
    \right.
    \end{equation}
    in the sense that $V$ is differentiable in $t$, linearly functionally differentiable in $\mu$, with $x\mapsto \frac{dV}{d\nu} (t,\mu)(x)$ differentiable, and Estimates \eqref{V:growth} and \eqref{DV:growth} hold. 
\end{theorem}

The HJB equation \eqref{HJB} includes the HJB equation for first-order MFTC problems (see, e.g., \cite{AB6}) as a special case; indeed, \eqref{HJB} reduces to the standard mean-field type HJB equation when $f$, $g$, and $\varphi$ vanish.  

\subsection{Higher-order regularity of $Y^{t,x,\mu}_t$}\label{subsec:U}

The HJB equation \eqref{HJB} inspires us to study the map $U:[0,T]\times\brn\times\pr_2(\brn)\to\brn$ defined as 
\begin{align}\label{def:map_U}
    U(t,x,\mu):= Y^{t,x,\mu}_t,
\end{align}
with which the HJB equation \eqref{HJB} can also be written as
\begin{equation*}\label{HJB(U)}
    \left\{
    \begin{aligned}
        &\dd_t V(t,\mu)+\int_\brn H\left(t,x,U(t,x,\mu),\left(I_n\otimes U(t,\cdot,\mu)\right) \#\mu,D_{x}\frac{dV}{d\nu} (t,\mu)(x),- \varphi_y \circ U (t,x,\mu)\right)\mu(dx)=0, \\
        &V(T,\mu)= \int_\brn  \left[\Phi(x,\mu)+\varphi\left(g(x,\mu)\right)\right] \mu(dx).
    \end{aligned}
    \right.
\end{equation*}
From the well-posedness of System \eqref{FBSDEs:x} for $\Theta^{t,x,\mu}$, we know that the map $U$ is well-defined; and Estimate \eqref{tmux-0} shows that $U$ is Lipschitz-continuous in $x$. Moreover, we can also give the differentiability of the map $U$ by studying that of the map $(x,\mu)\mapsto\Theta^{t,x,\mu}$. We first give the derivatives in $x$. 

\begin{lemma}\label{lem:D_x}
    Under Assumptions (A1), (A2'), (A3) and (B), the map $\brn\ni x\mapsto\Theta^{t,x,\mu}_s\in(\brn)^4\times\brd$ is differentiable, with the derivatives $D_x\Theta^{t,x,\mu}=\left(D_x X^{t,x,\mu},D_x Y^{t,x,\mu},D_x P^{t,x,\mu},D_x R^{t,x,\mu},D_x u^{t,x,\mu}\right)$ satisfying
    \begin{align}\label{tmux_Dx}
        &\sup_{(s,x)\in[t,T]\times\brn} \left|\left(D_x X^{t,x,\mu}_s,D_x Y^{t,x,\mu}_s,D_x P^{t,x,\mu}_s,D_x R^{t,x,\mu}_s,D_x u^{t,x,\mu}_s\right)\right|^2 \le  C(L,T,\lambda). 
    \end{align}
\end{lemma}

The idea of the proof is to characterize $D_x\Theta^{t,x,\mu}$ as the solution of a new system of forward-backward ODEs, obtained by formally differentiating System \eqref{FBSDEs:x} with respect to $x$. A rigorous proof requires establishing the well-posedness of this new system and proving the convergence of the difference quotient to its solution. A sketch of the proof is provided in Appendix~\ref{App:lem:D_x}.

To establish the differentiability of the map $\mu\mapsto \Theta^{t,x,\mu}$, the idea is to  first use the lifting method (see \cite{CP,cardaliaguet2015weak}) to study the G\^ateaux derivatives of $\xi\mapsto \Theta^{t,x,\lr(\xi)}$ along some direction $\eta$; these derivatives are then used to obtain the linear functional derivatives of $\Theta^{t,x,\mu}$ with respect to $\mu$.

\begin{lemma}\label{Lem:D_xi}
    Under Assumptions (A1), (A2'), (A3) and (B), the map $\mu\mapsto \Theta^{t,x,\mu}$ is linearly functionally differentiable, with its derivative $\frac{d\Theta^{t,x,\mu}}{d\nu}(x')$ being differentiable in $x'$ and satisfying 
    \begin{align}
        \sup_{(s,x')\in[t,T]\times\brn} \bigg|\bigg(& D_{x'}\frac{dX^{t,x,\mu}_s}{d\nu}(x'),D_{x'}\frac{dY^{t,x,\mu}_s}{d\nu}(x'),D_{x'}\frac{dP^{t,x,\mu}_s}{d\nu}(x'),D_{x'}\frac{dR^{t,x,\mu}_s}{d\nu}(x'), \notag \\
        & D_{x'}\frac{du^{t,x,\mu}_s}{d\nu}(x')\bigg)\bigg|^2 \le  C(L,T,\lambda).  \label{tmux_Dmu}
    \end{align}
    For any $\eta\in L_{\f_t}^2$ independent of $\xi$, the directional derivative of $\xi\mapsto \Theta^{t,x,\lr(\xi)}$ along $\eta$, denoted by $D_\eta  \Theta^{t,x,\xi}$, satisfies
    \begin{align}\label{relation:D_eta-D_mu}
        D_\eta \Theta_s^{t,x,\xi}=\e\left[D_{x'}\frac{d\Theta^{t,x,\mu}_s}{d\nu}\left({\xi}\right) {\eta} \right].
    \end{align}  
\end{lemma}

A sketch of proof is given in Appendix~\ref{App:Lem:D_xi}. The idea is to characterize $D_\eta  \Theta^{t,\xi}$ (the derivative of the map $\xi\mapsto \Theta^{t,\xi}$ along the direction $\eta$) and $D_\eta  \Theta^{t,x,\xi}$ as the solutions of two auxiliary systems of forward-backward equations, respectively. The key step in passing from directional derivative to the linear functional derivative is the identity
\begin{align}\label{relation:D_xi}
    D_\eta \Theta^{t,\xi} = \dd_\epsilon \left(\Theta^{t,x,\lr\left(\xi+\epsilon\eta\right)}\big|_{x=\xi+\epsilon\eta}\right) \Big|_{\epsilon=0} = D_x \Theta^{t,x,\lr(\xi)} \big|_{x=\xi}\ \eta + D_\eta  \Theta^{t,x,\xi} \Big|_{x=\xi},
\end{align}
which follows directly from  the relation $\Theta^{t,\xi}=\Theta^{t,x,\lr(\xi)}\big|_{x=\xi}$ in Lemma~\ref{lem:tmux}. 

Based on Lemmas~\ref{lem:D_x} and \ref{Lem:D_xi}, we give the following regularity result for the map $U$ defined in \eqref{def:map_U} with respect to $x$ and $\mu$, and also an estimate for $D_x\frac{dV}{d\nu}$. An explicit formulation of $\frac{dV}{d\nu}(t,\mu)(x)$ is also given in Appendix~\ref{App:prop:U_x-mu}. 

\begin{proposition}\label{prop:U_x-mu}
    Under Assumptions (A1), (A2'), (A3) and (B), the map $U(t,x,\mu)$ defined in \eqref{def:map_U} is differentiable in $x$ and linearly functional-differentiable in $\mu$, and the derivatives satisfy the relations
    \begin{equation}\label{prop:U_x-mu_1}
        U_x(t,x,\mu)=D_x Y^{t,x,\mu}_t,\qquad D_{x'}\frac{dU}{d\nu}(t,x,\mu)(x')=D_{x'}\frac{dY^{t,x,\mu}_t}{d\nu}(x').
    \end{equation}
    Moreover, the following estimates for $U$ hold
    \begin{align}
        & |U(t,x,\mu)|\le C(L,T,\lambda) \left(1+|x|+|\mu|_2\right); \label{prop:U_x-mu_2'}\\ 
        & |U_x(t,x,\mu)|,\quad \left|D_{x'}\frac{dU}{d\nu}(t,x,\mu)(x')\right|\le C(L,T,\lambda); \label{prop:U_x-mu_2}
    \end{align}
    and the following estimate for $V$ holds
    \begin{equation}\label{prop:DdV_x}
        \begin{aligned}
            &\left|D_x \frac{dV}{d\nu}(t,\mu)(x) \right|\le C(L,T,\lambda)\left(1+|x|+|\mu|_2\right).
        \end{aligned}
    \end{equation}
\end{proposition}

To proceed to study the differentiability of $U(t,x,\mu)$ in $t$, we follow the same procedure as in Subsection~\ref{subsec:auxiliary} and consider the limit of $\frac{1}{\epsilon}\left(\Theta_s^{t+\epsilon,x,\mu}-\Theta_s^{t,x,\mu}\right)$ for $s\in(t+\epsilon,T)$ as $\epsilon\to 0$. We have the following result.

\begin{lemma}\label{lem:cirly_X_mu}
    Under Assumptions (A1), (A2'), (A3) and (B), there exists 
    \begin{align*}
        \left(\x^{t,x,\mu},\y^{t,x,\mu},\p^{t,x,\mu},\R^{t,x,\mu},\U^{t,x,\mu}\right):[t,T]\to (\brn)^4\times \brd
    \end{align*}
    such that the following convergence holds
    \begin{equation}\label{derivative_t'}
    \begin{aligned}
        \lim_{\epsilon\to 0}\ \bigg[&\sup_{t+\epsilon\le s\le T}\left|\left(\delta \x^\epsilon_s,\ \delta\y_s^\epsilon,\ \delta\p_s^\epsilon,\ \delta\R_s^\epsilon,\ \delta\U_s^\epsilon\right)\right|_2^2 \bigg] =0. 
    \end{aligned}
    \end{equation}
    where $\delta \x^\epsilon_s:=\frac{1}{\epsilon}\left(X_s^{t+\epsilon,x,\mu}-X_s^{t,x,\mu}\right) - \x^{t,x,\mu}_s$, and the other terms are defined analogously. Moreover, the following estimate holds:
    \begin{align}
        \sup_{t\le s\le T}\left|\left(\x^{t,x,\mu}_s,\y^{t,x,\mu}_s,\p^{t,x,\mu}_s,\R^{t,x,\mu}_s,\U^{t,x,\mu}_s\right)\right|^2 \le C(L,T,\lambda) \left( 1+|x|^2+|\mu|^2_2\right), \label{lem:xypru_0_txmu}
    \end{align}    
\end{lemma}

Following the same strategy as in the previous arguments of this subsection, we characterize the limit  $\left(\x^{t,x,\mu},\y^{t,x,\mu},\p^{t,x,\mu},\R^{t,x,\mu},\U^{t,x,\mu}\right)$ as the unique solution of a system of deterministic ODEs (similar to \eqref{FBSDEs:high}). The explicit equations for the limiting process are given in Appendix~\ref{App:lem:cirly_X_mu}.

We are now ready to state the following regularity result for the map $U$ defined in \eqref{def:map_U} with respect to the time variable.

\begin{proposition}\label{prop:dt_U}
    Under Assumptions (A1), (A2'), (A3) and (B), the map $U(t,x,\mu)$ defined in \eqref{def:map_U} is differentiable in $t$, and the derivatives satisfy the relations
    \begin{equation}\label{prop:dt_U_01}
    \begin{aligned}
        \dd_t U(t,x,\mu)=\y^{t,x,\mu}_t - f\left(t, x, U(t,x,\mu), \left(I_n\otimes U(t,\cdot,\mu)\right) \#\mu, u^{t,x,\mu}_t\right), 
    \end{aligned}
    \end{equation}
    where 
    \begin{equation*}\label{prop:dt_U_01'}
        u^{t,x,\mu}_t = \hv\left(t,x,U(t,x,\mu),\left(I_n\otimes U(t,\cdot,\mu)\right) \#\mu, D_{x}\frac{dV}{d\nu}(t,\mu)(x), -\varphi_y\circ U(t,x,\mu) \right).
    \end{equation*}
    Moreover, the following estimate holds
    \begin{equation}\label{prop:dt_U_02}
        \left|\dd_t U(t,x,\mu)\right|\le C(L,T,\lambda) \left( 1+|x|+|\mu|_2\right). 
    \end{equation}
\end{proposition}

\begin{proof}
For $\epsilon\in(0,1)$, by using the definition of $U$ and the equation for $Y^{t,x,\mu}_s$, we have
\begin{align*}
    \frac{U(t+\epsilon,x,\mu)-U(t,x,\mu)}{\epsilon}=\ &  \frac{1}{\epsilon}\left(Y^{t+\epsilon,x,\mu}_{t+\epsilon}- Y^{t,x,\mu}_t\right)\\
    =\ & \frac{1}{\epsilon}\e\left[Y^{t+\epsilon,x,\mu}_{t+\epsilon}- Y^{t,x,\mu}_{t+\epsilon}\right] + \frac{1}{\epsilon}\e\left[Y^{t,x,\mu}_{t+\epsilon}- Y^{t,x,\mu}_{t}\right] \notag \\
    =\ & \frac{1}{\epsilon}\e\left[Y^{t+\epsilon,x,\mu}_{t+\epsilon}- Y^{t,x,\mu}_{t+\epsilon}\right] - \frac{1}{\epsilon}\e\int_t^{t+\epsilon} f \left(r,\theta^{t,x,\mu}_r,\lr\left(\theta^{t,\xi}_r\right),u^{t,x,\mu}_r\right)dr.
\end{align*}
Then, from the differentiability of $Y^{t,x,\mu}_s$ in $t$ in Lemma~\ref{lem:cirly_X_mu}, and the continuity of $\Theta^{t,x,\mu}_s$ in $s$, we know that $U(t,x,\mu)$ is differentiable in $t$, and its derivative satisfies \eqref{prop:dt_U_01}. From the last equation in \eqref{FBSDEs:x}, the definition of the map $\hv$ and Assumption (B), we know that
\begin{align*}
    u^{t,x,\mu}_t=\hv \left(t,x,Y^{t,x,\mu}_t, \lr\left(\xi,Y^{t,\xi}_t\right),P^{t,x,\mu}_t,R^{t,x,\mu}_t\right).
\end{align*}
Substituting \eqref{DV_P}, \eqref{def:map_U} and the initial condition for $R^{t,x,\mu}_t$ into the last equation, we obtain the formulation of $u^{t,x,\mu}_t$. Then, \eqref{prop:dt_U_02} is a direct consequence of the growth conditions for the map $\hv$ and the coefficients $f$, $\varphi_y$, $D_{x}\frac{dV}{d\nu}$ and Estimate \eqref{prop:U_x-mu_2'}.
\end{proof}

Equation \eqref{prop:dt_U_01} gives the derivative of the map $U$, but it does not constitute a PDE for $U$, since we cannot yet characterize $\y^{t,x,\mu}_t$ in terms of $U$. Indeed, a PDE for $U$ cannot generally be derived for control problems of FBSDEs (especially when the dimension of the backward equation is larger than one) due to the lack of the flow property; see Section~\ref{sec:flow} below for a detailed discussion. In Section~\ref{sec:flow}, we will also provide cases where the flow property holds true for second-order problems (under certain structural assumptions on the coefficient $h$), and derive the coupled PDEs for the maps $(V,U)$ (see Theorem~\ref{thm:flow:HJB:main}).

\section{Second-order HJB equation: a feedback case}\label{sec:feedback}

In this section, we drop the assumption $\sigma\equiv 0$ made in the previous section and return to the second-order problem in the presence of Brownian motion. Instead, we impose the following assumption, which requires that the diffusion coefficient be independent of the control.

{\bf (A4)} The diffusion is independent of control, i.e., $\sigma_3=0$. The coefficients $b_i$, $\sigma_j$, $f_k$ (with $0\le i\le 3$, $0\le j\le 2$ and $0\le k\le 7$) are $\frac{1}{2}$-H\"older continuous in time; and the coefficient $h$ satisfies the same continuity condition in the time argument as in Assumption (B). 

Note that, under Assumption (A4), the minimizer map $\hv(\cdot)$ defined in \eqref{minimizer} is independent of the $q$ argument. Since $q$ no longer appears in the minimization, we define the reduced Hamiltonian $\hr$ by suppressing the $q$-dependence:
\begin{align}\label{def:Ham'}
    \hr(s,x,y,z,\varrho,p,r):=\inf_{v\in\brd} \left\{b\left(s,x,{ \varrho_x},v\right)^\top p -f(s,x,y,z,{ \varrho},v)^\top r + h(s,x,y,z,{ \varrho},v)\right\}.
\end{align}
In view of the definition of the full Hamiltonian $H$ in \eqref{def:Ham}, we have the decomposition
\begin{align*}
    H(s,x,y,z,\varrho,p,q,r)=\hr(s,x,y,z,\varrho,p,r)+ \sum_{j=1}^n\sigma^j\left(s,x, \varrho_x \right)^\top q^j.
\end{align*}

In this section, we consider a special case in which the process $Y^{t,\xi,v}$ admits the following feedback representation in terms of $X^{t,\xi,v}$:
\begin{align}\label{feedback:Y}
    Y^{t,\xi,v}_s= F\left(s,X^{t,\xi,v}_s\right),\quad s\in[t,T],
\end{align}
for some map $F:[0,T]\times\brn\to\brn$ that is $C^1$ in time and $C^2$ in space. In this simpler setting, the control problem $\left(\mathsf{P}_{t,\xi}\right)$ defined in \eqref{FBSDE^v(Introduction)}-\eqref{cost^v} can be reformulated as a standard MFTC problem for a controlled SDE (rather than for controlled FBSDEs). We shall also characterize the role of the adjoint processes $\left(P^{t,\xi},Q^{t,\xi},R^{t,\xi}\right)$, and provide an alternative derivation of the associated HJB equation.

Under Assumption (A1), applying It\^o's formula to $F\left(s,X^{t,\xi,v}_s\right)$ and comparing the resulting drift and diffusion terms with the BSDE for $Y^{t,\xi,v}$, we see that $F$ should be linear in $x$, i.e., $F(s,x)=F_0(s)+F_1(s)x$. More precisely, the maps $F_0$, $F_1$, and the coefficients $f$ and $g$ are required to satisfy the following structural condition, ensuring consistency with the feedback ansatz \eqref{feedback:Y}:

{\bf (C)} The functions $F_0:[0,T]\to\brn$ and $F_1:[0,T]\to\brnn$ are bounded by $L$, differentiable, and their derivatives $\dot F_0$, $\dot F_1$ are bounded by $L$ and $\frac{1}{2}$-H\"older continuous. Moreover, $F_0(T)=0$. The coefficients $f$ and $g$ are assumed to satisfy
\begin{equation}\label{relation:f-F}
\begin{aligned}
    &f_0(s)=-\dot{F}_0(s)- F_1(s) b_0(s),\quad f_1(s)= -\dot{F}_1(s)-F_1(s) b_1(s),\quad f_4(s)=-F_1(s) b_2(s),\\
    &f_7(s)=-F_1(s) b_3(s), \quad f_2(s)=0,\quad f_3(s)=0,\quad f_5(s)=0,\quad f_6(s)=0;\\
    &g_1=F_1(T),\quad g_2=0.
\end{aligned}
\end{equation}

It is straightforward to verify that the coefficients $f$ and $g$ in Assumption (C) satisfy Assumption (A1). The relations \eqref{relation:f-F} guarantee the consistency of the BSDE for $Y^{t,\xi,v}$ with the feedback representation \eqref{feedback:Y}. Moreover, by the uniqueness of the BSDE solution, we obtain the following relation for $Z^{t,\xi,v}_s$: 
\begin{equation}\label{relation:YZ}
    \begin{aligned}
        Z^{t,\xi,v}_s =\ &  F_1(s)\; \sigma\left(s,X^{t,\xi,v}_s,\lr\left(X^{t,\xi,v}_s\right)\right).
    \end{aligned}
\end{equation}

\subsection{Reformulation as a mean-field type control problem}

With the relation \eqref{feedback:Y}, the control problem $\left(\mathsf{P}_{t,\xi}\right)$ can be reformulated as follows:
\begin{equation}\label{cost^v'}
    \begin{aligned}
        &X^{t,\xi,v}_s= \xi+\int_t^s b \left(r,X^{t,\xi,v}_r, \lr\left(X^{t,\xi,v}_r\right), v_r \right)dr + \int_t^s \sigma\left(r,X^{t,\xi,v}_r, \lr\left(X^{t,\xi,v}_r\right) \right)dB_r,\\
        &J^{t,\xi}(v):= \e\bigg[\varphi\Big(F_0(t)+F_1(t)\xi\Big)+\int_t^T \mathbf{h}\left(s,X^{t,\xi,v}_s, \lr\left(X^{t,\xi,v}_s\right), v_s \right) ds  + \Phi\left(X_T^{t,\xi,v},\lr\left(X_T^{t,\xi,v}\right)\right)\bigg],
    \end{aligned}
\end{equation}
where the functional $\mathbf{h}$ is defined as
\begin{align}\label{def:hh}
    &\mathbf{h} (s,x,m,v):= h\Big(s,x, F(s,x), F_1(s)\sigma(s,x,m), \left(I_n\otimes F(s,\cdot)\otimes F_1(s)\sigma(s,\cdot,m)\right)\#m, v \Big).
\end{align}
Here, $\left(I_n\otimes F(s,\cdot)\otimes F_1(s)\sigma(s,\cdot,m)\right) \# m$ represents the push-forward of $m\in\pr_2(\brn)$ by the map $x\mapsto \left(x,F(s,x),F_1(s)\sigma(s,x,m)\right)$. We have the following properties for $\mathbf{h}$. 

\begin{lemma}\label{lem:der_h}
    Under Assumptions (A1), (A2) and (C), the function $\mathbf{h}$ defined in \eqref{def:hh} is continuous, and is differentiable in $(x,v)$ as well as linearly functionally differentiable in $m\in\pr_2(\brn)$, with $x'\mapsto \frac{d\mathbf{h}}{d\nu}(s,x,m,v)(x')$ being differentiable. The derivatives $\mathbf{h}_x$, $\mathbf{h}_v$, and $D_{x'}\frac{d\mathbf{h}}{d\nu}$ are $C(L)$-Lipschitz continuous in $(x,m,v,x')$. Moreover, under Assumption (A3), $\mathbf{h}$ is convex in the following sense: for any $s\in[0,T]$, $x^1,x^2\in\brn$, $v^1,v^2\in\brd$, and $\xi^1,\xi^2\in L^2(\Omega,\f,\mathbb{P};\brn)$,
    \begin{equation}\label{convex:h}
    \begin{aligned}
        &\mathbf{h}\left(s,x^2,\lr\left(\xi^2\right),v^2\right)-\mathbf{h}\left(s,x^1,\lr\left(\xi^1\right),v^1\right)\\
        \geq\ & \mathbf{h}_v \left(s,x^1,\lr\left(\xi^1\right),v^1\right)^\top \left(v^2-v^1\right) + \mathbf{h}_x \left(s,x^1,\lr\left(\xi^1\right),v^1\right)^\top \left(x^2-x^1\right) \\
        &+ \e\left[\left(D_{x'}\frac{d \mathbf{h}}{d\nu}\left(s,x^1,\lr\left(\xi^1\right),v^1\right)\left(\xi^1\right)\right)^\top \left(\xi^2-\xi^1\right)\right] +\lambda\left|v^2-v^1\right|^2;
    \end{aligned}
    \end{equation}
    and the function $\frac{d\mathbf{h}}{d\nu}(s,x,m,v)(\cdot):\brn\to\mathbb{R}$ is convex. 
\end{lemma}

The proof of Lemma~\ref{lem:der_h}, along with the explicit formula for the derivatives of $\mathbf{h}$, is provided in Appendix~\ref{App:lem:der_h}. With these properties, we know that the coefficients for the MFTC problem \eqref{cost^v'} satisfy the linear-convex assumptions as in \cite{AB10}. As a consequence, its unique optimal control $u^{t,\xi}$ is given by the following optimality condition:
\begin{align}\label{optimality_condition}
    b_3(s)^\top \mathbf{P}^{t,\xi}_s + \mathbf{h}_v\left(s,X^{t,\xi}_s,\lr\left(X^{t,\xi}_s\right), u^{t,\xi}_s\right)=0, 
\end{align}
where $\left(\mathbf{P}^{t,\xi},\mathbf{Q}^{t,\xi}\right)$ is the adjoint process corresponding to the optimal control $u^{t,\xi}$, satisfying the following BSDE:
\begin{equation}\label{BSDE:PQ}
    \begin{aligned}
        \mathbf{P}^{t,\xi}_s=\ & -\int_s^T \mathbf{Q}^{t,\xi}_r dB_r +\Phi_x\left(X^{t,\xi}_T, \lr\left(X^{t,\xi}_T\right) \right) +\widetilde{\e}\left[D_{x'}\frac{d\Phi}{d\nu} \left(\widetilde{X^{t,\xi}_T}, \lr\left(X^{t,\xi}_T\right)\right)\left(X^{t,\xi}_T\right) \right] \\
        &+ \int_s^T \bigg\{ b_1(r)^\top \mathbf{P}^{t,\xi}_r + \sum_{j=1}^n \sigma^j_1(r)^\top \mathbf{Q}^{t,\xi,j}_r +b_2(r)^\top \e\left[\mathbf{P}^{t,\xi}_r\right] +\sum_{j=1}^n\sigma^j_2(r)^\top \e\left[\mathbf{Q}^{t,\xi,j}_r\right]\\
        &\qquad \qquad + \mathbf{h}_x \left( r,X^{t,\xi}_r,\lr\left(X^{t,\xi}_r \right), u^{t,\xi}_r \right) 
        + \widetilde{\e} \left[D_{x'}\frac{d\mathbf{h}}{d\nu} \left( r,\widetilde{X^{t,\xi}_r},\lr\left(X^{t,\xi}_r\right), \widetilde{u^{t,\xi}_r}\right)\left(X^{t,\xi}_r\right) \right] \bigg\} dr. 
    \end{aligned}
\end{equation}
We also refer to \cite[Section 6]{AB10} for the well-posedness of the FBSDE system for $(X^{t,\xi},\mathbf{P}^{t,\xi},\mathbf{Q}^{t,\xi})$ defined by \eqref{cost^v'}, \eqref{BSDE:PQ}, and \eqref{optimality_condition}. Since the reformulated problem \eqref{cost^v'} is equivalent to the original problem $(\mathsf{P}_{t,\xi})$, it follows from Theorem~\ref{thm:solvability}, the relation for $f_7$ in \eqref{relation:f-F}, and Assumption (A4) that the optimal control $u^{t,\xi}$ satisfies
\begin{align*}
    b_3(s)^\top \left(P^{t,\xi}_s + F_1(s)^\top R^{t,\xi}_s \right) + h_v\left(s,\theta^{t,\xi}_s, \lr\left(\theta^{t,\xi}_s\right),u^{t,\xi}_s\right) =0, 
\end{align*}
where $\theta^{t,\xi}_s:= \left(X^{t,\xi}_s, Y^{t,\xi}_s, Z^{t,\xi}_s\right)$. The following relations hold between the adjoint processes $(\mathbf{P}^{t,\xi},\mathbf{Q}^{t,\xi})$ of the reformulated problem and $(P^{t,\xi},Q^{t,\xi},R^{t,\xi})$ of the original one.

\begin{proposition}\label{prop:relation:PQR}
    Under Assumptions (A1)-(A4) and (C), we have
    \begin{equation}\label{relation_PQR-PQ}
    \begin{aligned}
        \mathbf{P}^{t,\xi}_s=\ & P^{t,\xi}_s+F_1(s)^\top R^{t,\xi}_s,\\
        \mathbf{Q}^{t,\xi,j}_s=\ & Q^{t,\xi,j}_s - F_1(s)^\top h_{z^j} \left(s,\theta^{t,\xi}_s,\lr\left(\theta^{t,\xi}_s\right),u^{t,\xi}_s\right) \\
        &- \widetilde{\e} \bigg[F_1(s)^\top D_{{z'}^j}\frac{dh}{d\nu} \left(r,\widetilde{\theta^{t,\xi}_s},\lr\left(\theta^{t,\xi}_s\right),\widetilde{u^{t,\xi}_s}\right)\left(\theta^{t,\xi}_s\right) \bigg].
    \end{aligned}
    \end{equation}
\end{proposition}

The proof consists of applying It\^o's formula to the process $P^{t,\xi}_s + F_1(s)^\top R^{t,\xi}_s$, comparing the resulting BSDE with that satisfied by $(\mathbf{P}^{t,\xi},\mathbf{Q}^{t,\xi})$ using the relations between the derivatives of $\mathbf{h}$ and $h$ obtained in the proof of Lemma~\ref{lem:der_h}, and then invoking the uniqueness of the FBSDE system for $(X^{t,\xi},\mathbf{P}^{t,\xi},\mathbf{Q}^{t,\xi})$. The details are provided in Appendix~\ref{App:prop:relation:PQR}.

\subsection{Value function and HJB equation}

We define the Lagrangian $\mathbf{L}$ associated with $\mathbf{h}$ as
\begin{align*}
    \mathbf{L}(s,x,m,v,\mathbf{p}) := b(s,x,m,v)^\top \mathbf{p} + \mathbf{h}(s,x,m,v),
\end{align*}
and the corresponding Hamiltonian $\mathbf{H}$ as $\mathbf{H}(s,x,m,\mathbf{p}) := \inf_{v\in\brd} \mathbf{L}(s,x,m,v,\mathbf{p})$. We also define the value function $\mathbf{V}:[0,T]\times\pr_2(\brn)\to\br$ by
\begin{align*}
    \mathbf{V} (t,\lr(\xi)):= \inf_{v\in \lr_\f^2(t,T)} \e\bigg[\int_t^T \mathbf{h}\left(s,X^{t,\xi,v}_s, \lr\left(X^{t,\xi,v}_s\right), v_s \right) ds  + \Phi\left(X_T^{t,\xi,v},\lr\left(X_T^{t,\xi,v}\right)\right)\bigg].  
\end{align*}
From \eqref{cost^v'}, the relation between $V$ and $\mathbf{V}$ is given by
\begin{align}\label{relation:V-VV}
    V(t,\lr(\xi))= \e\left[\varphi\Big(F_0(t)+F_1(t)\xi \Big)\right] +\mathbf{V} (t,\lr(\xi)).
\end{align}
By \cite[Theorem 7.1]{AB10}, under Assumptions (A1), (A2'), (A3) and (C), the functional $\mathbf{V}(t,\lr(\xi))$ is twice G\^ateaux differentiable in $\xi$, and the following relation holds:
\begin{align}
    \mathbf{P}^{t,\xi}_t=\ & D_{x}\frac{d\mathbf{V}}{d\nu}(t,\lr(\xi))(\xi). \label{AB10:relation_P}
\end{align}
Moreover, the second-order derivatives $D_\xi^2 \mathbf{V}$,  $D_{x}^2\frac{d\mathbf{V}}{d\nu}(t,\lr(\xi))(\xi)$ and $D_{x'}D_{x}^2\frac{d^2\mathbf{V}}{d\nu^2}(t,\lr(\xi))\left(\xi,\widetilde{\xi}\right)$ are also characterized in \cite[Theorem 7.1]{AB10}, under the second-order differentiability conditions on the coefficients stated in (A2'). We are now ready to state the HJB equation corresponding to the feedback case \eqref{feedback:Y}.

\begin{theorem}
    Under Assumptions (A1), (A2'), (A3), (A4) and (C), the value function $V$ is a classical solution of the following HJB equation on $[0,T]\times\pr_2(\brn)$:
    \begin{equation}\label{HJB:second-order}
    \left\{
    \begin{aligned}
        &\dd_t V(t,\mu) + \frac{1}{2}\int_\brn \text{Tr}\left[ \left(\sigma\sigma^\top\right) (t,x,\mu)  \left(D_{x}^2\frac{dV}{d\nu}(t,\mu)(x)-F_1(t)^\top \varphi_{yy}\Big(F(t,x) \Big)F_1(t)\right)\right]\mu(dx) \\
        &\,\ \ \quad\qquad +\int_\brn \hr \bigg(t,\ x,\  F(t,x),\  F_1(t)\sigma(t,x,\mu),\ \left(I_n\otimes F(t,\cdot)\otimes F_1(t)\sigma(t,\cdot,\mu)\right) \# \mu, \\
        &\ \qquad\qquad\qquad\qquad D_{x} \frac{dV}{d\nu}(t,\mu)(x),\ -\varphi_y\left(F(t,x)\right)  \bigg) \mu(dx)=0,\\
        &V(T,\mu)=\int_\brn \left[\varphi(F(T,x))+\Phi(x,\mu)\right]\mu(dx),
    \end{aligned}
    \right.
    \end{equation}
    in the sense that $V$ is differentiable in $t$, linearly functionally differentiable in $\mu$, with $x\mapsto \frac{dV}{d\nu} (t,\mu)(x)$ twice differentiable.  Here, the function $\hr$ is defined in \eqref{def:Ham'}. 
\end{theorem}

\begin{proof}
From the relation \eqref{relation:V-VV}, we see that 
\begin{equation}\label{feedback:HJB_1}
\begin{aligned}
    \dd_t V(t,\lr(\xi))&= \dd_t\mathbf{V} (t,\lr(\xi)) + \e\left\{ \left[\dot{F}_0(t)+\dot{F}_1(t)\xi\right]^\top \varphi_y\Big(F_0(t)+F_1(t)\xi \Big) \right\};\\
    %D_{x}\frac{dV}{d\nu}(t,\lr(\xi))(\xi)&=D_{x}^2\frac{d\mathbf{V}}{d\nu}(t,\lr(\xi))(\xi)+F_1(t)^\top \varphi_y\Big(F_0(t)+F_1(t)\xi \Big);\\
    D_{x}^2\frac{dV}{d\nu}(t,\lr(\xi))(\xi)&=D_{x}^2\frac{d\mathbf{V}}{d\nu}(t,\lr(\xi))(\xi)+F_1(t)^\top \varphi_{yy}\Big(F_0(t)+F_1(t)\xi \Big)F_1(t).
\end{aligned}
\end{equation}
From \cite[Theorem 7.2]{AB10}, we know that $\mathbf{V}$ is the classical solution of the HJB equation corresponding to the MFTC problem \eqref{cost^v'}; and then, by using the characterization of the derivative of $\mathbf{V}$ in \eqref{AB10:relation_P}, we know that 
\begin{align*}
    -\dd_t \mathbf{V} (t,\lr(\xi))=\ & \e\bigg\{\frac{1}{2}\text{Tr}\left[\left(\sigma\sigma^\top\right) (t,\xi,\lr(\xi)) D^2_{x} \frac{d\mathbf{V}}{d\nu}(t,\lr(\xi))(\xi)\right] \\
    &\qquad +\mathbf{H}\left(t, \xi, \lr(\xi), D_{x}\frac{d\mathbf{V}}{d\nu}(t,\lr(\xi))(\xi) \right) \bigg\}\\
    =\ & \e\bigg\{\frac{1}{2}\text{Tr}\left[\left(\sigma\sigma^\top\right) (t,\xi,\lr(\xi)) D^2_{x} \frac{d\mathbf{V}}{d\nu}(t,\lr(\xi))(\xi)\right] \\
    &\qquad + b\left(t,\xi,\lr(\xi),u^{t,\xi}_t \right)^\top \mathbf{P}^{t,\xi}_t + \mathbf{h} \left(t,\xi,\lr(\xi),u^{t,\xi}_t \right) \bigg\}.
\end{align*}
Then, by applying the relation \eqref{feedback:Y}, the characterization of $Z^{t,\xi}$ in \eqref{relation:YZ}, the definition of the functional $\mathbf{h}$ in \eqref{def:hh} and the relation between $\mathbf{P}^{t,\xi}$ and $P^{t,\xi}$ in \eqref{relation_PQR-PQ}, we have
\begin{align*}
    &-\dd_t \mathbf{V} (t,\lr(\xi))\\
    =\ & \e\bigg\{\frac{1}{2}\text{Tr}\left[\left(\sigma\sigma^\top\right) (t,\xi,\lr(\xi)) D^2_{x} \frac{d\mathbf{V}}{d\nu}(t,\lr(\xi))(\xi)\right] + b\left(t,\xi,\lr(\xi),u^{t,\xi}_t \right)^\top \left(P^{t,\xi}_t+F_1(t)^\top R^{t,\xi}_t\right) \\
    &\quad + h \left(t,\xi,Y^{t,\xi}_t,Z^{t,\xi}_t,\lr\left(\xi,Y^{t,\xi}_t,Z^{t,\xi}_t\right),u^{t,\xi}_t \right) \bigg\}.
\end{align*}
Combining the last equality with \eqref{feedback:HJB_1}, by using the initial condition for $R^{t,\xi}_t$ in System \eqref{FBSDE:MFTC}, and the relations in \eqref{relation:f-F}, we know that 
\begin{align*}
    &-\dd_t V(t,\lr(\xi))\\
    =\ & \e\Bigg\{\frac{1}{2}\text{Tr}\left\{\left(\sigma\sigma^\top\right) (t,\xi,\lr(\xi)) \left[D_{x}^2\frac{dV}{d\nu}(t,\lr(\xi))(\xi)-F_1(t)^\top \varphi_{yy}\Big(F_0(t)+F_1(t)\xi \Big)F_1(t)\right]\right\} \\
    &\quad + b\left(t,\xi,\lr(\xi),u^{t,\xi}_t \right)^\top \left(P^{t,\xi}_t+F_1(t)^\top R^{t,\xi}_t\right) + h \left(t,\xi,Y^{t,\xi}_t,Z^{t,\xi}_t,\lr\left(\xi,Y^{t,\xi}_t,Z^{t,\xi}_t\right),u^{t,\xi}_t \right) \\
    &\quad -\left[\dot{F}_0(t)+\dot{F}_1(t)\xi\right]^\top \varphi_y\Big(Y^{t,\xi}_t \Big) \Bigg\} \\
    %=\ & \e\Bigg\{\frac{1}{2}\text{Tr}\left\{\left(\sigma\sigma^\top\right) (t,\xi,\lr(\xi))  \left[D_{x}^2\frac{dV}{d\nu}(t,\lr(\xi))(\xi)-F_1(t)^\top \varphi_{yy}\Big(F_0(t)+F_1(t)\xi \Big)F_1(t)\right]\right\} \\
    %&\quad + b\left(t,\xi,\lr(\xi),u^{t,\xi}_t \right)^\top P^{t,\xi}_t + h \left(t,\xi,Y^{t,\xi}_t,Z^{t,\xi}_t,\lr\left(\xi,Y^{t,\xi}_t,Z^{t,\xi}_t\right),u^{t,\xi}_t, Y^{t,\xi}_t,Z^{t,\xi}_t \right) \\
    %&\quad + \left[\dot{F}_0(t)+\dot{F}_1(t)\xi + F_1(t) b\left(t,\xi,\lr(\xi),u^{t,\xi}_t \right)\right]^\top  R^{t,\xi}_t \Bigg\} \\
    =\ & \e\Bigg\{\frac{1}{2}\text{Tr}\left\{ \left(\sigma\sigma^\top\right) (t,\xi,\lr(\xi))  \left[D_{x}^2\frac{dV}{d\nu}(t,\lr(\xi))(\xi)-F_1(t)^\top \varphi_{yy}\Big(F_0(t)+F_1(t)\xi \Big)F_1(t)\right]\right\} \\
    &\quad + b\left(t,\xi,\lr(\xi),u^{t,\xi}_t \right)^\top P^{t,\xi}_t + h \left(t,\xi,Y^{t,\xi}_t,Z^{t,\xi}_t,\lr\left(\xi,Y^{t,\xi}_t,Z^{t,\xi}_t\right),u^{t,\xi}_t \right) \\
    &\quad -f\left(t,\xi,Y^{t,\xi}_t,Z^{t,\xi}_t,\lr\left(\xi,Y^{t,\xi}_t,Z^{t,\xi}_t\right), u^{t,\xi}_t \right)^\top  R^{t,\xi}_t \Bigg\}.
\end{align*}
Then, from the definition of the function $\hr$ defined in \eqref{def:Ham'}, the characterization for $P^{t,\xi}_t$ in Theorem~\ref{thm:value_regu_mu} and again the initial condition for $R^{t,\xi}_t$, we have 
\begin{align*}
    &-\dd_t V(t,\lr(\xi))\\
    %=\ & \e\Bigg\{ \frac{1}{2}\text{Tr}\left\{ \left(\sigma\sigma^\top\right) (t,\xi,\lr(\xi))  \left[D_{x}^2\frac{dV}{d\nu}(t,\lr(\xi))(\xi)-F_1(t)^\top \varphi_{yy}\Big(F_0(t)+F_1(t)\xi \Big)F_1(t)\right]\right\} \\
    %&\quad +\hr\bigg(t, \xi, Y^{t,\xi}_t,  Z^{t,\xi}_t, \lr\left(\xi,Y^{t,\xi}_t,  Z^{t,\xi}_t\right), P^{t,\xi}_t, R^{t,\xi}_t  \bigg) \Bigg\} \\
    =\ & \e\Bigg\{ \frac{1}{2}\text{Tr}\left\{ \left(\sigma\sigma^\top\right) (t,\xi,\lr(\xi))  \left[D_{x}^2\frac{dV}{d\nu}(t,\lr(\xi))(\xi)-F_1(t)^\top \varphi_{yy}\Big(F_0(t)+F_1(t)\xi \Big)F_1(t)\right]\right\} \\
    &\quad +\hr\bigg(t, \xi, Y^{t,\xi}_t,  Z^{t,\xi}_t, \lr\left(\xi,Y^{t,\xi}_t,  Z^{t,\xi}_t\right), D_{x} \frac{dV}{d\nu}(t,\lr(\xi))(\xi),  -\varphi_y\left(Y^{t,\xi}_t\right)  \bigg) \Bigg\}.
\end{align*}
From \eqref{feedback:Y} and \eqref{relation:YZ}, we know that $Y^{t,\xi}_t=F(t,\xi)$ and $Z^{t,\xi}_t =  F_1(t)\; \sigma\left(t,\xi,\lr(\xi)\right)$, and therefore,
\begin{align*}
    \lr\left(\xi,Y^{t,\xi}_t,  Z^{t,\xi}_t \right)=\left(I_n\otimes F(t,\cdot)\otimes F_1(t)\sigma(t,\cdot,\lr(\xi))\right) \# \lr(\xi). 
\end{align*}
With these results, by using the fact that $\lr(\xi)=\mu$, we obtain \eqref{HJB:second-order}. %\begin{align*}
    %&-\dd_t V(t,\lr(\xi))\\
    %=\ & \e\Bigg\{\frac{1}{2}\text{Tr}\left\{ \left(\sigma\sigma^\top\right) (t,\xi,\lr(\xi))  \left[D_{x}^2\frac{dV}{d\nu}(t,\lr(\xi))(\xi)-F_1(t)^\top \varphi_{yy}\Big(F_0(t)+F_1(t)\xi \Big)F_1(t)\right]\right\}\\
    %&\quad +\hr\bigg(t,\ \xi,\ F(t,\xi),\  F_1(t)\sigma(t,\xi,\lr(\xi)),\ \left(I_n\otimes F(t,\cdot)\otimes F_1(t)\sigma(t,\cdot,\lr(\xi))\right) \# \lr(\xi),\\
    %&\qquad\qquad D_{x} \frac{dV}{d\nu}(t,\lr(\xi))(\xi), \ -\varphi_y\left(F(t,\xi)\right)  \bigg) \Bigg\}.
%\end{align*}
%Together with the terminal condition, 
\end{proof}

The HJB equation \eqref{HJB:second-order} contains the standard second-order MFTC HJB equation (see, e.g., \cite{AB10}) as a special case, obtained by setting $F\equiv 0$ and $\varphi\equiv 0$. The HJB equation \eqref{HJB:second-order} is consistent in form with \eqref{HJB}, with the only difference being that the latter does not contain the following two terms:
\begin{equation}\label{feedback:ZQ}
\begin{aligned}
    &F_1(t)\sigma(t,\xi,\lr(\xi))\qquad \text{(corresponding to $Z^{t,\xi}_t$)},\\
    &\left[D_{x}^2\frac{dV}{d\nu}(t,\lr(\xi))(\xi)-F_1(t)^\top \varphi_{yy}\Big(F(t,\xi) \Big)F_1(t)\right]\sigma (t,\xi,\lr(\xi))\qquad \text{(corresponding to $\mathbf{Q}^{t,\xi}_t$)}.
\end{aligned}
\end{equation}
because \eqref{HJB} is derived under the first-order assumption $\sigma\equiv 0$. However, \eqref{HJB:second-order} cannot subsume \eqref{HJB}, as the former requires the feedback structure \eqref{feedback:Y}, while the latter allows an arbitrary linear controlled BSDE.

\section{Second-order HJB Equation: the case with flow property}\label{sec:flow}

In the proof of the differentiability of $V$ in $t$ for the first-order problem in Theorem~\ref{thm:HJB}, we directly compute the derivatives of the processes $\Theta^{t+\epsilon,\xi}_s,\ t+\epsilon\le s\le T$  with respect to $\epsilon$, and characterize the limiting derivatives as solutions of a new system of FBSDEs of the form \eqref{FBSDEs:high}. This approach differs significantly from the standard stochastic control framework for SDEs, where the flow property $X^{t,\xi}_s=X^{r,X^{t,\xi}_r}_s$, $t\le r\le s\le T$ plays a central role in the derivation of the dynamic programming principle and the HJB equation. In our setting, however, the presence of both a controlled SDE and a controlled BSDE means that the flow property generally fails, rendering the standard approach inapplicable. We therefore develop an alternative method (based on Lemma~\ref{lem:de_bound} through Theorem~\ref{thm:HJB}) to derive the HJB equation.

In this section, we continue to study the second-order problem. We first characterize the flow property by an equivalent condition, and then use it to derive the associated HJB equation; finally, we provide sufficient conditions under which the flow property holds. The flow property is defined as follows:
\begin{align}\label{flow_property}
    \Pi^{t,\xi}_s=\Pi^{r,X^{t,\xi}_r}_s,\quad t\le r\le s\le T, 
\end{align}
where $\Pi^{t,\xi}:=\left(\theta^{t,\xi},P^{t,\xi},Q^{t,\xi},R^{t,\xi},u^{t,\xi}\right):=\left(X^{t,\xi},Y^{t,\xi},Z^{t,\xi},P^{t,\xi},Q^{t,\xi},R^{t,\xi},u^{t,\xi}\right)$ denotes the unique adapted solution of System \eqref{FBSDE:MFTC} with initial condition $(t,\xi)$. 

\subsection{Flow property and dynamic programming principle}

We begin by giving the following necessary and sufficient condition for the validity of the flow property \eqref{flow_property}, whose proof is given in Appendix~\ref{App:prop:s&n}.

\begin{proposition}\label{prop:s&n}
    Under Assumptions (A1)-(A3), suppose that $\varphi\in C^3$. Then, the validity of the flow property \eqref{flow_property} is equivalent to the validity of the following relations: for $s\in[t,T)$,
    \begin{align}
        & -\varphi_{yy} \left(Y^{t,\xi}_s\right) f\left(s,\theta^{t,\xi}_s,\lr\left(\theta^{t,\xi}_s\right), u^{t,\xi}_s \right) + \frac{1}{2}\text{Tr}\left[\varphi_{yyy} \left(Y^{t,\xi}_s\right)Z^{t,\xi}_s\left(Z^{t,\xi}_s\right)^\top \right] \notag\\
        &\qquad =-f_2(s)^\top  R^{t,\xi}_s -f_5(s)^\top  \e\left[R^{t,\xi}_s\right] + h_y \left(s,\theta^{t,\xi}_s,\lr\left(\theta^{t,\xi}_s\right),u^{t,\xi}_s\right) \notag\\
        &\quad\qquad +\widetilde{\e} \bigg[D_{y'}\frac{dh}{d\nu} \left(s,\widetilde{\theta^{t,\xi}_s},\lr\left({ \theta^{t,\xi}_s}\right),\widetilde{u^{t,\xi}_s}\right)\left({ \theta^{t,\xi}_s}\right) \bigg]; \label{condition:1} \\
        &\varphi_{yy} \left(Y^{t,\xi}_s\right) Z^{t,\xi,j}_s = -f^j_3(s)^\top  R^{t,\xi}_s-f^j_6(s)^\top  \e\left[R^{t,\xi}_s\right] +h_{z^j} \left(s,\theta^{t,\xi}_s,\lr\left(\theta^{t,\xi}_s\right),u^{t,\xi}_s\right) \notag\\
        &\qquad\qquad\qquad\qquad +\widetilde{\e} \bigg[D_{{z'}^j}\frac{dh}{d\nu} \left(s,\widetilde{\theta^{t,\xi}_s},\lr\left({ \theta^{t,\xi}_s}\right),\widetilde{u^{t,\xi}_s}\right)\left({ \theta^{t,\xi}_s}\right) \bigg]. \label{condition:2}
    \end{align}
    Moreover, if \eqref{condition:1}-\eqref{condition:2} hold, then we have the following relations
    \begin{align}\label{feedback_RP}
        &R^{t,\xi}_s= - \varphi_y \left(Y^{t,\xi}_s\right),\qquad P^{t,\xi}_s=D_{x}\frac{dV}{d\nu}\left(s,\lr\left(X_s^{t,\xi}\right)\right)\left(X_s^{t,\xi}\right).
    \end{align}
\end{proposition}

From Proposition~\ref{prop:s&n}, the flow property is equivalent to Conditions \eqref{condition:1}-\eqref{condition:2}. These conditions are not generally satisfied, so the flow property typically fails for control problems of mean-field FBSDEs. In Subsection~\ref{subsec:example}, we shall provide sufficient conditions on $\varphi$ and $h$ that ensure \eqref{condition:1}-\eqref{condition:2} and hence the flow property \eqref{flow_property}. Before doing so, we assume that these conditions hold and proceed to derive the dynamic programming principle, the HJB equation, and the decoupling field. We emphasize that all subsequent results remain valid if the validity of these conditions is replaced by Assumption (D), introduced in Subsection~\ref{subsec:example}.

With the flow property \eqref{flow_property} in hand, we can establish the following dynamic programming principle for the optimal control problem of mean-field FBSDEs.

\begin{theorem}\label{thm:DPP}
    Under Assumptions (A1)-(A3) and the validity of Conditions \eqref{condition:1}-\eqref{condition:2}, for any $0\le t\le s\le T$, we have
    \begin{align}
        &V(t,\lr(\xi))-V\left(s,\lr\left(X^{t,\xi}_s\right)\right) \notag \\
        =\ & \e \int_t^s \bigg\{ h\left(r,\theta^{t,\xi}_r,\lr\left(\theta^{t,\xi}_r\right),u^{t,\xi}_r\right) -  \left(R^{t,\xi}_r\right)^\top f\left(r,\theta^{t,\xi}_r, \lr\left(\theta^{t,\xi}_r\right), u^{t,\xi}_r \right) \notag \\
        &\quad\qquad - \frac{1}{2} \text{Tr}\left[\varphi_{yy} \left(Y^{t,\xi}_r\right)Z^{t,\xi}_r\left(Z^{t,\xi}_r\right)^\top \right]\bigg\}dr. \label{thm:DPP_0}
    \end{align}
\end{theorem}

\begin{proof}
We first apply It\^o's formula on $\varphi\left(Y^{t,\xi}_r\right)$, integrate on $[s,T]$ and take expectation:  
\begin{align*}
    &\e\left[\varphi\left(Y^{t,\xi}_T\right)-\varphi\left(Y^{t,\xi}_s\right)\right]\\
    =\ & \e\int_s^T\left\{ -f\left(r,\theta^{t,\xi}_r,\lr\left({ \theta^{t,\xi}_r}\right),u^{t,\xi}_r\right)^\top \varphi_y\left(Y^{t,\xi}_r\right) + \frac{1}{2}\text{Tr}\left[\varphi_{yy} \left(Y^{t,\xi}_r\right)Z^{t,\xi}_r\left(Z^{t,\xi}_r\right)^\top \right] \right\}dr.
\end{align*}
By using the first relation in \eqref{feedback_RP}, we have
\begin{align}
        &\e\left[\varphi\left(Y^{t,\xi}_s\right)\right] \notag \\
        =\ & \e\left[\varphi\left(Y^{t,\xi}_T\right)\right]-\e\int_s^T\left\{ f\left(r,\theta^{t,\xi}_r,\lr\left({ \theta^{t,\xi}_r}\right),u^{t,\xi}_r\right)^\top R^{t,\xi}_r + \frac{1}{2}\text{Tr}\left[\varphi_{yy} \left(Y^{t,\xi}_r\right)Z^{t,\xi}_r\left(Z^{t,\xi}_r\right)^\top \right] \right\}. \label{lem:DPP_0'}
\end{align}
From \eqref{value_equation} and the relation \eqref{lem:DPP_0'}, we know that 
\begin{align*}
    V(t,\lr(\xi))%=\ & \e\left[\int_t^T h\left(r,\theta_r^{t,\xi},\lr\left(\theta^{t,\xi}_r\right), u^{t,\xi}_r\right) dr + \Phi\left(X_T^{t,\xi},\lr\left(X_T^{t,\xi}\right)\right)+\varphi\left(Y_t^{t,\xi}\right) \right]\\
    =\ &  \e\bigg\{\int_t^T \bigg\{ h\left(r,\theta_r^{t,\xi},\lr\left(\theta^{t,\xi}_r\right), u^{t,\xi}_r\right) -  \left(R^{t,\xi}_r\right)^\top f\left(r,\theta^{t,\xi}_r,\lr\left(\theta^{t,\xi}_r\right), u^{t,\xi}_r \right) \\
    &\qquad\qquad  - \frac{1}{2}\text{Tr}\left[\varphi_{yy} \left(Y^{t,\xi}_r\right)Z^{t,\xi}_r\left(Z^{t,\xi}_r\right)^\top \right] \bigg\}dr  +\Phi\left(X_T^{t,\xi},\lr\left(X_T^{t,\xi}\right)\right)+\varphi\left(Y^{t,\xi}_T\right)\bigg\}.
\end{align*}
Applying the above relation with $V\left(s,\lr\left(X^{t,\xi}_s\right)\right)$, and using the flow property \eqref{flow_property}, we have
\begin{align*}
    &V\left(s,\lr\left(X^{t,\xi}_s\right)\right)\\
    =\ &  \e\int_s^T \bigg\{ h\left(r,\theta_r^{s,X^{t,\xi}_s},\lr\left(\theta_r^{s,X^{t,\xi}_s}\right),u^{s,X^{t,\xi}_s}_r\right) - \left(R^{s,X^{t,\xi}_s}_r\right)^\top f\left(r,\theta^{s,X^{t,\xi}_s}_r, \lr\left(\theta_r^{s,X^{t,\xi}_s}\right), u^{s,X^{t,\xi}_s}_r \right) \\
    &\qquad\qquad - \frac{1}{2}\text{Tr}\left[\varphi_{yy} \left(Y^{s,X^{t,\xi}_s}_r\right)Z^{s,X^{t,\xi}_s}_r\left(Z^{s,X^{t,\xi}_s}_r\right)^\top \right]  \bigg\}dr \\
    &+\e\bigg[\Phi\left(X_T^{s,X^{t,\xi}_s},\lr\left(X_T^{s,X^{t,\xi}_s}\right)\right)+\varphi\left(Y^{s,X^{t,\xi}_s}_T\right)\bigg]\\
    =\ &  \e\bigg\{\int_s^T \bigg\{ h\left(r,\theta_r^{t,\xi},\lr\left(\theta_r^{t,\xi}\right),u^{t,\xi}_r\right) -  \left(R^{t,\xi}_r\right)^\top f\left(r,\theta^{t,\xi}_r,\lr\left(\theta_r^{t,\xi}\right), u^{t,\xi}_r \right) \\
    &\qquad\qquad - \frac{1}{2} \text{Tr}\left[\varphi_{yy} \left(Y^{t,\xi}_r\right)Z^{t,\xi}_r\left(Z^{t,\xi}_r\right)^\top \right]\bigg\}dr +\Phi\left(X_T^{t,\xi},\lr\left(X_T^{t,\xi}\right)\right)+\varphi\left(Y^{t,\xi}_T\right)\bigg\}.
\end{align*}
By combining the above two equalities, we obtain \eqref{thm:DPP_0}.
\end{proof}

To proceed to derive the HJB equation, we henceforth impose Assumption (A4) (i.e., the diffusion is independent of the control, and all coefficients are $\frac{1}{2}$-H\"older continuous in time). Using the flow property \eqref{flow_property}, we first establish the following time-regularity estimates for the processes.

\begin{proposition}\label{prop:flow:cinti:t}
    Under Assumptions (A1)-(A4) and the validity of Conditions \eqref{condition:1}-\eqref{condition:2}, we have
    \begin{align}\label{prop:conti_t_0}
        &\e\left[\left|X^{t,\xi}_s - X^{t,\xi}_t \right|^2+\left|Y^{t,\xi}_s - Y^{t,\xi}_t \right|^2+\left|P^{t,\xi}_s - P^{t,\xi}_t \right|^2+\left|R^{t,\xi}_s - R^{t,\xi}_t \right|^2\right] \notag \\
        \le\ & C(L,T,\lambda)\left(1+\|\xi\|_2^2\right)(s-t),\quad 0\le t\le s\le T.
    \end{align}
    The value function $V$ satisfies
    \begin{align}
        \left|D_{x}\frac{dV}{d\nu}(s,\mu)(x) - D_{x}\frac{dV}{d\nu}(t,\mu)(x) \right|  \le C(L,T,\lambda)\left(1+|x|+|\mu|_2\right)(s-t)^{\frac{1}{2}}. \label{prop:conti_t_0'}
    \end{align}
    Moreover, if Assumption (A2') holds, then $D^2_{x}\frac{dV}{d\nu}(t,\mu)(x)$ is continuous in $(t,\mu,x)$ and
    \begin{align}
        \left|D^2_{x}\frac{dV}{d\nu}(t,\mu)(x) \right|  \le C(L,T,\lambda). \label{prop:conti_t_0''}
    \end{align}
\end{proposition}

The proof of Proposition~\ref{prop:flow:cinti:t} is given in Appendix~\ref{App:prop:flow:cinti:t}. We now state the following result, which gives an expansion of the value function $V$ in the time variable.

\begin{theorem}\label{thm:flow:HJB:pre}
    Under Assumptions (A1), (A2'), (A3), (A4) and the validity of Conditions \eqref{condition:1}-\eqref{condition:2}, for any $0\le t\le s\le T$, we have   
    \begin{align}
        &V(s,\lr(\xi))-V(t,\lr(\xi)) \notag \\
        =\ & -\e\int_t^s \bigg\{\frac{1}{2}\text{Tr}\left[D^2_{x}\frac{dV}{d\nu}\left(r,\lr\left(X^{t,\xi}_r\right)\right)\left(X^{t,\xi}_r\right) \left(\sigma\sigma^\top\right) \left(r,X^{t,\xi}_r,\lr\left(X^{t,\xi}_r\right)\right)\right] \notag \\
        &\qquad\qquad + \hr\bigg(r,\theta^{t,\xi}_r,\lr\left(\theta^{t,\xi}_r\right), D_{x}\frac{dV}{d\nu}\left(r,\lr\left(X_r^{t,\xi}\right)\right)\left(X_r^{t,\xi}\right),  -\varphi_y \left(Y^{t,\xi}_r\right) \bigg)  \notag \\
        &\qquad\qquad - \frac{1}{2} \text{Tr}\left[\varphi_{yy} \left(Y^{t,\xi}_r\right)Z^{t,\xi}_r\left(Z^{t,\xi}_r\right)^\top \right] \bigg\}dr + o\left(|s-t|\right), \label{discuss:Bell_0}
    \end{align}
    where the functional $\hr$ is defined in \eqref{def:Ham'}.
\end{theorem}

\begin{proof}
By using Theorem~\ref{thm:DPP}, we can write 
\begin{align}
    &V(s,\lr(\xi))-V(t,\lr(\xi)) \notag \\
    =\ &V(s,\lr(\xi))-V\left(s,\lr\left(X^{t,\xi}_s\right)\right)+V\left(s,\lr\left(X^{t,\xi}_s\right)\right)-V(t,\lr(\xi)) \notag \\
    =\ & V(s,\lr(\xi))-V\left(s,\lr\left(X^{t,\xi}_s\right)\right) \notag \\
    &- \e \int_t^s \bigg\{ h\left(r,\theta_r^{t,\xi},\lr\left(\theta^{t,\xi}_r\right),u^{t,\xi}_r\right) -  \left(R^{t,\xi}_r\right)^\top f\left(r,\theta^{t,\xi}_r,\lr\left(\theta^{t,\xi}_r\right), u^{t,\xi}_r \right) \notag \\
    &\qquad\qquad - \frac{1}{2} \text{Tr}\left[\varphi_{yy} \left(Y^{t,\xi}_r\right)Z^{t,\xi}_r\left(Z^{t,\xi}_r\right)^\top \right] \bigg\}dr. \label{discuss:Bell_1}
\end{align}
By applying It\^o's formula on the term $V\left(s,\lr\left(X^{t,\xi}_s\right)\right)-V(s,\lr(\xi)))$ and using the continuity estimates in \eqref{prop:conti_t_0'} and \eqref{prop:conti_t_0''}, 
%and the boundedness estimates of $\left(X^{t,\xi},u^{t,\xi}\right)$ (see \eqref{lem:kappa_0} and \eqref{special_h_v}), 
we have
\begin{align*}
    &V\left(s,\lr\left(X^{t,\xi}_s\right)\right)-V(s,\lr(\xi))\\
    %=\ & \e\int_t^s \bigg\{ D_{x}\frac{dV}{d\nu}\left(s,\lr\left(X^{t,\xi}_r\right)\right)\left(X^{t,\xi}_r\right)^\top b \left(r,X^{t,\xi}_r,\lr\left(X^{t,\xi}_r\right), u^{t,\xi}_r\right) \\
    %&\quad\qquad + \frac{1}{2}\text{Tr}\left[D^2_{x}\frac{dV}{d\nu}\left(s,\lr\left(X^{t,\xi}_r\right)\right)\left(X^{t,\xi}_r\right) \left(\sigma\sigma^\top\right) \left(r,X^{t,\xi}_r,\lr\left(X^{t,\xi}_r\right)\right)\right]  \bigg\}dr\\
    =\ & \e\int_t^s \bigg\{ D_{x}\frac{dV}{d\nu}\left(r,\lr\left(X^{t,\xi}_r\right)\right)\left(X^{t,\xi}_r\right)^\top b \left(r,X^{t,\xi}_r,\lr\left(X^{t,\xi}_r\right), u^{t,\xi}_r\right) \\
    &\quad\qquad + \frac{1}{2}\text{Tr}\left[D^2_{x}\frac{dV}{d\nu}\left(r,\lr\left(X^{t,\xi}_r\right)\right)\left(X^{t,\xi}_r\right) \left(\sigma\sigma^\top\right) \left(r,X^{t,\xi}_r,\lr\left(X^{t,\xi}_r\right)\right)\right]  \bigg\}dr + o\left(|s-t|\right).
\end{align*}
Combining the last equality with \eqref{discuss:Bell_1}, %we have
and by using the relation \eqref{feedback_RP} and the definition of the function $\hr$ in \eqref{def:Ham'}, we obtain \eqref{discuss:Bell_0}. 
\end{proof}

Equation \eqref{discuss:Bell_0} provides an expansion of the value function in time that involves the state processes $\theta^{t,\xi}_s=(X^{t,\xi}_s,Y^{t,\xi}_s,Z^{t,\xi}_s)$. From the estimate \eqref{prop:conti_t_0}, we already know the continuity of $(X^{t,\xi}_s,Y^{t,\xi}_s)$ in $s$. If we can further establish the continuity of $Z^{t,\xi}_s$ in $s$, then \eqref{discuss:Bell_0} would imply differentiability of $V$ in $t$. Moreover, if we can characterize $(Y^{t,\xi},Z^{t,\xi})$ in terms of $X^{t,\xi}$, the HJB equation can be written in a more explicit form.

\subsection{Decoupling field and HJB equation}

Under Assumption (A4), the diffusion coefficient $\sigma$ is independent of the control, so the minimizer map $\hv(\cdot)$ defined in \eqref{minimizer} does not depend on the $q$ argument. Using the relations in \eqref{feedback_RP}, the state processes $(X^{t,\xi},Y^{t,\xi},Z^{t,\xi})$ satisfy the following coupled FBSDEs:
\begin{equation}\label{System:XYZ}
    \left\{
        \begin{aligned}
            &X^{t,\xi}_s= \xi+\int_t^s B\left(r, \theta^{t,\xi}_r, \lr\left(\theta^{t,\xi}_r\right)\right) dr  +\int_t^s \sigma\left(r, X^{t,\xi}_r, \lr\left(X^{t,\xi}_r\right)\right) dB_r,\\
            &Y^{t,\xi}_s= g \left(X^{t,\xi}_T,\lr\left( X^{t,\xi}_T\right)\right) + \int_s^T  F\left(r, \theta^{t,\xi}_r, \lr\left(\theta^{t,\xi}_r\right) \right) dr -\int_s^T Z^{t,\xi}_r dB_r,
        \end{aligned}
    \right.
\end{equation}
where the coefficients $B$ and $F$ are defined by 
\begin{align*}
    B(s,x,y,z,\varrho):=\ & \hr_p \left(s,x,y,z,\varrho,D_{x}\frac{dV}{d\nu}(s,\varrho_x)\left(x\right), -\varphi_y(y)\right),\\
    F(s,x,y,z,\varrho):=\ &  -\hr_r \left(s,x,y,z,\varrho,D_{x}\frac{dV}{d\nu}(s,\varrho_x)\left(x\right), -\varphi_y(y)\right).
\end{align*}
Here, the functional $\hr$ is defined in \eqref{def:Ham'}. Let $\Pi^{t,x,\mu}:=\left(\theta^{t,x,\mu},P^{t,x,\mu},Q^{t,x,\mu},R^{t,x,\mu},u^{t,x,\mu}\right)=\left(X^{t,x,\mu},Y^{t,x,\mu},Z^{t,x,\mu},P^{t,x,\mu},Q^{t,x,\mu},R^{t,x,\mu},u^{t,x,\mu}\right)$ denote the solution of System \eqref{FBSDE:MFTC=gamma} with the initial condition being replaced by $X_t=x$ and with coefficients depending on $\lr(\theta^{t,\xi}_s)$. This is the second-order analogue of the first-order system $\eqref{FBSDEs:x}$, whose well-posedness and regularity follow by essentially the same argument as in Lemma~\ref{lem:tmux}. From the relations in \eqref{feedback_RP}, we know that $\theta^{t,x,\mu}$ is the solution of the following FBSDEs:
\begin{equation}\label{System:XYZ'}
    \left\{
        \begin{aligned}
            &X^{t,x,\mu}_s= x +\int_t^s B\left(r, \theta^{t,x,\mu}_r, \lr\left(\theta^{t,\xi}_r\right)\right) dr  +\int_t^s \sigma\left(r, X^{t,x,\mu}_r, \lr\left(X^{t,\xi}_r\right)\right) dB_r,\\
            &Y^{t,x,\mu}_s= g \left(X^{t,x,\mu}_T,\lr\left( X^{t,\xi}_T\right)\right) + \int_s^T  F\left(r, \theta^{t,x,\mu}_r, \lr\left(\theta^{t,\xi}_r\right) \right) dr -\int_s^T Z^{t,x,\mu}_r dB_r.
        \end{aligned}
    \right.
\end{equation}
As in \eqref{def:map_U}, we define the map $U:[0,T]\times\brn\times\pr_2(\brn)\to\brn$ by 
\begin{align}\label{def:U}
    U(t,x,\mu):=Y^{t,x,\mu}_t.
\end{align}
From the flow property \eqref{flow_property}, we see that 
\begin{align}\label{def:U'}
    &U\left(s, X_s^{t,x,\mu},\lr\left(X^{t,\xi}_s\right)\right)=Y^{s, X_s^{t,x,\mu},\lr\left(X^{t,\xi}_s\right)}_s=Y^{t,x,\mu}_s,
\end{align}
which shows that $U$ is indeed a decoupling field for FBSDEs \eqref{System:XYZ'}. In Theorem~\ref{main2_thm}, we established the Lipschitz continuity of the map $\xi\mapsto \Pi^{t,\xi}$. By following a similar argument, we can also obtain the Lipschitz continuity of $x\mapsto \theta^{t,x,\mu}$ and $\mu\mapsto \theta^{t,x,\mu}$. The proof is analogous and additionally relies on the continuity result of Theorem~\ref{main2_thm}. Together, these two continuity results imply the Lipschitz continuity of $U$ with respect to both $x$ and $\mu$. 

Moreover, following a similar approach as in Subsection~\ref{subsec:U}, we can further establish the second-order differentiability of $U$ with respect to $x$ and $\mu$ by studying the G\^ateaux derivatives of the map $(x,\xi)\mapsto \theta^{t,x,\lr(\xi)}$ in the second-order setting. This requires the following additional assumption, which is a regularity-enhanced version of Assumption (A2').

\textbf{(A2'')} The functionals $h$, $\Phi$ and $\varphi$ satisfy (A2'). Moreover, the following derivatives exist, and they are continuous in all arguments, and are globally bounded by $L$ in norm:
\begin{align*}
    &\left(D_{x},D_{y},D_{z},D_{v}\right)^3 h, \quad \left(D_{x'},D_{y'},D_{z'}\right)^2 \frac{d\left(h_x,h_y,h_z,h_v\right)}{d\nu},\quad \left(D_{x''},D_{y''},D_{z''}\right)\left(D_{x'},D_{y'},D_{z'}\right)^2 \frac{d^2h}{d\nu^2},\\
    &\left(D_{x'},D_{y'},D_{z'}\right)^3 \frac{dh}{d\nu}; \qquad \Phi_{xxx},\quad  D^2_{x'} \frac{d\Phi_x}{d\nu},\quad D_{x''}D^2_{x'} \frac{d^2\Phi}{d\nu^2},\quad D_{x'}^3 \frac{d\Phi}{d\nu}.
\end{align*}

We now state the following regularity properties of $U$.

\begin{proposition}\label{prop:regu_U}
    Under Assumptions (A1), (A2''), (A3), (A4) and the validity of Conditions \eqref{condition:1}-\eqref{condition:2}, the map $U$ defined in \eqref{def:U} is $C^2$ in $x$ and linearly functionally differentiable in $\mu$, with the map $x'\mapsto \frac{dU}{d\nu}(t,x,\mu)(x')$ being $C^2$. Moreover, the derivatives $\left(U_x,U_{xx},D_{x'}\frac{dU}{d\nu}, D^2_{x'}\frac{dU}{d\nu}\right)$ are continuous in the time variable. 
\end{proposition}

The proof of the differentiability in $x$ relies on the study of the first- and second-order derivatives of the map $x\mapsto \theta^{t,x,\mu}$, denoted respectively by $D_x \theta^{t,x,\mu}$ and $D^2_x \theta^{t,x,\mu}$. For the differentiability with respect to $\mu$, we study the derivatives $D_{x'}\frac{d\theta^{t,x,\mu}}{d\nu}(x')$ and $D^2_{x'}\frac{d\theta^{t,x,\mu}}{d\nu}(x')$, which correspond to $D_{x'}\frac{dU}{d\nu}$ and $D^2_{x'}\frac{dU}{d\nu}$ respectively. In the latter case, we also need the G\^ateaux differentiability of the map $\xi\mapsto \theta^{t,\xi}$. All these derivatives can be characterized successively as solutions of FBSDE systems. Since the proof relies on the same technical arguments as those used several times earlier in this article, we do not reproduce them here; a brief outline is provided in Appendix~\ref{App:prop:regu_U} for completeness. 

We now characterize the process $Z^{t,\xi}$ as a function of $X^{t,\xi}$. Using \eqref{def:U'}, we can write
\begin{align*}
    Y^{t,x,\mu}_s - Y^{t,x,\mu}_t =\ & Y^{t,x,\mu}_s - Y^{s,x,\mu}_s + Y^{s,x,\mu}_s - Y^{t,x,\mu}_t  \\
    =\ &U\left(s, X_s^{t,x,\mu},\lr\left(X^{t,\xi}_s\right)\right)-U(s,x,\mu) + \left[Y^{s,x,\mu}_s - Y^{t,x,\mu}_t\right].
\end{align*}
From the regularity properties of $U$ in $(x,\mu)$ established in Proposition~\ref{prop:regu_U}, we can apply It\^o's formula to the term $U\left(s, X_s^{t,x,\mu},\lr\left(X^{t,\xi}_s\right)\right)-U(s,x,\mu)$. Using the continuity in time of the derivatives $\left(U_x,U_{xx},D_{x'}\frac{dU}{d\nu}, D^2_{x'}\frac{dU}{d\nu}\right)$, and the fact that the term $\left[Y^{s,x,\mu}_s - Y^{t,x,\mu}_t\right]$ is deterministic, the uniqueness of the BSDE for $Y^{t,x,\mu}$ yields
\begin{align}\label{relation:Z}
    Z^{t,x,\mu}_s=U_x \left(s, X^{t,x,\mu}_s,\lr\left(X^{t,\xi}_s\right)\right)\sigma \left(s, X^{t,x,\mu}_s, \lr\left(X^{t,\xi}_s\right)\right).
\end{align}

Using \eqref{def:U'} and \eqref{relation:Z}, we obtain
\begin{align}\label{distribution}
    \lr\left(\xi, Y^{t,\xi}_t,Z^{t,\xi}_t\right)= \Big(I_n\otimes U\left(t, \cdot,\mu\right) \otimes \left[U_x \left(t, \cdot,\mu\right)\sigma \left(t, \cdot, \mu\right)\right]\Big)\#\mu,
\end{align}
where the right-hand side denotes the push-forward of $\mu$ under the map $$x\mapsto \left(x,U(t,x,\mu),U_x(t,x,\mu)\sigma(t,x,\mu)\right).$$ The relation \eqref{relation:Z}, together with Propositions~\ref{prop:flow:cinti:t} and \ref{prop:regu_U}, guarantees the continuity of $Z^{t,\xi}_s$ in $s$. Consequently, this continuity and Theorem~\ref{thm:flow:HJB:pre} imply the differentiability of $V(t,\mu)$ in $t$. We are now ready to state the following result, which characterizes the HJB equation for $V$ and the decoupling field $U$ as the solution of a coupled system of backward PDEs on the measure space $\pr_2(\brn)$.

\begin{theorem}\label{thm:flow:HJB:main}
    Under Assumptions (A1), (A2''), (A3), (A4) and the validity of Conditions \eqref{condition:1}-\eqref{condition:2}, the value function $V$ and the decoupling field $U$ give a classical solution of the following coupled system of backward PDEs: for $(t,x,\mu)\in[0,T]\times\brn\times\pr_2(\brn)$,
    \begin{equation}\label{discuss:Bell_main}
        \left\{
        \begin{aligned}
            &\dd_t V(t,\mu) + \int_\brn \bigg\{ \frac{1}{2}\text{Tr}\left[\left(D^2_{x}\frac{dV}{d\nu}\left(t,\mu\right)(x)-\left[U_x^\top\varphi_{yy} (U)U_x\right](t, x,\mu)\right) \left(\sigma\sigma^\top\right) (t,x,\mu)\right]  \\
            &\quad\qquad\qquad\qquad + \hr\left(\varpi(t,x,\mu) \right) \bigg\} \mu(dx)=0,\\
            &\dd_t U \left(t, x,\mu \right) +\frac{1}{2}\text{Tr}\left[\left(\sigma\sigma^\top\right) \left(t, x, \mu \right) U_{xx} \left(t, x,\mu \right)\right] + U_x \left(t, x,\mu \right) \hr_p \left(\varpi(t,x,\mu) \right) - \hr_r (\varpi(t,x,\mu)) \\
            &\quad\qquad\qquad + \int_\brn \bigg\{ \frac{1}{2}\text{Tr}\left[\left(\sigma\sigma^\top\right) \left(t, x', \mu \right) D^2_{x'}\frac{dU}{d\nu} \left(t, x,\mu\right)\left(x'\right)\right] \\
            &\qquad\qquad\qquad\qquad + D_{x'}\frac{dU}{d\nu} \left(t, x,\mu \right)(x') \hr_p \left(\varpi(t,x',\mu)\right)\bigg\} \mu(dx') =0,\\
            &V(T,\mu) = \int_\brn \left[\varphi(g(x,\mu))+\Phi(x,\mu)\right]\mu(dx),\quad U \left(T, x,\mu \right)= g(x,\mu),
        \end{aligned}
        \right.
    \end{equation}
    where we use the notation $\varpi(t,x,\mu)$ to denote the tuple
    \begin{align*}
        \varpi(t,x,\mu):=\bigg(&t,\ x,\ U(t,x,\mu),\ U_x(t,x,\mu)\sigma(t,x,\mu),\ \Big(I_n\otimes U\left(t, \cdot,\mu\right) \otimes \left[U_x \left(t, \cdot,\mu\right)\sigma \left(t, \cdot, \mu\right)\right]\Big)\#\mu, \\
        &D_{x}\frac{dV}{d\nu}(t,\mu)(x),\  -\varphi_y (U(t,x,\mu)) \bigg).
    \end{align*}
\end{theorem}

\begin{proof}
The HJB equation for $V$ is a direct consequence of Theorem~\ref{thm:flow:HJB:pre} and the relations \eqref{def:U'}, \eqref{relation:Z} and \eqref{distribution}. For the decoupling field $U$, we can write
\begin{align*}
    &U(s,x,\mu)-U(t,x,\mu)= Y^{s,x,\mu}_s - Y^{t,x,\mu}_t = \e \left[Y^{s,x,\mu}_s-Y^{t, x,\mu}_s + Y^{t,x,\mu}_s - Y^{t,x,\mu}_t  \right]\\
    =\ & \e\left[U(s,x,\mu)- U\left(s, X_s^{t,x,\mu},\lr\left(X^{t,\xi}_s\right)\right) \right] - \e \int_t^s F \left(r, \theta^{t,x,\mu}_r, \lr\left(\theta^{t,\xi}_r\right) \right) dr. 
\end{align*}
From Proposition~\ref{prop:regu_U} and the continuity of the map $r\mapsto\left(X^{t,x,\mu}_r, X^{t,\xi}_r, Y^{t,x,\mu}_r, Z^{t,x,\mu}_r \right)$ in $r$, we know that $U$ is differentiable in the time argument. Then, by applying It\^o's formula for the term $\left[U(s,x,\mu)- U\left(s, X_s^{t,x,\mu},\lr\left(X^{t,\xi}_s\right)\right) \right]$ and taking the limit as $s\to t$, we can obtain the PDE for $U$. 
\end{proof}

\subsection{Examples ensuring flow property}\label{subsec:example}

In this subsection, we provide examples where Conditions \eqref{condition:1}-\eqref{condition:2} are satisfied. These conditions suggest considering the case in which the functions $h$ and $\varphi$ take the following quadratic form in $(y,z)$:

\textbf{(D)} The functionals $\varphi$ and $h$ are assumed to take the following quadratic forms:
\begin{align*}
    \varphi(y)=\frac{1}{2} y^\top \Psi y, \qquad \text{$\Psi$ is positive definite},
\end{align*}
and
\begin{equation}\label{Condition:h}
\begin{aligned}
    h(s,x,y,z,\varrho,v)=\ & h^0\left(s,x,\varrho_x,v\right) -y^\top \Psi  f_0(s) - y^\top \Psi f_1(s)x - y^\top \Psi f_4(s)\overline{\varrho_x} \\
    &- \frac{1}{2}y^\top \left(f_2(s)^\top \Psi + \Psi f_2(s) \right) y - \frac{1}{2}\overline{\varrho_y}^\top \left(f_5(s)^\top \Psi + \Psi f_5(s) \right) \overline{\varrho_y} \\
    &- \sum_{j=1}^n (z^j)^\top \big[f^j_3(s)^\top \Psi y + f^j_6(s)^\top \Psi \overline{\varrho_y} \big] + \frac{1}{2}\sum_{j=1}^n (z^j)^\top \Psi z^j - y^\top \Psi f_7(s)v.
\end{aligned}
\end{equation}
The function $h^0:[0,T]\times\brn\times\pr_2(\brn)\times\brd\to\br$ is convex in the following sense: for some $\lambda>0$ and $\lambda_v,\lambda_x,\lambda_m\ge0$,
\begin{align}
    &h^0\left(s,x^2,\lr\left(\xi^2\right),v^2\right)-h^0\left(s,x^1,\lr\left(\xi^1\right),v^1\right) \notag \\
    \geq\ & h^0_x\left(s,x^1,\lr\left(\xi^1\right),v^1\right)^\top \left(x^2-x^1\right) +h^0_v\left(s,x^1,\lr\left(\xi^1\right),v^1\right)^\top \left(v^2-v^1\right) \notag \\
    &+\e\left[\left(D_{x'}\frac{dh^0}{d\nu}\left(s,x^1,\lr\left(\xi^1\right),v^1\right)\left(\xi^1\right)\right)^\top \left(\xi^2-\xi^1\right)\right] \notag \\
    &+(\lambda+\lambda_v)\left|v^2-v^1\right|^2 + \lambda_x \left|x^2-x^1\right|^2 + \lambda_m \left|\e\left[\xi^2-\xi^1\right]\right|^2, \label{convex_h^0}
\end{align}
and the function $x'\mapsto \frac{dh^0}{d\nu}(s,x,m,v)(x')$ is convex. The terminal cost $\Phi$ is assumed convex in $(x,m)$ in the analogous sense. Moreover, the following block matrix is positive semi-definite for every $s\in[0,T]$:
\small
\begin{equation}\label{matrix}
\left(
\begin{array}{cccccccc}
   2\lambda_x I_n  & \  & \  & -f_1(s)^\top \Psi  & \ & \  & \  & \  \\
   \   & 2\lambda_m I_n  & \  & -f_4(s)^\top \Psi & \ &  \ & \ & \   \\
   \   & \  & 2\lambda_v I_d & -f_7(s)^\top \Psi & \ &  \  & \  & \   \\
   -\Psi f_1(s) & -\Psi f_4(s) & -\Psi f_7(s) & -f_2(s)^\top \Psi-\Psi f_2(s) & 0 &  -\Psi f^1_3(s) & \cdots & -\Psi f^n_3(s) \\
   \ & \ & \ & 0 & -f_5(s)^\top \Psi-\Psi f_5(s) & -\Psi f^1_6(s) & \cdots & -\Psi f^n_6(s) \\
   \  & \  & \  & -f^1_3(s)^\top \Psi & -f^1_6(s)^\top \Psi &  \Psi & \  & \  \\
   \  & \  & \  & \vdots & \vdots &  \  & \ddots  & \  \\
   \  & \  & \  & -f^n_3(s)^\top \Psi & -f^n_6(s)^\top \Psi &  \  & \  & \Psi 
\end{array}
\right).
\end{equation}
\normalsize

The following result shows that Assumption (D) indeed implies the convexity conditions on $h$ in (A3). Its proof is given in Appendix~\ref{App:convex_h}.

\begin{lemma}\label{convex_h}
    Under Assumptions (A1), (A2) and (D), the convexity condition on $h$ in (A3) is satisfied.
\end{lemma}

The next result shows that Assumption (D) ensures the conditions \eqref{condition:1}-\eqref{condition:2}, and hence the flow property \eqref{flow_property}.

\begin{proposition}\label{prop:relation:Y-R}
    Under Assumptions (A1), (A2) and (D), we have the relation
    \begin{align}\label{relation:Y-R}
        R^{t,\xi}_s=-\Psi Y^{t,\xi}_s,\quad s\in[t,T],
    \end{align}
    and Conditions \eqref{condition:1} and \eqref{condition:2} hold true. As a consequence, the flow property \eqref{flow_property} and the relation \eqref{feedback_RP} hold true. 
\end{proposition}

\begin{proof}
From \eqref{Condition:h}, we can compute that
\begin{equation}\label{prop:flow_1}
\begin{aligned}
    &h_y(s,x,y,z,\varrho,v)=-\Psi f(s,x,y,z,\varrho,v) - f_2(s)^\top \Psi y + \Psi f_5(s) \overline{\varrho_y} + \sum_{j=1}^n \Psi f^j_6(s) \overline{\varrho_{z^j}},\\
    &h_{z^j}(s,x,y,z,\varrho,v)=-f^j_3(s)^\top \Psi y -f^j_6(s)^\top \Psi \overline{\varrho_y} + \Psi z^j;
\end{aligned}
\end{equation}
and also the following derivatives in distribution variable
\begin{equation}\label{prop:flow_2}
\begin{aligned}
    &D_{y'}\frac{dh}{d\nu}(s,x,y,z,\varrho,v)(x',y',z')= -\left(f_5(s)^\top \Psi + \Psi f_5(s) \right)\overline{\varrho_y} - \sum_{j=1}^n \Psi f^j_6(s) z^j,\\
    &D_{z'}\frac{dh}{d\nu}(s,x,y,z,\varrho,v)(x',y',z')=0.
\end{aligned}
\end{equation}
Therefore, the SDE for $R^{t,\xi}$ can also be written as
\begin{align*}
    dR^{t,\xi}_s=\ &  \bigg\{ f_2(s)^\top  \left(R^{t,\xi}_s + \Psi Y^{t,\xi}_s \right)  +f_5(s)^\top  \e\left[R^{t,\xi}_s+ \Psi Y^{t,\xi}_s\right]  +\Psi f \left(s,\theta^{t,\xi}_s,\lr\left({ \theta^{t,\xi}_s}\right),u^{t,\xi}_s\right) \bigg\} ds \\
    &+ \sum_{j=1}^n \bigg\{ f^j_3(s)^\top  \left(R^{t,\xi}_s + \Psi Y^{t,\xi}_s \right)  +f^j_6(s)^\top  \e\left[R^{t,\xi}_s+ \Psi Y^{t,\xi}_s\right]  -\Psi Z^{t,\xi,j}_s \bigg\} dB^j_s. 
\end{align*}
Then, by using the BSDE for $\left(Y^{t,\xi},Z^{t,\xi}\right)$, we have 
\begin{align*}
    d\left(\Psi Y^{t,\xi}_s + R^{t,\xi}_s\right)=\ &  \bigg\{ f_2(s)^\top  \left(R^{t,\xi}_s + \Psi Y^{t,\xi}_s \right)  +f_5(s)^\top  \e\left[R^{t,\xi}_s+ \Psi Y^{t,\xi}_s\right]  \bigg\} ds \\
    &+ \sum_{j=1}^n \bigg\{ f^j_3(s)^\top  \left(R^{t,\xi}_s + \Psi Y^{t,\xi}_s \right)  +f^j_6(s)^\top  \e\left[R^{t,\xi}_s+ \Psi Y^{t,\xi}_s\right] \bigg\} dB^j_s. 
\end{align*}
From the initial condition of $R^{t,\xi}_t$, we also see that $\Psi Y^{t,\xi}_t+ R^{t,\xi}_t=0$. Therefore, from the uniqueness result of the linear SDE, we know that \eqref{relation:Y-R} holds true, which together with \eqref{prop:flow_1} and \eqref{prop:flow_2} imply Conditions \eqref{condition:1} and \eqref{condition:2}. Then, the validity of \eqref{flow_property} and \eqref{feedback_RP} are consequences of Proposition~\ref{prop:s&n}. 
\end{proof}

Assumption (D) provides a class of cases with the flow property. In fact, by using \eqref{relation:Y-R} and \eqref{prop:flow_1}, we know that the optimality condition can also be written as
\begin{align*}
    0=\ & b_3(s)^\top P^{t,\xi}_s - f_7(s)^\top R^{t,\xi}_s + h^0_v\left(s,X^{t,\xi}_s,\lr\left(X^{t,\xi}_s\right),u^{t,\xi}_s\right) -  f_7(s)^\top \Psi Y^{t,\xi}_s \\
    =\ & b_3(s)^\top P^{t,\xi}_s + h^0_v\left(s,X^{t,\xi}_s,\lr\left(X^{t,\xi}_s\right),u^{t,\xi}_s\right),
\end{align*}
which shows that the optimal control $u^{t,\xi}_s$ is a feedback function of $X^{t,\xi}_s$ and $P^{t,\xi}_s$, not $Y^{t,\xi}_s$ or $Z^{t,\xi}_s$.  This means that, under Assumption (D), the system for $\Pi^{t,\xi}$ decouples into two FBSDEs: one for $\left(X^{t,\xi},P^{t,\xi},Q^{t,\xi}\right)$ and another for $\left(Y^{t,\xi},Z^{t,\xi},R^{t,\xi}\right)$. Specifically, under Assumption (D), the equations in \eqref{System:XYZ} can be rewritten in the following decoupled form:
\begin{equation*}
    \left\{
        \begin{aligned}
            &X^{t,\xi}_s= \xi+\int_t^s \hr_p\left(r, X^{t,\xi}_r, \lr\left(X^{t,\xi}_r\right),D_{x}\frac{dV}{d\nu}\left(r, \lr\left(X^{t,\xi}_r\right)\right)\left(X^{t,\xi}_r\right)\right) dr  \\
            &\quad\qquad +\int_t^s \sigma\left(r, X^{t,\xi}_r, \lr\left(X^{t,\xi}_r\right)\right) dB_r,\\
            &Y^{t,\xi}_s= g \left(X^{t,\xi}_T,\lr\left( X^{t,\xi}_T\right)\right) - \int_s^T  \hr_r\left(r, \theta^{t,\xi}_r, \lr\left(\theta^{t,\xi}_r\right),D_{x}\frac{dV}{d\nu}\left(r, \lr\left(X^{t,\xi}_r\right)\right)\left(X^{t,\xi}_r\right),-\Psi Y^{t,\xi}_r \right) dr \\
            &\quad\qquad -\int_s^T Z^{t,\xi}_r dB_r.
        \end{aligned}
    \right.
\end{equation*}
Of course, the validity of Conditions \eqref{condition:1}–\eqref{condition:2} is not equivalent to that of Assumption (D). Nevertheless, at least in the linear-quadratic setting, we observe that the flow property essentially implies the decoupling of the system for $\left(X^{t,\xi},Y^{t,\xi},Z^{t,\xi}\right)$ --- a phenomenon that is not generally true for control problems of mean-field FBSDEs.

%%%%%%%%%%%%%%%%%%%%%%%%%%%%%%%%%%%%%%%%%%%%%%%%%%%%%%%%%%%%%%%%%%%%%%%%%%%%
%\newpage

\normalsize 
\appendix
\appendixpage
\addappheadtotoc

\section{Proof of some statements in Section~\ref{sec:MFTC}}\label{App:sec:MFTC}

\subsection{Proof of Lemma~\ref{lem:wellposedness_adjoint}}\label{App:lem:wellposedness_adjoint}
In view of \eqref{Lagrangian} and Assumption (A1), the equation for $R^u$ can also be written as 
\begin{align*}
    R^u_s= \ & -\varphi_y \left(Y^u_t\right) + \int_t^s \bigg\{ f_2(r)^\top R^u_r { + f_5(r)^\top \e\left[R^u_r\right]} -h_y \left(r,\theta^u_r,\lr\left({ \theta^u_r}\right), u_r \right) \\
    &\qquad\qquad\qquad\qquad\qquad { - \widetilde{\e} \left[D_{y'}\frac{dh}{d\nu} \left(r,\widetilde{\theta^u_r},\lr\left(\theta^u_r\right), \widetilde{u_r} \right)\left(\theta^u_r\right) \right]} \bigg\} dr \\
    &+\sum_{j=1}^n \int_t^s \bigg\{f^j_3(r)^\top R^u_r { + f^j_6(r)^\top \e\left[R^u_r\right]} - h_{z^j}\left(r,\theta^u_r,\lr\left({ \theta^u_r}\right), u_r \right)\\
    &\qquad\qquad\qquad { - \widetilde{\e} \left[D_{{z'}^j}\frac{dh}{d\nu} \left(r,\widetilde{\theta^u_r},\lr\left(\theta^u_r\right), \widetilde{u_r} \right)\left(\theta^u_r\right) \right]}  \bigg\} dB^j_r.
\end{align*}
With the processes $u$ and $\theta^u$ being given, the last equation is a linear SDE for $R^u$, which has a unique adapted solution. Moreover, by using Assumption (A2) and standard argument for SDEs (see \cite{MR1696772} for instance), we have 
\begin{align}\label{lem:wellposedness_adjoint_1}
    \e\bigg[\sup_{s\in[t,T]}\left|R^u_s\right|^2\bigg]\le C(L,T)\; \e\left[1+\left|Y^u_t\right|^2+\int_t^T \left(|u_s|^2+\left|\theta^u_s\right|^2\right) ds\right].
\end{align}
Similarly, the BSDE for $\left(P^u,Q^u\right)$ can also be written as
\begin{align*}
    P^u_s=\ & -\int_s^T Q^u_r dB_r - g_1^\top  R^u_T - g_2^\top  \e\left[R^u_T\right] \\
    &+\Phi_x\left(X^u_T, \lr\left(X^u_T\right) \right) +\widetilde{\e}\left[D_{x'}\frac{d\Phi}{d\nu} \left(\widetilde{X^u_T}, \lr\left(X^u_T\right)\right)\left(X^u_T\right) \right] \\
    &+ \int_s^T \bigg\{b_1(r)^\top P^u_r +\sum_{j=1}^n \sigma^j_1(r)^\top Q^{u,j}_r -f_1(r)^\top R^u_r \\
    &\quad\qquad + b_2(r)^\top \e\left[P^u_r\right] +\sum_{j=1}^n \sigma^j_2(r)^\top \e\left[Q^{u,j}_r\right] -f_4(r)^\top \e\left[R^u_r\right]\\
    &\quad\qquad + h_x\left(r,\theta^u_r,\lr\left({ \theta^u_r}\right), u_r\right) + \widetilde{\e} \bigg[D_{x'}\frac{dh}{d\nu} \left(r,\widetilde{\theta^u_r}, \lr\left({ \theta^u_r}\right), \widetilde{u_r} \right)\left({ \theta^u_r}\right) \bigg] \bigg\} dr. 
\end{align*}
With the processes $u$, $\theta^u$ and $R^u$ being given, the above equation is a linear BSDE for $\left(P^u,Q^u\right)$, which has a unique adapted solution. By using standard argument for BSDEs, Assumption (A2) and \eqref{lem:wellposedness_adjoint_1}, we have 
\begin{align}
    & \e\left[\sup_{s\in[t,T]}\left|P^u_s\right|^2+\int_t^T \left|Q^u_s\right|^2 ds \right] \notag \\
    \le\ & C(L,T)\; \e\left[1+\left|X^u_T\right|^2+\left|R^u_T\right|^2+\int_t^T \left(|u_s|^2+\left|\theta^u_s\right|^2+\left|R^u_s\right|^2\right) ds\right]\notag  \\
    \le\ & C(L,T)\; \e\left[1+\left|X^u_T\right|^2+\left|Y^u_t\right|^2+\int_t^T \left(|u_s|^2+\left|\theta^u_s\right|^2\right) ds\right]. \label{lem:wellposedness_adjoint_2}
\end{align}
Substituting \eqref{SDE:bound} and \eqref{BSDE:bound} into \eqref{lem:wellposedness_adjoint_1} and \eqref{lem:wellposedness_adjoint_2}, we obtain \eqref{lem:wellposedness_adjoint_0}.

\subsection{Proof of Lemma~\ref{lem:estimate}}\label{App:lem:estimate}
We denote by $\de\xi:=\xi^2-\xi^1$, $\de \ir:=\ir^2-\ir^1$ and $\de \Pi:= \left(\Pi^2-\Pi^1\right)$. Then, from the FBSDEs for $\Pi^1$ and $\Pi^2$, we know that the processes $\de \Pi$ satisfy the following equations:
\begin{align}
            &\de X_s= \de \xi+\int_t^s \left\{\gamma \left[ b_1(r)\de X_r + b_2(r)\e\left[\de X_r\right]+b_3(r)\de u_r\right] + \de \ir^b_r \right\} dr \notag \\
            &\qquad\qquad + \int_t^s \left\{\gamma \left[ \sigma_1(r)\de X_r + \sigma_2(r)\e\left[\de X_r\right]+\sigma_3(r)\de u_r\right] + \de \ir^\sigma_r\right\} dB_r, \notag \\
            &\de Y_s= -\int_s^T \de Z_r dB_r+ \gamma g_1\de X_T+\gamma g_2\e\left[\de X_T\right] +\de\ir^g \notag \\
            &\qquad\qquad + \int_s^T \bigg\{\gamma \bigg[f_1(r)\de X_r+ f_2(r)\de Y_r +\sum_{j=1}^n f^j_3(r) \de Z^j_r + f_4(r) \e\left[\de X_r\right]{ + f_5(r) \e\left[\de Y_r\right]} \notag \\
            &\quad\qquad\qquad\qquad\qquad { +\sum_{j=1}^n f^j_6(r) \e\left[\de Z^j_r\right]} + f_7(r)\de u_r  \bigg] + \de \ir^f_r\bigg\}dr, \notag \\
            &\de P_s=  \gamma\bigg\{- g_1^\top  \de R_T - g_2^\top \e\left[\de R_T\right]  +\Phi_x\left(X^2_T, \lr\left(X^2_T\right) \right) - \Phi_x\left(X^1_T, \lr\left(X^1_T\right) \right)  \notag \\
            &\qquad\qquad +\widetilde{\e}\left[D_{x'}\frac{d\Phi}{d\nu} \left(\widetilde{X^2_T}, \lr\left(X^2_T\right)\right)\left(X^2_T\right) - D_{x'}\frac{d\Phi}{d\nu} \left(\widetilde{X^1_T}, \lr\left(X^1_T\right)\right)\left(X^1_T\right) \right] \bigg\} + \de \ir^\Phi  \notag \\
            &\quad\qquad + \int_s^T \bigg\{\gamma \bigg[ b_1(r)^\top \de P_r + \sum_{j=1}^n \sigma^j_1(r)\de Q^j_r - f_1(r)^\top \de R_r + b_2(r)^\top \e[\de P_r]  \notag \\
            &\qquad\qquad\qquad\qquad + \sum_{j=1}^n \sigma^j_2(r)\e\left[\de Q^j_r\right]- f_4(r)^\top \e[\de R_r] + h_x\left(r,\cdot\right)\big|^{\left(\theta^2_r,\lr({ \theta^2_r}),u^2_r\right)}_{\left(\theta^1_r,\lr({ \theta^1_r}),u^1_r\right)} \notag \\
            &\qquad\qquad\qquad\qquad  + \widetilde{\e} \bigg[D_{x'}\frac{dh}{d\nu} \left(r,\cdot \right)(\cdot)\Big|^{\left(\widetilde{\theta^2_r},\lr\left(\theta^2_r\right),\widetilde{u^2_r}\right)\left({ \theta^2_r}\right)}_{\left(\widetilde{\theta^1_r},\lr\left(\theta^1_r\right),\widetilde{u^1_r}\right)\left({ \theta^1_r}\right)} \bigg]\bigg]  + \de \ir^x_r \bigg\} dr -\int_s^T \de Q_r dB_r, \notag \\
            &\de R_s= -\gamma \left[\varphi_y \left(Y^2_t\right)-\varphi_y \left(Y^1_t\right)\right] - \de \ir^\varphi \notag \\
            &\quad\qquad +\int_t^s \bigg\{\gamma \bigg[f_2(r)^\top \de R_r { + f_5(r)^\top \e[\de R_r]} -h_y \left(r,\cdot\right)\big|^{\left(\theta^2_r,\lr({ \theta^2_r}),u^2_r\right)}_{\left(\theta^1_r,\lr({\theta^1_r}),u^1_r\right)} \notag\\
            &\qquad\qquad\qquad\qquad { - \widetilde{\e} \bigg[D_{y'}\frac{dh}{d\nu} \left(r,\cdot\right)(\cdot)\Big|^{\left(\widetilde{\theta^2_r},\lr\left(\theta^2_r\right),\widetilde{u^2_r}\right)\left(\theta^2_r\right)}_{\left(\widetilde{\theta^1_r},\lr\left(\theta^1_r\right),\widetilde{u^1_r}\right)\left(\theta^1_r\right)} \bigg]}\bigg] +\de \ir^y_r\bigg\} dr \notag \\
            &\qquad\qquad +\sum_{j=1}^n \int_t^s \bigg\{\gamma \bigg[f^j_3(r)^\top \de R_r{ +f^j_6(r)^\top \e[\de R_r]} -h_{z^j}\left(r,\cdot\right)\big|^{\left(\theta^2_r,\lr({ \theta^2_r}),u^2_r\right)}_{\left(\theta^1_r,\lr({ \theta^1_r}),u^1_r\right)}  \notag \\
            &\quad\qquad\qquad\qquad\qquad { - \widetilde{\e} \bigg[D_{{z'}^j}\frac{dh}{d\nu} \left(r,\cdot\right)(\cdot)\Big|^{\left(\widetilde{\theta^2_r},\lr\left(\theta^2_r\right),\widetilde{u^2_r}\right)\left(\theta^2_r\right)}_{\left(\widetilde{\theta^1_r},\lr\left(\theta^1_r\right),\widetilde{u^1_r}\right)\left(\theta^1_r\right)}\bigg]}\bigg] + \de \ir^{z^j}_r\bigg\} dB^j_r, \notag \\
            &b_3(s)^\top \de P_s+ \sum_{j=1}^n \sigma^j_3(s)^\top \de Q^j_s -f_7(s)^\top \de R_s + h_v\left(s,\cdot\right)\big|^{\left(\theta^2_s,\lr\left({\theta^2_s}\right),u^2_s\right)}_{\left(\theta^1_s,\lr\left({\theta^1_s}\right),u^1_s\right)} =0,\quad s\in[t,T]. \label{lem:estimate_FBSDE}
\end{align}
By applying It\^o's formula for $\de P_s^\top \de X_s$, taking expectation and using Fubini's theorem, we have
\begin{align}
    &\e\left[\de P_T^\top \de X_T -\de P_t^\top \de\xi\right] \notag\\
    =\ & \e\int_t^T \bigg\{ \gamma \de P_s^\top b_3(s)  \de u_s + \sum_{j=1}^n \gamma \left(\de Q^j_s\right)^\top \sigma^j_3(s) \de u_s + \gamma \de R_s^\top  f_1(s) \de X_s +\gamma \e\left[\de R_s\right]^\top f_4(s)\de X_s  \notag \\
    &\quad\qquad - \gamma \left[h_x\left(s,\cdot\right)\big|^{\left(\theta^2_s,\lr\left({ \theta^2_s}\right),u^2_s\right)}_{\left(\theta^1_s,\lr\left({ \theta^1_s}\right),u^1_s\right)}\right]^\top \de X_s  - \gamma \widetilde{\e} \bigg[D_{x'}\frac{dh}{d\nu} \left(s,\cdot\right)(\cdot)\Big|^{\left(\widetilde{\theta^2_s},\lr\left({ \theta^2_s}\right),\widetilde{u^2_s}\right)\left({ \theta^2_s}\right)}_{\left(\widetilde{\theta^1_s},\lr\left({ \theta^1_s}\right),\widetilde{u^1_s}\right)\left({ \theta^1_s}\right)} \bigg]^\top \de X_s  \notag\\
    &\quad\qquad + \de P_s^\top\de \ir^b_s  +  \sum_{j=1}^n \left(\de Q^j_s\right)^\top \de \ir^{\sigma^j}_s - \de X_s^\top \de \ir^x_s  \bigg\} ds \notag \\
    =\ & \e\int_t^T \bigg\{ \gamma \de u_s^\top \bigg[b_3(s)^\top \de P_s  + \sum_{j=1}^n  \sigma^j_3(s)^\top \de Q^j_s\bigg] + \gamma \de R_s^\top  \Big[f_1(s) \de X_s + f_4(s)\e\left[\de X_s\right] \Big]  \notag \\
    &\quad\qquad - \gamma \left[h_x\left(s,\cdot\right)\big|^{\left(\theta^2_s,\lr\left({ \theta^2_s}\right),u^2_s\right)}_{\left(\theta^1_s,\lr\left({ \theta^1_s}\right),u^1_s\right)}\right]^\top \de X_s  - \gamma \widetilde{\e} \bigg[\bigg(D_{x'}\frac{dh}{d\nu} \left(s,\cdot\right)(\cdot)\Big|^{\left(\theta^2_s,\lr\left({ \theta^2_s}\right),u^2_s\right)\left(\widetilde{ \theta^2_s}\right)}_{\left(\theta^1_s,\lr\left({ \theta^1_s}\right),u^1_s\right)\left(\widetilde{ \theta^1_s}\right)} \bigg)^\top \widetilde{\de X_s} \bigg]  \notag\\
    &\quad\qquad + \de P_s^\top\de \ir^b_s  +  \sum_{j=1}^n \left(\de Q^j_s\right)^\top \de \ir^{\sigma^j}_s - \de X_s^\top \de \ir^x_s  \bigg\} ds. \label{lem:estimate_1}
\end{align}
Similarly, by using It\^o's formula and Fubini's theorem for the term $\de R_s^\top \de Y_s$, we can deduce that 
\begin{align}
    &\e\left[\de R_T^\top \de Y_T -\de R_t^\top \de Y_t \right] \notag \\
    =\ & \e\int_t^T \bigg\{ -\gamma \de R_s^\top \bigg[f_1(s)\de X_s + f_4(s) \e\left[\de X_s\right]  \bigg] -\gamma \de u_s^\top f_7(s)\de R_s \notag\\
    &\quad\qquad  - \gamma \left[h_y \left(s,\cdot\right)\big|^{\left(\theta^2_s,\lr\left({ \theta^2_s}\right),u^2_s\right)}_{\left(\theta^1_s,\lr\left({ \theta^1_s}\right),u^1_s\right)}\right]^\top \de Y_s  -\gamma \sum_{j=1}^n   \left[h_{z^j}\left(s,\cdot\right)\big|^{\left(\theta^2_s,\lr\left({ \theta^2_s}\right),u^2_s\right)}_{\left(\theta^1_s,\lr\left({ \theta^1_s}\right),u^1_s\right)} \right]^\top \de Z^j_s \notag\\
    & \quad\qquad  - \gamma \widetilde{\e} \bigg[\bigg(D_{y'}\frac{dh}{d\nu} \left(s,\cdot\right)(\cdot)\Big|^{\left(\theta^2_s,\lr\left({ \theta^2_s}\right),u^2_s\right)\left(\widetilde{ \theta^2_s}\right)}_{\left(\theta^1_s,\lr\left({ \theta^1_s}\right),u^1_s\right)\left(\widetilde{ \theta^1_s}\right)} \bigg)^\top \widetilde{\de Y_s} \bigg]  \notag\\
    & \quad\qquad - \gamma\sum_{j=1}^n \widetilde{\e} \bigg[\bigg(D_{{z'}^j}\frac{dh}{d\nu} \left(s,\cdot\right)(\cdot)\Big|^{\left(\theta^2_s,\lr\left({ \theta^2_s}\right),u^2_s\right)\left(\widetilde{ \theta^2_s}\right)}_{\left(\theta^1_s,\lr\left({ \theta^1_s}\right),u^1_s\right)\left(\widetilde{ \theta^1_s}\right)} \bigg)^\top \widetilde{\de Z^j_s} \bigg] \notag \\
    &\quad\qquad -\de R_s^\top \de \ir^f_s + \de Y_s^\top \de \ir^y_s  +\sum_{j=1}^n \left(\de Z^j_s\right)^\top \de \ir^{z^j}_s  \bigg\} ds. \label{lem:estimate_2}
\end{align}
Plugging \eqref{lem:estimate_2} with \eqref{lem:estimate_1}, by using the last equality in \eqref{lem:estimate_FBSDE} and the convexity of $h$ in Assumption (A3), we know that
\begin{align}
    &\e\left[\de P_T^\top \de X_T + \de R_T^\top \de Y_T -\de P_t^\top \de\xi -\de R_t^\top \de Y_t \right] \notag\\
    =\ & \e\int_t^T \bigg\{ \gamma \de u_s^\top \bigg[b_3(s)^\top \de P_s  + \sum_{j=1}^n  \sigma^j_3(s)^\top \de Q^j_s - f_7(s)\de R_s \bigg] - \gamma \left[h_x(s,\cdot)\big|^{\left(\theta^2_s,\lr\left({ \theta^2_s}\right),u^2_s\right)}_{\left(\theta^1_s,\lr\left({ \theta^1_s}\right),u^1_s\right)}\right]^\top \de X_s  \notag \\
    &\quad\qquad - \gamma \left[h_y\left(s,\cdot\right)\big|^{\left(\theta^2_s,\lr\left({ \theta^2_s}\right),u^2_s\right)}_{\left(\theta^1_s,\lr\left({ \theta^1_s}\right),u^1_s\right)}\right]^\top \de Y_s  -\gamma \sum_{j=1}^n   \left[h_{z^j}\left(s,\cdot\right)\big|^{\left(\theta^2_s,\lr\left({ \theta^2_s}\right),u^2_s\right)}_{\left(\theta^1_s,\lr\left({ \theta^1_s}\right),u^1_s\right)} \right]^\top \de Z^j_s  \notag\\
    &\quad\qquad  - \gamma \widetilde{\e} \bigg[\bigg(D_{x'}\frac{dh}{d\nu} \left(s,\cdot\right)(\cdot)\Big|^{\left(\theta^2_s,\lr\left({ \theta^2_s}\right),u^2_s\right)\left(\widetilde{ \theta^2_s}\right)}_{\left(\theta^1_s,\lr\left({ \theta^1_s}\right),u^1_s\right)\left(\widetilde{ \theta^1_s}\right)} \bigg)^\top \widetilde{\de X_s} \bigg]   \notag\\
    & \quad\qquad - \gamma \widetilde{\e} \bigg[\bigg(D_{y'}\frac{dh}{d\nu} \left(s,\cdot\right)(\cdot)\Big|^{\left(\theta^2_s,\lr\left({ \theta^2_s}\right),u^2_s\right)\left(\widetilde{ \theta^2_s}\right)}_{\left(\theta^1_s,\lr\left({ \theta^1_s}\right),u^1_s\right)\left(\widetilde{ \theta^1_s}\right)} \bigg)^\top \widetilde{\de Y_s} \bigg] \notag \\
    & \quad\qquad - \gamma\sum_{j=1}^n \widetilde{\e} \bigg[\bigg(D_{{z'}^j}\frac{dh}{d\nu} \left(s,\cdot\right)(\cdot)\Big|^{\left(\theta^2_s,\lr\left({ \theta^2_s}\right),u^2_s\right)\left(\widetilde{ \theta^2_s}\right)}_{\left(\theta^1_s,\lr\left({ \theta^1_s}\right),u^1_s\right)\left(\widetilde{ \theta^1_s}\right)} \bigg)^\top \widetilde{\de Z^j_s} \bigg] + \de P_s^\top\de \ir^b_s  +  \sum_{j=1}^n \left(\de Q^j_s\right)^\top \de \ir^{\sigma^j}_s  \notag \\
    &\quad\qquad  - \de X_s^\top \de \ir^x_s-\de R_s^\top \de \ir^f_s + \de Y_s^\top \de \ir^y_s  +\sum_{j=1}^n \left(\de Z^j_s\right)^\top \de \ir^{z^j}_s\bigg\} ds \notag\\
    =\ & \e\int_t^T \bigg\{ -\gamma \left[h_v(s,\cdot)\big|^{\left(\theta^2_s,\lr\left({ \theta^2_s}\right),u^2_s\right)}_{\left(\theta^1_s,\lr\left({ \theta^1_s}\right),u^1_s\right)} \right]^\top \de u_s  - \gamma \left[h_x(s,\cdot)\big|^{\left(\theta^2_s,\lr\left({ \theta^2_s}\right),u^2_s\right)}_{\left(\theta^1_s,\lr\left({ \theta^1_s}\right),u^1_s\right)}\right]^\top \de X_s   \notag\\
    &\quad\qquad  - \gamma \left[h_y (s,\cdot)\big|^{\left(\theta^2_s,\lr\left({ \theta^2_s}\right),u^2_s\right)}_{\left(\theta^1_s,\lr\left({ \theta^1_s}\right),u^1_s\right)}\right]^\top \de Y_s  -\gamma \sum_{j=1}^n   \left[h_{z^j}(s,\cdot)\big|^{\left(\theta^2_s,\lr\left({ \theta^2_s}\right),u^2_s\right)}_{\left(\theta^1_s,\lr\left({ \theta^1_s}\right),u^1_s\right)} \right]^\top \de Z^j_s \notag\\
    &\quad\qquad  - \gamma \widetilde{\e} \bigg[\bigg(D_{x'}\frac{dh}{d\nu} \left(s,\cdot\right)(\cdot)\Big|^{\left(\theta^2_s,\lr\left({ \theta^2_s}\right),u^2_s\right)\left(\widetilde{ \theta^2_s}\right)}_{\left(\theta^1_s,\lr\left({ \theta^1_s}\right),u^1_s\right)\left(\widetilde{ \theta^1_s}\right)} \bigg)^\top \widetilde{\de X_s} \bigg]   \notag\\
    & \quad\qquad - \gamma \widetilde{\e} \bigg[\bigg(D_{y'}\frac{dh}{d\nu} \left(s,\cdot\right)(\cdot)\Big|^{\left(\theta^2_s,\lr\left({ \theta^2_s}\right),u^2_s\right)\left(\widetilde{ \theta^2_s}\right)}_{\left(\theta^1_s,\lr\left({ \theta^1_s}\right),u^1_s\right)\left(\widetilde{ \theta^1_s}\right)} \bigg)^\top \widetilde{\de Y_s} \bigg] \notag \\
    & \quad\qquad - \gamma\sum_{j=1}^n \widetilde{\e} \bigg[\bigg(D_{{z'}^j}\frac{dh}{d\nu} \left(s,\cdot\right)(\cdot)\Big|^{\left(\theta^2_s,\lr\left({ \theta^2_s}\right),u^2_s\right)\left(\widetilde{ \theta^2_s}\right)}_{\left(\theta^1_s,\lr\left({ \theta^1_s}\right),u^1_s\right)\left(\widetilde{ \theta^1_s}\right)} \bigg)^\top \widetilde{\de Z^j_s} \bigg] + \de P_s^\top\de \ir^b_s  +  \sum_{j=1}^n \left(\de Q^j_s\right)^\top \de \ir^{\sigma^j}_s  \notag \\
    &\quad\qquad - \de X_s^\top \de \ir^x_s-\de R_s^\top \de \ir^f_s + \de Y_s^\top \de \ir^y_s  +\sum_{j=1}^n \left(\de Z^j_s\right)^\top \de \ir^{z^j}_s\bigg\} ds \notag\\
    \le\ & \e\int_t^T \bigg[ -2\gamma \lambda |\de u_s|^2 + \de P_s^\top\de \ir^b_s  +  \sum_{j=1}^n \left(\de Q^j_s\right)^\top \de \ir^{\sigma^j}_s - \de X_s^\top \de \ir^x_s-\de R_s^\top \de \ir^f_s \notag\\
    &\quad\qquad   + \de Y_s^\top \de \ir^y_s  +\sum_{j=1}^n \left(\de Z^j_s\right)^\top \de \ir^{z^j}_s\bigg] ds. \label{lem:estimate_3}
\end{align}
By using the terminal conditions for $\de P_T$ and $\de Y_T$, the initial condition for $\de R_t$, and then  applying Fubini's theorem and also the convexity conditions on $\varphi$ and $\Phi$ in Assumption (A3), we have
\begin{align}
    &\e\left[\de P_T^\top \de X_T + \de R_T^\top \de Y_T-\de R_t^\top \de Y_t \right] \notag \\
    =\ & \e\bigg\{ \gamma \de Y_t^\top \left[\varphi_y \left(Y^2_t\right)-\varphi_y \left(Y^1_t\right)\right]+ \gamma \left[\Phi_x\left(X^2_T, \lr\left(X^2_T\right) \right) - \Phi_x\left(X^1_T, \lr\left(X^1_T\right) \right)\right]^\top \de X_T\notag \\
    &\quad +\gamma\widetilde{\e}\bigg[\left(D_{x'}\frac{d\Phi}{d\nu} \left(X^2_T, \lr\left(X^2_T\right)\right)\left(\widetilde{X^2_T}\right) - D_{x'}\frac{d\Phi}{d\nu} \left(X^1_T, \lr\left(X^1_T\right)\right)\left(\widetilde{X^1_T}\right) \right)^\top \widetilde{\de X_T} \bigg] \notag\\
    &\quad + \de X_T^\top \de \ir^\Phi +  \de R_T^\top \de\ir^g + \de Y_t^\top \de \ir^\varphi  \bigg\} \notag\\
    \geq\ & \e\left[ \de X_T^\top \de \ir^\Phi +  \de R_T^\top \de\ir^g + \de Y_t^\top \de \ir^\varphi  \right]. \notag
\end{align}
Substituting the last inequality back into \eqref{lem:estimate_3}, we know that for any $\epsilon\in(0,1)$,  
\begin{align}
    &2\gamma \lambda\int_t^T  \left\|\de u_s\right\|_2^2 ds \notag\\
    \le\ & \e\bigg[ \int_t^T \bigg( \de P_s^\top\de \ir^b_s  +  \sum_{j=1}^n \left(\de Q^j_s\right)^\top \de \ir^{\sigma^j}_s - \de X_s^\top \de \ir^x_s-\de R_s^\top \de \ir^f_s  + \de Y_s^\top \de \ir^y_s   \notag\\
    &\qquad\qquad  +\sum_{j=1}^n \left(\de Z^j_s\right)^\top \de \ir^{z^j}_s\bigg) ds - \de X_T^\top \de \ir^\Phi -  \de R_T^\top \de\ir^g - \de Y_t^\top \de \ir^\varphi  +\de P_t^\top \de\xi  \bigg]\notag\\
    \le\ & \int_t^T \Big( \left\|\de P_s\right\|_2 \left\|\de \ir^b_s\right\|_2 + \left\|\de Q_s\right\|_2 \left\|\de \ir^{\sigma}_s\right\|_2 + \left\|\de X_s\right\|_2 \left\|\de \ir^x_s\right\|_2  + \left\|\de R_s\right\|_2 \left\|\de \ir^f_s\right\|_2  \notag \\
    &\qquad + \left\|\de Y_s\right\|_2 \left\|\de \ir^y_s\right\|_2 +\left\|\de Z_s\right\|_2 \left\|\de \ir^z_s\right\|_2\Big)ds +  \left\|\de X_T\right\|_2  \left\|\de \ir^\Phi\right\|_2 + \left\| \de R_T\right\|_2 \left\|\de\ir^g\right\|_2 \notag \\
    & + \left\|\de Y_t\right\|_2 \left\|\de \ir^\varphi\right\|_2  + \left\|\de P_t\right\|_2 \left\|\de\xi\right\|_2 \notag\\
    \le\ & \epsilon \|(\de\pi,0)\|^2_{\mathbb{S}} + \frac{C(L,T)}{\epsilon} \left(\|\de\ir\|^2_{\mathbb{I}} + \|\de\xi\|^2_2\right). \label{lem:estimate_4}
\end{align}
Here, the tuple of processes $(\de \pi,0)$ is $\de \Pi$ with $\de u$ being replaced by $0$, and it is an element in $\mathbb{S}$. By using standard estimate for the SDE of $\de X$, we have
\begin{align}
    \e\bigg[\sup_{t\le s\le T} \left|\de X_s\right|^2\bigg]\le C(L,T) \left[ \|\de\xi\|^2_2 + \gamma^2 \int_t^T  \left\|\de u_s\right\|_2^2 ds + \int_t^T \left(\left\|\de \ir^b_s\right\|_2^2 + \left\|\de \ir^\sigma_s\right\|_2^2 \right) ds \right]; \label{lem:estimate_5} 
\end{align}
then, by using standard estimate for the BSDE of $(\de Y,\de Z)$ and \eqref{lem:estimate_5} (and also the fact that $0\le \gamma\le 1$), we have
\begin{align}
    &\e\bigg[\sup_{t\le s\le T} \left|\de Y_s\right|^2+\int_t^T \left|\de Z_s\right|^2ds \bigg] \notag \\
    \le\ & C(L,T) \bigg[\gamma^2\sup_{t\le s\le T} \left\|\de X_s\right\|_2^2+\left\|\de\ir^g\right\|^2_2 +\gamma^2 \int_t^T  \left\|\de u_s\right\|_2^2 ds + \int_t^T \left\|\de \ir^f_s\right\|_2^2  ds \bigg]\notag\\
    \le\ & C(L,T) \bigg[\|\de\xi\|^2_2 +\left\|\de\ir^g\right\|^2_2 +\gamma^2 \int_t^T  \left\|\de u_s\right\|_2^2 ds + \int_t^T\left(\left\|\de \ir^b_s\right\|_2^2 + \left\|\de \ir^\sigma_s\right\|_2^2 + \left\|\de \ir^f_s\right\|_2^2 \right)  ds \bigg]. \label{lem:estimate_6} 
\end{align}
Similarly, from standard estimate for the SDE of $\de R$ and Estimates \eqref{lem:estimate_5} and \eqref{lem:estimate_6}, we have 
\begin{align}
    \e\bigg[\sup_{t\le s\le T} \left|\de R_s\right|^2\bigg]\le\ & C(L,T) \bigg[ \gamma^2\sup_{t\le s\le T} \left\|\de Y_s\right\|^2_2 + \gamma^2 \int_t^T  \left(\left\|\de X_s\right\|_2^2 + \left\|\de Z_s\right\|_2^2 + \left\|\de u_s\right\|_2^2 \right)ds \notag \\
    &\qquad\qquad +\left\|\de\ir^\varphi\right\|_2^2 + \int_t^T \left(\left\|\de \ir^y_s\right\|_2^2 + \left\|\de \ir^z_s\right\|_2^2 \right) ds \bigg] \notag\\
    \le\ & C(L,T) \bigg[ \|\de\xi\|^2_2 +\left\|\de\ir^g\right\|^2_2 +\left\|\de\ir^\varphi\right\|_2^2 + \gamma^2 \int_t^T  \left\|\de u_s\right\|_2^2 ds \notag \\
    &\qquad\qquad + \int_t^T \left(\left\|\de \ir^b_s\right\|_2^2 + \left\|\de \ir^\sigma_s\right\|_2^2 + \left\|\de \ir^f_s\right\|_2^2 + \left\|\de \ir^y_s\right\|_2^2 + \left\|\de \ir^z_s\right\|_2^2 \right) ds \bigg]; \label{lem:estimate_7} 
\end{align}
and finally, by using standard estimate for the BSDE of $(\de P,\de Q)$ and Estimates \eqref{lem:estimate_5}-\eqref{lem:estimate_7}, we know that
\begin{align}
    &\e\bigg[\sup_{t\le s\le T} \left|\de P_s\right|^2+\int_t^T \left|\de Q_s\right|^2ds \bigg] \notag \\
    \le\ & C(L,T) \bigg[\gamma^2\sup_{t\le s\le T} \left\|\de X_s\right\|_2^2+\gamma^2\sup_{t\le s\le T} \left\|\de R_s\right\|_2^2+\left\|\de\ir^\Phi\right\|^2_2 \notag \\
    &\qquad\qquad +\gamma^2 \int_t^T  \left(\left\|\de Y_s\right\|_2^2+\left\|\de Z_s\right\|_2^2+\left\|\de u_s\right\|_2^2\right) ds + \int_t^T \left\|\de \ir^x_s\right\|_2^2  ds \bigg]\notag\\
    \le\ & C(L,T) \bigg(\|\de\xi\|^2_2 +\left\|\de\ir\right\|^2_{\mathbb{I}}  +\gamma^2 \int_t^T  \left\|\de u_s\right\|_2^2 ds \bigg). \label{lem:estimate_8} 
\end{align}
Combining \eqref{lem:estimate_5}-\eqref{lem:estimate_8}, we have
\begin{align}
    \|(\de\pi,0)\|^2_{\mathbb{S}} \le C(L,T) \bigg(\|\de\xi\|^2_2 +\left\|\de\ir\right\|^2_{\mathbb{I}}  +\gamma^2 \int_t^T  \left\|\de u_s\right\|_2^2 ds \bigg), \label{lem:estimate_9} 
\end{align}
which together with \eqref{lem:estimate_4} (and the fact that $0\le \epsilon,\gamma\le 1$) imply
\begin{align}
    \gamma \left[2\lambda- \epsilon C(L,T) \right]\int_t^T  \left\|\de u_s\right\|_2^2 ds \le    C(L,T)\left(1+\frac{1}{\epsilon} \right)\left(\|\de\ir\|^2_{\mathbb{I}} + \|\de\xi\|^2_2\right). \notag
\end{align}
By choosing $\epsilon:=\frac{\lambda}{C(L,T)}$, we obtain that
\begin{align}
    \gamma \int_t^T  \left\|\de u_s\right\|_2^2 ds \le    C(L,T,\lambda)\left(\|\de\ir\|^2_{\mathbb{I}} + \|\de\xi\|^2_2\right). \notag
\end{align}
Substituting the last estimate back into \eqref{lem:estimate_9}, we have 
\begin{align}
    \|(\de\pi,0)\|^2_{\mathbb{S}} \le C(L,T) \left(\|\de\xi\|^2_2 +\left\|\de\ir\right\|^2_{\mathbb{I}} \right). \label{lem:estimate_10} 
\end{align}
From the last equality of \eqref{lem:estimate_FBSDE}, the Lipschitz condition on $h_v$ and the convexity of $h$ in $v$, we have
\begin{align*}
    0=\ & \de u_s^\top \bigg(b_3(s)^\top \de P_s+ \sum_{j=1}^n \sigma^j_3(s)^\top \de Q^j_s -f_7(s)^\top \de R_s + h_v(s,\cdot)\big|^{\left(\theta^2_s,\lr\left({ \theta^2_s}\right),u^2_s\right)}_{\left(\theta^1_s,\lr\left({ \theta^1_s}\right),u^1_s\right)} \bigg)\\
    =\ & \de u_s^\top \bigg(b_3(s)^\top \de P_s+ \sum_{j=1}^n \sigma^j_3(s)^\top \de Q^j_s - f_7(s)^\top \de R_s + h_v(s,\cdot,u_2)\big|^{\left(\theta^2_s,\lr\left({ \theta^2_s}\right)\right)}_{\left(\theta^1_s,\lr\left({ \theta^1_s}\right)\right)} \bigg)\\
    & + \Big(h_v\left(s,\theta^1_s,\lr\left({ \theta^1_s}\right),\cdot\right)\big|^{u^2_s}_{u^1_s} \Big)^\top \de u_s\\
    \geq \ & -C(L) \left| \de u_s\right|\cdot \bigg( \left|\de P_s\right| + \left|\de Q_s\right| + \left|\de R_s\right| + \left|\de \theta_s\right| +\left\|{ \de \theta_s}\right\|_2 \bigg) + 2\lambda \left|\de u_s\right|^2,
\end{align*}
from which we know that 
\begin{align}
    \left|\de u_s\right| \le C(L,\lambda) \bigg(\left|\de \theta_s\right| +\left\|{ \de \theta_s}\right\|_2 + \left|\de P_s\right| + \left|\de Q_s\right| + \left|\de R_s\right| \bigg). \label{lem:estimate_11} 
\end{align}
By combining \eqref{lem:estimate_11} and \eqref{lem:estimate_10}, we obtain \eqref{lem:estimate_0}.

\section{Proof of some statements in Section~\ref{sec:HJB:1st}}\label{App:sec:HJB:1st}

\subsection{Proof of Lemma~\ref{lem:estimate_s}}\label{App:lem:estimate_s}
In this proof, we temporarily omit the superscript $t,\xi$ of all the processes. From \eqref{lem:kappa_0}, we have already known that 
\begin{align}
    \e\bigg[\sup_{t\le s\le T}\left|X_s\right|^2+\sup_{t\le s\le T}\left|Y_s\right|^2+\sup_{t\le s\le T}\left|P_s\right|^2+\sup_{t\le s\le T}\left|R_s\right|^2\bigg]\le C(L,T,\lambda) \left( 1+\left\|\xi\right\|^2_2\right),\label{1order-FBSDE-1}
\end{align}  
and it remains the estimate for $\e\bigg[\sup_{t\le s\le T}\left|u_s\right|^2\bigg]$. From \eqref{FBSDEs:first} and the convexity condition on $h$ in $v$ in Assumption (A3), we see that
\begin{align*}
    0
    %=\ & \left(u_s\right)^\top b_3(s)^\top  P_s -\left(u_s\right)^\top f_7(s)^\top  R_s + h_v\left(s,\theta_s,\lr({ \theta_s}),u_s\right)^\top u_s\\
    =\ & \left(u_s\right)^\top b_3(s)^\top  P_s -\left(u_s\right)^\top f_7(s)^\top  R_s +h_v\left(s,\theta_s,\lr({ \theta_s}),0\right)^\top u_s \\
    &+ \left[h_v\left(s,\theta_s,\lr({ \theta_s}),u_s\right)-h_v\left(s,\theta_s,\lr({ \theta_s}),0\right)\right]^\top u_s\\
    \geq\ & -L \left|u_s\right| \cdot \left(\left|P_s\right| + \left|R_s\right| + \left|X_s\right| + \left\|X_s\right\|_2 + \left|Y_s\right|{ + \left\|Y_s\right\|_2} \right) + 2\lambda \left|u_s\right|^2, 
\end{align*}
which implies
\begin{align*}
    \left|u_s\right| \le \frac{L}{2\lambda}  \left(\left|P_s\right| + \left|R_s\right| + \left|X_s\right| + \left\|X_s\right\|_2 + \left|Y_s\right| { + \left\|Y_s\right\|_2} \right).
\end{align*}
Then, by using \eqref{1order-FBSDE-1}, we obtain the estimate for $\e\bigg[\sup_{t\le s\le T}\left|u_s\right|^2\bigg]$. For $t\le s\le s'\le T$, by using Cauchy–Schwarz inequality and Estimate \eqref{1order-FBSDE-0}, we have 
\begin{align*}
    &\e\left[\left|X_{s'}-X_s\right|^2 + \left|Y_{s'}-Y_s\right|^2 + \left|P_{s'}-P_s\right|^2 + \left|R_{s'}-R_s\right|^2\right]\notag\\
    \le\ & C(L) (s'-s) \e\left[\int_s^{s'} \left(1+ \left|X_s\right|^2 + \left\|X_s\right\|_2^2 + \left|u_s\right|^2 + \left|Y_s\right|^2+ \left\|Y_s\right\|_2^2 + \left|P_s\right|^2 + \left|R_s\right|^2\right)dr\right] \notag \\
    \le\ & C(L,T,\lambda) (s'-s)^2, 
\end{align*}
from which we obtain \eqref{lem:estimate_s_1}. 
We use again \eqref{FBSDEs:first} to see that
\begin{align*}
    0=\ & \left[b_3(s')^\top  P_{s'} - f_7(s')^\top  R_{s'} + h_v\left(s',\theta_{s'},\lr\left({ \theta_{s'}}\right),u_{s'}\right)\right] \\
    &-\left[b_3(s)^\top  P_s - f_7(s)^\top  R_s + h_v\left(s,\theta_s,\lr\left({ \theta_s}\right),u_s\right)\right]\\
    =\ & b_3(s')^\top  \left(P_{s'}-P_{s}\right) - f_7(s')^\top  \left(R_{s'}-R_{s}\right) + h_v\left(s',\theta_{s'},\lr\left({ \theta_{s'}}\right),u_{s'}\right)-h_v\left(s',\theta_{s},\lr\left({ \theta_s}\right),u_{s'}\right)\\
    &+ \left(b_3(s')-b_3(s)\right)^\top P_{s} - \left(f_7(s')-f_7(s)\right)^\top R_{s} + h_v\left(s',\theta_{s},\lr\left({ \theta_s}\right),u_{s}\right)-h_v\left(s,\theta_{s},\lr\left({ \theta_s}\right),u_{s}\right)\\
    &+h_v\left(s',\theta_{s},\lr\left({ \theta_s}\right),u_{s'}\right)-h_v\left(s',\theta_{s},\lr\left({ \theta_s}\right),u_{s}\right),
\end{align*}
then, by using the convexity condition on $h$ in $v$ in Assumption (A3), we have 
\begin{align*}
    &2\lambda\left|u_{s'}-u_s\right|\\
    \le\ & L \left( \left|X_{s'}-X_{s}\right| +\left\|X_{s'}-X_{s}\right\|_2 +\left|Y_{s'}-Y_{s}\right| +\left\|Y_{s'}-Y_{s}\right\|_2 +\left|P_{s'}-P_{s}\right| + \left|R_{s'}-R_{s}\right| \right)\\
    &+ \left|b_3(s')-b_3(s)\right|\cdot \left|P_{s}\right| + \left|f_7(s')-f_7(s)\right|\cdot \left|R_{s}\right| + \left|h_v\left(s',\theta_{s},\lr\left({ \theta_s}\right),u_{s}\right)-h_v\left(s,\theta_{s},\lr\left({ \theta_s}\right),u_{s}\right)\right|.  
\end{align*}
Then, by using \eqref{1order-FBSDE-0} and \eqref{lem:estimate_s_1} and the continuity conditions with respect to $s$ in Assumption (B), we obtain \eqref{lem:estimate_s_2}.

\subsection{Proof of Lemma~\ref{lem:de_bound}}\label{App:lem:de_bound}
From \eqref{def:x_s^epsilon}, we can write the equations for $\left(\x^\epsilon,\y^\epsilon,\p^\epsilon,\R^\epsilon,\U^\epsilon\right)$ as 
\begin{align}
    &\x^\epsilon_s= -\int_t^{t+\epsilon} \frac{1}{\epsilon}\Big[ b_0(r) + b_1(r) X_r + b_2(r)\e\left[X_r\right]+b_3(r)u_r\Big] dr \notag \\
    &\quad\qquad + \int_{t+\epsilon}^s \left[b_1(r) \x^\epsilon_r + b_2(r)\e\left[\x^\epsilon_r\right]+b_3(r)\U^\epsilon_r\right] dr, \notag\\
    &\y^\epsilon_s= g_1 \x^\epsilon_T+ g_2 \e\left[\x^\epsilon_T\right]  + \int_s^T  \left[f_1(r) \x^\epsilon_r  + f_2(r) \y^\epsilon_r + f_4(r) \e\left[\x^\epsilon_r\right] + f_5(r) \e\left[\y^\epsilon_r\right] + f_7(r) \U^\epsilon_r \right] dr, \notag\\
    &\p^\epsilon_s= - g_1^\top \R^\epsilon_T - g_2^\top \e\left[\R^\epsilon_T\right] +\int_0^1 \Phi_{xx}\left(X^{\epsilon,\delta}_T, \lr\left(X^{\epsilon,\delta}_T\right)\right) \x^\epsilon_T d\delta  \notag\\
    &\quad\qquad + \int_0^1 \widetilde{\e}\left[D_{x'}\frac{d\Phi_x}{d\nu}\left(X^{\epsilon,\delta}_T, \lr\left(X^{\epsilon,\delta}_T\right)\right)  \left(\widetilde{X^{\epsilon,\delta}_T}\right) \widetilde{\x^\epsilon_T}\right] d\delta \notag\\
    &\quad\qquad +\int_0^1 \widetilde{\e}\bigg[\left(D_{x'}\frac{d\Phi_x}{d\nu}\left(\widetilde{X^{\epsilon,\delta}_T}, \lr\left(X^{\epsilon,\delta}_T\right)\right) \left(X^{\epsilon,\delta}_T\right)\right)^\top \widetilde{\x^\epsilon_T} \bigg]d\delta  \notag\\
    &\quad\qquad +\int_0^1 \he\widetilde{\e}\bigg[D_{x''}D_{x'}\frac{d^2\Phi}{d\nu^2}\left(\widetilde{X^{\epsilon,\delta}_T}, \lr\left(X^{\epsilon,\delta}_T\right)\right) \left(X^{\epsilon,\delta}_T,\widehat{X^{\epsilon,\delta}_T}\right)  \widehat{\x^\epsilon_T} \bigg]d\delta  \notag\\
    &\quad\qquad +\int_0^1 \widetilde{\e}\bigg[D_{x'}^2\frac{d\Phi}{d\nu}\left(\widetilde{X^{\epsilon,\delta}_T}, \lr\left(X^{\epsilon,\delta}_T\right)\right) \left(X^{\epsilon,\delta}_T\right) \x^\epsilon_T \bigg]d\delta  \notag\\
    &\quad\qquad + \int_s^T \int_0^1  b_1(r)^\top  \p^\epsilon_r - f_1(r)^\top \R^\epsilon_r + b_2(r)^\top  \e\left[\p^\epsilon_r\right] - f_4(r)^\top \left[\R^\epsilon_r\right] \notag\\
    &\qquad\qquad\qquad +\left[\left(
    \begin{array}{ccc}
        h_{xx} & h_{xy} & h_{xv} 
    \end{array}\right)\left(r,\theta^{\epsilon,\delta}_r,\lr\left({ \theta^{\epsilon,\delta}_r}\right),u^{\epsilon,\delta}_r\right)\right] \left(
    \begin{array}{c}
         \x^\epsilon_r  \\
         \y^\epsilon_r \\
         \U^\epsilon_r 
    \end{array}
    \right) \notag \\
    & \qquad\qquad\qquad +\widetilde{\e}\left\{\left[\left(\begin{array}{cc}
        D_{x'}\frac{dh_{x}}{d\nu} & D_{y'}\frac{dh_{x}}{d\nu}
    \end{array}\right)\left(r,\theta^{\epsilon,\delta}_r,\lr\left({ \theta^{\epsilon,\delta}_r}\right),u^{\epsilon,\delta}_r\right)\left(\widetilde{ \theta^{\epsilon,\delta}_r}\right)\right] \left(\begin{array}{cc}
        \widetilde{\x^\epsilon_r}  \\
        \widetilde{\y^\epsilon_r} 
    \end{array}
    \right) \right\}  \notag\\
    &\qquad\qquad\qquad +\widetilde{\e}\left\{\left[\left(\begin{array}{ccc}
        D_{x'}\frac{dh_x}{d\nu} \\ 
        D_{x'}\frac{dh_y}{d\nu} \\
        D_{x'}\frac{dh_v}{d\nu}
    \end{array}\right)\left(r,\widetilde{\theta^{\epsilon,\delta}_r},\lr\left({ \theta^{\epsilon,\delta}_r}\right),\widetilde{u^{\epsilon,\delta}_r}\right)\left({ \theta^{\epsilon,\delta}_r}\right)\right]^\top  \left(\begin{array}{c}
        \widetilde{\x^\epsilon_r}  \\
        \widetilde{\y^\epsilon_r}  \\
        \widetilde{\U^\epsilon_r} 
    \end{array}
    \right) \right\}  \notag\\
    & \qquad\qquad\qquad  + \he\widetilde{\e} \left\{ \left[\left( \begin{array}{cc}
         D_{x''}D_{x'}\frac{d^2h}{d\nu^2} & D_{y''}D_{x'}\frac{d^2h}{d\nu^2} \\
    \end{array}\right) \left(r,\widetilde{\theta^{\epsilon,\delta}_r},\lr\left({ \theta^{\epsilon,\delta}_r}\right),\widetilde{u^{\epsilon,\delta}_r}\right)\left({ \theta^{\epsilon,\delta}_r},\widehat{{ \theta^{\epsilon,\delta}_r}}\right) \right] \left(\begin{array}{c}
         \widehat{\x^\epsilon_r}   \\
         \widehat{\y^\epsilon_r} 
    \end{array}\right) \right\}  \notag\\
    & \qquad\qquad\qquad  + \widetilde{\e} \bigg\{ \left[\left( \begin{array}{cc}
         D^2_{x'}\frac{dh}{d\nu} & D_{y'}D_{x'}\frac{dh}{d\nu} \\
    \end{array}\right) \left(r,\widetilde{\theta^{\epsilon,\delta}_r},\lr\left({ \theta^{\epsilon,\delta}_r}\right),\widetilde{u^{\epsilon,\delta}_r}\right)\left({ \theta^{\epsilon,\delta}_r}\right) \right] \left(\begin{array}{c}
         \x^\epsilon_r   \\
         \y^\epsilon_r 
    \end{array}\right) \bigg\}  d\delta dr, \notag\\
    & \R^\epsilon_s= - \int_0^1\varphi_{yy} \left(Y^{\epsilon,\delta}_{t+\epsilon}\right) \y^\epsilon_{t+\epsilon} d\delta + \frac{1}{\epsilon}\int_t^{t+\epsilon} \varphi_{yy}\left(Y_r\right) f\left(r,X_r,Y_r, { \lr(X_r,Y_r)}, u_r\right) dr \notag \\
    &\quad\qquad -\frac{1}{\epsilon}\int_t^{t+\epsilon} \left\{f_2(r)^\top  R_r { + f_5(r)^\top  \e[R_r]} -h_y \left(r,\theta_r,\lr({ \theta_r}),u_r\right) { -\widetilde{\e}\left[D_{y'} \frac{dh}{d\nu}\left(r,\widetilde{\theta_r},\lr(\theta_r),\widetilde{u_r} \right)\left(\theta_r\right) \right]} \right\} dr \notag \\
    &\quad\qquad +\int_{t+\epsilon}^s \int_0^1 f_2(r)^\top  \R^\epsilon_r { + f_5(r)^\top  \e\left[\R^\epsilon_r\right]} - \left[\left( \begin{array}{ccc}
         h_{yx} & h_{yy} & h_{yv}
    \end{array}\right)\left(r,\theta^{\epsilon,\delta}_r,\lr\left({ \theta^{\epsilon,\delta}_r}\right),u^{\epsilon,\delta}_r\right)\right] \left(\begin{array}{c}
         \x^\epsilon_r \\
         \y^\epsilon_r \\
         \U^\epsilon_r
    \end{array}\right) \notag \\
    & \qquad\qquad\qquad -\widetilde{\e}\left\{ \left[\left(\begin{array}{cc}
         D_{x'}\frac{dh_{y}}{d\nu} & D_{y'}\frac{dh_{y}}{d\nu} 
    \end{array} \right) \left(r,\theta^{\epsilon,\delta}_r,\lr\left(\theta^{\epsilon,\delta}_r\right),u^{\epsilon,\delta}_r\right)\left(\widetilde{\theta^{\epsilon,\delta}_r}\right) \right] \left( \begin{array}{cc}
         \widetilde{\x^\epsilon_r}  \\
         \widetilde{\y^\epsilon_r} 
    \end{array}\right) \right\}  \notag\\
    & \qquad\qquad\qquad - \widetilde{\e}\left\{\left[\left(\begin{array}{ccc}
        D_{y'}\frac{dh_x}{d\nu} \\ 
        D_{y'}\frac{dh_y}{d\nu} \\
        D_{y'}\frac{dh_v}{d\nu}
    \end{array}\right)\left(r,\widetilde{\theta^{\epsilon,\delta}_r},\lr\left({ \theta^{\epsilon,\delta}_r}\right),\widetilde{u^{\epsilon,\delta}_r}\right)\left({ \theta^{\epsilon,\delta}_r}\right)\right]^\top  \left(\begin{array}{c}
        \widetilde{\x^\epsilon_r}  \\
        \widetilde{\y^\epsilon_r}  \\
        \widetilde{\U^\epsilon_r} 
    \end{array}
    \right) \right\}  \notag\\
    & \qquad\qquad\qquad  - \he\widetilde{\e} \left\{ \left[\left( \begin{array}{cc}
         D_{x''}D_{y'}\frac{d^2h}{d\nu^2} & D_{y''}D_{y'}\frac{d^2h}{d\nu^2} \\
    \end{array}\right) \left(r,\widetilde{\theta^{\epsilon,\delta}_r},\lr\left({ \theta^{\epsilon,\delta}_r}\right),\widetilde{u^{\epsilon,\delta}_r}\right)\left({ \theta^{\epsilon,\delta}_r},\widehat{{ \theta^{\epsilon,\delta}_r}}\right) \right] \left(\begin{array}{c}
         \widehat{\x^\epsilon_r}   \\
         \widehat{\y^\epsilon_r} 
    \end{array}\right) \right\}  \notag\\
    & \qquad\qquad\qquad  - \widetilde{\e} \bigg\{ \left[\left( \begin{array}{cc}
         D_{x'}D_{y'}\frac{dh}{d\nu} & D_{y'}^2\frac{dh}{d\nu} \\
    \end{array}\right) \left(r,\widetilde{\theta^{\epsilon,\delta}_r},\lr\left({ \theta^{\epsilon,\delta}_r}\right),\widetilde{u^{\epsilon,\delta}_r}\right)\left({ \theta^{\epsilon,\delta}_r}\right) \right] \left(\begin{array}{c}
         \x^\epsilon_r   \\
         \y^\epsilon_r 
    \end{array}\right) \bigg\} d\delta dr, \notag \\
    &0=b_3(s)^\top  \p^\epsilon_s -f_7(s)^\top  \R^\epsilon_s + \left[\left(\begin{array}{ccc}
         h_{vx} & h_{vy} & h_{vv} 
    \end{array} \right)\left(s,\theta^{\epsilon,\delta}_s,\lr\left({ \theta^{\epsilon,\delta}_s}\right),u^{\epsilon,\delta}_s\right) \left(\begin{array}{c}
         \x^\epsilon_s \\
         \y^\epsilon_s \\
         \U^\epsilon_s 
    \end{array}\right)\right] \notag \\
    & \qquad +\widetilde{\e}\left\{ \left[\left(\begin{array}{cc}
         D_{x'}\frac{dh_{v}}{d\nu} & D_{y'}\frac{dh_{v}}{d\nu} 
    \end{array} \right) \left(s,\theta^{\epsilon,\delta}_s,\lr\left(\theta^{\epsilon,\delta}_s\right),u^{\epsilon,\delta}_s\right)\left(\widetilde{\theta^{\epsilon,\delta}_s}\right) \right] \left( \begin{array}{cc}
         \widetilde{\x^\epsilon_s}  \\
         \widetilde{\y^\epsilon_s} 
    \end{array}\right) \right\},\quad s\in[t+\epsilon,T], \label{lem:xypru_11'}
\end{align}
where 
\begin{align*}
    &X^{\epsilon,\delta}_s:=X_s + \delta \left(X^{\epsilon}_s-X_s\right),\quad Y^{\epsilon,\delta}_s:=Y_s + \delta \left(Y^{\epsilon}_s-Y_s\right),\quad u^{\epsilon,\delta}_s:=u_s + \delta \left(u^{\epsilon}_s-u_s\right),
\end{align*}
and $\theta^{\epsilon,\delta}_s:=\left( X^{\epsilon,\delta}_s,Y^{\epsilon,\delta}_s\right)$. By applying It\^o's formula for $\left(\p_s^\epsilon\right)^\top \x_s^\epsilon$ for $t+\epsilon\le s\le T$, taking expectation and using Fubini's theorem, we have
\begin{align}
    &\e \left[\left(\p_T^\epsilon\right)^\top \x_T^\epsilon - \left(\p_{t+\epsilon}^\epsilon\right)^\top \x_{t+\epsilon}^\epsilon\right] \notag \\
    =\ & \e\int_{t+\epsilon}^T \int_0^1   \left(\x_s^\epsilon\right)^\top f_1(s)^\top  \R_s^\epsilon +\left(\x_s^\epsilon\right)^\top  f_4(s)^\top \e\left[\R_s^\epsilon \right] +\left(\p_s^\epsilon\right)^\top  b_3(s)\U_s^\epsilon \notag\\
    &\qquad -\left(\x_s^\epsilon\right)^\top \left[\left(
    \begin{array}{ccc}
        h_{xx} & h_{xy} & h_{xv} \\
    \end{array}\right)\left(s,\theta^{\epsilon,\delta}_s,\lr\left({ \theta^{\epsilon,\delta}_s}\right),u^{\epsilon,\delta}_s\right)\right] \left(
    \begin{array}{c}
         \x^\epsilon_s  \\
         \y^\epsilon_s \\
         \U^\epsilon_s 
    \end{array}
    \right) \notag \\
    & \qquad -\widetilde{\e}\left\{\left(\x_s^\epsilon\right)^\top\left[\left(\begin{array}{cc}
        D_{x'}\frac{dh_{x}}{d\nu} & D_{y'}\frac{dh_{x}}{d\nu}
    \end{array}\right)\left(s,\theta^{\epsilon,\delta}_s,\lr\left({ \theta^{\epsilon,\delta}_s}\right),u^{\epsilon,\delta}_s\right)\left(\widetilde{ \theta^{\epsilon,\delta}_s}\right)\right] \left(\begin{array}{cc}
        \widetilde{\x^\epsilon_s}  \\
        \widetilde{\y^\epsilon_s} 
    \end{array}
    \right) \right\} \notag \\
    &\qquad -\widetilde{\e}\left\{\left(\x_s^\epsilon\right)^\top \left[\left(\begin{array}{ccc}
        D_{x'}\frac{dh_x}{d\nu} \\ 
        D_{x'}\frac{dh_y}{d\nu} \\
        D_{x'}\frac{dh_v}{d\nu}
    \end{array}\right)\left(s,\widetilde{\theta^{\epsilon,\delta}_s},\lr\left({ \theta^{\epsilon,\delta}_s}\right),\widetilde{u^{\epsilon,\delta}_s}\right)\left({ \theta^{\epsilon,\delta}_s}\right)\right]^\top  \left(\begin{array}{c}
        \widetilde{\x^\epsilon_s}  \\
        \widetilde{\y^\epsilon_s}  \\
        \widetilde{\U^\epsilon_s} 
    \end{array}
    \right) \right\} \notag \\
    & \qquad  - \he\widetilde{\e} \left\{ \left(\x_s^\epsilon\right)^\top \left[\left( \begin{array}{cc}
         D_{x''}D_{x'}\frac{d^2h}{d\nu^2} & D_{y''}D_{x'}\frac{d^2h}{d\nu^2} \\
    \end{array}\right) \left(s,\widetilde{\theta^{\epsilon,\delta}_s},\lr\left({ \theta^{\epsilon,\delta}_s}\right),\widetilde{u^{\epsilon,\delta}_s}\right)\left({ \theta^{\epsilon,\delta}_s},\widehat{{ \theta^{\epsilon,\delta}_s}}\right) \right] \left(\begin{array}{c}
         \widehat{\x^\epsilon_s}   \\
         \widehat{\y^\epsilon_s} 
    \end{array}\right) \right\} \notag \\
    & \qquad  - \widetilde{\e} \bigg\{ \left(\x_s^\epsilon\right)^\top \left[\left( \begin{array}{cc}
         D^2_{x'}\frac{dh}{d\nu} & D_{y'}D_{x'}\frac{dh}{d\nu} \\
    \end{array}\right) \left(s,\widetilde{\theta^{\epsilon,\delta}_s},\lr\left({ \theta^{\epsilon,\delta}_s}\right),\widetilde{u^{\epsilon,\delta}_s}\right)\left({ \theta^{\epsilon,\delta}_s}\right) \right] \left(\begin{array}{c}
         \x^\epsilon_s   \\
         \y^\epsilon_s 
    \end{array}\right) \bigg\}  d\delta ds \notag \\
    =\ & \e\int_{t+\epsilon}^T \int_0^1 \left(\R_s^\epsilon\right)^\top f_1(s)  \x_s^\epsilon +\left(\R_s^\epsilon\right)^\top  f_4(s) \e\left[\x_s^\epsilon \right] +\left(\p_s^\epsilon\right)^\top  b_3(s)\U_s^\epsilon \notag\\
    &\qquad -\left(\x_s^\epsilon\right)^\top \left[\left(
    \begin{array}{ccc}
        h_{xx} & h_{xy} & h_{xv} \\
    \end{array}\right)\left(s,\theta^{\epsilon,\delta}_s,\lr\left({ \theta^{\epsilon,\delta}_s}\right),u^{\epsilon,\delta}_s\right)\right] \left(
    \begin{array}{c}
         \x^\epsilon_s  \\
         \y^\epsilon_s \\
         \U^\epsilon_s 
    \end{array}
    \right) \notag \\
    & \qquad -\widetilde{\e}\left\{\left(\x_s^\epsilon\right)^\top\left[\left(\begin{array}{cc}
        D_{x'}\frac{dh_{x}}{d\nu} & D_{y'}\frac{dh_{x}}{d\nu}
    \end{array}\right)\left(s,\theta^{\epsilon,\delta}_s,\lr\left({ \theta^{\epsilon,\delta}_s}\right),u^{\epsilon,\delta}_s\right)\left(\widetilde{ \theta^{\epsilon,\delta}_s}\right)\right] \left(\begin{array}{cc}
        \widetilde{\x^\epsilon_s}  \\
        \widetilde{\y^\epsilon_s} 
    \end{array}
    \right) \right\} \notag \\
    &\qquad -\widetilde{\e}\left\{\left(\begin{array}{c}
        \x^\epsilon_s  \\
        \y^\epsilon_s  \\
        \U^\epsilon_s 
    \end{array}
    \right)^\top  \left[\left(\begin{array}{ccc}
        D_{x'}\frac{dh_x}{d\nu} \\ 
        D_{x'}\frac{dh_y}{d\nu} \\
        D_{x'}\frac{dh_v}{d\nu}
    \end{array}\right)\left(s,\theta^{\epsilon,\delta}_s,\lr\left({ \theta^{\epsilon,\delta}_s}\right),u^{\epsilon,\delta}_s\right)\left(\widetilde{ \theta^{\epsilon,\delta}_s}\right)\right] \widetilde{\x_s^\epsilon}  \right\} \notag \\
    & \qquad  - \he\widetilde{\e} \left\{ \left(\widetilde{\x_s^\epsilon}\right)^\top \left[\left( \begin{array}{cc}
         D_{x''}D_{x'}\frac{d^2h}{d\nu^2} & D_{y''}D_{x'}\frac{d^2h}{d\nu^2} \\
    \end{array}\right) \left(s,\theta^{\epsilon,\delta}_s,\lr\left({ \theta^{\epsilon,\delta}_s}\right),u^{\epsilon,\delta}_s\right)\left(\widetilde{ \theta^{\epsilon,\delta}_s},\widehat{{ \theta^{\epsilon,\delta}_s}}\right) \right] \left(\begin{array}{c}
         \widehat{\x^\epsilon_s}   \\
         \widehat{\y^\epsilon_s} 
    \end{array}\right) \right\} \notag \\
    & \qquad  - \widetilde{\e} \left\{ \left(\widetilde{\x_s^\epsilon}\right)^\top \left[\left( \begin{array}{cc}
         D^2_{x'}\frac{dh}{d\nu} & D_{y'}D_{x'}\frac{dh}{d\nu} \\
    \end{array}\right) \left(s,\theta^{\epsilon,\delta}_s,\lr\left({ \theta^{\epsilon,\delta}_s}\right),u^{\epsilon,\delta}_s \right)\left(\widetilde{ \theta^{\epsilon,\delta}_s}\right) \right] \left(\begin{array}{c}
         \widetilde{\x^\epsilon_s}   \\
         \widetilde{\y^\epsilon_s} 
    \end{array}\right) \right\}  d\delta ds. \label{lem:xypru_1'}
\end{align}
Similarly, by using It\^o's formula and Fubini's theorem for $\left(\R_s^\epsilon\right)^\top \y_s^\epsilon$, we have 
\begin{align}
    &\e\left[\left(\R_T^\epsilon\right)^\top \y_T^\epsilon -\left(\R_{t+\epsilon}^\epsilon\right)^\top \y_{t+\epsilon}^\epsilon \right] \notag \\
    =\ & \e\int_{t+\epsilon}^T \int_0^1 -\left(\R_s^\epsilon\right)^\top f_1(r) \x_s^\epsilon -\left(\R_s^\epsilon\right)^\top f_4(s) \e\left[ \x_s^\epsilon\right] -\left(\R_s^\epsilon\right)^\top f_7(s) \U_s^\epsilon \notag\\
    &\qquad - \left(\y_s^\epsilon\right)^\top \left[\left( \begin{array}{ccc}
         h_{yx} & h_{yy} & h_{yv}
    \end{array}\right)\left(s,\theta^{\epsilon,\delta}_s,\lr\left({ \theta^{\epsilon,\delta}_s}\right),u^{\epsilon,\delta}_s\right)\right] \left(\begin{array}{c}
         \x^\epsilon_s \\
         \y^\epsilon_s \\
         \U^\epsilon_s
    \end{array}\right) \notag \\
    & \qquad -\widetilde{\e}\left\{\left(\y_s^\epsilon\right)^\top \left[\left(\begin{array}{cc}
         D_{x'}\frac{dh_{y}}{d\nu} & D_{y'}\frac{dh_{y}}{d\nu} 
    \end{array} \right) \left(s,\theta^{\epsilon,\delta}_s,\lr\left(\theta^{\epsilon,\delta}_s\right),u^{\epsilon,\delta}_s\right)\left(\widetilde{\theta^{\epsilon,\delta}_s}\right) \right] \left( \begin{array}{cc}
         \widetilde{\x^\epsilon_s}  \\
         \widetilde{\y^\epsilon_s} 
    \end{array}\right) \right\}  \notag \\
    & \qquad - \widetilde{\e}\left\{\left(\begin{array}{c}
        \x^\epsilon_s  \\
        \y^\epsilon_s  \\
        \U^\epsilon_s 
    \end{array}
    \right)^\top \left[\left(\begin{array}{ccc}
        D_{y'}\frac{dh_x}{d\nu} \\ 
        D_{y'}\frac{dh_y}{d\nu} \\
        D_{y'}\frac{dh_v}{d\nu}
    \end{array}\right)\left(s,\theta^{\epsilon,\delta}_s,\lr\left({ \theta^{\epsilon,\delta}_s}\right),u^{\epsilon,\delta}_s\right)\left(\widetilde{ \theta^{\epsilon,\delta}_s}\right)\right] \widetilde{\y_s^\epsilon} \right\} \notag \\
    & \qquad  - \he\widetilde{\e} \left\{\left(\widetilde{\y_s^\epsilon}\right)^\top \left[\left( \begin{array}{cc}
         D_{x''}D_{y'}\frac{d^2h}{d\nu^2} & D_{y''}D_{y'}\frac{d^2h}{d\nu^2} \\
    \end{array}\right) \left(s,\theta^{\epsilon,\delta}_s,\lr\left({ \theta^{\epsilon,\delta}_s}\right),u^{\epsilon,\delta}_s\right)\left(\widetilde{ \theta^{\epsilon,\delta}_s},\widehat{{ \theta^{\epsilon,\delta}_s}}\right) \right] \left(\begin{array}{c}
         \widehat{\x^\epsilon_s}   \\
         \widehat{\y^\epsilon_s} 
    \end{array}\right) \right\} \notag \\
    & \qquad  - \widetilde{\e} \left\{ \left(\widetilde{\y_s^\epsilon}\right)^\top \left[\left( \begin{array}{cc}
         D_{x'}D_{y'}\frac{dh}{d\nu} & D_{y'}^2\frac{dh}{d\nu} \\
    \end{array}\right) \left(s,\theta^{\epsilon,\delta}_s,\lr\left({ \theta^{\epsilon,\delta}_s}\right),u^{\epsilon,\delta}_s\right)\left(\widetilde{ \theta^{\epsilon,\delta}_s}\right) \right] \left(\begin{array}{c}
         \widetilde{\x^\epsilon_s}   \\
         \widetilde{\y^\epsilon_s} 
    \end{array}\right) \right\} d\delta ds \label{lem:xypru_2'}. 
\end{align}
Plugging \eqref{lem:xypru_2'} with \eqref{lem:xypru_1'}, by using the last equality in \eqref{lem:xypru_11'} and the convexity condition \eqref{def:convex_h}, we know that
\begin{align}
    &\e \left[\left(\p_T^\epsilon\right)^\top \x_T^\epsilon  - \left(\p_{t+\epsilon}^\epsilon\right)^\top \x_{t+\epsilon}^\epsilon + \left(\R_T^\epsilon\right)^\top \y_T^\epsilon -\left(\R_{t+\epsilon}^\epsilon\right)^\top \y_{t+\epsilon}^\epsilon \right] \notag \\
    =\ & \e\int_{t+\epsilon}^T \int_0^1  \left(\U_s^\epsilon\right)^\top \left(b_3(s)^\top \p_s^\epsilon -f_7(s)^\top \R_s^\epsilon \right) \notag \\
    & -\left(\begin{array}{cc}
         \x_s^\epsilon  \\
         \y_s^\epsilon 
    \end{array}\right)^\top \left[\left(
    \begin{array}{ccc}
        h_{xx} & h_{xy} & h_{xv} \\
        h_{yx} & h_{yy} & h_{yv}
    \end{array}\right)\left(s,\theta^{\epsilon,\delta}_s,\lr\left({ \theta^{\epsilon,\delta}_s}\right),u^{\epsilon,\delta}_s\right)\right] \left(
    \begin{array}{c}
         \x^\epsilon_s  \\
         \y^\epsilon_s \\
         \U^\epsilon_s 
    \end{array}
    \right) \notag \\
    &  -\widetilde{\e}\left\{\left(\begin{array}{c}
         \x_s^\epsilon  \\
         \y_s^\epsilon 
    \end{array}\right)^\top\left[\left(\begin{array}{cc}
        D_{x'}\frac{dh_{x}}{d\nu} & D_{y'}\frac{dh_{x}}{d\nu} \\
        D_{x'}\frac{dh_{y}}{d\nu} & D_{y'}\frac{dh_{y}}{d\nu}
    \end{array}\right)\left(s,\theta^{\epsilon,\delta}_s,\lr\left({ \theta^{\epsilon,\delta}_s}\right),u^{\epsilon,\delta}_s\right)\left(\widetilde{ \theta^{\epsilon,\delta}_s}\right)\right] \left(\begin{array}{cc}
        \widetilde{\x^\epsilon_s}  \\
        \widetilde{\y^\epsilon_s} 
    \end{array}
    \right) \right\} \notag \\
    & -\widetilde{\e}\left\{\left(\begin{array}{c}
        \x^\epsilon_s  \\
        \y^\epsilon_s  \\
        \U^\epsilon_s 
    \end{array}
    \right)^\top  \left[\left(\begin{array}{cc}
        D_{x'}\frac{dh_x}{d\nu} & D_{y'}\frac{dh_x}{d\nu} \\ 
        D_{x'}\frac{dh_y}{d\nu} & D_{y'}\frac{dh_y}{d\nu} \\
        D_{x'}\frac{dh_v}{d\nu} & D_{y'}\frac{dh_v}{d\nu}
    \end{array}\right)\left(s,\theta^{\epsilon,\delta}_s,\lr\left({ \theta^{\epsilon,\delta}_s}\right),u^{\epsilon,\delta}_s\right)\left(\widetilde{ \theta^{\epsilon,\delta}_s}\right)\right] \left(\begin{array}{c}
         \widetilde{\x_s^\epsilon}   \\
         \widetilde{\y_s^\epsilon}  
    \end{array}\right) \right\} \notag \\
    &   - \he\widetilde{\e} \left\{ \left(\begin{array}{c}
         \widetilde{\x_s^\epsilon}  \\
         \widetilde{\y_s^\epsilon} 
    \end{array}\right)^\top \left[\left( \begin{array}{cc}
         D_{x''}D_{x'}\frac{d^2h}{d\nu^2} & D_{y''}D_{x'}\frac{d^2h}{d\nu^2} \\
         D_{x''}D_{y'}\frac{d^2h}{d\nu^2} & D_{y''}D_{y'}\frac{d^2h}{d\nu^2} 
    \end{array}\right) \left(s,\theta^{\epsilon,\delta}_s,\lr\left({ \theta^{\epsilon,\delta}_s}\right),u^{\epsilon,\delta}_s\right)\left(\widetilde{ \theta^{\epsilon,\delta}_s},\widehat{{ \theta^{\epsilon,\delta}_s}}\right) \right] \left(\begin{array}{c}
         \widehat{\x^\epsilon_s}   \\
         \widehat{\y^\epsilon_s} 
    \end{array}\right) \right\} \notag \\
    &   - \widetilde{\e} \bigg\{ \left[\left( \begin{array}{cc}
         D^2_{x'}\frac{dh}{d\nu} & D_{y'}D_{x'}\frac{dh}{d\nu} \\
         D_{x'}D_{y'}\frac{dh}{d\nu} & D_{y'}^2\frac{dh}{d\nu} 
    \end{array}\right) \left(s,\theta^{\epsilon,\delta}_s,\lr\left({ \theta^{\epsilon,\delta}_s}\right),u^{\epsilon,\delta}_s \right)\left(\widetilde{ \theta^{\epsilon,\delta}_s}\right) \right] \left(\begin{array}{c}
         \widetilde{\x^\epsilon_s}   \\
         \widetilde{\y^\epsilon_s} 
    \end{array}\right)^{\otimes 2} \bigg\}  d\delta ds \notag \\
    =\ & \e\int_{t+\epsilon}^T \int_0^1  -\left[\left(
    \begin{array}{ccc}
        h_{xx} & h_{xy} & h_{xv} \\
        h_{yx} & h_{yy} & h_{yv} \\
        h_{vx} & h_{vy} & h_{vv}
    \end{array}\right)\left(s,\theta^{\epsilon,\delta}_s,\lr\left({ \theta^{\epsilon,\delta}_s}\right),u^{\epsilon,\delta}_s\right)\right] \left(
    \begin{array}{c}
         \x^\epsilon_s  \\
         \y^\epsilon_s \\
         \U^\epsilon_s 
    \end{array}
    \right)^{\otimes 2} \notag \\
    & -2 \widetilde{\e}\left\{\left(\begin{array}{c}
        \x^\epsilon_s  \\
        \y^\epsilon_s  \\
        \U^\epsilon_s 
    \end{array}
    \right)^\top  \left[\left(\begin{array}{cc}
        D_{x'}\frac{dh_x}{d\nu} & D_{y'}\frac{dh_x}{d\nu} \\ 
        D_{x'}\frac{dh_y}{d\nu} & D_{y'}\frac{dh_y}{d\nu} \\
        D_{x'}\frac{dh_v}{d\nu} & D_{y'}\frac{dh_v}{d\nu}
    \end{array}\right)\left(s,\theta^{\epsilon,\delta}_s,\lr\left({ \theta^{\epsilon,\delta}_s}\right),u^{\epsilon,\delta}_s\right)\left(\widetilde{ \theta^{\epsilon,\delta}_s}\right)\right] \left(\begin{array}{c}
         \widetilde{\x_s^\epsilon}   \\
         \widetilde{\y_s^\epsilon}  
    \end{array}\right) \right\} \notag \\
    &   - \he\widetilde{\e} \left\{ \left(\begin{array}{c}
         \widetilde{\x_s^\epsilon}  \\
         \widetilde{\y_s^\epsilon} 
    \end{array}\right)^\top \left[\left( \begin{array}{cc}
         D_{x''}D_{x'}\frac{d^2h}{d\nu^2} & D_{y''}D_{x'}\frac{d^2h}{d\nu^2} \\
         D_{x''}D_{y'}\frac{d^2h}{d\nu^2} & D_{y''}D_{y'}\frac{d^2h}{d\nu^2} 
    \end{array}\right) \left(s,\theta^{\epsilon,\delta}_s,\lr\left({ \theta^{\epsilon,\delta}_s}\right),u^{\epsilon,\delta}_s\right)\left(\widetilde{ \theta^{\epsilon,\delta}_s},\widehat{{ \theta^{\epsilon,\delta}_s}}\right) \right] \left(\begin{array}{c}
         \widehat{\x^\epsilon_s}   \\
         \widehat{\y^\epsilon_s} 
    \end{array}\right) \right\} \notag \\
    &   - \widetilde{\e} \bigg\{ \left[\left( \begin{array}{cc}
         D^2_{x'}\frac{dh}{d\nu} & D_{y'}D_{x'}\frac{dh}{d\nu} \\
         D_{x'}D_{y'}\frac{dh}{d\nu} & D_{y'}^2\frac{dh}{d\nu} 
    \end{array}\right) \left(s,\theta^{\epsilon,\delta}_s,\lr\left({ \theta^{\epsilon,\delta}_s}\right),u^{\epsilon,\delta}_s \right)\left(\widetilde{ \theta^{\epsilon,\delta}_s}\right) \right] \left(\begin{array}{c}
         \widetilde{\x^\epsilon_s}   \\
         \widetilde{\y^\epsilon_s} 
    \end{array}\right)^{\otimes 2} \bigg\}  d\delta ds \notag \\
    \le\ & -2\lambda \int_{t+\epsilon}^T \int_0^1 \left\|\U_s^\epsilon\right\|_2^2 d\delta ds. \label{lem:xypru_3'}
\end{align}
By using the terminal conditions for $\p_T^\epsilon$ and $\y_T^\epsilon$, the initial conditions for $\x_{t+\epsilon}^\epsilon$ and $\R_{t+\epsilon}^\epsilon$, and then  applying Fubini's theorem and also the convexity conditions on $\varphi$ and $\Phi$ in \eqref{def:convex_Phi} and \eqref{def:convex_varphi}, we have
\begin{align*}
    &\e \left[\left(\p_T^\epsilon\right)^\top \x_T^\epsilon  - \left(\p_{t+\epsilon}^\epsilon\right)^\top \x_{t+\epsilon}^\epsilon + \left(\R_T^\epsilon\right)^\top \y_T^\epsilon -\left(\R_{t+\epsilon}^\epsilon\right)^\top \y_{t+\epsilon}^\epsilon \right] \\
    =\ & \e\int_0^1\bigg\{ \left(\x_T^\epsilon\right)^\top \Phi_{xx}\left(X^{\epsilon,\delta}_T, \lr\left(X^{\epsilon,\delta}_T\right)\right) \x^\epsilon_T + \widetilde{\e}\left[\left(\x_T^\epsilon\right)^\top D_{x'}\frac{d\Phi_x}{d\nu}\left(X^{\epsilon,\delta}_T, \lr\left(X^{\epsilon,\delta}_T\right)\right)  \left(\widetilde{X^{\epsilon,\delta}_T}\right) \widetilde{\x^\epsilon_T}\right] \\
    &\quad\qquad + \widetilde{\e}\bigg[\left(\x_T^\epsilon\right)^\top \left(D_{x'}\frac{d\Phi_x}{d\nu}\left(\widetilde{X^{\epsilon,\delta}_T}, \lr\left(X^{\epsilon,\delta}_T\right)\right) \left(X^{\epsilon,\delta}_T\right)\right)^\top \widetilde{\x^\epsilon_T} \bigg]  \\
    &\quad\qquad + \he\widetilde{\e}\bigg[\left(\x_T^\epsilon\right)^\top D_{x''}D_{x'}\frac{d^2\Phi}{d\nu^2}\left(\widetilde{X^{\epsilon,\delta}_T}, \lr\left(X^{\epsilon,\delta}_T\right)\right) \left(X^{\epsilon,\delta}_T,\widehat{X^{\epsilon,\delta}_T}\right)  \widehat{\x^\epsilon_T} \bigg]  \\
    &\quad\qquad + \widetilde{\e}\bigg[\left(\x_T^\epsilon\right)^\top D_{x'}^2\frac{d\Phi}{d\nu}\left(\widetilde{X^{\epsilon,\delta}_T}, \lr\left(X^{\epsilon,\delta}_T\right)\right) \left(X^{\epsilon,\delta}_T\right) \x^\epsilon_T \bigg] +\left(\y_{t+\epsilon}^\epsilon\right)^\top   \varphi_{yy} \left(Y^{\epsilon,\delta}_{t+\epsilon}\right) \y^\epsilon_{t+\epsilon} \bigg\}d\delta  \\
    & + \e\bigg\{ \left(\p_{t+\epsilon}^\epsilon\right)^\top  \int_t^{t+\epsilon} \frac{1}{\epsilon}\Big[ b_0(r) + b_1(r) X_r + b_2(r)\e\left[X_r\right]+b_3(r)u_r\Big] dr \\
    &\quad\qquad  - \left(\y_{t+\epsilon}^\epsilon\right)^\top \int_t^{t+\epsilon} \frac{1}{\epsilon}\varphi_{yy}\left(Y_r\right)f\left(r,\theta_r,\lr({ \theta_r}),u_r\right)dr \\
    &\quad\qquad +\left(\y_{t+\epsilon}^\epsilon\right)^\top\int_t^{t+\epsilon} \frac{1}{\epsilon}\bigg[f_2(r)^\top  R_r +f_5(r)^\top \e[R_r] -h_y \left(r,\theta_r,\lr({ \theta_r}),u_r\right)\\
    &\qquad\qquad\qquad\qquad\qquad\qquad { -\widetilde{\e}\left[D_{y'} \frac{dh}{d\nu}\left(r,\widetilde{\theta_r},\lr(\theta_r),\widetilde{u_r} \right)\left(\theta_r\right) \right]}\bigg] dr\bigg\}\\
    =\ & \e\int_0^1\bigg\{ \left(\x_T^\epsilon\right)^\top \Phi_{xx}\left(X^{\epsilon,\delta}_T, \lr\left(X^{\epsilon,\delta}_T\right)\right) \x^\epsilon_T + 2\widetilde{\e}\left[\left(\x_T^\epsilon\right)^\top D_{x'}\frac{d\Phi_x}{d\nu}\left(X^{\epsilon,\delta}_T, \lr\left(X^{\epsilon,\delta}_T\right)\right)  \left(\widetilde{X^{\epsilon,\delta}_T}\right) \widetilde{\x^\epsilon_T}\right] \\
    &\quad\qquad + \he\widetilde{\e}\bigg[\left(\widetilde{\x_T^\epsilon}\right)^\top D_{x''}D_{x'}\frac{d^2\Phi}{d\nu^2}\left(X^{\epsilon,\delta}_T, \lr\left(X^{\epsilon,\delta}_T\right)\right) \left(\widetilde{X^{\epsilon,\delta}_T},\widehat{X^{\epsilon,\delta}_T}\right)  \widehat{\x^\epsilon_T} \bigg]  \\
    &\quad\qquad + \widetilde{\e}\bigg[\left(\widetilde{\x_T^\epsilon}\right)^\top D_{x'}^2\frac{d\Phi}{d\nu}\left(X^{\epsilon,\delta}_T, \lr\left(X^{\epsilon,\delta}_T\right)\right) \left(\widetilde{X^{\epsilon,\delta}_T}\right) \widetilde{\x^\epsilon_T} \bigg] +\left(\y_{t+\epsilon}^\epsilon\right)^\top   \varphi_{yy} \left(Y^{\epsilon,\delta}_{t+\epsilon}\right) \y^\epsilon_{t+\epsilon} \bigg\}d\delta  \\
    & + \e\bigg\{ \left(\p_{t+\epsilon}^\epsilon\right)^\top  \int_t^{t+\epsilon} \frac{1}{\epsilon}\Big[ b_0(r) + b_1(r) X_r + b_2(r)\e\left[X_r\right]+b_3(r)u_r\Big] dr \\
    &\quad\qquad  - \left(\y_{t+\epsilon}^\epsilon\right)^\top \int_t^{t+\epsilon} \frac{1}{\epsilon}\varphi_{yy}\left(Y_r\right)f\left(r,\theta_r,\lr({ \theta_r}),u_r\right)dr \\
    &\quad\qquad +\left(\y_{t+\epsilon}^\epsilon\right)^\top\int_t^{t+\epsilon} \frac{1}{\epsilon}\bigg[f_2(r)^\top  R_r +f_5(r)^\top \e[R_r] -h_y \left(r,\theta_r,\lr({ \theta_r}),u_r\right)\\
    &\qquad\qquad\qquad\qquad\qquad\qquad { -\widetilde{\e}\left[D_{y'} \frac{dh}{d\nu}\left(r,\widetilde{\theta_r},\lr(\theta_r),\widetilde{u_r} \right)\left(\theta_r\right) \right]}\bigg] dr\bigg\}\\
    \geq\ & \e\bigg\{ \left(\p_{t+\epsilon}^\epsilon\right)^\top  \int_t^{t+\epsilon} \frac{1}{\epsilon}\Big[ b_0(r) + b_1(r) X_r + b_2(r)\e\left[X_r\right]+b_3(r)u_r\Big] dr \\
    &\quad  - \left(\y_{t+\epsilon}^\epsilon\right)^\top \int_t^{t+\epsilon} \frac{1}{\epsilon}\varphi_{yy}\left(Y_r\right)f\left(r,\theta_r,\lr({ \theta_r}),u_r\right)dr \\
    &\quad +\left(\y_{t+\epsilon}^\epsilon\right)^\top\int_t^{t+\epsilon} \frac{1}{\epsilon}\bigg[f_2(r)^\top  R_r +f_5(r)^\top \e[R_r] -h_y \left(r,\theta_r,\lr({ \theta_r}),u_r\right)\\
    &\qquad\qquad\qquad\qquad\qquad { -\widetilde{\e}\left[D_{y'} \frac{dh}{d\nu}\left(r,\widetilde{\theta_r},\lr(\theta_r),\widetilde{u_r} \right)\left(\theta_r\right) \right]}\bigg] dr\bigg\}.
\end{align*}
Substituting the last inequality back into \eqref{lem:xypru_3'}, by using Estimate \eqref{1order-FBSDE-0}, we know that for any $\epsilon\in(0,1)$, 
\begin{align}
    &2\lambda \int_{t+\epsilon}^T \int_0^1 \left\|\U_s^\epsilon\right\|_2^2 d\delta ds \notag \\
    \le\ & \e\bigg\{ -  \left(\p_{t+\epsilon}^\epsilon\right)^\top  \int_t^{t+\epsilon} \frac{1}{\epsilon}\Big[ b_0(r) + b_1(r) X_r + b_2(r)\e\left[X_r\right]+b_3(r)u_r\Big] dr \notag \\
    &\quad  + \left(\y_{t+\epsilon}^\epsilon\right)^\top \int_t^{t+\epsilon} \frac{1}{\epsilon}\varphi_{yy}\left(Y_r\right)f\left(r,\theta_r,\lr({ \theta_r}),u_r\right)dr \notag\\
    &\quad +\left(\y_{t+\epsilon}^\epsilon\right)^\top\int_t^{t+\epsilon} \frac{1}{\epsilon}\bigg[f_2(r)^\top  R_r +f_5(r)^\top \e[R_r] -h_y \left(r,\theta_r,\lr({ \theta_r}),u_r\right) \notag \\
    &\quad\qquad\qquad\qquad\qquad { -\widetilde{\e}\left[D_{y'} \frac{dh}{d\nu}\left(r,\widetilde{\theta_r},\lr(\theta_r),\widetilde{u_r} \right)\left(\theta_r\right) \right]}\bigg] dr\bigg\} \notag\\
    \le\ &  L\left\|\p_{t+\epsilon}^\epsilon\right\|_2 \cdot   \int_t^{t+\epsilon} \frac{1}{\epsilon}\Big( 1 + 2 \|X_r\|_2 + \|u_r\|_2\Big) dr \notag \\
    &\quad  + C(L)\left\|\y_{t+\epsilon}^\epsilon\right\|_2\cdot \int_t^{t+\epsilon} \frac{1}{\epsilon} \Big(1 +\left\|X_r\right\|_2 + \|Y_r\|_2 + \|u_r\|_2  + \|R_r\|_2 \Big)dr \notag\\
    \le\ &  C(L,T,\lambda)\left(1+\|\xi\|_2\right) \left(\left\|\p_{t+\epsilon}^\epsilon\right\|_2  + \left\|\y_{t+\epsilon}^\epsilon\right\|_2\right) \notag \\
    \le\ & \epsilon\left( \left\|\p_{t+\epsilon}^\epsilon\right\|^2_2+ \left\|\y_{t+\epsilon}^\epsilon\right\|^2_2\right)+\frac{C(L,T,\lambda)}{\epsilon} \left(   1+\|\xi\|^2_2\right). \label{lem:xypru_4'}
\end{align}
By using standard estimate for the equation of $\x^\epsilon$ and Estimate \eqref{1order-FBSDE-0}, we have 
\begin{align}
    \e\bigg[\sup_{t+\epsilon\le s\le T} \left|\x_s^\epsilon\right|^2\bigg]\le C(L,T,\lambda) \left(1+\|\xi\|_2^2+\int_{t+\epsilon}^T \left\|\U_s^\epsilon\right\|_2^2 ds \right); \label{lem:xypru_5'}
\end{align}
then, by using standard estimate for the equation of $\y^\epsilon$ and Estimate \eqref{lem:xypru_5'}, we have
\begin{align}
    \e\bigg[\sup_{t+\epsilon\le s\le T} \left|\y_s^\epsilon\right|^2\bigg]\le C(L,T,\lambda) \left(1+\|\xi\|_2^2+\int_{t+\epsilon}^T \left\|\U_s^\epsilon\right\|_2^2 ds \right). \label{lem:xypru_6'}
\end{align}
Similarly, from standard estimate for the equation of $\R^\epsilon$, by using Estimates \eqref{lem:xypru_5'} and \eqref{lem:xypru_6'} and also Estimate \eqref{1order-FBSDE-0}, we have 
\begin{align}
    \e\bigg[\sup_{t+\epsilon\le s\le T} \left|\R_s^\epsilon\right|^2\bigg]\le C(L,T,\lambda) \left(1+\|\xi\|_2^2+\int_{t+\epsilon}^T \left\|\U_s^\epsilon\right\|_2^2 ds \right); \label{lem:xypru_7'}
\end{align}
and finally, by using standard estimate for the equation of $\p^\epsilon$ and Estimates \eqref{lem:xypru_5'}-\eqref{lem:xypru_7'}, we know that
\begin{align}
    \e\bigg[\sup_{t+\epsilon\le s\le T} \left|\p_s^\epsilon\right|^2\bigg]\le C(L,T,\lambda) \left(1+\|\xi\|_2^2+\int_{t+\epsilon}^T \left\|\U_s^\epsilon\right\|_2^2 ds \right). \label{lem:xypru_8'}
\end{align}
Substituting \eqref{lem:xypru_5'} and \eqref{lem:xypru_8'} back into \eqref{lem:xypru_4'}, we have 
\begin{align}
    &2\lambda \int_{t+\epsilon}^T \left\|\U_s^\epsilon\right\|_2^2 ds \le \epsilon C(L,T,\lambda) \int_t^T \left\|\U_s^\epsilon\right\|_2^2 ds +C(L,T,\lambda)\left(1+\frac{1}{\epsilon} \right) \left(   1+\|\xi\|^2_2\right), \notag
\end{align}
and then, by choosing $\epsilon:= \frac{\lambda}{C(L,T,\lambda)}$, we have
\begin{align}
    \int_{t+\epsilon}^T \left\|\U_s^\epsilon\right\|_2^2 ds \le\ & C(L,T,\lambda)\left(   1+\|\xi\|^2_2\right). \notag
\end{align}
Substituting the last estimate back into \eqref{lem:xypru_5}-\eqref{lem:xypru_8}, we obtain 
\begin{align}
    \e\bigg[\sup_{t+\epsilon\le s\le T}\left|\x_s^\epsilon\right|^2+\sup_{t+\epsilon\le s\le T}\left|\y_s^\epsilon\right|^2+\sup_{t+\epsilon\le s\le T}\left|\p_s^\epsilon\right|^2+\sup_{t+\epsilon\le s\le T}\left|\R_s^\epsilon\right|^2\bigg]\le C \left( 1+\left\|\xi\right\|^2_2\right). \label{lem:xypru_9'}
\end{align}
By using the equality \eqref{lem:xypru_11'} and the convexity of $h$ in $v$, we have
\begin{align*}
    0=\ & \left(\U_s^\epsilon\right)^\top h_{vv} \left(s,\theta^{\epsilon,\delta}_s,\lr\left({ \theta^{\epsilon,\delta}_s}\right),u^{\epsilon,\delta}_s\right) \U^\epsilon_s + \left(\U_s^\epsilon\right)^\top b_3(s)^\top  \p^\epsilon_s -\left(\U_s^\epsilon\right)^\top f_7(s)^\top  \R^\epsilon_s \\
    &+ \left(\U_s^\epsilon\right)^\top\left[\left(\begin{array}{ccc}
         h_{vx} & h_{vy} 
    \end{array} \right)\left(s,\theta^{\epsilon,\delta}_s,\lr\left({ \theta^{\epsilon,\delta}_s}\right),u^{\epsilon,\delta}_s\right) \left(\begin{array}{c}
         \x^\epsilon_s \\
         \y^\epsilon_s 
    \end{array}\right)\right] \notag \\
    &  +\widetilde{\e}\left\{ \left(\U_s^\epsilon\right)^\top\left[\left(\begin{array}{cc}
         D_{x'}\frac{dh_{v}}{d\nu} & D_{y'}\frac{dh_{v}}{d\nu} 
    \end{array} \right) \left(r,\theta^{\epsilon,\delta}_r,\lr\left(\theta^{\epsilon,\delta}_r\right),u^{\epsilon,\delta}_r\right)\left(X^{\epsilon,\delta}_r\right) \right] \left( \begin{array}{cc}
         \widetilde{\x^\epsilon_r}  \\
         \widetilde{\y^\epsilon_r} 
    \end{array}\right) \right\}\\
    \geq\ & 2\lambda \left|\U_s^\epsilon\right|^2 - L \left|\U_s^\epsilon\right|\cdot \left(\left|\p_s^\epsilon\right| + \left|\R_s^\epsilon\right| +\left|\x_s^\epsilon\right| + \left|\y_s^\epsilon\right| +\left\|\x_s^\epsilon\right\|_2 { +\left\|\y_s^\epsilon\right\|_2 } \right),
\end{align*}
which implies that 
\begin{align}\label{lem:xypru_10'}
    \left|\U_s^\epsilon\right|  
    \le\ &  \frac{L}{2\lambda} \left(\left|\p_s^\epsilon\right| + \left|\R_s^\epsilon\right| +\left|\x_s^\epsilon\right| + \left|\y_s^\epsilon\right| +\left\|\x_s^\epsilon\right\|_2 { +\left\|\y_s^\epsilon\right\|_2 }\right).
\end{align}
From \eqref{lem:xypru_9'} and \eqref{lem:xypru_10'}, we obtain \eqref{lem:xypru_0'}. 

\subsection{Proof of Theorem~\ref{thm:convergence}}\label{App:thm:convergence}

For notational convenience, we simply drop the superscript ${t,\xi}$ on the processes. The proof of the well-posedness result follows from a similar way as in Subsection~\ref{subsec:FBSDE}, and here, we prove the estimates \eqref{lem:xypru_0} and \eqref{derivative_t} in order. 

{\bf Step 1: Proof of \eqref{lem:xypru_0}.} By applying It\^o's formula for $\p_s^\top \x_s$, taking expectation and using Fubini's theorem, we have
\begin{align}
    &\e \left[\p_T^\top \x_T - \p_t^\top \x_t\right] \notag \\
    =\ & \e\int_t^T \Bigg\{  \x_s^\top f_1(s)^\top  \R_s +\x_s^\top  f_4(s)^\top \e\left[R_s \right] +\p_s^\top  b_3(s)\U_s \notag\\
    &\quad\qquad -\left(\x_s\right)^\top \left[\left(
    \begin{array}{ccc}
        h_{xx} & h_{xy} & h_{xv} \\
    \end{array}\right)\left(s,\theta_s,\lr\left({ \theta_s}\right),u_s\right)\right] \left(
    \begin{array}{c}
         \x_s  \\
         \y_s \\
         \U_s 
    \end{array}
    \right) \notag \\
    & \quad\qquad -\widetilde{\e}\left\{\left(\x_s\right)^\top\left[\left(\begin{array}{cc}
        D_{x'}\frac{dh_{x}}{d\nu} & D_{y'}\frac{dh_{x}}{d\nu}
    \end{array}\right)\left(s,\theta_s,\lr\left({ \theta_s}\right),u_s\right)\left(\widetilde{ \theta_s}\right)\right] \left(\begin{array}{cc}
        \widetilde{\x_s}  \\
        \widetilde{\y_s} 
    \end{array}
    \right) \right\} \notag \\
    &\quad\qquad -\widetilde{\e}\left\{\left(\x_s\right)^\top \left[\left(\begin{array}{ccc}
        D_{x'}\frac{dh_x}{d\nu} \\ 
        D_{x'}\frac{dh_y}{d\nu} \\
        D_{x'}\frac{dh_v}{d\nu}
    \end{array}\right)\left(s,\widetilde{\theta_s},\lr\left({ \theta_s}\right),\widetilde{u_s}\right)\left({ \theta_s}\right)\right]^\top  \left(\begin{array}{c}
        \widetilde{\x_s}  \\
        \widetilde{\y_s}  \\
        \widetilde{\U_s} 
    \end{array}
    \right) \right\} \notag \\
    & \quad\qquad  - \he\widetilde{\e} \left\{ \left(\x_s\right)^\top \left[\left( \begin{array}{cc}
         D_{x''}D_{x'}\frac{d^2h}{d\nu^2} & D_{y''}D_{x'}\frac{d^2h}{d\nu^2} \\
    \end{array}\right) \left(s,\widetilde{\theta_s},\lr\left({ \theta_s}\right),\widetilde{u_s}\right)\left({ \theta_s},\widehat{{ \theta_s}}\right) \right] \left(\begin{array}{c}
         \widehat{\x_s}   \\
         \widehat{\y_s} 
    \end{array}\right) \right\} \notag \\
    & \quad\qquad  - \widetilde{\e} \bigg\{ \left(\x_s\right)^\top \left[\left( \begin{array}{cc}
         D^2_{x'}\frac{dh}{d\nu} & D_{y'}D_{x'}\frac{dh}{d\nu} \\
    \end{array}\right) \left(s,\widetilde{\theta_s},\lr\left({ \theta_s}\right),\widetilde{u_s}\right)\left({ \theta_s}\right) \right] \left(\begin{array}{c}
         \x_s   \\
         \y_s 
    \end{array}\right) \bigg\} \Bigg\}  ds \notag \\
    =\ & \e\int_t^T \Bigg\{ \R_s^\top f_1(s)  \x_s +\R_s^\top  f_4(s) \e\left[\x_s \right] +\p_s^\top  b_3(s)\U_s  \notag\\
    &\quad\qquad -\left(\x_s\right)^\top \left[\left(
    \begin{array}{ccc}
        h_{xx} & h_{xy} & h_{xv} \\
    \end{array}\right)\left(s,\theta_s,\lr\left({ \theta_s}\right),u_s\right)\right] \left(
    \begin{array}{c}
         \x_s  \\
         \y_s \\
         \U_s 
    \end{array}
    \right) \notag \\
    & \quad\qquad -\widetilde{\e}\left\{\left(\x_s\right)^\top\left[\left(\begin{array}{cc}
        D_{x'}\frac{dh_{x}}{d\nu} & D_{y'}\frac{dh_{x}}{d\nu}
    \end{array}\right)\left(s,\theta_s,\lr\left({ \theta_s}\right),u_s\right)\left(\widetilde{ \theta_s}\right)\right] \left(\begin{array}{cc}
        \widetilde{\x_s}  \\
        \widetilde{\y_s} 
    \end{array}
    \right) \right\} \notag \\
    &\quad\qquad -\widetilde{\e}\left\{\left(\begin{array}{c}
        \x_s  \\
        \y_s  \\
        \U_s 
    \end{array}
    \right)^\top  \left[\left(\begin{array}{ccc}
        D_{x'}\frac{dh_x}{d\nu} \\ 
        D_{x'}\frac{dh_y}{d\nu} \\
        D_{x'}\frac{dh_v}{d\nu}
    \end{array}\right)\left(s,\theta_s,\lr\left({ \theta_s}\right),u_s\right)\left(\widetilde{ \theta_s}\right)\right] \widetilde{\x_s}  \right\} \notag \\
    & \quad\qquad  - \he\widetilde{\e} \left\{ \left(\widetilde{\x_s}\right)^\top \left[\left( \begin{array}{cc}
         D_{x''}D_{x'}\frac{d^2h}{d\nu^2} & D_{y''}D_{x'}\frac{d^2h}{d\nu^2} \\
    \end{array}\right) \left(s,\theta_s,\lr\left({ \theta_s}\right),u_s\right)\left(\widetilde{ \theta_s},\widehat{{ \theta_s}}\right) \right] \left(\begin{array}{c}
         \widehat{\x_s}   \\
         \widehat{\y_s} 
    \end{array}\right) \right\} \notag \\
    & \quad\qquad  - \widetilde{\e} \left\{ \left(\widetilde{\x_s}\right)^\top \left[\left( \begin{array}{cc}
         D^2_{x'}\frac{dh}{d\nu} & D_{y'}D_{x'}\frac{dh}{d\nu} \\
    \end{array}\right) \left(s,\theta_s,\lr\left({ \theta_s}\right),u_s \right)\left(\widetilde{ \theta_s}\right) \right] \left(\begin{array}{c}
         \widetilde{\x_s}   \\
         \widetilde{\y_s} 
    \end{array}\right) \right\} \Bigg\} ds. \label{lem:xypru_1}
\end{align}
By using It\^o's formula and Fubini's theorem on the term $\R_s^\top \y_s$, we deduce that  
\begin{align}
    &\e\left[\R_T^\top \y_T -\R_t^\top \y_t \right] \notag \\
    =\ & \e\int_t^T \Bigg\{ -\R_s^\top f_1(r) \x_s -\R_s^\top f_4(s) \e\left[ \x_s\right] -\R_s^\top f_7(s) \U_s \notag\\
    &\quad\qquad - \left(\y_s\right)^\top \left[\left( \begin{array}{ccc}
         h_{yx} & h_{yy} & h_{yv}
    \end{array}\right)\left(s,\theta_s,\lr\left({ \theta_s}\right),u_s\right)\right] \left(\begin{array}{c}
         \x_s \\
         \y_s \\
         \U_s
    \end{array}\right) \notag \\
    & \quad\qquad -\widetilde{\e}\left\{\left(\y_s\right)^\top \left[\left(\begin{array}{cc}
         D_{x'}\frac{dh_{y}}{d\nu} & D_{y'}\frac{dh_{y}}{d\nu} 
    \end{array} \right) \left(s,\theta_s,\lr\left(\theta_s\right),u_s\right)\left(\widetilde{\theta_s}\right) \right] \left( \begin{array}{cc}
         \widetilde{\x_s}  \\
         \widetilde{\y_s} 
    \end{array}\right) \right\}  \notag \\
    & \quad\qquad - \widetilde{\e}\left\{\left(\begin{array}{c}
        \x_s  \\
        \y_s  \\
        \U_s 
    \end{array}
    \right)^\top \left[\left(\begin{array}{ccc}
        D_{y'}\frac{dh_x}{d\nu} \\ 
        D_{y'}\frac{dh_y}{d\nu} \\
        D_{y'}\frac{dh_v}{d\nu}
    \end{array}\right)\left(s,\theta_s,\lr\left({ \theta_s}\right),u_s\right)\left(\widetilde{ \theta_s}\right)\right] \left(\widetilde{\y_s}\right) \right\} \notag \\
    & \quad\qquad  - \he\widetilde{\e} \left\{\left(\widetilde{\y_s}\right)^\top \left[\left( \begin{array}{cc}
         D_{x''}D_{y'}\frac{d^2h}{d\nu^2} & D_{y''}D_{y'}\frac{d^2h}{d\nu^2} \\
    \end{array}\right) \left(s,\theta_s,\lr\left({ \theta_s}\right),u_s\right)\left(\widetilde{ \theta_s},\widehat{{ \theta_s}}\right) \right] \left(\begin{array}{c}
         \widehat{\x_s}   \\
         \widehat{\y_s} 
    \end{array}\right) \right\} \notag \\
    & \quad\qquad  - \widetilde{\e} \left\{ \left(\widetilde{\y_s}\right)^\top \left[\left( \begin{array}{cc}
         D_{x'}D_{y'}\frac{dh}{d\nu} & D_{y'}^2\frac{dh}{d\nu} \\
    \end{array}\right) \left(s,\theta_s,\lr\left({ \theta_s}\right),u_s\right)\left(\widetilde{ \theta_s}\right) \right] \left(\begin{array}{c}
         \widetilde{\x_s}   \\
         \widetilde{\y_s} 
    \end{array}\right) \right\} \Bigg\} ds. \label{lem:xypru_2}
\end{align}
Plugging \eqref{lem:xypru_2} with \eqref{lem:xypru_1}, by using the last equality in \eqref{FBSDEs:high} and the convexity condition \eqref{def:convex_h}, we know that
\begin{align}
    &\e \left[\p_T^\top \x_T  - \p_t^\top \x_t + \R_T^\top \y_T -\R_t^\top \y_t \right] \notag \\
    =\ & \e\int_t^T \Bigg\{  \U_s^\top \left(b_3(s)^\top \p_s -f_7(s)^\top \R_s \right) \notag \\
    &\quad -\left(\begin{array}{cc}
         \x_s  \\
         \y_s 
    \end{array}\right)^\top \left[\left(
    \begin{array}{ccc}
        h_{xx} & h_{xy} & h_{xv} \\
        h_{yx} & h_{yy} & h_{yv}
    \end{array}\right)\left(s,\theta_s,\lr\left({ \theta_s}\right),u_s\right)\right] \left(
    \begin{array}{c}
         \x_s  \\
         \y_s \\
         \U_s 
    \end{array}
    \right) \notag \\
    & \quad -\widetilde{\e}\left\{\left(\begin{array}{c}
         \x_s  \\
         \y_s 
    \end{array}\right)^\top\left[\left(\begin{array}{cc}
        D_{x'}\frac{dh_{x}}{d\nu} & D_{y'}\frac{dh_{x}}{d\nu} \\
        D_{x'}\frac{dh_{y}}{d\nu} & D_{y'}\frac{dh_{y}}{d\nu}
    \end{array}\right)\left(s,\theta_s,\lr\left({ \theta_s}\right),u_s\right)\left(\widetilde{ \theta_s}\right)\right] \left(\begin{array}{cc}
        \widetilde{\x_s}  \\
        \widetilde{\y_s} 
    \end{array}
    \right) \right\} \notag \\
    &\quad -\widetilde{\e}\left\{\left(\begin{array}{c}
        \x_s  \\
        \y_s  \\
        \U_s 
    \end{array}
    \right)^\top  \left[\left(\begin{array}{cc}
        D_{x'}\frac{dh_x}{d\nu} & D_{y'}\frac{dh_x}{d\nu} \\ 
        D_{x'}\frac{dh_y}{d\nu} & D_{y'}\frac{dh_y}{d\nu} \\
        D_{x'}\frac{dh_v}{d\nu} & D_{y'}\frac{dh_v}{d\nu}
    \end{array}\right)\left(s,\theta_s,\lr\left({ \theta_s}\right),u_s\right)\left(\widetilde{ \theta_s}\right)\right] \left(\begin{array}{c}
         \widetilde{\x_s}   \\
         \widetilde{\y_s}  
    \end{array}\right) \right\} \notag \\
    & \quad  - \he\widetilde{\e} \left\{ \left(\begin{array}{c}
         \widetilde{\x_s}  \\
         \widetilde{\y_s} 
    \end{array}\right)^\top \left[\left( \begin{array}{cc}
         D_{x''}D_{x'}\frac{d^2h}{d\nu^2} & D_{y''}D_{x'}\frac{d^2h}{d\nu^2} \\
         D_{x''}D_{y'}\frac{d^2h}{d\nu^2} & D_{y''}D_{y'}\frac{d^2h}{d\nu^2} 
    \end{array}\right) \left(s,\theta_s,\lr\left({ \theta_s}\right),u_s\right)\left(\widetilde{ \theta_s},\widehat{{ \theta_s}}\right) \right] \left(\begin{array}{c}
         \widehat{\x_s}   \\
         \widehat{\y_s} 
    \end{array}\right) \right\} \notag \\
    & \quad  - \widetilde{\e} \bigg\{ \left[\left( \begin{array}{cc}
         D^2_{x'}\frac{dh}{d\nu} & D_{y'}D_{x'}\frac{dh}{d\nu} \\
         D_{x'}D_{y'}\frac{dh}{d\nu} & D_{y'}^2\frac{dh}{d\nu} 
    \end{array}\right) \left(s,\theta_s,\lr\left({ \theta_s}\right),u_s \right)\left(\widetilde{ \theta_s}\right) \right] \left(\begin{array}{c}
         \widetilde{\x_s}   \\
         \widetilde{\y_s} 
    \end{array}\right)^{\otimes 2} \bigg\}  \Bigg\} ds \notag \\
    =\ & \e\int_t^T \Bigg\{  -\left[\left(
    \begin{array}{ccc}
        h_{xx} & h_{xy} & h_{xv} \\
        h_{yx} & h_{yy} & h_{yv} \\
        h_{vx} & h_{vy} & h_{vv}
    \end{array}\right)\left(s,\theta_s,\lr\left({ \theta_s}\right),u_s\right)\right] \left(
    \begin{array}{c}
         \x_s  \\
         \y_s \\
         \U_s 
    \end{array}
    \right)^{\otimes 2} \notag \\
    &\quad -2 \widetilde{\e}\left\{\left(\begin{array}{c}
        \x_s  \\
        \y_s  \\
        \U_s 
    \end{array}
    \right)^\top  \left[\left(\begin{array}{cc}
        D_{x'}\frac{dh_x}{d\nu} & D_{y'}\frac{dh_x}{d\nu} \\ 
        D_{x'}\frac{dh_y}{d\nu} & D_{y'}\frac{dh_y}{d\nu} \\
        D_{x'}\frac{dh_v}{d\nu} & D_{y'}\frac{dh_v}{d\nu}
    \end{array}\right)\left(s,\theta_s,\lr\left({ \theta_s}\right),u_s\right)\left(\widetilde{ \theta_s}\right)\right] \left(\begin{array}{c}
         \widetilde{\x_s}   \\
         \widetilde{\y_s}  
    \end{array}\right) \right\} \notag \\
    & \quad  - \he\widetilde{\e} \left\{ \left(\begin{array}{c}
         \widetilde{\x_s}  \\
         \widetilde{\y_s} 
    \end{array}\right)^\top \left[\left( \begin{array}{cc}
         D_{x''}D_{x'}\frac{d^2h}{d\nu^2} & D_{y''}D_{x'}\frac{d^2h}{d\nu^2} \\
         D_{x''}D_{y'}\frac{d^2h}{d\nu^2} & D_{y''}D_{y'}\frac{d^2h}{d\nu^2} 
    \end{array}\right) \left(s,\theta_s,\lr\left({ \theta_s}\right),u_s\right)\left(\widetilde{ \theta_s},\widehat{{ \theta_s}}\right) \right] \left(\begin{array}{c}
         \widehat{\x_s}   \\
         \widehat{\y_s} 
    \end{array}\right) \right\} \notag \\
    & \quad  - \widetilde{\e} \bigg\{ \left[\left( \begin{array}{cc}
         D^2_{x'}\frac{dh}{d\nu} & D_{y'}D_{x'}\frac{dh}{d\nu} \\
         D_{x'}D_{y'}\frac{dh}{d\nu} & D_{y'}^2\frac{dh}{d\nu} 
    \end{array}\right) \left(s,\theta_s,\lr\left({ \theta_s}\right),u_s \right)\left(\widetilde{ \theta_s}\right) \right] \left(\begin{array}{c}
         \widetilde{\x_s}   \\
         \widetilde{\y_s} 
    \end{array}\right)^{\otimes 2} \bigg\}  \Bigg\}  ds \notag \\
    \le\ & -2\lambda \int_t^T \left\|\U_s\right\|_2^2 ds. \label{lem:xypru_3}
\end{align}
By using the terminal conditions for $\p_T$ and $\y_T$, the initial conditions for $\x_t$ and $\R_t$, and then  applying Fubini's theorem and also the convexity conditions on $\varphi$ and $\Phi$ in \eqref{def:convex_Phi} and \eqref{def:convex_varphi}, we have
\begin{align*}
    &\e \left[\p_T^\top \x_T - \p_t^\top \x_t + \R_T^\top \y_T -\R_t^\top \y_t \right] \notag \\
    =\ & \e\bigg\{ \x_T^\top  \Phi_{xx}\left(X_T, \lr\left(X_T\right) \right) \x_T +\widetilde{\e}\bigg[\x_T^\top  D_{x'}\frac{d\Phi_x}{d\nu}\left(X_T, \lr\left(X_T\right) \right)\left(\widetilde{X_T}\right) \widetilde{\x_T}\bigg] \notag \\
    &\quad +\widetilde{\e}\bigg[\x_T^\top\left(D_{x'}\frac{d\Phi_x}{d\nu} \left(\widetilde{X_T}, \lr\left(X_T\right)\right)\left(X_T\right)\right)^\top \widetilde{\x_T}  \bigg] +  \widetilde{\e}\left[\x_T^\top D_{x'}^2\frac{d\Phi}{d\nu} \left(\widetilde{X_T}, \lr\left(X_T\right)\right)\left(X_T\right) \x_T \right] \notag \\
    &\quad +  \he\widetilde{\e}\left[\x_T^\top D_{x''}D_{x'}\frac{d^2\Phi}{d\nu^2} \left(\widetilde{X_T}, \lr\left(X_T\right)\right)\left(X_T,\widehat{X_T}\right) \widehat{\x_T} \right] + \y_t^\top   \varphi_{yy} \left(Y_t\right) \y_t \bigg\}\\
    & + \e\bigg\{ \p_t^\top  \Big[ b_0(t) + b_1(t) \xi + b_2(t)\e\left[\xi\right]+b_3(t)u_t\Big]  \\
    &\quad\qquad  - \y_t^\top \bigg[\varphi_{yy}\left(Y_t\right)f\left(t,\xi, Y_t,\lr({ \xi,Y_t}),u_t\right) -f_2(t)^\top  R_t -f_5(t)^\top \e[R_t]  \\
    &\qquad\qquad\qquad +h_y \left(t,\xi,Y_t,\lr({ \xi, Y_t}),u_t\right) { +\widetilde{\e}\left[D_{y'} \frac{dh}{d\nu}\left(t,\widetilde{\xi},\widetilde{Y_t},\lr(\xi, Y_t),\widetilde{u_t} \right)\left(\xi,Y_t\right) \right]}\bigg]  \bigg\}\\
    =\ & \e\bigg\{ \x_T^\top  \Phi_{xx}\left(X_T, \lr\left(X_T\right) \right) \x_T +2 \widetilde{\e}\bigg[\x_T^\top  D_{x'}\frac{d\Phi_x}{d\nu}\left(X_T, \lr\left(X_T\right) \right)\left(\widetilde{X_T}\right) \widetilde{\x_T}\bigg] \notag \\
    &\quad +  \he\widetilde{\e}\left[\widetilde{\x_T}^\top D_{x''}D_{x'}\frac{d^2\Phi}{d\nu^2} \left(X_T, \lr\left(X_T\right)\right)\left(\widetilde{X_T},\widehat{X_T}\right) \widehat{\x_T} \right] \\
    &\quad +  \widetilde{\e}\left[\widetilde{\x_T}^\top D_{x'}^2\frac{d\Phi}{d\nu} \left(X_T, \lr\left(X_T\right)\right)\left(\widetilde{X_T}\right) \widetilde{\x_T} \right] + \y_t^\top   \varphi_{yy} \left(Y_t\right) \y_t  \bigg\}\\
    & + \e\bigg\{ \p_t^\top  \Big[ b_0(t) + b_1(t) \xi + b_2(t)\e\left[\xi\right]+b_3(t)u_t\Big]  \\
    &\quad\qquad  - \y_t^\top \bigg[\varphi_{yy}\left(Y_t\right)f\left(t,\xi, Y_t,\lr({ \xi,Y_t}),u_t\right) -f_2(t)^\top  R_t -f_5(t)^\top \e[R_t]  \\
    &\qquad\qquad\qquad +h_y \left(t,\xi,Y_t,\lr({ \xi, Y_t}),u_t\right) { +\widetilde{\e}\left[D_{y'} \frac{dh}{d\nu}\left(t,\widetilde{\xi},\widetilde{Y_t},\lr(\xi, Y_t),\widetilde{u_t} \right)\left(\xi,Y_t\right) \right]}\bigg]  \bigg\}\\
    \geq\ &  \e\bigg\{ \p_t^\top  \Big[ b_0(t) + b_1(t) \xi + b_2(t)\e\left[\xi\right]+b_3(t)u_t\Big]  \\
    &\quad  - \y_t^\top \bigg[\varphi_{yy}\left(Y_t\right)f\left(t,\xi, Y_t,\lr({ \xi,Y_t}),u_t\right) -f_2(t)^\top  R_t -f_5(t)^\top \e[R_t]  \\
    &\qquad\qquad +h_y \left(t,\xi,Y_t,\lr({ \xi, Y_t}),u_t\right) { +\widetilde{\e}\left[D_{y'} \frac{dh}{d\nu}\left(t,\widetilde{\xi},\widetilde{Y_t},\lr(\xi, Y_t),\widetilde{u_t} \right)\left(\xi,Y_t\right) \right]}\bigg]  \bigg\}.
\end{align*}
Substituting the last inequality back into \eqref{lem:xypru_3}, by using Estimate \eqref{1order-FBSDE-0}, we know that for any $\epsilon\in(0,1)$,  
\begin{align}
    &2\lambda \int_t^T \left\|\U_s\right\|_2^2 ds \notag\\
    \le\ & -\e\bigg\{ \p_t^\top  \Big[ b_0(t) + b_1(t) \xi + b_2(t)\e\left[\xi\right]+b_3(t)u_t\Big] \notag \\
    &\qquad  - \y_t^\top \bigg[\varphi_{yy}\left(Y_t\right)f\left(t,\xi, Y_t,\lr({ \xi,Y_t}),u_t\right) -f_2(t)^\top  R_t -f_5(t)^\top \e[R_t] \notag \\
    &\quad\qquad\qquad +h_y \left(t,\xi,Y_t,\lr({ \xi, Y_t}),u_t\right) { +\widetilde{\e}\left[D_{y'} \frac{dh}{d\nu}\left(t,\widetilde{\xi},\widetilde{Y_t},\lr(\xi, Y_t),\widetilde{u_t} \right)\left(\xi,Y_t\right) \right]}\bigg]  \bigg\} \notag\\
    \le\ & L\left\|\p_t\right\|_2\cdot  \left(1+2\|\xi\|_2 + \left\|u_t\right\|_2 \right) \notag\\
    & + C(L) \left\|\y_t\right\|_2 \cdot \left(   1+\|\xi\|_2+\left\|Y_t\right\|_2+\left\|u_t\right\|_2 + \left\|R_t\right\|_2 \right) \notag\\
    \le\ & \epsilon\left( \left\|\p_t\right\|^2_2+ \left\|\y_t\right\|^2_2\right)+\frac{C(L)}{\epsilon} \left(   1+\|\xi\|^2_2+\left\|Y_t\right\|^2_2 +\left\|u_t\right\|^2_2 + \left\|R_t\right\|^2_2 \right) \notag \\
    \le\ & \epsilon\left( \left\|\p_t\right\|^2_2+ \left\|\y_t\right\|^2_2\right)+\frac{C(L,T,\lambda)}{\epsilon} \left(   1+\|\xi\|^2_2\right). \label{lem:xypru_4}
\end{align}
By using standard estimate for the equation of $\x$ and Estimate \eqref{1order-FBSDE-0}, we have 
\begin{align}
    \e\bigg[\sup_{t\le s\le T} \left|\x_s\right|^2\bigg]\le C(L,T,\lambda) \left(1+\|\xi\|_2^2+\int_t^T \left\|\U_s\right\|_2^2 ds \right); \label{lem:xypru_5}
\end{align}
then, by using standard estimate for the equation of $\y$ and Estimate \eqref{lem:xypru_5}, we have
\begin{align}
    \e\bigg[\sup_{t\le s\le T} \left|\y_s\right|^2\bigg]\le C(L,T,\lambda) \left(1+\|\xi\|_2^2+\int_t^T \left\|\U_s\right\|_2^2 ds \right). \label{lem:xypru_6}
\end{align}
Similarly, from standard estimate for the equation of $\R$, by using Estimates \eqref{lem:xypru_5} and \eqref{lem:xypru_6} and also Estimate \eqref{1order-FBSDE-0}, we have 
\begin{align}
    \e\bigg[\sup_{t\le s\le T} \left|\R_s\right|^2\bigg]\le C(L,T,\lambda) \left(1+\|\xi\|_2^2+\int_t^T \left\|\U_s\right\|_2^2 ds \right); \label{lem:xypru_7}
\end{align}
and finally, by using standard estimate for the equation of $\p$ and Estimates \eqref{lem:xypru_5}-\eqref{lem:xypru_7}, we know that
\begin{align}
    \e\bigg[\sup_{t\le s\le T} \left|\p_s\right|^2\bigg]\le C(L,T,\lambda) \left(1+\|\xi\|_2^2+\int_t^T \left\|\U_s\right\|_2^2 ds \right). \label{lem:xypru_8}
\end{align}
Substituting \eqref{lem:xypru_5} and \eqref{lem:xypru_8} back into \eqref{lem:xypru_4}, we have 
\begin{align}
    &2\lambda \int_t^T \left\|\U_s\right\|_2^2 ds \le \epsilon C(L,T,\lambda) \int_t^T \left\|\U_s\right\|_2^2 ds +C(L,T,\lambda)\left(1+\frac{1}{\epsilon} \right) \left(   1+\|\xi\|^2_2\right), \notag
\end{align}
and then, by choosing $\epsilon:= \frac{\lambda}{C(L,T,\lambda)}$, we have
\begin{align}
    \int_t^T \left\|\U_s\right\|_2^2 ds \le\ & C(L,T,\lambda)\left(   1+\|\xi\|^2_2\right). \notag
\end{align}
Substituting the last estimate back into \eqref{lem:xypru_5}-\eqref{lem:xypru_8}, we obtain 
\begin{align}
    \e\bigg[\sup_{t\le s\le T}\left|\x_s\right|^2+\sup_{t\le s\le T}\left|\y_s\right|^2+\sup_{t\le s\le T}\left|\p_s\right|^2+\sup_{t\le s\le T}\left|\R_s\right|^2\bigg]\le C \left( 1+\left\|\xi\right\|^2_2\right). \label{lem:xypru_9}
\end{align}
By using the last equality in \eqref{FBSDEs:high} and the convexity of $h$ in $v$, we have
\begin{align*}
    0=\ & \U_s^\top h_{vv} \left(s,\theta_s,\lr\left({ \theta_s}\right),u_s\right) \U_s + \U_s^\top b_3(s)^\top  \p_s - \U_s^\top f_7(s)^\top  \R_s \\
    &+ \U_s^\top\left[\left(\begin{array}{ccc}
         h_{vx} & h_{vy} 
    \end{array} \right)\left(s,\theta_s,\lr\left({ \theta_s}\right),u_s\right) \left(\begin{array}{c}
         \x_s \\
         \y_s 
    \end{array}\right)\right] \notag \\
    &  +\widetilde{\e}\left\{ \U_s \left[\left(\begin{array}{cc}
         D_{x'}\frac{dh_{v}}{d\nu} & D_{y'}\frac{dh_{v}}{d\nu} 
    \end{array} \right) \left(r,\theta_r,\lr\left(\theta_r\right),u_r\right)\left(X_r\right) \right] \left( \begin{array}{cc}
         \widetilde{\x_r}  \\
         \widetilde{\y_r} 
    \end{array}\right) \right\}\\
    \geq\ & 2\lambda \left|\U_s\right|^2 - L \left|\U_s\right|\cdot \left(\left|\p_s\right| + \left|\R_s\right| +\left|\x_s\right| + \left|\y_s\right| +\left\|\x_s\right\|_2 { +\left\|\y_s\right\|_2}  \right),
\end{align*}
which implies that 
\begin{align}\label{lem:xypru_10}
    \left|\U_s\right|  
    \le\ &  \frac{L}{2\lambda} \left(\left|\p_s\right| + \left|\R_s\right| +\left|\x_s\right| + \left|\y_s\right| +\left\|\x_s\right\|_2 { +\left\|\y_s\right\|_2} \right).
\end{align}
From \eqref{lem:xypru_9} and \eqref{lem:xypru_10}, we obtain \eqref{lem:xypru_0}.\\

{\bf Step 2: Proof of \eqref{derivative_t}.} From \eqref{def:x_s^epsilon}, \eqref{def:de x^epsilon} and the equations \eqref{FBSDEs:high} and \eqref{lem:xypru_11'}, we can write the following equations for the processes $\left(\de \x^\epsilon,\de \y^\epsilon,\de \p^\epsilon,\de \R^\epsilon,\de \U^\epsilon\right)$ as 
\begin{align*}
    &\de \x^\epsilon_s= \zeta_X^\epsilon + \int_{t+\epsilon}^s \left[b_1(r) \de \x^\epsilon_r + b_2(r)\e\left[\de \x^\epsilon_r\right]+b_3(r)\de \U^\epsilon_r\right] dr, \\
    &\de \y^\epsilon_s= g_1 \de\x^\epsilon_T+ g_2 \e\left[\de\x^\epsilon_T\right]  + \int_s^T  \Big[f_1(r) \de\x^\epsilon_r + f_2(r) \de\y^\epsilon_r + f_4(r)\e\left[\de\x^\epsilon_r\right] \\
    &\qquad\qquad\qquad\qquad\qquad\qquad\qquad\qquad + f_5(r)\e\left[\de\y^\epsilon_r\Big] + f_7(r) \de\U^\epsilon_r   \right] dr,\\
    &\de \p^\epsilon_s= - g_1^\top \de \R^\epsilon_T - g_2^\top \e\left[\de \R^\epsilon_T\right] +\Phi_{xx}\left(X_T, \lr\left(X_T\right) \right) \de\x^\epsilon_T \\
    &\quad\qquad +\widetilde{\e}\bigg[D_{x'}\frac{d\Phi_x}{d\nu}\left(X_T, \lr\left(X_T\right) \right)\left(\widetilde{X_T}\right) \widetilde{\de\x^\epsilon_T}\bigg]\\
    &\quad\qquad +\widetilde{\e}\bigg[\left(D_{x'}\frac{d\Phi_x}{d\nu} \left(\widetilde{X_T}, \lr\left(X_T\right)\right)\left(X_T\right)\right)^\top \widetilde{\de\x^\epsilon_T}  \bigg]\\
    &\quad\qquad +\he\widetilde{\e}\left[D_{x''}D_{x'}\frac{d^2\Phi}{d\nu^2} \left(\widetilde{X_T}, \lr\left(X_T\right)\right)\left(X_T,\widehat{X_T}\right) \widehat{\de\x^\epsilon_T} \right] \notag \\
    &\quad\qquad +\widetilde{\e}\left[D_{x'}^2\frac{d\Phi}{d\nu} \left(\widetilde{X_T}, \lr\left(X_T\right)\right)\left(X_T\right) \de\x^\epsilon_T \right] +\zeta_P^\epsilon  \notag \\
    &\quad\qquad + \int_s^T \Bigg\{  b_1(r)^\top  \de\p^\epsilon_r + b_2(r)^\top \e\left[\de\p^\epsilon_r\right] - f_1(r)^\top  \de\R^\epsilon_r - f_4(r)^\top \e\left[\de\R^\epsilon_r\right] \notag \\
    &\qquad\qquad\qquad +\left[\left(
    \begin{array}{ccc}
        h_{xx} & h_{xy} & h_{xv} 
    \end{array}\right)\left(r,\theta_r,\lr\left({ \theta_r}\right),u_r\right)\right] \left(
    \begin{array}{c}
         \de \x^\epsilon_r  \\
         \de \y^\epsilon_r \\
         \de \U^\epsilon_r 
    \end{array}
    \right) \notag \\
    & \qquad\qquad\qquad +\widetilde{\e}\left\{\left[\left(\begin{array}{cc}
        D_{x'}\frac{dh_{x}}{d\nu} & D_{y'}\frac{dh_{x}}{d\nu}
    \end{array}\right)\left(r,\theta_r,\lr\left({ \theta_r}\right),u_r\right)\left(\widetilde{ \theta_r}\right)\right] \left(\begin{array}{cc}
        \de \widetilde{\x^\epsilon_r}  \\
        \de \widetilde{\y^\epsilon_r} 
    \end{array}
    \right) \right\}  \notag\\
    &\qquad\qquad\qquad +\widetilde{\e}\left\{\left[\left(\begin{array}{ccc}
        D_{x'}\frac{dh_x}{d\nu} \\ 
        D_{x'}\frac{dh_y}{d\nu} \\
        D_{x'}\frac{dh_v}{d\nu}
    \end{array}\right)\left(r,\widetilde{\theta_r},\lr\left({ \theta_r}\right),\widetilde{u_r}\right)\left({ \theta_r}\right)\right]^\top  \left(\begin{array}{c}
        \de \widetilde{\x^\epsilon_r}  \\
        \de \widetilde{\y^\epsilon_r}  \\
        \de \widetilde{\U^\epsilon_r} 
    \end{array}
    \right) \right\}  \notag\\
    & \qquad\qquad\qquad  + \he\widetilde{\e} \left\{ \left[\left( \begin{array}{cc}
         D_{x''}D_{x'}\frac{d^2h}{d\nu^2} & D_{y''}D_{x'}\frac{d^2h}{d\nu^2} \\
    \end{array}\right) \left(r,\widetilde{\theta_r},\lr\left({ \theta_r}\right),\widetilde{u_r}\right)\left({ \theta_r},\widehat{{ \theta_r}}\right) \right] \left(\begin{array}{c}
         \de \widehat{\x^\epsilon_r}   \\
         \de \widehat{\y^\epsilon_r} 
    \end{array}\right) \right\}  \notag\\
    & \qquad\qquad\qquad  + \widetilde{\e} \bigg\{ \left[\left( \begin{array}{cc}
         D^2_{x'}\frac{dh}{d\nu} & D_{y'}D_{x'}\frac{dh}{d\nu} \\
    \end{array}\right) \left(r,\widetilde{\theta_r},\lr\left({ \theta_r}\right),\widetilde{u_r}\right)\left({ \theta_r}\right) \right] \left(\begin{array}{c}
         \de \x^\epsilon_r   \\
         \de \y^\epsilon_r 
    \end{array}\right) \bigg\} + \eta^\epsilon_P(r) \Bigg\} dr,\\
    & \de\R^\epsilon_s= - \varphi_{yy} \left(Y_{t+\epsilon}\right) \de\y^\epsilon_{t+\epsilon} +\zeta_R^\epsilon  \\
    &\quad\qquad +\int_{t+\epsilon}^s \Bigg\{ f_2(r)^\top  \de \R^\epsilon_r { + f_5(r)^\top  \e\left[\de \R^\epsilon_r\right]}  - \left[\left( \begin{array}{ccc}
         h_{yx} & h_{yy} & h_{yv}
    \end{array}\right)\left(r,\theta_r,\lr\left({ \theta_r}\right),u_r\right)\right] \left(\begin{array}{c}
         \de \x^\epsilon_r \\
         \de \y^\epsilon_r \\
         \de \U^\epsilon_r
    \end{array}\right) \notag \\
    & \qquad\qquad\qquad -\widetilde{\e}\left\{ \left[\left(\begin{array}{cc}
         D_{x'}\frac{dh_{y}}{d\nu} & D_{y'}\frac{dh_{y}}{d\nu} 
    \end{array} \right) \left(r,\theta_r,\lr\left(\theta_r\right),u_r\right)\left(\widetilde{\theta_r}\right) \right] \left( \begin{array}{cc}
         \de \widetilde{\x^\epsilon_r}  \\
         \de \widetilde{\y^\epsilon_r} 
    \end{array}\right) \right\}  \notag\\
    & \qquad\qquad\qquad - \widetilde{\e}\left\{\left[\left(\begin{array}{ccc}
        D_{y'}\frac{dh_x}{d\nu} \\ 
        D_{y'}\frac{dh_y}{d\nu} \\
        D_{y'}\frac{dh_v}{d\nu}
    \end{array}\right)\left(r,\widetilde{\theta_r},\lr\left({ \theta_r}\right),\widetilde{u_r}\right)\left({ \theta_r}\right)\right]^\top  \left(\begin{array}{c}
        \de \widetilde{\x^\epsilon_r}  \\
        \de \widetilde{\y^\epsilon_r}  \\
        \de \widetilde{\U^\epsilon_r} 
    \end{array}
    \right) \right\}  \notag\\
    & \qquad\qquad\qquad  - \he\widetilde{\e} \left\{ \left[\left( \begin{array}{cc}
         D_{x''}D_{y'}\frac{d^2h}{d\nu^2} & D_{y''}D_{y'}\frac{d^2h}{d\nu^2} \\
    \end{array}\right) \left(r,\widetilde{\theta_r},\lr\left({ \theta_r}\right),\widetilde{u_r}\right)\left({ \theta_r},\widehat{{ \theta_r}}\right) \right] \left(\begin{array}{c}
         \de \widehat{\x^\epsilon_r}   \\
         \de \widehat{\y^\epsilon_r} 
    \end{array}\right) \right\}  \notag\\
    & \qquad\qquad\qquad  - \widetilde{\e} \bigg\{ \left[\left( \begin{array}{cc}
         D_{x'}D_{y'}\frac{dh}{d\nu} & D_{y'}^2\frac{dh}{d\nu} \\
    \end{array}\right) \left(r,\widetilde{\theta_r},\lr\left({ \theta_r}\right),\widetilde{u_r}\right)\left({ \theta_r}\right) \right] \left(\begin{array}{c}
         \de \x^\epsilon_r   \\
         \de \y^\epsilon_r 
    \end{array}\right) \bigg\} -\eta_R^\epsilon(r) \Bigg\} dr, \notag \\ 
    &0=b_3(s)^\top  \de \p^\epsilon_s -f_7(s)^\top  \de \R^\epsilon_s + \left[\left(\begin{array}{ccc}
         h_{vx} & h_{vy} & h_{vv} 
    \end{array} \right)\left(s,\theta_s,\lr\left({ \theta_s}\right),u_s\right) \left(\begin{array}{c}
         \de \x^\epsilon_s \\
         \de \y^\epsilon_s \\
         \de \U^\epsilon_s 
    \end{array}\right)\right] \notag \\
    & \qquad +\widetilde{\e}\left\{ \left[\left(\begin{array}{cc}
         D_{x'}\frac{dh_{v}}{d\nu} & D_{y'}\frac{dh_{v}}{d\nu} 
    \end{array} \right) \left(s,\theta_s,\lr\left(\theta_s\right),u_s\right)\left(\widetilde{\theta_s}\right) \right] \left( \begin{array}{cc}
         \widetilde{\x^\epsilon_s}  \\
         \widetilde{\y^\epsilon_s} 
    \end{array}\right) \right\} + \eta_u^\epsilon(s),\quad s\in[t+\epsilon,T],
\end{align*}
where
\begin{align*}
    \zeta_X^\epsilon:=\ & -\int_t^{t+\epsilon} \frac{1}{\epsilon}\Big[b\big|^{\left(r,X_r,\lr(X_r),u_r\right)}_{\left(t,\xi,\lr(\xi),u_t\right) }\Big] dr -\int_t^{t+\epsilon} \left\{ b_1(r) \x_r + b_2(r)\e\left[\x_r\right]+b_3(r)\U_r\right\} dr,\\
    \zeta_P^\epsilon:=\ & \int_0^1 \left[\Phi_{xx}\big|^{\left(X^{\epsilon,\delta}_T, \lr\left(X_T\right)\right)}_{\left(X_T, \lr\left(X_T\right) \right)}\right]\x^\epsilon_T d\delta + \int_0^1 \widetilde{\e}\left[\left(D_{x'}\frac{d\Phi_x}{d\nu}\bigg|^{\left(X^{\epsilon,\delta}_T, \lr\left(X^{\epsilon,\delta}_T\right)\right)  \left(\widetilde{X^{\epsilon,\delta}_T}\right)}_{\left(X_T, \lr\left(X_T\right)\right)  \left(\widetilde{X_T}\right)}\right) \widetilde{\x^\epsilon_T}\right] d\delta \\
    & +\int_0^1 \widetilde{\e}\left[\left(D_{x'}\frac{d\Phi_x}{d\nu}\bigg|^{\left(\widetilde{X^{\epsilon,\delta}_T}, \lr\left(X^{\epsilon,\delta}_T\right)\right) \left(X^{\epsilon,\delta}_T\right)}_{\left(\widetilde{X_T}, \lr\left(X_T\right)\right) \left(X_T\right)}\right)^\top \widetilde{\x^\epsilon_T} \right]d\delta  \\
    & +\int_0^1 \he\widetilde{\e}\left[\left(D_{x''}D_{x'}\frac{d^2\Phi}{d\nu^2}\bigg|^{\left(\widetilde{X^{\epsilon,\delta}_T}, \lr\left(X^{\epsilon,\delta}_T\right)\right) \left(X^{\epsilon,\delta}_T,\widehat{X^{\epsilon,\delta}_T}\right)}_{\left(\widetilde{X_T}, \lr\left(X_T\right)\right) \left(X_T,\widehat{X^{\epsilon,\delta}_T}\right)}  \right)\widehat{\x^\epsilon_T} \right]d\delta  \\
    & +\int_0^1 \widetilde{\e}\left[\left(D_{x'}^2\frac{d\Phi}{d\nu}\bigg|^{\left(\widetilde{X^{\epsilon,\delta}_T}, \lr\left(X^{\epsilon,\delta}_T\right)\right) \left(X^{\epsilon,\delta}_T\right)}_{\left(\widetilde{X_T}, \lr\left(X_T\right)\right) \left(X_T\right)} \right) \x^\epsilon_T \right]d\delta,  \\
    \zeta_R^\epsilon:=\ & \int_0^1 \left[\varphi_{yy}\big|^{Y^{\epsilon,\delta}_{t+\epsilon}}_{Y_{t+\epsilon}} \right] \y^\epsilon_{t+\epsilon} d\delta + \varphi_{yy} \left(Y_{t+\epsilon}\right) \y_{t+\epsilon} - \varphi_{yy} \left(Y_t\right) \y_t\\
    & + \frac{1}{\epsilon}\int_t^{t+\epsilon}\left[ \varphi_{yy}\left(Y_r\right)f\left(r,X_r,Y_r,\lr\left({ X_r,Y_r}\right),u_r\right) -\varphi_{yy} \left(Y_t\right) f\left(t,\xi,Y_t,\lr({ \xi,Y_t}),u_t\right) \right]  dr\\
    & -\frac{1}{\epsilon}\int_t^{t+\epsilon} \Bigg\{f_2(r)^\top  R_r -f_2(t)^\top  R_t + f_5(r)^\top  \e[R_r] -f_5(t)^\top \e [R_t]  \\
    &\quad\qquad\qquad -h_y\big|^{\left(r,X_r,Y_r,\lr({ X_r,Y_r}),u_r\right)}_{\left(t,\xi,Y_t,\lr({ \xi,Y_t}),u_t\right)} { - \widetilde{\e} \left[D_{y'}\frac{dh}{d\nu}\bigg|^{\left(r,\widetilde{X_r},\widetilde{Y_r},\lr\left({ X_r,Y_r}\right),\widetilde{u_r}\right)\left({ X_r,Y_r}\right)}_{\left(t,\widetilde{\xi},\widetilde{Y_t},\lr\left( \xi,Y_t\right),\widetilde{u_t}\right)\left({ \xi,Y_t}\right)} \right]} \Bigg\} dr \\
    & -\int_t^{t+\epsilon} \Bigg\{ f_2(r)^\top  \R_r { + f_5(r)^\top  \e\left[\R_r\right]} - \left[\left( \begin{array}{ccc}
         h_{yx} & h_{yy} & h_{yv}
    \end{array}\right)\left(r,\theta_r,\lr\left({ \theta_r}\right),u_r\right)\right] \left(\begin{array}{c}
         \x_r \\
         \y_r \\
         \U_r
    \end{array}\right) \notag \\
    & \qquad\qquad -\widetilde{\e}\left\{ \left[\left(\begin{array}{cc}
         D_{x'}\frac{dh_{y}}{d\nu} & D_{y'}\frac{dh_{y}}{d\nu} 
    \end{array} \right) \left(r,\theta_r,\lr\left(\theta_r\right),u_r\right)\left(\widetilde{\theta_r}\right) \right] \left( \begin{array}{cc}
         \widetilde{\x_r}  \\
         \widetilde{\y_r} 
    \end{array}\right) \right\}  \notag\\
    & \qquad\qquad - \widetilde{\e}\left\{\left[\left(\begin{array}{ccc}
        D_{y'}\frac{dh_x}{d\nu} \\ 
        D_{y'}\frac{dh_y}{d\nu} \\
        D_{y'}\frac{dh_v}{d\nu}
    \end{array}\right)\left(r,\widetilde{\theta_r},\lr\left({ \theta_r}\right),\widetilde{u_r}\right)\left({ \theta_r}\right)\right]^\top  \left(\begin{array}{c}
        \widetilde{\x_r}  \\
        \widetilde{\y_r}  \\
        \widetilde{\U_r} 
    \end{array}
    \right) \right\}  \notag\\
    & \qquad\qquad  - \he\widetilde{\e} \left\{ \left[\left( \begin{array}{cc}
         D_{x''}D_{y'}\frac{d^2h}{d\nu^2} & D_{y''}D_{y'}\frac{d^2h}{d\nu^2} \\
    \end{array}\right) \left(r,\widetilde{\theta_r},\lr\left({ \theta_r}\right),\widetilde{u_r}\right)\left({ \theta_r},\widehat{{ \theta_r}}\right) \right] \left(\begin{array}{c}
         \widehat{\x_r}   \\
         \widehat{\y_r} 
    \end{array}\right) \right\}  \notag\\
    & \qquad\qquad  - \widetilde{\e} \left\{ \left[\left( \begin{array}{cc}
         D_{x'}D_{y'}\frac{dh}{d\nu} & D_{y'}^2\frac{dh}{d\nu} \\
    \end{array}\right) \left(r,\widetilde{\theta_r},\lr\left({ \theta_r}\right),\widetilde{u_r}\right)\left({ \theta_r}\right) \right] \left(\begin{array}{c}
         \x_r   \\
         \y_r 
    \end{array}\right) \right\} \Bigg\} dr,\\
    \eta^\epsilon_P(r):=\ & \int_0^1  \Bigg\{ \left[\left(
    \begin{array}{ccc}
        h_{xx} & h_{xy} & h_{xv} 
    \end{array}\right)(r,\cdot)\Big|^{\left(\theta^{\epsilon,\delta}_r,\lr\left({ \theta^{\epsilon,\delta}_r}\right),u^{\epsilon,\delta}_r\right)}_{\left(\theta_r,\lr\left({ \theta_r}\right),u_r\right)} \right] \left(
    \begin{array}{c}
         \x^\epsilon_r  \\
         \y^\epsilon_r \\
         \U^\epsilon_r 
    \end{array}
    \right) \notag \\
    & \quad +\widetilde{\e}\left\{\left[\left(\begin{array}{cc}
        D_{x'}\frac{dh_{x}}{d\nu} & D_{y'}\frac{dh_{x}}{d\nu}
    \end{array}\right)(r,\cdot)(\cdot)\Big|^{\left(\theta^{\epsilon,\delta}_r,\lr\left({ \theta^{\epsilon,\delta}_r}\right),u^{\epsilon,\delta}_r\right)\left(\widetilde{ \theta^{\epsilon,\delta}_r}\right)}_{\left(\theta_r,\lr\left({ \theta_r}\right),u_r\right)\left(\widetilde{ \theta_r}\right)} \right] \left(\begin{array}{cc}
        \widetilde{\x^\epsilon_r}  \\
        \widetilde{\y^\epsilon_r} 
    \end{array}
    \right) \right\}  \notag\\
    &\quad +\widetilde{\e}\left\{\left[\left(\begin{array}{ccc}
        D_{x'}\frac{dh_x}{d\nu} \\ 
        D_{x'}\frac{dh_y}{d\nu} \\
        D_{x'}\frac{dh_v}{d\nu}
    \end{array}\right)(r,\cdot)(\cdot) \bigg|^{\left(\widetilde{\theta^{\epsilon,\delta}_r},\lr\left({ \theta^{\epsilon,\delta}_r}\right),\widetilde{u^{\epsilon,\delta}_r}\right)\left({ \theta^{\epsilon,\delta}_r}\right)}_{\left(\widetilde{\theta_r},\lr\left({ \theta_r}\right),\widetilde{u_r}\right)\left({ \theta_r}\right)} \right]^\top  \left(\begin{array}{c}
        \widetilde{\x^\epsilon_r}  \\
        \widetilde{\y^\epsilon_r}  \\
        \widetilde{\U^\epsilon_r} 
    \end{array}
    \right) \right\}  \notag\\
    & \quad  + \he\widetilde{\e} \left\{ \left[\left( \begin{array}{cc}
         D_{x''}D_{x'}\frac{d^2h}{d\nu^2} & D_{y''}D_{x'}\frac{d^2h}{d\nu^2} \\
    \end{array}\right)(r,\cdot)(\cdot)\bigg|^{\left(\widetilde{\theta^{\epsilon,\delta}_r},\lr\left({ \theta^{\epsilon,\delta}_r}\right),\widetilde{u^{\epsilon,\delta}_r}\right)\left({ \theta^{\epsilon,\delta}_r},\widehat{{ \theta^{\epsilon,\delta}_r}}\right)}_{\left(\widetilde{\theta_r},\lr\left({ \theta_r}\right),\widetilde{u_r}\right)\left({ \theta_r},\widehat{{ \theta_r}}\right)}  \right] \left(\begin{array}{c}
         \widehat{\x^\epsilon_r}   \\
         \widehat{\y^\epsilon_r} 
    \end{array}\right) \right\}  \notag\\
    & \quad  + \widetilde{\e} \left\{ \left[\left( \begin{array}{cc}
         D^2_{x'}\frac{dh}{d\nu} & D_{y'}D_{x'}\frac{dh}{d\nu} \\
    \end{array}\right)(r,\cdot)(\cdot)\Big|^{\left(\widetilde{\theta^{\epsilon,\delta}_r},\lr\left({ \theta^{\epsilon,\delta}_r}\right),\widetilde{u^{\epsilon,\delta}_r}\right)\left({ \theta^{\epsilon,\delta}_r}\right) }_{\left(\widetilde{\theta_r},\lr\left({ \theta_r}\right),\widetilde{u_r}\right)\left({ \theta_r}\right) } \right] \left(\begin{array}{c}
         \x^\epsilon_r   \\
         \y^\epsilon_r 
    \end{array}\right) \right\} \Bigg\} d\delta, \\
    \eta_R^\epsilon(r):=\ & \int_0^1 \Bigg\{ \left[\left( \begin{array}{ccc}
         h_{yx} & h_{yy} & h_{yv}
    \end{array}\right)(r,\cdot)\Big|^{\left(\theta^{\epsilon,\delta}_r,\lr\left({ \theta^{\epsilon,\delta}_r}\right),u^{\epsilon,\delta}_r\right)}_{\left(\theta_r,\lr\left({ \theta_r}\right),u_r\right)} \right] \left(\begin{array}{c}
         \x^\epsilon_r \\
         \y^\epsilon_r \\
         \U^\epsilon_r
    \end{array}\right) \notag \\
    & \quad + \widetilde{\e}\left\{ \left[\left(\begin{array}{cc}
         D_{x'}\frac{dh_{y}}{d\nu} & D_{y'}\frac{dh_{y}}{d\nu} 
    \end{array} \right)(r,\cdot)(\cdot) \Big|^{\left(\theta^{\epsilon,\delta}_r,\lr\left(\theta^{\epsilon,\delta}_r\right),u^{\epsilon,\delta}_r\right)\left(\widetilde{\theta^{\epsilon,\delta}_r}\right)}_{\left(\theta_r,\lr\left(\theta_r\right),u_r\right)\left(\widetilde{\theta_r}\right)}  \right] \left( \begin{array}{cc}
         \widetilde{\x^\epsilon_r}  \\
         \widetilde{\y^\epsilon_r} 
    \end{array}\right) \right\}  \notag\\
    & \quad + \widetilde{\e}\left\{\left[\left(\begin{array}{ccc}
        D_{y'}\frac{dh_x}{d\nu} \\ 
        D_{y'}\frac{dh_y}{d\nu} \\
        D_{y'}\frac{dh_v}{d\nu}
    \end{array}\right)(r,\cdot)(\cdot)\bigg|^{\left(\widetilde{\theta^{\epsilon,\delta}_r},\lr\left({ \theta^{\epsilon,\delta}_r}\right),\widetilde{u^{\epsilon,\delta}_r}\right)\left({ \theta^{\epsilon,\delta}_r}\right)}_{\left(\widetilde{\theta_r},\lr\left({ \theta_r}\right),\widetilde{u_r}\right)\left({ \theta_r}\right)} \right]^\top  \left(\begin{array}{c}
        \widetilde{\x^\epsilon_r}  \\
        \widetilde{\y^\epsilon_r}  \\
        \widetilde{\U^\epsilon_r} 
    \end{array}
    \right) \right\}  \notag\\
    & \quad  + \he\widetilde{\e} \left\{ \left[\left( \begin{array}{cc}
         D_{x''}D_{y'}\frac{d^2h}{d\nu^2} & D_{y''}D_{y'}\frac{d^2h}{d\nu^2} \\
    \end{array}\right) (r,\cdot)(\cdot) \Big|^{\left(\widetilde{\theta^{\epsilon,\delta}_r},\lr\left({ \theta^{\epsilon,\delta}_r}\right),\widetilde{u^{\epsilon,\delta}_r}\right)\left({ \theta^{\epsilon,\delta}_r},\widehat{{ \theta^{\epsilon,\delta}_r}}\right)}_{\left(\widetilde{\theta_r},\lr\left({ \theta_r}\right),\widetilde{u_r}\right)\left({ \theta_r},\widehat{{ \theta_r}}\right)}   \right] \left(\begin{array}{c}
         \widehat{\x^\epsilon_r}   \\
         \widehat{\y^\epsilon_r} 
    \end{array}\right) \right\}  \notag\\
    & \quad  + \widetilde{\e} \left\{ \left[\left( \begin{array}{cc}
         D_{x'}D_{y'}\frac{dh}{d\nu} & D_{y'}^2\frac{dh}{d\nu} \\
    \end{array}\right) (r,\cdot)(\cdot) \Big|^{\left(\widetilde{\theta^{\epsilon,\delta}_r},\lr\left({ \theta^{\epsilon,\delta}_r}\right),\widetilde{u^{\epsilon,\delta}_r}\right)\left({ \theta^{\epsilon,\delta}_r}\right)}_{\left(\widetilde{\theta_r},\lr\left({ \theta_r}\right),\widetilde{u_r}\right)\left({ \theta_r}\right)}  \right] \left(\begin{array}{c}
         \x^\epsilon_r   \\
         \y^\epsilon_r 
    \end{array}\right) \right\} \Bigg\} d\delta,\\
    \eta_u^\epsilon(s):=\ & \left[\left(\begin{array}{ccc}
         h_{vx} & h_{vy} & h_{vv} 
    \end{array} \right)(s,\cdot)\Big|^{\left(\theta^{\epsilon,\delta}_s,\lr\left({ \theta^{\epsilon,\delta}_s}\right),u^{\epsilon,\delta}_s\right)}_{\left(\theta_s,\lr\left({ \theta_s}\right),u_s\right)}  \left(\begin{array}{c}
         \x^\epsilon_s \\
         \y^\epsilon_s \\
         \U^\epsilon_s 
    \end{array}\right)\right] \notag \\
    & +\widetilde{\e}\left\{ \left[\left(\begin{array}{cc}
         D_{x'}\frac{dh_{v}}{d\nu} & D_{y'}\frac{dh_{v}}{d\nu} 
    \end{array} \right) (s,\cdot)(\cdot)\Big|^{\left(\theta^{\epsilon,\delta}_s,\lr\left(\theta^{\epsilon,\delta}_s\right),u^{\epsilon,\delta}_s\right)\left(\widetilde{\theta^{\epsilon,\delta}_s}\right)}_{\left(\theta_s,\lr\left(\theta_s\right),u_s\right)\left(\widetilde{\theta_s}\right)}  \right] \left( \begin{array}{cc}
         \widetilde{\x^\epsilon_s}  \\
         \widetilde{\y^\epsilon_s} 
    \end{array}\right) \right\}.
\end{align*}
Here, recall that 
\begin{align*}
    &X^{\epsilon,\delta}_s:=X_s + \delta \left(X^{\epsilon}_s-X_s\right),\quad Y^{\epsilon,\delta}_s:=Y_s + \delta \left(Y^{\epsilon}_s-Y_s\right),\quad u^{\epsilon,\delta}_s:=u_s + \delta \left(u^{\epsilon}_s-u_s\right),
\end{align*}
and $\theta^{\epsilon,\delta}_s:=\left( X^{\epsilon,\delta}_s,Y^{\epsilon,\delta}_s\right)$.

\noindent\textbf{(Step 2(a))} We first show that the following estimates hold:
\begin{equation}\label{zeta}
    \lim_{\epsilon\to 0}\ \bigg[\left\|\zeta_X^\epsilon\right\|_2^2+\left\|\zeta_P^\epsilon\right\|_2^2+\left\|\zeta_R^\epsilon\right\|_2^2 + \sup_{t+\epsilon\le s\le T}\left\|\eta_u^\epsilon(s)\right\|_2^2 + \int_{t+\epsilon}^{T} \left(\left\|\eta_P^\epsilon(s)\right\|_2^2 + \left\|\eta_R^\epsilon(s)\right\|_2^2 \right) ds\bigg]=0.  
\end{equation}
Indeed, from Assumption (B), Estimates \eqref{1order-FBSDE-0}, \eqref{lem:estimate_s_1}, \eqref{lem:estimate_s_2} and \eqref{lem:xypru_0}, we see
\begin{align*}
    \e\left[\left|\zeta_X^\epsilon\right|^2\right]\le\ & C(L) \frac{1}{\epsilon} \int_t^{t+\epsilon} \Big\|b\left(r,X_r,\lr(X_r),u_r\right)-b\left(t,\xi,\lr(\xi),u_t\right) \Big\|_2^2 dr \\
    &+C(L)\epsilon \int_t^{t+\epsilon} \Big\| b_1(r) \x_r + b_2(r)\e\left[\x_r\right]+b_3(r)\U_r \Big\|_2^2 dr\\
    \le\ & C(L) \frac{1}{\epsilon} \int_t^{t+\epsilon} \left(|r-t|\left(1+\|X_r\|_2^2+ \|u_r\|_2^2\right)+\|X_r-\xi\|_2^2 +\|u_r-u_t\|_2^2 \right) dr \\
    &+C(L)\epsilon \int_t^{t+\epsilon} \left( \|\x_r\|_2^2 + \|\U_r\|_2^2 \right) dr\\
    \le\ & C(L,T,\lambda) \left( 1+\left\|\xi\right\|^2_2\right) \epsilon,
\end{align*}
from which we know that 
\begin{equation}\label{zeta_X}
    \lim_{\epsilon\to 0}\ \left\|\zeta_X^\epsilon\right\|_2^2 =0.
\end{equation}
For the term $\zeta_R^\epsilon$, by using Assumption (B), Estimates \eqref{1order-FBSDE-0}, \eqref{lem:estimate_s_1}, \eqref{lem:estimate_s_2} and \eqref{lem:xypru_0},
\begin{align}
    \left\|\zeta_R^\epsilon\right\|_2^2 \le\ & C(L) \bigg\{ \left\|\y_{t+\epsilon} - \y_t\right\|_2^2 +\epsilon \int_t^{t+\epsilon} \left(  \|\R_r\|_2^2  +\|\x_r\|_2^2 + \|\y_r\|_2^2 + \|\U_r\|_2^2 \right) dr \notag \\
    &\qquad +\frac{1}{\epsilon}\int_t^{t+\epsilon} \Big[\left(1+\|\xi\|_2^2+\|Y_t\|_2^2+\|u_t\|_2^2+\|R_t\|_2^2\right)|r-t| \notag \\
    &\quad\qquad\qquad\qquad +\|X_r-\xi\|_2^2+\|Y_r-Y_t\|_2^2 +\|u_r-u_t\|_2^2+ \|R_r-R_t\|_2^2 \Big] dr \notag \\
    &\qquad+\e\left[\int_0^1 \left|\varphi_{yy} \left(Y^{\epsilon,\delta}_{t+\epsilon}\right)-\varphi_{yy} \left(Y_{t+\epsilon}\right)\right|^2 \cdot \left|\y^\epsilon_{t+\epsilon}\right|^2 d\delta \right] \notag \\
    &\qquad + \e\left[\sup_{t\le r\le t+\epsilon}\left|\varphi_{yy}\left(Y_r\right) -\varphi_{yy} \left(Y_t\right) \right|^2 \cdot \left( |\y_t|^2+ \left|f\left(t,\xi,Y_t,\lr({ \xi,Y_t}),u_t\right)\right|^2\right) \right] \bigg\} \notag \\
    \le\ & C(L,T,\lambda) \bigg\{ \left\|\y_{t+\epsilon} - \y_t\right\|_2^2 + \left(1+\|\xi\|_2^2\right)\epsilon \notag \\
    &\qquad\qquad +\e\left[\int_0^1 \left|\varphi_{yy} \left(Y^{\epsilon,\delta}_{t+\epsilon}\right)-\varphi_{yy} \left(Y_{t+\epsilon}\right)\right|^2 \cdot \left|\y^\epsilon_{t+\epsilon}\right|^2 d\delta \right] \notag \\
    &\qquad\qquad + \e\left[\sup_{t\le r\le t+\epsilon}\left|\varphi_{yy}\left(Y_r\right) -\varphi_{yy} \left(Y_t\right) \right|^2 \cdot \left( |\y_t|^2+ \left|f\left(t,\xi,Y_t,\lr({ \xi,Y_t}),u_t\right)\right|^2\right) \right] \bigg\}. \label{zeta_R-1}
\end{align}
For the first term of the right-hand side of the last inequality, since
\begin{align*}
    \y_{t+\epsilon} - \y_t= -\int_{t}^{t+\epsilon}  \left[f_1(r) \x_r + f_2(r) \y_r  + f_4(r) \e\left[ \x_r\right] + f_5(r) \e\left[ \y_r\right] + f_7(r) \U_r \right] dr,
\end{align*}
then, by using \eqref{lem:xypru_0}, we know that 
\begin{align}\label{zeta_R-2}
    \left\|\y_{t+\epsilon} - \y_t\right\|_2^2\le C(L,T,\lambda)\left(1+\|\xi\|_2^2\right)\epsilon^2. 
\end{align}
For the last two terms of the right-hand side of \eqref{zeta_R-1}, in view of the dominated convergence theorem, also noting that Estimates \eqref{lem:xypru_0'} and \eqref{lem:xypru_0}, we know that  
\begin{align}
    \lim_{\epsilon\to 0}\ \bigg\{&\e\left[\int_0^1 \left|\varphi_{yy} \left(Y^{\epsilon,\delta}_{t+\epsilon}\right)-\varphi_{yy} \left(Y_{t+\epsilon}\right)\right|^2 \cdot \left|\y^\epsilon_{t+\epsilon}\right|^2 d\delta \right] \notag \\
    &+ \e\left[\sup_{t\le r\le t+\epsilon}\left|\varphi_{yy}\left(Y_r\right) -\varphi_{yy} \left(Y_t\right) \right|^2 \cdot \left( |\y_t|^2+ \left|f\left(t,\xi,Y_t,\lr({ \xi,Y_t}),u_t\right)\right|^2\right) \right] \bigg\}=0.  \label{zeta_R-3}
\end{align}
From \eqref{zeta_R-1}, \eqref{zeta_R-2} and \eqref{zeta_R-3}, we have 
\begin{equation}\label{zeta_R}
    \lim_{\epsilon\to 0}\ \left\|\zeta_R^\epsilon\right\|_2^2 =0.
\end{equation}
Similar as the proof of \eqref{zeta_R}, we can also obtain
\begin{equation*}
    \lim_{\epsilon\to 0}\ \bigg[\left\|\zeta_P^\epsilon\right\|_2^2+ \sup_{t+\epsilon\le s\le T}\left\|\eta_u^\epsilon(s)\right\|_2^2 + \int_{t+\epsilon}^{T} \left(\left\|\eta_P^\epsilon(s)\right\|_2^2 + \left\|\eta_R^\epsilon(s)\right\|_2^2 \right) ds\bigg]=0.  
\end{equation*}
%From the last estimate together with \eqref{zeta_X} and \eqref{zeta_R}, we obtain \eqref{zeta}.

\noindent\textbf{(Step 2(b))} We now prove \eqref{derivative_t}.\\
Similar as in \eqref{lem:xypru_1'}-\eqref{lem:xypru_3'} and \eqref{lem:xypru_1}-\eqref{lem:xypru_3}, by applying It\^o's formula for $\left(\de\p_s^\epsilon\right)^\top \de\x_s^\epsilon$ and $\left(\de\R_s^\epsilon\right)^\top \de\y_s^\epsilon$, taking expectation and using Fubini's theorem and the convexity condition \eqref{def:convex_h}, we have
\begin{align}
    &\e \left[\left(\de\p_T^\epsilon\right)^\top \de\x_T^\epsilon  - \left(\de\p_{t+\epsilon}^\epsilon\right)^\top \de\x_{t+\epsilon}^\epsilon + \left(\de\R_T^\epsilon\right)^\top \de\y_T^\epsilon -\left(\de\R_{t+\epsilon}^\epsilon\right)^\top \de\y_{t+\epsilon}^\epsilon \right] \notag \\
    \le\ & -2\lambda \int_{t+\epsilon}^T  \left\|\de\U_s^\epsilon\right\|_2^2  ds + \int_{t+\epsilon}^T \left(\left\|\de\x_s^\epsilon\right\|_2\cdot \left\|\eta_P^\epsilon(s)\right\|_2 +\left\|\de\y_s^\epsilon\right\|_2\cdot \left\|\eta_R^\epsilon(s)\right\|_2\right) ds. \label{lem:xypru_3''}
\end{align}
By using the terminal conditions for $\de\p_T^\epsilon$ and $\de\y_T^\epsilon$, the initial conditions for $\de\x_{t+\epsilon}^\epsilon$ and $\de\R_{t+\epsilon}^\epsilon$, and then applying again Fubini's theorem and also the convexity conditions on $\varphi$ and $\Phi$ in \eqref{def:convex_Phi} and \eqref{def:convex_varphi}, we have
\begin{align*}
    &\e \left[\left(\de\p_T^\epsilon\right)^\top \de\x_T^\epsilon  - \left(\de\p_{t+\epsilon}^\epsilon\right)^\top \de\x_{t+\epsilon}^\epsilon + \left(\de\R_T^\epsilon\right)^\top \de\y_T^\epsilon -\left(\de\R_{t+\epsilon}^\epsilon\right)^\top \de\y_{t+\epsilon}^\epsilon \right]\\
    \geq\ & \e\left[\left(\de\x_T^\epsilon\right)^\top \zeta_P^\epsilon - \left(\de\p_{t+\epsilon}^\epsilon\right)^\top \zeta_X^\epsilon - \left(\de\y_{t+\epsilon}^\epsilon\right)^\top \zeta_R^\epsilon\right].
\end{align*}
Substituting the last estimate into \eqref{lem:xypru_3''}, we know that for any $\epsilon\in(0,1)$, 
\begin{align}
    &2\lambda \int_{t+\epsilon}^T  \left\|\de\U_s^\epsilon\right\|_2^2  ds\notag\\
    \le\ & \left\|\de\x_T^\epsilon\right\|_2\cdot \left\|\zeta_P^\epsilon\right\|_2 + \left\|\de\p_{t+\epsilon}^\epsilon\right\|_2 \cdot \left\|\zeta_X^\epsilon\right\|_2 + \left\|\de\y_{t+\epsilon}^\epsilon\right\|_2 \cdot  \left\|\zeta_R^\epsilon\right\|_2 \notag\\
    &+ \int_{t+\epsilon}^T \left(\left\|\de\x_T^\epsilon\right\|_2\cdot \left\|\eta_P^\epsilon(s)\right\|_2 +\left\|\de\y_T^\epsilon\right\|_2\cdot \left\|\eta_R^\epsilon(s)\right\|_2  \right) ds \notag\\
    \le\ &\epsilon \bigg(\sup_{t+\epsilon\le s\le T}\left\|\de\x_s^\epsilon\right\|_2^2 + \sup_{t+\epsilon\le s\le T}\left\|\de\p_{s}^\epsilon\right\|_2^2 + \sup_{t+\epsilon\le s\le T} \left\|\de\y_{s}^\epsilon\right\|_2^2 \bigg) \notag\\
    &+ \frac{C(L,T)}{\epsilon} \left[\left\|\zeta_X^\epsilon\right\|_2^2+\left\|\zeta_P^\epsilon\right\|_2^2+\left\|\zeta_R^\epsilon\right\|_2^2 + \int_{t+\epsilon}^{T} \left(\left\|\eta_P^\epsilon(s)\right\|_2^2 + \left\|\eta_R^\epsilon(s)\right\|_2^2 \right) ds\right]. \label{lem:xypru_4''}
\end{align}
By using standard estimate for the equation of $\de\x^\epsilon$, we have 
\begin{align}
    \e\bigg[\sup_{t+\epsilon\le s\le T} \left|\de\x_s^\epsilon\right|^2\bigg]\le C(L,T) \left(\left\|\zeta_X^\epsilon\right\|_2^2+\int_{t+\epsilon}^T \left\|\de\U_s^\epsilon\right\|_2^2 ds \right); \label{lem:xypru_5''}
\end{align}
then, by using standard estimate for the equation of $\de\y^\epsilon$ and Estimate \eqref{lem:xypru_5''}, we have
\begin{align}
    \e\bigg[\sup_{t+\epsilon\le s\le T} \left|\de\y_s^\epsilon\right|^2\bigg]\le C(L,T) \left(\left\|\zeta_X^\epsilon\right\|_2^2+\int_{t+\epsilon}^T \left\|\de\U_s^\epsilon\right\|_2^2 ds \right). \label{lem:xypru_6''}
\end{align}
Similarly, from standard estimate for the equation of $\de\R^\epsilon$, by using Estimates \eqref{lem:xypru_5''} and \eqref{lem:xypru_6''}, we have 
\begin{align}
    \e\bigg[\sup_{t+\epsilon\le s\le T} \left|\de\R_s^\epsilon\right|^2\bigg]\le C(L,T,\lambda) \left(\left\|\zeta_X^\epsilon\right\|_2^2+\left\|\zeta_R^\epsilon\right\|_2^2 +\int_{t+\epsilon}^T \left(\left\|\de\U_s^\epsilon\right\|_2^2 +\left\|\eta_R^\epsilon(s)\right\|_2^2\right) ds \right); \label{lem:xypru_7''}
\end{align}
and finally, by using standard estimate for the equation of $\de\p^\epsilon$ and Estimates \eqref{lem:xypru_5}-\eqref{lem:xypru_7}, we know that
\begin{align}
    &\e\bigg[\sup_{t+\epsilon\le s\le T} \left|\de\p_s^\epsilon\right|^2\bigg]\notag\\
    \le\ & C(L,T,\lambda) \left(\left\|\zeta_X^\epsilon\right\|_2^2+\left\|\zeta_R^\epsilon\right\|_2^2 +\left\|\zeta_P^\epsilon\right\|_2^2+\int_{t+\epsilon}^T \left(\left\|\eta_R^\epsilon(s)\right\|_2^2+\left\|\de\U_s^\epsilon\right\|_2^2 + \left\|\eta_P^\epsilon (s) \right\|_2^2 \right) ds \right). \label{lem:xypru_8''}
\end{align}
Substituting \eqref{lem:xypru_5''} and \eqref{lem:xypru_8''} back into \eqref{lem:xypru_4''}, we have 
\begin{align*}
    &2\lambda \int_{t+\epsilon}^T  \left\|\de\U_s^\epsilon\right\|_2^2  ds\notag\\
    \le\ &\epsilon C(L,T,\lambda) \int_{t+\epsilon}^T \left\|\de\U_s^\epsilon\right\|_2^2  ds \notag\\
    &+ C(L,T,\lambda)\left(1+\frac{1}{\epsilon}\right) \left[\left\|\zeta_X^\epsilon\right\|_2^2+\left\|\zeta_P^\epsilon\right\|_2^2+\left\|\zeta_R^\epsilon\right\|_2^2 + \int_{t+\epsilon}^{T} \left(\left\|\eta_P^\epsilon(s)\right\|_2^2 + \left\|\eta_R^\epsilon(s)\right\|_2^2 \right) ds\right],
\end{align*}
and then, by choosing $\epsilon:= \frac{\lambda}{C(L,T,\lambda)}$, we have
\begin{align}
    \int_t^T \left\|\de\U_s^\epsilon\right\|_2^2 ds \le\ & C(L,T,\lambda) \left[\left\|\zeta_X^\epsilon\right\|_2^2+\left\|\zeta_P^\epsilon\right\|_2^2+\left\|\zeta_R^\epsilon\right\|_2^2 + \int_{t+\epsilon}^{T} \left(\left\|\eta_P^\epsilon(s)\right\|_2^2 + \left\|\eta_R^\epsilon(s)\right\|_2^2 \right) ds\right]. \notag
\end{align}
Substituting the last estimate back into \eqref{lem:xypru_5}-\eqref{lem:xypru_8}, we obtain 
\begin{align}
    &\e\bigg[\sup_{t\le s\le T}\left|\de\x_s^\epsilon\right|^2+\sup_{t\le s\le T}\left|\de\y_s^\epsilon\right|^2+\sup_{t\le s\le T}\left|\de\p_s^\epsilon\right|^2+\sup_{t\le s\le T}\left|\de\R_s^\epsilon\right|^2\bigg] \notag\\
    \le\ & C(L,T,\lambda) \left[\left\|\zeta_X^\epsilon\right\|_2^2+\left\|\zeta_P^\epsilon\right\|_2^2+\left\|\zeta_R^\epsilon\right\|_2^2 + \int_{t+\epsilon}^{T} \left(\left\|\eta_P^\epsilon(s)\right\|_2^2 + \left\|\eta_R^\epsilon(s)\right\|_2^2 \right) ds\right]. \label{lem:xypru_9''}
\end{align}
Now we use the condition  
\begin{align*}
    &0=b_3(s)^\top  \de \p^\epsilon_s -f_7(s)^\top  \de \R^\epsilon_s + \left[\left(\begin{array}{ccc}
         h_{vx} & h_{vy} & h_{vv} 
    \end{array} \right)\left(s,\theta_s,\lr\left({ \theta_s}\right),u_s\right) \left(\begin{array}{c}
         \de \x^\epsilon_s \\
         \de \y^\epsilon_s \\
         \de \U^\epsilon_s 
    \end{array}\right)\right] \notag \\
    & \qquad +\widetilde{\e}\left\{ \left[\left(\begin{array}{cc}
         D_{x'}\frac{dh_{v}}{d\nu} & D_{y'}\frac{dh_{v}}{d\nu} 
    \end{array} \right) \left(s,\theta_s,\lr\left(\theta_s\right),u_s\right)\left(\widetilde{\theta_s}\right) \right] \left( \begin{array}{cc}
         \widetilde{\x^\epsilon_s}  \\
         \widetilde{\y^\epsilon_s} 
    \end{array}\right) \right\} + \eta_u^\epsilon(s),
\end{align*}
and the convexity of $h$ in $v$ to get
\begin{align}\label{lem:xypru_10''}
    \left|\de\U_s^\epsilon\right|  
    \le\ &  \frac{L}{2\lambda} \left(\left|\de\p_s^\epsilon\right| + \left|\de\R_s^\epsilon\right| +\left|\de\x_s^\epsilon\right| + \left|\de\y_s^\epsilon\right| +\left\|\de\x_s^\epsilon\right\|_2 { +\left\|\de\y_s^\epsilon\right\|_2} \right)+\frac{1}{2\lambda} \left|\eta_u^\epsilon(s)\right|.
\end{align}
From \eqref{lem:xypru_9''}, \eqref{lem:xypru_10''} and the estimate \eqref{zeta}, we obtain \eqref{derivative_t}.

\subsection{Proof of Lemma~\ref{lem:tmux}}\label{App:lem:tmux}

The proof of the existence and uniqueness result follows a similar approach as in Subsection~\ref{subsec:FBSDE}, so we omit it here. Moreover, if $\Theta^{t,x,\mu}$ is a solution of \eqref{FBSDEs:x}, then, we see easily that the processes $\Theta^{t,x,\mu}\big|_{x=\xi}$ satisfy the forward-backward equations for $\Theta^{t,\xi}$. Then, from the uniqueness result, we have $\Theta^{t,x,\mu}\big|_{x=\xi}=\Theta^{t,\xi}$. Here, we only prove \eqref{tmux-0}; the proof of \eqref{tmux-0'} follows a similar approach, using the estimate \eqref{1order-FBSDE-0}. For notational convenience, we denote by $\de x:=x_2-x_1$, and define $\de \Theta:= \Theta^{t,x_2,\mu}-\Theta^{t,x_1,\mu}$, which satisfy the following equations: for $s\in[t,T]$, 
\begin{align}
    &\de X_s= x_2-x_1+\int_t^s \left[b_1(r) \de X_r +b_3(r) \de u_r\right] dr, \notag \\
    &\de Y_s= g_1 \de X_T  + \int_s^T  \left[f_1(r) \de X_r + f_2(r) \de Y_r  + f_7(r) \de u_r \right] dr, \notag \\
    &\de P_s= - g_1^\top \de R_T  +\Phi_x\left(\cdot, \lr\left(X^{t,\xi}_T\right)\right)\Big|^{X^{t,x_2,\mu}_T}_{X^{t,x_1,\mu}_T}  +{\e}\left[D_{x'}\frac{d\Phi}{d\nu} \left({X^{t,\xi}_T}, \lr\left(X^{t,\xi}_T\right)\right)\left(\cdot\right)\Big|^{X^{t,x_2,\mu}_T}_{X^{t,x_1,\mu}_T} \right]  \notag  \\
    &\quad\qquad + \int_s^T \bigg\{  b_1(r)^\top  \de P_r - f_1(r)^\top  \de R_r +h_x\left(r, \cdot, \lr\left( X^{t,\xi}_r,Y^{t,\xi}_r\right), \cdot \right)\Big|^{\left(X^{t,x_2,\mu}_r, Y^{t,x_2,\mu}_r,u^{t,x_2,\mu}_r\right)}_{\left(X^{t,x_1,\mu}_r, Y^{t,x_1,\mu}_r,u^{t,x_1,\mu}_r\right)}  \notag  \\
    &\qquad\qquad\qquad  + {\e} \bigg[D_{x'}\frac{dh}{d\nu} \left(r, X^{t,\xi}_r, Y^{t,\xi}_r , \lr\left( X^{t,\xi}_r, Y^{t,\xi}_r \right),  u^{t,\xi}_r \right)\left(\cdot\right)\Big|^{\theta^{t,x_2,\mu}_r}_{\theta^{t,x_1,\mu}_r} \bigg]  \bigg\} dr,  \notag \\
    &\de R_s= - \varphi_y \left(\cdot\right)\big|^{Y^{t,x_2,\mu}_t}_{Y^{t,x_1,\mu}_t} +\int_t^s \Bigg\{ f_2(r)^\top  \de R_r -h_y \left(r, \cdot, \lr\left( X^{t,\xi}_r, Y^{t,\xi}_r\right),  \cdot \right)\Big|^{\left(X^{t,x_2,\mu}_r, Y^{t,x_2,\mu}_r,u^{t,x_2,\mu}_r\right)}_{\left(X^{t,x_1,\mu}_r, Y^{t,x_1,\mu}_r,u^{t,x_1,\mu}_r\right)} \notag\\
    &\qquad\qquad\qquad\qquad\qquad\qquad  - {\e} \bigg[D_{y'}\frac{dh}{d\nu} \left(r, X^{t,\xi}_r, Y^{t,\xi}_r , \lr\left( X^{t,\xi}_r, Y^{t,\xi}_r \right),  u^{t,\xi}_r  \right)\left(\cdot\right)\Big|^{\theta^{t,x_2,\mu}_r}_{\theta^{t,x_1,\mu}_r} \bigg] \Bigg\} dr,  \notag \\
    &b_3(s)^\top  \de P_s -f_7(s)^\top \de R_s + h_v\left(r, \cdot, \lr\left( X^{t,\xi}_r, Y^{t,\xi}_r\right),  \cdot \right)\Big|^{\left(X^{t,x_2,\mu}_r, Y^{t,x_2,\mu}_r,u^{t,x_2,\mu}_r\right)}_{\left(X^{t,x_1,\mu}_r, Y^{t,x_1,\mu}_r,u^{t,x_1,\mu}_r\right)}  =0. \label{tmux-5}
\end{align}
By computing the derivative of $\dd_r\left(\de P_r^\top \de X_r\right)$ and $\dd_r\left(\de R_r^\top \de Y_r\right)$ and integrating $r$ from $t$ to $T$, and using the last equation of \eqref{tmux-5}, we have 
\begin{align}
    &\de P_T^\top \de X_T - \de P_t^\top \de x + \de R_T^\top \de Y_T - \de R_t^\top \de Y_t \notag \\
    =\ & \int_t^T \bigg\{ \left[\de P_r^\top b_3(r) -\de R_r^\top f_7(r) \right] \de u_r - \bigg[h_x\left(r, \cdot, \lr\left( X^{t,\xi}_r, Y^{t,\xi}_r\right),  \cdot \right)\Big|^{\left(X^{t,x_2,\mu}_r, Y^{t,x_2,\mu}_r,u^{t,x_2,\mu}_r\right)}_{\left(X^{t,x_1,\mu}_r, Y^{t,x_1,\mu}_r,u^{t,x_1,\mu}_r\right)} \bigg]^\top \de X_r \notag \\
    &\quad\qquad - \bigg[h_y \left(r, \cdot, \lr\left( X^{t,\xi}_r, Y^{t,\xi}_r\right),  \cdot \right)\Big|^{\left(X^{t,x_2,\mu}_r, Y^{t,x_2,\mu}_r,u^{t,x_2,\mu}_r\right)}_{\left(X^{t,x_1,\mu}_r, Y^{t,x_1,\mu}_r,u^{t,x_1,\mu}_r\right)} \bigg]^\top \de Y_r \notag \\
    &\quad\qquad  - {\e} \bigg[D_{x'}\frac{dh}{d\nu} \left(r, X^{t,\xi}_r, Y^{t,\xi}_r , \lr\left( X^{t,\xi}_r, Y^{t,\xi}_r \right),  u^{t,\xi}_r \right)\left(\cdot\right)\Big|^{\theta^{t,x_2,\mu}_r}_{\theta^{t,x_1,\mu}_r}  \bigg]^\top \de X_r \notag\\
    &\quad\qquad  - {\e} \bigg[D_{y'}\frac{dh}{d\nu} \left(r, X^{t,\xi}_r, Y^{t,\xi}_r , \lr\left( X^{t,\xi}_r, Y^{t,\xi}_r \right),  u^{t,\xi}_r \right)\left(\cdot\right)\Big|^{\theta^{t,x_2,\mu}_r}_{\theta^{t,x_1,\mu}_r}  \bigg]^\top \de Y_r \bigg\} dr \notag \\
    =\ & \int_t^T \bigg\{  - \bigg[h_v \left(r, \cdot, \lr\left( X^{t,\xi}_r, Y^{t,\xi}_r\right),  \cdot \right)\Big|^{\left(X^{t,x_2,\mu}_r, Y^{t,x_2,\mu}_r,u^{t,x_2,\mu}_r\right)}_{\left(X^{t,x_1,\mu}_r, Y^{t,x_1,\mu}_r,u^{t,x_1,\mu}_r\right)} \bigg]^\top \de u_r \notag \\
    &\quad\qquad - \bigg[h_x\left(r, \cdot, \lr\left( X^{t,\xi}_r, Y^{t,\xi}_r\right),  \cdot \right)\Big|^{\left(X^{t,x_2,\mu}_r, Y^{t,x_2,\mu}_r,u^{t,x_2,\mu}_r\right)}_{\left(X^{t,x_1,\mu}_r, Y^{t,x_1,\mu}_r,u^{t,x_1,\mu}_r\right)} \bigg]^\top \de X_r \notag \\
    &\quad\qquad - \bigg[h_y \left(r, \cdot, \lr\left( X^{t,\xi}_r, Y^{t,\xi}_r\right),  \cdot \right)\Big|^{\left(X^{t,x_2,\mu}_r, Y^{t,x_2,\mu}_r,u^{t,x_2,\mu}_r\right)}_{\left(X^{t,x_1,\mu}_r, Y^{t,x_1,\mu}_r,u^{t,x_1,\mu}_r\right)} \bigg]^\top \de Y_r \notag \\
    &\quad\qquad  - {\e} \bigg[D_{x'}\frac{dh}{d\nu} \left(r, X^{t,\xi}_r, Y^{t,\xi}_r , \lr\left( X^{t,\xi}_r, Y^{t,\xi}_r \right),  u^{t,\xi}_r \right)\left(\cdot\right)\Big|^{\theta^{t,x_2,\mu}_r}_{\theta^{t,x_1,\mu}_r}  \bigg]^\top \de X_r \notag\\
    &\quad\qquad  - {\e} \bigg[D_{y'}\frac{dh}{d\nu} \left(r, X^{t,\xi}_r, Y^{t,\xi}_r , \lr\left( X^{t,\xi}_r, Y^{t,\xi}_r \right),  u^{t,\xi}_r \right)\left(\cdot\right)\Big|^{\theta^{t,x_2,\mu}_r}_{\theta^{t,x_1,\mu}_r}  \bigg]^\top \de Y_r \bigg\} dr. \notag 
\end{align}
Then, by using the convexity of $h$ in $(x,y,v)$ and also the convexity of $\frac{dh}{d\nu}(s,x,y,\varrho,v)(x',y')$ in $(x',y')$ (see Assumption (A3)),  we have
\begin{align}\label{tmux-1}
    &\de P_T^\top \de X_T + \de R_T^\top \de Y_T - \de R_t^\top \de Y_t \le -2\lambda \int_t^T |\de u_r|^2   dr. 
\end{align}
On the other hand, by using the terminal conditions for $\de P_T$ and $\de Y_T$ and the initial condition for $\de R_t$, and then using the convexity of $\Phi$ and $\varphi$, we have 
\begin{align}
    &\de P_T^\top \de X_T - \de P_t^\top \de x + \de R_T^\top \de Y_T - \de R_t^\top \de Y_t \notag \\
    =\ & \left[\Phi_x\left(X^{t,x_2,\mu}_T, \lr\left(X^{t,\xi}_T\right) \right) -\Phi_x\left(X^{t,x_1,\mu}_T, \lr\left(X^{t,\xi}_T\right) \right) \right]^\top \de X_T  \notag \\
    &+{\e}\left[D_{x'}\frac{d\Phi}{d\nu} \left({X^{t,\xi}_T}, \lr\left(X^{t,\xi}_T\right)\right)\left(X^{t,x_2,\mu}_T\right) - D_{x'}\frac{d\Phi}{d\nu} \left({X^{t,\xi}_T}, \lr\left(X^{t,\xi}_T\right)\right)\left(X^{t,x_1,\mu}_T\right) \right]^\top \de X_T \notag \\
    &  + \left[\varphi_y \left(Y^{t,x_2,\mu}_t\right) -\varphi_y \left(Y^{t,x_1,\mu}_t\right)\right]^\top \de Y_t - \de P_t^\top \de x \notag\\
    \geq\ & - \de P_t^\top \de x. \label{tmux-2}
\end{align}
Combining \eqref{tmux-1} and \eqref{tmux-2}, we know that for any $\epsilon\in(0,1)$,
\begin{align}
    2\lambda\int_t^T |\de u_r|^2 dr \le \de P_t^\top \de x \le \epsilon |\de P_t|^2 + \frac{C(L)}{\epsilon} |\de x|^2. \label{tmux-3}
\end{align}
From the ODEs for $\de X_s$, $\de Y_s$, $\de R_s$ and $\de P_s$, we know that 
\begin{align}\label{tmux-4}
    \sup_{t\le r\le T}|\de X_r|^2 + \sup_{t\le r\le T}|\de Y_r|^2 + \sup_{t\le r\le T}|\de R_r|^2 +\sup_{t\le r\le T}|\de P_r|^2  \le C(L,T) \left( |\de x|^2+\int_t^T |\de u_r|^2 dr\right). 
\end{align}
Substituting \eqref{tmux-4} into \eqref{tmux-3} and choosing $\epsilon$ small enough, we have 
\begin{align*}
    \int_t^T |\de u_r|^2 dr \le C(L,T,\lambda) |\de x|^2. 
\end{align*}
By substituting the last estimate back into \eqref{tmux-4}, and then using the last equation of \eqref{tmux-5} (and the convexity of $h$ in $v$), we finally obtain \eqref{tmux-0}. Then, as a consequence of Theorem~\ref{thm:value_regu_mu}, we obtain $D_x\frac{dV}{d\nu}(t,\mu)(\cdot)=P^{t,\cdot,\mu}_t,\ \mu-a.e.$.

\subsection{Proof of Lemma~\ref{lem:D_x}}\label{App:lem:D_x}
The derivative $D_x\Theta^{t,x,\mu}=\left(D_x X^{t,x,\mu},D_x Y^{t,x,\mu},D_x P^{t,x,\mu},D_x R^{t,x,\mu},D_x u^{t,x,\mu}\right)$ can be characterized as the solution of the following system of ODEs:
\begin{align}
            &D_x X^{t,x,\mu}_s= I_n +\int_t^s \left[b_1(r) D_x X^{t,x,\mu}_r +b_3(r) D_x u^{t,x,\mu}_r\right]   dr, \notag  \\
            &D_x Y^{t,x,\mu}_s= g_1 D_x X^{t,x,\mu}_T + \int_s^T  \left[f_1(r) D_x X^{t,x,\mu}_r +f_2(r) D_x Y^{t,x,\mu}_r +f_7(r) D_x u^{t,x,\mu}_r \right] dr,  \notag \\
            &D_x P^{t,x,\mu}_s= - g_1^\top D_x R^{t,x,\mu}_T +\Phi_{xx}\left(X^{t,x,\mu}_T, \lr\left(X^{t,\xi}_T\right) \right)D_x X^{t,x,\mu}_T  \notag \\
            &\quad\qquad\qquad +{\e}\left[D^2_{x'}\frac{d\Phi}{d\nu} \left({X^{t,\xi}_T}, \lr\left(X^{t,\xi}_T\right)\right)\left(X^{t,x,\mu}_T\right) D_x X^{t,x,\mu}_T \right]  \notag  \\
            &\quad\qquad\qquad + \int_s^T \Bigg\{  b_1(r)^\top  D_x P^{t,x,\mu}_r - f_1(r)^\top  D_x R^{t,x,\mu}_r  \notag \\
            &\qquad\qquad\qquad +\left[\left(\begin{array}{ccc}
                 h_{xx} & h_{xy} & h_{xv}
            \end{array}\right)\left(r,  \theta^{t,x,\mu}_r, \lr\left( \theta^{t,\xi}_r\right),  u^{t,x,\mu}_r  \right)\right] \left(\begin{array}{c}
                 D_x X^{t,x,\mu}_r  \\
                 D_x Y^{t,x,\mu}_r  \\
                 D_x u^{t,x,\mu}_r
            \end{array}\right)  \notag \\
            &\qquad\qquad\qquad  + {\e} \bigg\{ \left[\left(\begin{array}{cc}
                 D^2_{x'}\frac{dh}{d\nu} &  D_{y'}D_{x'}\frac{dh}{d\nu}
            \end{array}\right)\left(r, \theta^{t,\xi}_r,\lr\left( \theta^{t,\xi}_r\right),  u^{t,\xi}_r  \right)\left( \theta^{t,x,\mu}_r\right)\right] \left(\begin{array}{c}
                 D_x X^{t,x,\mu}_r   \\
                 D_x Y^{t,x,\mu}_r 
            \end{array}\right)  \bigg\}  \Bigg\} dr,  \notag \\
            &D_x R^{t,x,\mu}_s= - \varphi_{yy} \left(Y^{t,x,\mu}_t\right)D_x Y^{t,x,\mu}_t +\int_t^s \Bigg\{ f_2(r)^\top D_x R^{t,x,\mu}_r  \notag \\
            &\qquad\qquad\qquad -\left[\left(\begin{array}{ccc}
                 h_{yx} & h_{yy} & h_{yv}
            \end{array}\right)\left(r,  \theta^{t,x,\mu}_r, \lr\left( \theta^{t,\xi}_r\right),  u^{t,x,\mu}_r  \right)\right] \left(\begin{array}{c}
                 D_x X^{t,x,\mu}_r  \\
                 D_x Y^{t,x,\mu}_r  \\
                 D_x u^{t,x,\mu}_r
            \end{array}\right)  \notag  \\
            &\qquad\qquad\qquad  -{\e} \bigg\{ \left[\left(\begin{array}{cc}
                 D_{x'}D_{y'}\frac{dh}{d\nu} &  D^2_{y'}\frac{dh}{d\nu}
            \end{array}\right)\left(r, \theta^{t,\xi}_r,\lr\left( \theta^{t,\xi}_r\right),  u^{t,\xi}_r  \right)\left( \theta^{t,x,\mu}_r\right)\right] \left(\begin{array}{c}
                 D_x X^{t,x,\mu}_r   \\
                 D_x Y^{t,x,\mu}_r 
            \end{array}\right)  \bigg\} \Bigg\} dr,  \notag \\
            &0= b_3(s)^\top  D_x P^{t,x,\mu}_s -f_7(s)^\top  D_x R^{t,x,\mu}_s \notag \\
            &\qquad +\left[\left(\begin{array}{ccc}
                 h_{vx} & h_{vy} & h_{vv}
            \end{array}\right)\left(r,  \theta^{t,x,\mu}_r, \lr\left( \theta^{t,\xi}_r\right),  u^{t,x,\mu}_r  \right)\right] \left(\begin{array}{c}
                 D_x X^{t,x,\mu}_r  \\
                 D_x Y^{t,x,\mu}_r  \\
                 D_x u^{t,x,\mu}_r
            \end{array}\right),\quad s\in[t,T]. \label{FBSDEs:Dx}
\end{align}
The well-posedness result for System \eqref{FBSDEs:Dx} follows a similar approach as in Subsection~\ref{subsec:FBSDE}; and by computing the derivatives of $\dd_s\left(\left(D_x P^{t,x,\mu}_s\right)^\top D_x X^{t,x,\mu}_s\right)$ and $\dd_s\left(\left(D_x R^{t,x,\mu}_s\right)^\top D_x Y^{t,x,\mu}_s\right)$ and using the convexity assumptions on the cost coefficients, we can obtain \eqref{tmux_Dx}. The proof of the differentiability of $\Theta^{t,x,\mu}$ in $x$ also requires the estimate for the convergence of the term $\frac{1}{\epsilon}\left[\Theta^{t,x+\epsilon y,\mu}_s - \Theta^{t,x,\mu}_s\right]-D_x \Theta^{t,x,\mu}_s \ y$ as $\epsilon$ goes to $0$, uniformly in $s\in[t,T]$ and $x,y\in\brn$, which follows a similar way as in the proof of Theorem~\ref{thm:convergence}. 

\subsection{Proof of Lemma~\ref{Lem:D_xi}}\label{App:Lem:D_xi}
The forward-backward systems for $D_\eta  \Theta^{t,\xi}$ and $D_\eta  \Theta^{t,x,\xi}$ can be derived by formally differentiating the equations satisfied by  $\Theta^{t,\xi+\epsilon\eta}$ and  $\Theta^{t,x,\lr(\xi+\epsilon\eta)}$ with respect to $\epsilon$ at $\epsilon=0$. A rigorous proof requires proving the well-posedness of these two systems in order and then establishing the convergence of the corresponding difference quotients. To keep the presentation concise, we only write down the system for $D_\eta  \Theta^{t,x,\xi}$; the system for $D_\eta  \Theta^{t,\xi}$ is omitted, as its structure is analogous. Note that the coefficients in the displayed system depend explicitly on $D_\eta \Theta^{t,\xi}$. The system is given by
\begin{align}
    &D_\eta X^{t,x,\xi}_s= \int_t^s \left\{ b_1(r) D_\eta X^{t,x,\xi}_r + b_2(r)\e\left[D_\eta X^{t,\xi}_r\right]+b_3(r)D_\eta u^{t,x,\xi}_r\right\} dr, \notag \\
    &D_\eta Y^{t,x,\xi}_s= \int_s^T  \Big\{ f_1(r) D_\eta X^{t,x,\xi}_r + f_2(r) D_\eta Y^{t,x,\xi}_r + f_4(r)\e\left[D_\eta X^{t,\xi}_r\right] + f_5(r)\e\left[D_\eta Y^{t,\xi}_r\right]  \notag \\
    &\quad\qquad\qquad\qquad +f_7(r)D_\eta u^{t,x,\xi}_r\Big\} dr +g_1 D_\eta X^{t,x,\xi}_T+ g_2 \e\left[D_\eta X^{t,\xi}_T\right], \notag \\
    & D_\eta P^{t,x,\xi}_s= - g_1^\top D_\eta R^{t,x,\xi}_T - g_2^\top \e\left[ D_\eta R^{t,\xi}_T\right] +\Phi_{xx}\left(X^{t,x,\mu}_T, \lr\left(X^{t,\xi}_T\right) \right)D_\eta X^{t,x,\xi}_T  \notag \\
    &\qquad + \e\bigg\{\bigg[D_{x'}\frac{d\Phi_x}{d\nu}\left(X^{t,x,\mu}_T, \lr\left(X^{t,\xi}_T\right) \right)\left(X^{t,\xi}_T\right)    +\left(D_{x'}\frac{d\Phi_x}{d\nu} \left({X^{t,\xi}_T}, \lr\left(X^{t,\xi}_T\right)\right)\left(X^{t,x,\mu}_T\right)\right)^\top \bigg] D_\eta X^{t,\xi}_T  \bigg\}  \notag  \\
    &\qquad +{\e}\widetilde{\e}\left[D_{x''}D_{x'}\frac{d^2\Phi}{d\nu^2} \left({X^{t,\xi}_T}, \lr\left(X^{t,\xi}_T\right)\right)\left(X^{t,x,\mu}_T,\widetilde{X^{t,\xi}_T}\right) \widetilde{ D_\eta X^{t,\xi}_T}  \right]  \notag  \\
    &\qquad +{\e}\left[D^2_{x'}\frac{d\Phi}{d\nu} \left({X^{t,\xi}_T}, \lr\left(X^{t,\xi}_T\right)\right)\left(X^{t,x,\mu}_T\right) D_\eta X^{t,x,\xi}_T \right]   \notag \\
    &\qquad  + \int_s^T \Bigg\{  b_1(r)^\top  D_\eta P^{t,x,\xi}_r + b_2(r)^\top \e\left[D_\eta P^{t,\xi}_r\right] - f_1(r)^\top  D_\eta R^{t,x,\xi}_r - f_4(r)^\top \e\left[D_\eta R^{t,\xi}_r\right] \notag \\
    &\qquad\qquad +\left[\left(\begin{array}{ccc}
         h_{xx} & h_{xy} & h_{xv}   
    \end{array}\right) \left(r, \theta^{t,x,\mu}_r, \lr\left( \theta^{t,\xi}_r\right),  u^{t,x,\mu}_r  \right) \right]\left(\begin{array}{c}
         D_\eta X^{t,x,\xi}_r  \\
         D_\eta Y^{t,x,\xi}_r  \\
         D_\eta u^{t,x,\xi}_r
    \end{array}\right)   \notag \\
    &\qquad\qquad + {\e} \left\{\left[\left(\begin{array}{cc}
         D_{x'}\frac{dh_x}{d\nu}  & D_{y'}\frac{dh_x}{d\nu} 
    \end{array}\right)\left(r, \theta^{t,x,\mu}_r, \lr\left( \theta^{t,\xi}_r\right),  u^{t,x,\mu}_r \right)\left(\theta^{t,\xi}_r\right)\right] \left(\begin{array}{c}
        D_\eta X^{t,\xi}_r  \\
        D_\eta Y^{t,\xi}_r
    \end{array}\right) \right\}  \notag \\
    &\qquad\qquad + \e\left\{\left[\left(\begin{array}{ccc}
        D_{x'}\frac{dh_x}{d\nu} \\ 
        D_{x'}\frac{dh_y}{d\nu} \\
        D_{x'}\frac{dh_v}{d\nu}
    \end{array}\right)\left(r,\theta^{t,\xi}_r,\lr\left( \theta^{t,\xi}_r\right),u^{t,\xi}_r\right)\left( \theta^{t,x,\mu}_r\right)\right]^\top  \left(\begin{array}{c}
        D_\eta X^{t,\xi}_r  \\
        D_\eta Y^{t,\xi}_r  \\
        D_\eta u^{t,\xi}_r
    \end{array}
    \right) \right\}  \notag\\
    &\qquad\qquad  +  \e\widetilde{\e} \left\{ \left[\left(\begin{array}{cc}
         D_{x''}D_{x'}\frac{d^2h}{d\nu^2} & D_{y''}D_{x'}\frac{d^2h}{d\nu^2}  
    \end{array}\right) \left(r,\theta^{t,\xi}_r,\lr\left(\theta^{t,\xi}_r\right), u^{t,\xi}_r\right)\left(\theta^{t,x,\mu}_r,\widetilde{\theta^{t,\xi}_r}\right) \right]  \left(\begin{array}{cc}
         \widetilde{D_\eta X^{t,\xi}_r}  \\
         \widetilde{D_\eta Y^{t,\xi}_r} 
    \end{array} \right) \right\} \notag \\
    &\qquad\qquad  +  \e \left\{ \left[\left(\begin{array}{cc}
         D^2_{x'}\frac{dh}{d\nu} & D_{y'}D_{x'}\frac{dh}{d\nu}  
    \end{array}\right) \left(r,\theta^{t,\xi}_r,\lr\left(\theta^{t,\xi}_r\right),u^{t,\xi}_r\right)\left(\theta^{t,x,\mu}_r\right) \right]  \left(\begin{array}{cc}
         D_\eta X^{t,x,\xi}_r   \\
         D_\eta Y^{t,x,\xi}_r
    \end{array} \right) \right\}  \Bigg\} dr,  \notag \\
    & D_\eta R^{t,x,\xi}_s= - \varphi_{yy} \left(Y^{t,x,\mu}_t\right) D_\eta Y^{t,x,\xi}_t +\int_t^s \Bigg\{ f_2(r)^\top  D_\eta R^{t,x,\xi}_r + f_5(r)^\top \e\left[D_\eta R^{t,\xi}_r\right]  \notag  \\
    &\qquad\qquad -\left[\left(\begin{array}{ccc}
         h_{yx} & h_{yy} & h_{yv}   
    \end{array}\right) \left(r, \theta^{t,x,\mu}_r, \lr\left( \theta^{t,\xi}_r\right),  u^{t,x,\mu}_r  \right) \right]\left(\begin{array}{c}
         D_\eta X^{t,x,\xi}_r  \\
         D_\eta Y^{t,x,\xi}_r  \\
         D_\eta u^{t,x,\xi}_r
    \end{array}\right)   \notag \\
    &\qquad\qquad - {\e} \left\{\left[\left(\begin{array}{cc}
         D_{x'}\frac{dh_y}{d\nu}  & D_{y'}\frac{dh_y}{d\nu} 
    \end{array}\right)\left(r, \theta^{t,x,\mu}_r, \lr\left( \theta^{t,\xi}_r\right),  u^{t,x,\mu}_r \right)\left(\theta^{t,\xi}_r\right)\right] \left(\begin{array}{c}
         D_\eta X^{t,\xi}_r  \\
         D_\eta Y^{t,\xi}_r
    \end{array}\right) \right\}  \notag \\
    &\qquad\qquad - \e\left\{\left[\left(\begin{array}{ccc}
        D_{y'}\frac{dh_x}{d\nu} \\ 
        D_{y'}\frac{dh_y}{d\nu} \\
        D_{y'}\frac{dh_v}{d\nu}
    \end{array}\right)\left(r,\theta^{t,\xi}_r,\lr\left( \theta^{t,\xi}_r\right),u^{t,\xi}_r\right)\left( \theta^{t,x,\mu}_r\right)\right]^\top  \left(\begin{array}{c}
        D_\eta X^{t,\xi}_r  \\
        D_\eta Y^{t,\xi}_r  \\
        D_\eta u^{t,\xi}_r
    \end{array}
    \right) \right\}  \notag\\
    &\qquad\qquad  -  \e\widetilde{\e} \left\{ \left[\left(\begin{array}{cc}
         D_{x''}D_{y'}\frac{d^2h}{d\nu^2} & D_{y''}D_{y'}\frac{d^2h}{d\nu^2}  
    \end{array}\right) \left(r,\theta^{t,\xi}_r,\lr\left(\theta^{t,\xi}_r\right), u^{t,\xi}_r\right)\left(\theta^{t,x,\mu}_r,\widetilde{\theta^{t,\xi}_r}\right) \right]  \left(\begin{array}{cc}
         \widetilde{D_\eta X^{t,\xi}_r}  \\
         \widetilde{D_\eta Y^{t,\xi}_r} 
    \end{array} \right) \right\} \notag \\
    &\qquad\qquad  -  \e \left\{ \left[\left(\begin{array}{cc}
         D_{x'}D_{y'}\frac{dh}{d\nu} & D^2_{y'}\frac{dh}{d\nu}  
    \end{array}\right) \left(r,\theta^{t,\xi}_r,\lr\left(\theta^{t,\xi}_r\right),u^{t,\xi}_r\right)\left(\theta^{t,x,\mu}_r\right) \right]  \left(\begin{array}{cc}
         D_\eta X^{t,x,\xi}_r   \\
         D_\eta Y^{t,x,\xi}_r
    \end{array} \right) \right\} \Bigg\} dr, \notag  \\
    &0=b_3(s)^\top D_\eta P^{t,x,\xi}_s -f_7(s)^\top  D_\eta R^{t,x,\xi}_s  \notag \\
    &\qquad\qquad +\left[\left(\begin{array}{ccc}
         h_{vx} & h_{vy} & h_{vv}   
    \end{array}\right) \left(s, \theta^{t,x,\mu}_s, \lr\left( \theta^{t,\xi}_s\right),  u^{t,x,\mu}_s  \right) \right]\left(\begin{array}{c}
         D_\eta X^{t,x,\xi}_s  \\
         D_\eta Y^{t,x,\xi}_s  \\
         D_\eta u^{t,x,\xi}_s
    \end{array}\right)  \notag \\
    &\qquad\qquad + {\e} \left\{\left[\left(\begin{array}{cc}
         D_{x'}\frac{dh_v}{d\nu}  & D_{y'}\frac{dh_v}{d\nu} 
    \end{array}\right)\left(s, \theta^{t,x,\mu}_s, \lr\left( \theta^{t,\xi}_s\right),  u^{t,x,\mu}_s \right)\left(\theta^{t,\xi}_s\right)\right] \left(\begin{array}{c}
         D_\eta X^{t,\xi}_s  \\
         D_\eta Y^{t,\xi}_s
    \end{array}\right) \right\}, \label{FBSDEs:Deta}
\end{align}
The well-posedness follows by the same argument as in Subsection~\ref{subsec:FBSDE}; and the convergence of the difference quotients 
\[
\frac{1}{\epsilon}\big(\Theta_s^{t,\xi+\epsilon\eta}-\Theta_s^{t,\xi}\big) \longrightarrow D_\eta \Theta_s^{t,\xi}
\quad\text{and}\quad
\frac{1}{\epsilon}\big(\Theta_s^{t,x,\xi+\epsilon\eta}-\Theta_s^{t,x,\xi}\big) \longrightarrow D_\eta \Theta_s^{t,x,\xi}
\]
is established in a similar way as that for \eqref{derivative_t}. Moreover, 
%similar as in the proof of Lemma~\ref{lem:tmux},
by using the initial condition $D_\eta X_t^{t,\xi}=\eta$, It\^o's formula and Fubini's theorem and standard estimates for forward and backward equations, we can obtain the following estimate
\begin{align*}
        \e\bigg[&\sup_{t\le s\le T} \left|\left(D_\eta X^{t,\xi}_s,D_\eta Y^{t,\xi}_s,D_\eta P^{t,\xi}_s,D_\eta R^{t,\xi}_s,D_\eta u^{t,\xi}_s\right)\right|^2 \notag \\
        &+\sup_{t\le s\le T} \left|\left(D_\eta X^{t,x,\xi}_s,D_\eta Y^{t,x,\xi}_s,D_\eta P^{t,x,\xi}_s,D_\eta R^{t,x,\xi}_s,D_\eta u^{t,x,\xi}_s\right)\right|^2 \bigg] \le  C(L,T,\lambda)\|\eta\|_2^2.  %\label{tmux_Deta}
\end{align*}
The important relation \eqref{relation:D_xi} then follows by writing down the forward-backward equations for
\begin{align*}
    D_x \Theta^{t,x,\lr(\xi)} \big|_{x=\xi}\,\eta + D_\eta \Theta^{t,x,\xi} \big|_{x=\xi}
\end{align*}
and comparing with the unique system for $D_\eta \Theta^{t,\xi}$. With \eqref{FBSDEs:Deta} and the identity \eqref{relation:D_xi}, we obtain the following forward-backward ODEs for the derivative $D_{x'}\frac{d\Theta^{t,x,\mu}_s}{d\nu}(x')$:
\begin{align}
    &D_{x'}\frac{dX^{t,x,\mu}_s}{d\nu}(x')= \int_t^s \bigg\{ b_1(r) D_{x'}\frac{dX^{t,x,\mu}_r}{d\nu}(x') + b_2(r)\bigg(D_x X^{t,x',\mu}_r + \e\bigg[D_{x'}\frac{dX^{t,x,\mu}_r}{d\nu}(x')\Big|_{x=\xi}\bigg]\bigg) \notag \\
    &\qquad\qquad\qquad\qquad\qquad +b_3(r)D_{x'}\frac{du^{t,x,\mu}_r}{d\nu}(x')\bigg\} dr, \notag \\
    &D_{x'}\frac{dY^{t,x,\mu}_s}{d\nu}(x')= g_1 D_{x'}\frac{dX^{t,x,\mu}_T}{d\nu}(x')+ g_2 \bigg(D_x X^{t,x',\mu}_T +  \e\bigg[D_{x'}\frac{dX^{t,x,\mu}_T}{d\nu}(x')\Big|_{x=\xi}\bigg] \bigg) \notag \\
    &\qquad\qquad\qquad\qquad + \int_s^T  \bigg\{ f_1(r) D_{x'}\frac{dX^{t,x,\mu}_r}{d\nu}(x') + f_2(r) D_{x'}\frac{dY^{t,x,\mu}_r}{d\nu}(x') \notag \\
    &\qquad\qquad\qquad\qquad\qquad + f_4(r)\bigg(D_x X^{t,x',\mu}_r + \e\bigg[D_{x'}\frac{dX^{t,x,\mu}_r}{d\nu}(x')\Big|_{x=\xi}\bigg]\bigg) \notag \\
    &\qquad\qquad\qquad\qquad\qquad + f_5(r)\bigg(D_x Y^{t,x',\mu}_r + \e\bigg[D_{x'}\frac{dY^{t,x,\mu}_r}{d\nu}(x')\Big|_{x=\xi}\bigg]\bigg) +f_7(r)D_{x'}\frac{du^{t,x,\mu}_r}{d\nu}(x') \bigg\} dr, \notag \\
    & D_{x'}\frac{dP^{t,x,\mu}_s}{d\nu}(x')= - g_1^\top D_{x'}\frac{dR^{t,x,\mu}_T}{d\nu}(x') - g_2^\top \bigg(D_x R^{t,x',\mu}_T + \e\bigg[D_{x'}\frac{dR^{t,x,\mu}_T}{d\nu}(x')\Big|_{x=\xi}\bigg]\bigg) \notag \\
    &\qquad +\bigg\{ \Phi_{xx}\left(X^{t,x,\mu}_T, \lr\left(X^{t,\xi}_T\right) \right) + {\e}\left[D^2_{x'}\frac{d\Phi}{d\nu} \left({X^{t,\xi}_T}, \lr\left(X^{t,\xi}_T\right)\right)\left(X^{t,x,\mu}_T\right) \right]\bigg\} D_{x'}\frac{dX^{t,x,\mu}_T}{d\nu}(x')  \notag \\
    &\qquad + \bigg\{ D_{x'}\frac{d\Phi_x}{d\nu}\left(X^{t,x,\mu}_T, \lr\left(X^{t,\xi}_T\right) \right)\left(X^{t,x',\mu}_T\right)  + \left(D_{x'}\frac{d\Phi_x}{d\nu} \left({X^{t,x',\mu}_T}, \lr\left(X^{t,\xi}_T\right)\right)\left(X^{t,x,\mu}_T\right)\right)^\top \notag \\
    &\qquad\qquad + \e\left[D_{x''}D_{x'}\frac{d^2\Phi}{d\nu^2} \left({X^{t,\xi}_T}, \lr\left(X^{t,\xi}_T\right)\right)\left(X^{t,x,\mu}_T, X^{t,x',\mu}_T\right) \right]\bigg\} D_x X^{t,x',\mu}_T   \notag \\
    &\qquad + \e\bigg\{ \bigg[D_{x'}\frac{d\Phi_x}{d\nu}\left(X^{t,x,\mu}_T, \lr\left(X^{t,\xi}_T\right) \right)\left(X^{t,\xi}_T\right) + \left(D_{x'}\frac{d\Phi}{d\nu} \left({X^{t,\xi}_T}, \lr\left(X^{t,\xi}_T\right)\right)\left(X^{t,x,\mu}_T\right)\right)^\top  \notag \\
    &\quad\qquad\qquad + \widetilde{\e}\bigg[D_{x''}D_{x'}\frac{d^2\Phi}{d\nu^2} \left(\widetilde{X^{t,\xi}_T}, \lr\left(X^{t,\xi}_T\right)\right)\left(X^{t,x,\mu}_T,X^{t,\xi}_T\right)  \bigg]\bigg] D_{x'}\frac{dX^{t,x,\mu}_T}{d\nu}(x')\Big|_{x=\xi}  \bigg\} \notag \\
    &\qquad  + \int_s^T \Bigg\{  b_1(r)^\top  D_{x'}\frac{dP^{t,x,\mu}_r}{d\nu}(x') + b_2(r)^\top \bigg(D_x P^{t,x',\mu}_r  + \e\bigg[D_{x'}\frac{dP^{t,x,\mu}_r}{d\nu}(x')\Big|_{x=\xi}\bigg] \bigg) \notag \\
    &\qquad\qquad  - f_1(r)^\top  D_{x'}\frac{dR^{t,x,\mu}_r}{d\nu}(x') - f_4(r)^\top \bigg(D_x R^{t,x',\mu}_r  + \e\bigg[D_{x'}\frac{dR^{t,x,\mu}_r}{d\nu}(x')\Big|_{x=\xi}\bigg] \bigg) \notag \\
    &\qquad\qquad +\left[\left(\begin{array}{ccc}
         h_{xx} & h_{xy} & h_{xv}   
    \end{array}\right) \left(r, \theta^{t,x,\mu}_r, \lr\left( \theta^{t,\xi}_r\right),  u^{t,x,\mu}_r  \right) \right]\left(\begin{array}{c}
         D_{x'}\frac{dX^{t,x,\mu}_r}{d\nu}(x')  \\
         D_{x'}\frac{dY^{t,x,\mu}_r}{d\nu}(x')  \\
         D_{x'}\frac{du^{t,x,\mu}_r}{d\nu}(x')
    \end{array}\right)  \notag \\
    & \qquad\qquad  +  \e \left[\left(\begin{array}{cc}
         D^2_{x'}\frac{dh}{d\nu} & D_{y'}D_{x'}\frac{dh}{d\nu}  
    \end{array}\right) \left(r,\theta^{t,\xi}_r,\lr\left(\theta^{t,\xi}_r\right),u^{t,\xi}_r\right)\left(\theta^{t,x,\mu}_r\right) \right] \left(\begin{array}{cc}
         D_{x'}\frac{dX^{t,x,\mu}_r}{d\nu}(x')   \\
         D_{x'}\frac{dY^{t,x,\mu}_r}{d\nu}(x')
    \end{array} \right) \notag  \\
    &\qquad\qquad + \left[\left(\begin{array}{cc}
         D_{x'}\frac{dh_x}{d\nu}  & D_{y'}\frac{dh_x}{d\nu} 
    \end{array}\right)\left(r, \theta^{t,x,\mu}_r, \lr\left( \theta^{t,\xi}_r\right),  u^{t,x,\mu}_r \right)\left(\theta^{t,x',\mu}_r\right)\right] \left(\begin{array}{c}
         D_x X^{t,x',\mu}_r   \\
         D_x Y^{t,x',\mu}_r
    \end{array}\right)  \notag \\
    & \qquad\qquad + \left[\left(\begin{array}{ccc}
        D_{x'}\frac{dh_x}{d\nu} \\ 
        D_{x'}\frac{dh_y}{d\nu} \\
        D_{x'}\frac{dh_v}{d\nu}
    \end{array}\right)\left(r,\theta^{t,x',\mu}_r,\lr\left( \theta^{t,\xi}_r\right),u^{t,x',\mu}_r\right)\left( \theta^{t,x,\mu}_r\right)\right]^\top  \left(\begin{array}{c}
        D_x X^{t,x',\mu}_r  \\
        D_x Y^{t,x',\mu}_r  \\
        D_x u^{t,x',\mu}_r
    \end{array}
    \right)  \notag\\
    & \qquad\qquad  +  \e \left[\left(\begin{array}{cc}
         D_{x''}D_{x'}\frac{d^2h}{d\nu^2} & D_{y''}D_{x'}\frac{d^2h}{d\nu^2}  
    \end{array}\right) \left(r,\theta^{t,\xi}_r,\lr\left(\theta^{t,\xi}_r\right), u^{t,\xi}_r\right)\left(\theta^{t,x,\mu}_r,\theta^{t,x',\mu}_r\right) \right]  \left(\begin{array}{cc}
         D_x X^{t,x',\mu}_r\\
         D_x Y^{t,x',\mu}_r
    \end{array} \right) \notag \\
    &\qquad\qquad + {\e} \left\{\left[\left(\begin{array}{cc}
         D_{x'}\frac{dh_x}{d\nu}  & D_{y'}\frac{dh_x}{d\nu} 
    \end{array}\right)\left(r, \theta^{t,x,\mu}_r, \lr\left( \theta^{t,\xi}_r\right),  u^{t,x,\mu}_r \right)\left(\theta^{t,\xi}_r\right)\right] \left(\begin{array}{c}
        D_{x'}\frac{dX^{t,x,\mu}_r}{d\nu}(x')\Big|_{x=\xi}  \\
        D_{x'}\frac{dY^{t,x,\mu}_r}{d\nu}(x')\Big|_{x=\xi} 
    \end{array}\right) \right\} \notag \\
    & \qquad\qquad + \e\left\{\left[\left(\begin{array}{ccc}
        D_{x'}\frac{dh_x}{d\nu} \\ 
        D_{x'}\frac{dh_y}{d\nu} \\
        D_{x'}\frac{dh_v}{d\nu}
    \end{array}\right)\left(r,\theta^{t,\xi}_r,\lr\left( \theta^{t,\xi}_r\right),u^{t,\xi}_r\right)\left( \theta^{t,x,\mu}_r\right)\right]^\top  \left(\begin{array}{c}
        D_{x'}\frac{dX^{t,x,\mu}_r}{d\nu}(x')\Big|_{x=\xi}  \\
        D_{x'}\frac{dY^{t,x,\mu}_r}{d\nu}(x')\Big|_{x=\xi}  \\
        D_{x'}\frac{du^{t,x,\mu}_r}{d\nu}(x')\Big|_{x=\xi}  
    \end{array}
    \right) \right\}  \notag\\
    & \qquad\qquad  +  \e \bigg\{\widetilde{\e} \left[\left(\begin{array}{cc}
         D_{x''}D_{x'}\frac{d^2h}{d\nu^2} & D_{y''}D_{x'}\frac{d^2h}{d\nu^2}  
    \end{array}\right) \left(r,\widetilde{\theta^{t,\xi}_r},\lr\left(\theta^{t,\xi}_r\right), \widetilde{u^{t,\xi}_r}\right)\left(\theta^{t,x,\mu}_r,\theta^{t,\xi}_r\right) \right]\notag\\
    &\quad\qquad\qquad\qquad \bigg(\begin{array}{cc}
         D_{x'}\frac{dX^{t,x,\mu}_r}{d\nu}(x')\Big|_{x=\xi}  \\
         D_{x'}\frac{dY^{t,x,\mu}_r}{d\nu}(x')\big|_{x=\xi} 
    \end{array} \bigg) \bigg\}  \Bigg\} dr, \notag \\
    & D_{x'}\frac{dR^{t,x,\mu}_s}{d\nu}(x')= - \varphi_{yy} \left(Y^{t,x,\mu}_t\right) D_{x'}\frac{dY^{t,x,\mu}_t}{d\nu}(x') \notag \\
    &\qquad +\int_t^s \Bigg\{ f_2(r)^\top  D_{x'}\frac{dR^{t,x,\mu}_r}{d\nu}(x') + f_5(r)^\top \bigg(D_x R^{t,x',\mu}_r + \e\bigg[D_{x'}\frac{dR^{t,x,\mu}_r}{d\nu}(x')\Big|_{x=\xi}\bigg]\bigg) \notag  \\
    &\qquad\qquad -\left[\left(\begin{array}{ccc}
         h_{yx} & h_{yy} & h_{yv}   
    \end{array}\right) \left(r, \theta^{t,x,\mu}_r, \lr\left( \theta^{t,\xi}_r\right),  u^{t,x,\mu}_r  \right) \right]\left(\begin{array}{c}
         D_{x'}\frac{dX^{t,x,\mu}_r}{d\nu}(x')  \\
         D_{x'}\frac{dY^{t,x,\mu}_r}{d\nu}(x')  \\
         D_{x'}\frac{du^{t,x,\mu}_r}{d\nu}(x')
    \end{array}\right) \notag  \\
    &\qquad\qquad  -  \e\left[\left(\begin{array}{cc}
         D_{x'}D_{y'}\frac{dh}{d\nu} & D^2_{y'}\frac{dh}{d\nu}  
    \end{array}\right) \left(r,\theta^{t,\xi}_r,\lr\left(\theta^{t,\xi}_r\right),u^{t,\xi}_r\right)\left(\theta^{t,x,\mu}_r\right) \right]  \left(\begin{array}{cc}
         D_{x'}\frac{dX^{t,x,\mu}_r}{d\nu}(x')   \\
         D_{x'}\frac{dY^{t,x,\mu}_r}{d\nu}(x')
    \end{array} \right) \notag \\
    &\qquad\qquad - \left[\left(\begin{array}{cc}
         D_{x'}\frac{dh_y}{d\nu}  & D_{y'}\frac{dh_y}{d\nu} 
    \end{array}\right)\left(r, \theta^{t,x,\mu}_r, \lr\left( \theta^{t,\xi}_r\right),  u^{t,x,\mu}_r \right)\left(\theta^{t,x',\mu}_r\right)\right] \left(\begin{array}{c}
         D_x X^{t,x',\mu}_r\\
         D_x Y^{t,x',\mu}_r
    \end{array}\right) \notag  \\
    & \qquad\qquad  -  \e\left[\left(\begin{array}{cc}
         D_{x''}D_{y'}\frac{d^2h}{d\nu^2} & D_{y''}D_{y'}\frac{d^2h}{d\nu^2}  
    \end{array}\right) \left(r,\theta^{t,\xi}_r,\lr\left(\theta^{t,\xi}_r\right), u^{t,\xi}_r\right)\left(\theta^{t,x,\mu}_r,\theta^{t,x',\mu}_r\right) \right]  \left(\begin{array}{cc}
         D_x X^{t,x',\mu}_r \\
         D_x Y^{t,x',\mu}_r 
    \end{array} \right) \notag \\
    & \qquad\qquad - \left[\left(\begin{array}{ccc}
        D_{y'}\frac{dh_x}{d\nu} \\ 
        D_{y'}\frac{dh_y}{d\nu} \\
        D_{y'}\frac{dh_v}{d\nu}
    \end{array}\right)\left(r,\theta^{t,x',\mu}_r,\lr\left( \theta^{t,\xi}_r\right),u^{t,x',\mu}_r\right)\left( \theta^{t,x,\mu}_r\right)\right]^\top  \left(\begin{array}{c}
        D_x X^{t,x',\mu}_r \\
        D_x Y^{t,x',\mu}_r \\
        D_x u^{t,x',\mu}_r
    \end{array}
    \right)   \notag\\
    &\qquad\qquad - \e \left\{\left[\left(\begin{array}{cc}
         D_{x'}\frac{dh_y}{d\nu}  & D_{y'}\frac{dh_y}{d\nu} 
    \end{array}\right)\left(r, \theta^{t,x,\mu}_r, \lr\left( \theta^{t,\xi}_r\right),  u^{t,x,\mu}_r \right)\left(\theta^{t,\xi}_r\right)\right] \left(\begin{array}{c}
         D_{x'}\frac{dX^{t,x,\mu}_r}{d\nu}(x')\Big|_{x=\xi}  \\
         D_{x'}\frac{dY^{t,x,\mu}_r}{d\nu}(x')\Big|_{x=\xi} 
    \end{array}\right) \right\} \notag \\
    & \qquad\qquad - \e\left\{\left[\left(\begin{array}{ccc}
        D_{y'}\frac{dh_x}{d\nu} \\ 
        D_{y'}\frac{dh_y}{d\nu} \\
        D_{y'}\frac{dh_v}{d\nu}
    \end{array}\right)\left(r,\theta^{t,\xi}_r,\lr\left( \theta^{t,\xi}_r\right),u^{t,\xi}_r\right)\left( \theta^{t,x,\mu}_r\right)\right]^\top  \left(\begin{array}{c}
        D_{x'}\frac{dX^{t,x,\mu}_r}{d\nu}(x')\Big|_{x=\xi}  \\
        D_{x'}\frac{dY^{t,x,\mu}_r}{d\nu}(x')\Big|_{x=\xi}  \\
        D_{x'}\frac{du^{t,x,\mu}_r}{d\nu}(x')\Big|_{x=\xi}  
    \end{array}
    \right) \right\}  \notag\\
    & \qquad\qquad  -  \e \bigg\{ \widetilde{\e}\left[\left(\begin{array}{cc}
         D_{x''}D_{y'}\frac{d^2h}{d\nu^2} & D_{y''}D_{y'}\frac{d^2h}{d\nu^2}  
    \end{array}\right) \left(r,\widetilde{\theta^{t,\xi}_r},\lr\left(\theta^{t,\xi}_r\right), \widetilde{u^{t,\xi}_r}\right)\left(\theta^{t,x,\mu}_r,\theta^{t,\xi}_r\right) \right] \notag\\
    &\qquad\qquad\qquad\qquad \bigg(\begin{array}{cc}
         D_{x'}\frac{dX^{t,x,\mu}_r}{d\nu}(x')\Big|_{x=\xi}  \\
         D_{x'}\frac{dY^{t,x,\mu}_r}{d\nu}(x')\big|_{x=\xi} 
    \end{array} \bigg) \bigg\} \notag  \Bigg\} dr, \notag \\
    &0=b_3(s)^\top D_{x'}\frac{dP^{t,x,\mu}_s}{d\nu}(x') -f_7(s)^\top  D_{x'}\frac{dR^{t,x,\mu}_s}{d\nu}(x') \notag  \\
    &\qquad +\left[\left(\begin{array}{ccc}
         h_{vx} & h_{vy} & h_{vv}   
    \end{array}\right) \left(s, \theta^{t,x,\mu}_s, \lr\left( \theta^{t,\xi}_s\right),  u^{t,x,\mu}_s  \right) \right]\left(\begin{array}{c}
         D_{x'}\frac{dX^{t,x,\mu}_s}{d\nu}(x')  \\
         D_{x'}\frac{dY^{t,x,\mu}_s}{d\nu}(x')  \\
         D_{x'}\frac{du^{t,x,\mu}_s}{d\nu}(x')
    \end{array}\right)  \notag \\
    &\qquad + \left[\left(\begin{array}{cc}
         D_{x'}\frac{dh_v}{d\nu}  & D_{y'}\frac{dh_v}{d\nu} 
    \end{array}\right)\left(s, \theta^{t,x,\mu}_s, \lr\left( \theta^{t,\xi}_s\right),  u^{t,x,\mu}_s \right)\left(\theta^{t,x',\mu}_s\right)\right] \left(\begin{array}{c}
         D_x X^{t,x',\mu}_s  \\
         D_x Y^{t,x',\mu}_s
    \end{array}\right) \notag \\
    &\qquad + {\e} \left\{\left[\left(\begin{array}{cc}
         D_{x'}\frac{dh_v}{d\nu}  & D_{y'}\frac{dh_v}{d\nu} 
    \end{array}\right)\left(s, \theta^{t,x,\mu}_s, \lr\left( \theta^{t,\xi}_s\right),  u^{t,x,\mu}_s \right)\left(\theta^{t,\xi}_s\right)\right] \left(\begin{array}{c}
         D_{x'}\frac{dX^{t,x,\mu}_s}{d\nu}(x')\Big|_{x=\xi}  \\
         D_{x'}\frac{dY^{t,x,\mu}_s}{d\nu}(x')\Big|_{x=\xi} 
    \end{array}\right) \right\}. \label{FBSDEs:Dmu}
\end{align}
The proof of \eqref{relation:D_eta-D_mu} consists of deriving the forward-backward equations for
$\e\left[D_{x'}\frac{d\Theta^{t,x,\mu}_s}{d\nu}\left({\xi}\right) {\eta} \right]$
and then applying Fubini's theorem, and using the uniqueness result of System \eqref{FBSDEs:Deta}. The system \eqref{FBSDEs:Dmu} involves the processes $D_x\Theta^{t,x',\mu}$. Using the estimate \eqref{tmux_Dx} for the latter, we obtain the desired bound \eqref{tmux_Dmu}.

\subsection{Proof of Proposition~\ref{prop:U_x-mu}}\label{App:prop:U_x-mu}
Estimate \eqref{prop:U_x-mu_2'} is a direct consequence of \eqref{tmux-0'}. The proof of the differentiability and the relations \eqref{prop:U_x-mu_1} are direct consequences of Lemmas~\ref{lem:D_x} and \ref{Lem:D_xi}. Then, the estimate \eqref{prop:U_x-mu_2} is a consequence of \eqref{tmux_Dx} and \eqref{tmux_Dmu}. From \eqref{value_equation}, we can also write that 
\begin{align*}
        V(t,\mu)=\int_\brn \left[\int_t^T h\left(s,\theta_s^{t,x,\mu},  \theta_s^{t,\cdot,\mu}\#\mu ,u^{t,x,\mu}_s\right) ds + \Phi\left(X_T^{t,x,\mu}, X_T^{t,\cdot,\mu} \# \mu \right)+\varphi\left(Y_t^{t,x,\mu}\right)\right]\mu(dx),
\end{align*}
where $\theta_s^{t,\cdot,\mu}\#\mu$ represents the push-forward of $\mu$ by the map $x\mapsto \theta_s^{t,x,\mu}$. Then, with the differentiable properties in Lemmas~\ref{lem:D_x} and \ref{Lem:D_xi}, we can explicitly compute that 
\begin{align*}
        &\frac{dV}{d\nu}(t,\mu)(x)\\
        =\ & \int_t^T h\left(s,\theta_s^{t,x,\mu},  \theta_s^{t,\cdot,\mu}\#\mu ,u^{t,x,\mu}_s\right) ds + \Phi\left(X_T^{t,x,\mu}, X_T^{t,\cdot,\mu} \# \mu \right)+\varphi\left(Y_t^{t,x,\mu}\right)\\
        &+ \int_\brn \bigg\{\int_t^T \bigg[\left[\left(\begin{array}{ccc}
         h_{x}^\top & h_{y}^\top & h_{v}^\top   
        \end{array}\right) (\star) \right]\left(\begin{array}{c}
         \frac{dX_s^{t,x',\mu}}{d\nu}(x)  \\
         \frac{dY_s^{t,x',\mu}}{d\nu}(x)  \\
         \frac{du_s^{t,x',\mu}}{d\nu}(x)
        \end{array}\right)  + \frac{dh}{d\nu}(\star)\left(\theta_s^{t,x,\mu}\right)\\
        &\qquad\qquad\qquad + \int_\brn \left[\left(\begin{array}{ccc}
         \left(D_{x'}\frac{dh}{d\nu}\right)^\top & \left(D_{y'}\frac{dh}{d\nu}\right)^\top 
        \end{array}\right) (\star)\left(\theta_s^{t,x'',\mu}\right) \right]\left(\begin{array}{c}
         \frac{dX_s^{t,x'',\mu}}{d\nu}(x)  \\
         \frac{dY_s^{t,x'',\mu}}{d\nu}(x)  
        \end{array}\right) \mu(dx'') \bigg] ds \\
        &\quad\qquad + \Phi_x\left(X_T^{t,x',\mu}, X_T^{t,\cdot,\mu} \# \mu \right)^\top \frac{dX_T^{t,x',\mu}}{d\nu}(x) + \frac{d\Phi}{d\nu}\left(X_T^{t,x',\mu}, X_T^{t,\cdot,\mu} \# \mu \right)\left(X_T^{t,x,\mu}\right) \\
        &\quad\qquad + \int_\brn \bigg(D_{x'}\frac{d\Phi}{d\nu}\left(X_T^{t,x',\mu}, X_T^{t,\cdot,\mu} \# \mu \right)\left(X_T^{t,x'',\mu}\right)\bigg)^\top  \frac{dX_T^{t,x'',\mu}}{d\nu}(x) \mu(dx'') \\
        &\quad\qquad +\varphi_y\left(Y_t^{t,x',\mu}\right)^\top  \frac{dY_t^{t,x',\mu}}{d\nu}(x) \bigg\} \mu(dx'),
\end{align*}
where $(\star):= \left(s,\theta_s^{t,x',\mu},  \theta_s^{t,\cdot,\mu}\#\mu ,u^{t,x',\mu}_s\right)$. By differentiating this formula with respect to $x$, we obtain the explicit formulation of $D_x\frac{dV}{d\nu}$; and by using Estimates \eqref{tmux_Dx} and \eqref{tmux_Dmu}, we obtain \eqref{prop:DdV_x}.

\subsection{Proof of Lemma~\ref{lem:cirly_X_mu}}\label{App:lem:cirly_X_mu}
The derivative $\left(\x^{t,x,\mu},\y^{t,x,\mu},\p^{t,x,\mu},\R^{t,x,\mu},\U^{t,x,\mu}\right)$ can be characterized as the solution of the following forward-backward system of ODEs: 
\begin{align}
    &\x^{t,x,\mu}_s=-\left\{b_0(t)+b_1(t) x + b_2(t)\overline{\mu}+b_3(t)u^{t,x,\mu}_t\right\} \notag \\
    &\qquad\qquad +\int_t^s \left\{ b_1(r) \x^{t,x,\mu}_r + b_2(r)\e\left[\x^{t,\xi}_r\right]+b_3(r)\U^{t,x,\mu}_r\right\}   dr, \notag \\
    &\y^{t,x,\mu}_s= g_1 \x^{t,x,\mu}_T+ g_2 \e\left[ \x^{t,\xi}_T\right]  + \int_s^T  \Big[f_1(r) \x^{t,x,\mu}_r + f_2(r) \y^{t,x,\mu}_r + f_4(r) \e\left[ \x^{t,\xi}_r\right] \notag\\
    &\quad\qquad\qquad\qquad\qquad\qquad\qquad\qquad\qquad + f_5(r) \e\left[ \y^{t,\xi}_r\right] + f_7(r) \U^{t,x,\mu}_r  \Big] dr, \notag \\
    & \p^{t,x,\mu}_s= - g_1^\top \R^{t,x,\mu}_T - g_2^\top \e\left[ \R^{t,\xi}_T\right]  +\Phi_{xx}\left(X^{t,x,\mu}_T, \lr\left(X^{t,\xi}_T\right) \right) \x^{t,x,\mu}_T \notag \\
    &\qquad +\e\bigg[D_{x'}\frac{d\Phi_x}{d\nu}\left(X^{t,x,\mu}_T, \lr\left(X^{t,\xi}_T\right) \right)\left({X^{t,\xi}_T}\right) {\x^{t,\xi}_T}+\left(D_{x'}\frac{d\Phi_x}{d\nu} \left({X^{t,\xi}_T}, \lr\left(X^{t,\xi}_T\right)\right)\left(X^{t,x,\mu}_T\right)\right)^\top {\x^{t,\xi}_T}  \bigg] \notag \\
    &\qquad +\e\widetilde{\e}\left[D_{x''}D_{x'}\frac{d^2\Phi}{d\nu^2} \left({X^{t,\xi}_T}, \lr\left(X^{t,\xi}_T\right)\right)\left(X^{t,x,\mu}_T,\widetilde{X^{t,\xi}_T}\right) \widetilde{\x^{t,\xi}_T} \right] \notag \\
    &\qquad +{\e}\left[D_{x'}^2\frac{d\Phi}{d\nu} \left({X^{t,\xi}_T}, \lr\left(X^{t,\xi}_T\right)\right)\left(X^{t,x,\mu}_T\right) \x^{t,x,\mu}_T \right] \notag \\
    &\qquad + \int_s^T \Bigg\{  b_1(r)^\top  \p^{t,x,\mu}_r + b_2(r)^\top \e\left[\p^{t,\xi}_r\right] - f_1(r)^\top  \R^{t,x,\mu}_r - f_4(r)^\top \e\left[\R^{t,\xi}_r\right] \notag \\
    &\quad\qquad +\left[\left(\begin{array}{ccc}
         h_{xx} & h_{xy} & h_{xv}
    \end{array} \right)\left(r,\theta^{t,x,\mu}_r,\lr\left(\theta^{t,\xi}_r\right),u^{t,x,\mu}_r\right) \right] \left(\begin{array}{c}
         \x^{t,x,\mu}_r  \\
         \y^{t,x,\mu}_r  \\
         \U^{t,x,\mu}_r
    \end{array}\right) \notag \\
    & \quad\qquad +{\e}\left\{\left[\left(\begin{array}{cc}
         D_{x'}\frac{dh_x}{d\nu} &  D_{y'}\frac{dh_x}{d\nu}
    \end{array}\right)\left(r,\theta^{t,x,\mu}_r,\lr\left(\theta^{t,\xi}_r\right),u^{t,x,\mu}_r\right)\left({\theta^{t,\xi}_r}\right) \right] \left(\begin{array}{cc}
         {\x^{t,\xi}_r}  \\
         {\y^{t,\xi}_r} 
    \end{array}\right) \right\} \notag \\
    & \quad\qquad +{\e}\left\{\left[\left(\begin{array}{c}
         D_{x'}\frac{dh_x}{d\nu} \\
         D_{x'}\frac{dh_y}{d\nu} \\
         D_{x'}\frac{dh_v}{d\nu}
    \end{array}\right)\left(r,{\theta^{t,\xi}_r},\lr\left(\theta^{t,\xi}_r\right),{u^{t,\xi}_r}\right)\left(\theta^{t,x,\mu}_r\right) \right] \left(\begin{array}{cc}
         {\x^{t,\xi}_r}  \\
         {\y^{t,\xi}_r}  \\
         {\U^{t,\xi}_r}
    \end{array}\right) \right\} \notag \\
    & \quad\qquad  +  \e\widetilde{\e} \left\{ \left[\left(\begin{array}{cc}
         D_{x''}D_{x'}\frac{d^2h}{d\nu^2} & D_{y''}D_{x'}\frac{d^2h}{d\nu^2}  
    \end{array}\right) \left(r,{\theta^{t,\xi}_r},\lr\left(\theta^{t,\xi}_r\right),{u^{t,\xi}_r}\right)\left(\theta^{t,x,\mu}_r,\widetilde{\theta^{t,\xi}_r}\right) \right]  \left(\begin{array}{cc}
         \widetilde{\x^{t,\xi}_r}   \\
         \widetilde{\y^{t,\xi}_r}  
    \end{array} \right) \right\} \notag \\
    & \quad\qquad  +  {\e} \left\{ \left[\left(\begin{array}{cc}
         D^2_{x'}\frac{dh}{d\nu} & D_{y'}D_{x'}\frac{dh}{d\nu}  
    \end{array}\right) \left(r,{\theta^{t,\xi}_r},\lr\left(\theta^{t,\xi}_r\right),{u^{t,\xi}_r}\right)\left(\theta^{t,x,\mu}_r\right) \right]  \left(\begin{array}{cc}
         \x^{t,x,\mu}_r   \\
         \y^{t,x,\mu}_r  
    \end{array} \right) \right\} \Bigg\} dr, \notag \\
    & \R^{t,x,\mu}_s= - \varphi_{yy} \left(Y^{t,x,\mu}_t\right) \left(\y^{t,x,\mu}_t - f\left(t,x,Y^{t,x,\mu}_t,\lr\left(\xi,Y^{t,\xi}_t\right),u^{t,x,\mu}_t\right) \right) - f_2(t)^\top  R^{t,x,\mu}_t   - f_5(t)^\top  \e\left[R^{t,\xi}_t\right]\notag \\
    &\qquad + h_y \left(t,x,Y^{t,x,\mu}_t,\lr\left({ \xi, Y^{t,\xi}_t}\right),u^{t,x,\mu}_t\right)  +{\e}\left[D_{y'} \frac{dh}{d\nu}\left(t,{\xi},{Y^{t,\xi}_t},\lr\left(\xi,Y^{t,\xi}_t\right),{u^{t,\xi}_t} \right)\left(x,Y^{t,x,\mu}_t\right) \right] \notag \\
    &\qquad +\int_t^s \Bigg\{ f_2(r)^\top  \R^{t,x,\mu}_r  + f_5(r)^\top  \e\left[\R^{t,\xi}_r\right] - \left[\left( \begin{array}{ccc}
         h_{yx} & h_{yy} & h_{yv}
    \end{array}\right)\left(r,\theta^{t,x,\mu}_r,\lr\left({ \theta^{t,\xi}_r}\right),u^{t,x,\mu}_r\right)\right] \left(\begin{array}{c}
         \x^{t,x,\mu}_r \\
         \y^{t,x,\mu}_r \\
         \U^{t,x,\mu}_r
    \end{array}\right) \notag \\
    & \quad\qquad -{\e}\left\{ \left[\left(\begin{array}{cc}
         D_{x'}\frac{dh_{y}}{d\nu} & D_{y'}\frac{dh_{y}}{d\nu} 
    \end{array} \right) \left(r,\theta^{t,x,\mu}_r,\lr\left(\theta^{t,\xi}_r\right),u^{t,x,\mu}_r\right)\left({\theta^{t,\xi}_r}\right) \right] \left( \begin{array}{cc}
         {\x^{t,\xi}_r}  \\
         {\y^{t,\xi}_r} 
    \end{array}\right) \right\}  \notag\\
    & \quad\qquad - {\e}\left\{\left[\left(\begin{array}{ccc}
        D_{y'}\frac{dh_x}{d\nu} \\ 
        D_{y'}\frac{dh_y}{d\nu} \\
        D_{y'}\frac{dh_v}{d\nu}
    \end{array}\right)\left(r,{\theta^{t,\xi}_r},\lr\left({ \theta^{t,\xi}_r}\right),{u^{t,\xi}_r}\right)\left( \theta^{t,x,\mu}_r\right)\right]^\top  \left(\begin{array}{c}
        {\x^{t,\xi}_r}  \\
        {\y^{t,\xi}_r}  \\
        {\U^{t,\xi}_r} 
    \end{array}
    \right) \right\}  \notag\\
    & \quad\qquad  - \e\widetilde{\e} \left\{ \left[\left( \begin{array}{cc}
         D_{x''}D_{y'}\frac{d^2h}{d\nu^2} & D_{y''}D_{y'}\frac{d^2h}{d\nu^2} \\
    \end{array}\right) \left(r,{\theta^{t,\xi}_r},\lr\left({ \theta^{t,\xi}_r}\right),{u^{t,\xi}_r}\right)\left( \theta^{t,x,\mu}_r,\widetilde{ \theta^{t,\xi}_r}\right) \right] \left(\begin{array}{c}
         \widetilde{\x^{t,\xi}_r}   \\
         \widetilde{\y^{t,\xi}_r} 
    \end{array}\right) \right\}  \notag\\
    & \quad\qquad  - {\e} \left\{ \left[\left( \begin{array}{cc}
         D_{x'}D_{y'}\frac{dh}{d\nu} & D_{y'}^2\frac{dh}{d\nu} \\
    \end{array}\right) \left(r,{\theta^{t,\xi}_r},\lr\left({ \theta^{t,\xi}_r}\right),{u^{t,\xi}_r}\right)\left( \theta^{t,x,\mu}_r\right) \right] \left(\begin{array}{c}
         \x^{t,x,\mu}_r   \\
         \y^{t,x,\mu}_r 
    \end{array}\right) \right\} \Bigg\} dr, \notag \\
    &0=b_3(s)^\top  \p^{t,x,\mu}_s -f_7(s)^\top  \R^{t,x,\mu}_s + \left[\left(\begin{array}{ccc}
         h_{vx} & h_{vy} & h_{vv} 
    \end{array} \right)\left(s,\theta^{t,x,\mu}_s,\lr\left({ \theta^{t,\xi}_s}\right),u^{t,x,\mu}_s\right) \left(\begin{array}{c}
         \x^{t,x,\mu}_s \\
         \y^{t,x,\mu}_s \\
         \U^{t,x,\mu}_s 
    \end{array}\right)\right] \notag \\
    & \qquad +{\e}\left\{ \left[\left(\begin{array}{cc}
         D_{x'}\frac{dh_{v}}{d\nu} & D_{y'}\frac{dh_{v}}{d\nu} 
    \end{array} \right) \left(s,\theta^{t,x,\mu}_s,\lr\left(\theta^{t,\xi}_s\right),u^{t,x,\mu}_s\right)\left({\theta^{t,\xi}_s}\right) \right] \left( \begin{array}{cc}
         {\x^{t,\xi}_s}  \\
         {\y^{t,\xi}_s} 
    \end{array}\right) \right\}. \label{FBSDEs:high_txmu}
\end{align}
System \eqref{FBSDEs:high_txmu} is structurally similar to \eqref{FBSDEs:high}, except that its coefficients depend on $\theta^{t,x,\mu}$, $\lr\left(\theta^{t,\xi}\right)$, and the solution of \eqref{FBSDEs:high}. Since the well-posedness and convergence follow by the same argument as in Theorem~\ref{thm:convergence}, we omit the details. The estimate \eqref{lem:xypru_0_txmu} then follows from \eqref{tmux-0} and \eqref{lem:xypru_0}.

\section{Proof of some statements in Section~\ref{sec:feedback}}\label{App:sec:feedback}

\subsection{Proof of Lemma~\ref{lem:der_h}}\label{App:lem:der_h}
The continuity and differentiability of $\mathbf{h}$ in $(x,v)$ follow directly from Assumption (A2). Specifically, we have
\begin{align*}
    \mathbf{h}_x (s,x,m,v)= h_x\left(\star \right) +F_1(s)^\top h_y \left(\star \right) +\sum_{j=1}^n \sigma^j_1(s)^\top F_1(s)^\top h_{z^j} \left(\star \right); \qquad \mathbf{h}_v (s,x,m,v)=  h_v\left(\star \right),
\end{align*}
where
\begin{align*}
    (\star):=\left(s,x, F(s,x),F_1(s)\sigma(s,x,m), \left(I_n\otimes F(s,\cdot)\otimes F_1(s)\sigma(s,\cdot,m)\right)\#m, v\right).
\end{align*}
We now compute the linear functional derivative of $\mathbf{h}$ with respect to $m$. For a linearly functionally differentiable function $k:\pr_2(\brn)\to\mathbb{R}$ such that $x'\mapsto \frac{dk}{d\nu}(m)(x')$ is differentiable, and a push-forward map $\phi:\brn\times\pr_2(\brn)\to\brn$ that is also linearly functionally differentiable, define $K(m):=k(\phi(\cdot,m)\#m)$. Then the chain rule gives
\begin{align*}
    \frac{dK}{d\nu}(m)(x')=\ & \frac{dk}{d\nu}(\phi(\cdot,m)\#m)(\phi(x',m)) \\
    &+ \int_\brn \left(\frac{d\phi}{d\nu}(x'',m)(x')\right)^\top D\frac{dk}{d\nu}(\phi(\cdot,m)\#m)(\phi(x'',m))   m(dx'').
\end{align*}
With this formula, by using \eqref{def:hh} and the linearity of $\sigma$ in Assumption (A1), we can explicitly compute that  
\begin{align*}
    &D_{x'}\frac{d \mathbf{h}}{d\nu} (s,x,m,v)(x')=  \sum_{j=1}^n \sigma^j_2(s)^\top F_1(s)^\top h_{z^j} \left(\star \right) + D_{x'}\frac{d h}{d\nu}\left(\star \right)\left(x',F(s,x'),F_1(s)\sigma(s,x',m)\right)\\
    &\qquad\qquad\qquad\qquad\qquad + F_1(s)^\top D_{y'}\frac{d h}{d\nu}\left(\star \right)\left(x',F(s,x'),F_1(s)\sigma(s,x',m)\right) \\
    &\qquad\qquad\qquad\qquad\qquad + \sum_{j=1}^n \sigma^j_1(s)^\top F_1(s)^\top D_{{z'}^j}\frac{d h}{d\nu}\left(\star \right)\left(x',F(s,x'),F_1(s)\sigma(s,x',m)\right)\\
    &\qquad\qquad\qquad\qquad\qquad + \sum_{j=1}^n \sigma^j_2(s)^\top F_1(s)^\top \int_\brn D_{{z'}^j}\frac{d h}{d\nu}\left(\star \right)\left(x'',F(s,x''),F_1(s)\sigma(s,x'',m)\right) m(dx'').
\end{align*}
The above formulas, together with the boundedness and Lipschitz continuity of the coefficients, imply the claimed Lipschitz continuity of $\mathbf{h}_x$, $\mathbf{h}_v$, and $D_{x'}\frac{d\mathbf{h}}{d\nu}$. It remains to prove the convexity. From the definition of $\mathbf{h}$ and the convexity of $h$ in Assumption (A3), by using Fubini's theorem and the above expressions of the derivatives, we see that, for any $s\in[0,T]$, $x^1,x^2\in\brn$, $v^1,v^2\in\brd$ and $\xi^1,\xi^2\in L^2(\Omega,\f,\mathbb{P};\brn)$,
\begin{align*}
    &\mathbf{h}\left(s,x^2,\lr\left(\xi^2\right),v^2\right)-\mathbf{h}\left(s,x^1,\lr\left(\xi^1\right),v^1\right)\\
    =\ & h\left(s,x^2,F\left(s,x^2\right),F_1(s)\sigma\left(s,x^2,\lr\left(\xi^2\right)\right),\lr\left(\xi^2,F\left(s,\xi^2\right),F_1(s)\sigma\left(s,\xi^2,\lr\left(\xi^2\right)\right)\right),v^2\right)\\
    & -h\left(s,x^1,F\left(s,x^1\right),F_1(s)\sigma\left(s,x^1,\lr\left(\xi^1\right)\right),\lr\left(\xi^1,F\left(s,\xi^1\right),F_1(s)\sigma\left(s,\xi^1,\lr\left(\xi^1\right)\right)\right),v^1\right)\\
    \geq\ & h_v\left(\dagger\right)^\top \left(v^2-v^1\right) + h_x\left(\dagger\right)^\top \left(x^2-x^1\right) +h_y\left(\dagger\right)^\top F_1(s)\left(x^2-x^1\right) \\
    &+\sum_{j=1}^n h_{z^j}\left(\dagger\right)^\top F_1(s) \left[\sigma^j_1(s) \left(x^2-x^1\right) + \sigma^j_2(s) \e\left(\xi^2-\xi^1\right)\right]\\
    &+\e\bigg[\left(D_{x'}\frac{dh}{d\nu}\left(\dagger\right)\left(\xi^1,F\left(s,\xi^1\right),F_1(s)\sigma\left(s,\xi^1,\lr\left(\xi^1\right)\right)\right)\right)^\top \left(\xi^2-\xi^1\right) \\
    &\qquad + \left(D_{y'}\frac{dh}{d\nu}\left(\dagger\right)\left(\xi^1,F\left(s,\xi^1\right),F_1(s)\sigma\left(s,\xi^1,\lr\left(\xi^1\right)\right)\right)\right)^\top F_1(s)\left(\xi^2-\xi^1\right)\\
    &\qquad + \sum_{j=1}^n \left(D_{{z'}^j}\frac{dh}{d\nu}\left(\dagger\right)\left(\xi^1,F\left(s,\xi^1\right),F_1(s)\sigma\left(s,\xi^1,\lr\left(\xi^1\right)\right)\right)\right)^\top F_1(s) \sigma^j_1(s)\left(\xi^2-\xi^1\right) \\
    &\qquad + \sum_{j=1}^n \left(D_{{z'}^j}\frac{dh}{d\nu}\left(\dagger\right)\left(\xi^1,F\left(s,\xi^1\right),F_1(s)\sigma\left(s,\xi^1,\lr\left(\xi^1\right)\right)\right)\right)^\top F_1(s)\sigma^j_2(s) \e\left[\xi^2-\xi^1\right] \bigg] \\
    & +\lambda\left|v^2-v^1\right|^2\\
    =\ & \mathbf{h}_v \left(s,x^1,\lr\left(\xi^1\right),v^1\right)^\top \left(v^2-v^1\right) + \mathbf{h}_x \left(s,x^1,\lr\left(\xi^1\right),v^1\right)^\top \left(x^2-x^1\right) \\
    &+ \e\left[\left(D_{x'}\frac{d \mathbf{h}}{d\nu}\left(s,x^1,\lr\left(\xi^1\right),v^1\right)\left(\xi^1\right)\right)^\top \left(\xi^2-\xi^1\right)\right] +\lambda\left|v^2-v^1\right|^2,
\end{align*}
where
\begin{align*}
    (\dagger):=\left(s,x^1,F\left(s,x^1\right),F_1(s)\sigma\left(s,x^1,\lr\left(\xi^1\right)\right),\lr\left(\xi^1,F\left(s,\xi^1\right),F_1(s)\sigma\left(s,\xi^1,\lr\left(\xi^1\right)\right)\right),v^1\right).
\end{align*}
This proves the first convexity assertion \eqref{convex:h}. The convexity of $x'\mapsto \frac{d\mathbf{h}}{d\nu}(s,x,m,v)(x')$ follows similarly from the corresponding convexity of $\frac{dh}{d\nu}$ in Assumption (A3), and the proof is complete.

\subsection{Proof of Proposition~\ref{prop:relation:PQR}}\label{App:prop:relation:PQR}
Using the relations between the derivatives of $\mathbf{h}$ and $h$ established in the proof of Lemma~\ref{lem:der_h}, the BSDE \eqref{BSDE:PQ} can be rewritten as
\begin{align}
        \mathbf{P}^{t,\xi}_s=\ & -\int_s^T \mathbf{Q}^{t,\xi}_r dB_r +\Phi_x\left(X^{t,\xi}_T, \lr\left(X^{t,\xi}_T\right) \right) +\widetilde{\e}\left[D_{x'}\frac{d\Phi}{d\nu} \left(\widetilde{X^{t,\xi}_T}, \lr\left(X^{t,\xi}_T\right)\right)\left(X^{t,\xi}_T\right) \right] \notag \\
        &+ \int_s^T \bigg\{ b_1(r)^\top \mathbf{P}^{t,\xi}_r + \sum_{j=1}^n \sigma^j_1(r)^\top \mathbf{Q}^{t,\xi,j}_r +b_2(r)^\top \e\left[\mathbf{P}^{t,\xi}_r\right] +\sigma^j_2(r)^\top \e\left[\mathbf{Q}^{t,\xi,j}_r\right] \notag \\
        &\qquad \qquad + h_x\left(r,\theta^{t,\xi}_r,\lr\left(\theta^{t,\xi}_r\right), u^{t,\xi}_r \right) +F_1(r)^\top h_y \left(r,\theta^{t,\xi}_r,\lr\left(\theta^{t,\xi}_r\right), u^{t,\xi}_r \right) \notag \\
        &\qquad \qquad +\sum_{j=1}^n \sigma^j_1(r)^\top F_1(r)^\top h_{z^j} \left(r,\theta^{t,\xi}_r, \lr\left(\theta^{t,\xi}_r\right), u^{t,\xi}_r \right) \notag \\
        &\qquad \qquad + \widetilde{\e} \bigg[\sum_{j=1}^n \sigma^j_2(r)^\top F_1(r)^\top \bigg(h_{z^j} \left(r,\widetilde{\theta^{t,\xi}_r}, \lr\left(\theta^{t,\xi}_r\right),\widetilde{u^{t,\xi}_r} \right) \notag \\
        &\qquad\qquad\qquad\qquad\qquad\qquad\qquad + \he\left[D_{{z'}^j}\frac{dh}{d\nu} \left(r,\widetilde{\theta^{t,\xi}_r}, \lr\left(\theta^{t,\xi}_r\right), \widetilde{u^{t,\xi}_r}\right)\left(\widehat{\theta^{t,\xi}_r}\right) \right] \bigg) \notag \\
        &\qquad\qquad\qquad + D_{x'}\frac{dh}{d\nu} \left(r,\widetilde{\theta^{t,\xi}_r}, \lr\left(\theta^{t,\xi}_r\right), \widetilde{u^{t,\xi}_r}\right)\left(\theta^{t,\xi}_r\right) \notag \\
        &\qquad\qquad\qquad + F_1(r)^\top D_{y'}\frac{dh}{d\nu} \left(r,\widetilde{\theta^{t,\xi}_r}, \lr\left(\theta^{t,\xi}_r\right), \widetilde{u^{t,\xi}_r}\right)\left(\theta^{t,\xi}_r\right)  \notag \\
        &\qquad\qquad\qquad + \sum_{j=1}^n \sigma^j_1(r)^\top F_1(r)^\top D_{{z'}^j}\frac{dh}{d\nu} \left(r,\widetilde{\theta^{t,\xi}_r}, \lr\left(\theta^{t,\xi}_r\right), \widetilde{u^{t,\xi}_r}\right)\left(\theta^{t,\xi}_r\right) \bigg] \bigg\} dr. \label{BSDE_other}
\end{align}
Similarly, the optimality condition \eqref{optimality_condition} can also be written as
\begin{align}
    b_3(s)^\top \mathbf{P}^{t,\xi}_s  + h_v\Big(s,\theta^{t,\xi}_s,\lr\left(\theta^{t,\xi}_s\right), u^{t,\xi}_s \Big)  =0. \label{coupling_condition} 
\end{align}
On the other hand, applying It\^o's formula to $P^{t,\xi}_s + F_1(s)^\top R^{t,\xi}_s$, and using the relations in \eqref{relation:f-F} together with the BSDE for $(P^{t,\xi},Q^{t,\xi},R^{t,\xi})$ in \eqref{FBSDE:MFTC}, we obtain
\begin{align*}
    &d\left(P^{t,\xi}_s+F_1(s)^\top R^{t,\xi}_s\right)\\
    =\ & - \bigg\{ b_1(s)^\top \left( P^{t,\xi}_s + F_1(s)^\top  R^{t,\xi}_s \right) + \sum_{j=1}^n \sigma^j_1(s)^\top  Q^{t,\xi,j}_s \\
    &\qquad + \e\bigg[ b_2(s)^\top \left(P^{t,\xi}_s + F_1(s)^\top R^{t,\xi}_s\right) + \sum_{j=1}^n \sigma^j_2(s)^\top  Q^{t,\xi,j}_s \bigg] \\
    &\qquad  +h_x\left(s,\theta^{t,\xi}_s,\lr\left(\theta^{t,\xi}_s\right),u^{t,\xi}_s\right) + F_1(s)^\top h_y \left(s,\theta^{t,\xi}_s,\lr\left(\theta^{t,\xi}_s\right),u^{t,\xi}_s\right) \\
    &\qquad + \widetilde{\e} \bigg[D_{x'}\frac{dh}{d\nu} \left(s,\widetilde{\theta^{t,\xi}_s},\lr\left(\theta^{t,\xi}_s\right),\widetilde{u^{t,\xi}_s}\right)\left(\theta^{t,\xi}_s\right) \\
    &\qquad\qquad + F_1(s)^\top D_{y'}\frac{dh}{d\nu} \left(s,\widetilde{\theta^{t,\xi}_s},\lr\left(\theta^{t,\xi}_s\right),\widetilde{u^{t,\xi}_s}\right)\left(\theta^{t,\xi}_s\right) \bigg]  \bigg\} dr +\sum_{j=1}^n \mathcal{A}^{j}_s dB^j_s \\
    =\ & - \bigg\{ b_1(s)^\top \left( P^{t,\xi}_s + F_1(s)^\top  R^{t,\xi}_s \right) + b_2(s)^\top \e\bigg[ \left(P^{t,\xi}_s + F_1(s)^\top R^{t,\xi}_s\right)\bigg] \\
    &\qquad + \sum_{j=1}^n \sigma^j_1(s)^\top  \mathcal{A}^{j}_s + \sum_{j=1}^n \sigma^j_2(s)^\top  \e \left[\mathcal{A}^{j}_s\right] \\
    &\qquad  +h_x\left(s,\theta^{t,\xi}_s,\lr\left(\theta^{t,\xi}_s\right),u^{t,\xi}_s\right) + F_1(s)^\top h_y \left(s,\theta^{t,\xi}_s,\lr\left(\theta^{t,\xi}_s\right),u^{t,\xi}_s\right) \\
    &\qquad + \sum_{j=1}^n \sigma^j_1(s)^\top F_1(s)^\top h_{z^j} \left(s,\theta^{t,\xi}_s,\lr\left(\theta^{t,\xi}_s\right),u^{t,\xi}_s\right) \\
    &\qquad +  \widetilde{\e} \bigg\{ \sum_{j=1}^n \sigma^j_2(s)^\top F_1(s)^\top  \bigg(h_{z^j} \left(s,\widetilde{\theta^{t,\xi}_s},\lr\left(\theta^{t,\xi}_s\right), \widetilde{u^{t,\xi}_s} \right) \\
    &\quad\qquad\qquad\qquad\qquad\qquad\qquad + \he\bigg[D_{{z'}^j}\frac{dh}{d\nu} \left(s,\widetilde{\theta^{t,\xi}_s},\lr\left(\theta^{t,\xi}_s\right),\widetilde{u^{t,\xi}_s}\right)\left(\widehat{\theta^{t,\xi}_s}\right) \bigg] \bigg) \\
    &\qquad\qquad + D_{x'}\frac{dh}{d\nu} \left(s,\widetilde{\theta^{t,\xi}_s},\lr\left(\theta^{t,\xi}_s\right),\widetilde{u^{t,\xi}_s}\right)\left(\theta^{t,\xi}_s\right) + F_1(s)^\top D_{y'}\frac{dh}{d\nu} \left(s,\widetilde{\theta^{t,\xi}_s},\lr\left(\theta^{t,\xi}_s\right),\widetilde{u^{t,\xi}_s}\right)\left(\theta^{t,\xi}_s\right) \\
    &\qquad\qquad + \sum_{j=1}^n \sigma^j_1(s)^\top  F_1(s)^\top D_{{z'}^j}\frac{dh}{d\nu} \left(s,\widetilde{\theta^{t,\xi}_s},\lr\left(\theta^{t,\xi}_s\right),\widetilde{u^{t,\xi}_s}\right)\left(\theta^{t,\xi}_s\right)\bigg\} \bigg\} ds +\sum_{j=1}^n \mathcal{A}^{j}_s dB^j_s,
\end{align*}
where
\begin{align*}
    \mathcal{A}^{j}_s:=Q^{t,\xi,j}_s - F_1(s)^\top h_{z^j} \left(s,\theta^{t,\xi}_s,\lr\left(\theta^{t,\xi}_s\right),u^{t,\xi}_s\right) - \widetilde{\e} \bigg[F_1(s)^\top D_{{z'}^j}\frac{dh}{d\nu} \left(s,\widetilde{\theta^{t,\xi}_s},\lr\left(\theta^{t,\xi}_s\right),\widetilde{u^{t,\xi}_s}\right)\left(\theta^{t,\xi}_s\right) \bigg].
\end{align*}
We also have the terminal condition
\begin{align*}
    P^{t,\xi}_T+F_1(T)^\top R^{t,\xi}_T=\Phi_x\left(X^{t,\xi}_T, \lr\left(X^{t,\xi}_T\right) \right) +\widetilde{\e}\left[D_{x'}\frac{d\Phi}{d\nu} \left(\widetilde{X^{t,\xi}_T}, \lr\left(X^{t,\xi}_T\right)\right)\left(X^{t,\xi}_T\right)  \right]. 
\end{align*}
Therefore, we see that the processes 
\begin{align*}
    \left(X^{t,\xi}_s,\quad P^{t,\xi}_s+F_1(s)^\top R^{t,\xi}_s,\quad \mathcal{A}_s \right),\quad t\le s\le T,
\end{align*}
also satisfy the FBSDEs for $\left(X^{t,\xi}_s, \mathbf{P}^{t,\xi}, \mathbf{Q}^{t,\xi} \right)$ defined by the SDE in \eqref{cost^v'}, the BSDE in \eqref{BSDE_other} and the coupling condition \eqref{coupling_condition} (or the BSDE in \eqref{BSDE:PQ} and the coupling condition \eqref{optimality_condition}). Then, from the uniqueness result for this system (see \cite[Section 6]{AB10} for instance), we obtain \eqref{relation_PQR-PQ}. 

\section{Proof of some statements in Section~\ref{sec:flow}}\label{App:sec:flow}

\subsection{Proof of Proposition~\ref{prop:s&n}}\label{App:prop:s&n}
Noting that under Assumptions (A1)-(A3), the equations for $\Pi^{t,\xi}_s$ can also be written as a system of FBSDEs defined on $s\in[r,T]$, with the initial condition for $X^{t,\xi}$ being $X_r^{t,\xi}$, and the initial condition for $R^{t,\xi}$ being 
\begin{align*}
    R^{t,\xi}_r= \ & - \varphi_y \left(Y^{t,\xi}_t\right) +\int_t^r \bigg\{ f_2(w)^\top  R^{t,\xi}_w  +f_5(w)^\top  \e\left[R^{t,\xi}_w\right]  -h_y \left(w,\theta^{t,\xi}_w,\lr\left({ \theta^{t,\xi}_w}\right),u^{t,\xi}_w\right) \\
    &\quad\qquad\qquad\qquad\qquad -\widetilde{\e} \bigg[D_{y'}\frac{dh}{d\nu} \left(w,\widetilde{\theta^{t,\xi}_w},\lr\left({ \theta^{t,\xi}_w}\right),\widetilde{u^{t,\xi}_w}\right)\left({ \theta^{t,\xi}_w}\right) \bigg]\bigg\} dw \\
    &+ \sum_{j=1}^n \int_t^r \bigg\{ f^j_3(w)^\top  R^{t,\xi}_w  +f^j_6(w)^\top  \e\left[R^{t,\xi}_w\right]  -h_{z^j} \left(w,\theta^{t,\xi}_w,\lr\left({ \theta^{t,\xi}_w}\right),u^{t,\xi}_w\right)\\
    &\quad\qquad\qquad -\widetilde{\e} \bigg[D_{{z'}^j}\frac{dh}{d\nu} \left(w,\widetilde{\theta^{t,\xi}_w},\lr\left({ \theta^{t,\xi}_w}\right),\widetilde{u^{t,\xi}_w}\right)\left({ \theta^{t,\xi}_w}\right) \bigg]\bigg\} dB^j_w;
\end{align*}
while in the equations for $\Pi^{r,X^{t,\xi}_r}$ defined on $s\in[t,T]$, the initial condition for $X^{r,X^{t,\xi}_r}$ is  $X_r^{t,\xi}$, and the initial condition for $R^{r,X^{t,\xi}_r}$ is $- \varphi_y \left(Y^{r,X^{t,\xi}_r}_r\right)$. Then, in view of the uniqueness result of the System for $\Pi^{r,X^{t,\xi}_r}$, the flow property \eqref{flow_property} holds if and only if the following relation holds:
\begin{align*}
    & \varphi_y \left(Y^{t,\xi}_r\right) - \varphi_y \left(Y^{t,\xi}_t\right)  \\
    =\ & - \int_t^r \bigg\{ f_2(w)^\top  R^{t,\xi}_w  +f_5(w)^\top  \e\left[R^{t,\xi}_w\right]  -h_y \left(w,\theta^{t,\xi}_w,\lr\left({ \theta^{t,\xi}_w}\right),u^{t,\xi}_w\right) \\
    &\quad\qquad -\widetilde{\e} \bigg[D_{y'}\frac{dh}{d\nu} \left(w,\widetilde{\theta^{t,\xi}_w},\lr\left({ \theta^{t,\xi}_w}\right),\widetilde{u^{t,\xi}_w}\right)\left({ \theta^{t,\xi}_w}\right) \bigg]\bigg\} dw\\
    &- \sum_{j=1}^n \int_t^r \bigg\{ f^j_3(w)^\top  R^{t,\xi}_w  +f^j_6(w)^\top  \e\left[R^{t,\xi}_w\right]  -h_{z^j} \left(w,\theta^{t,\xi}_w,\lr\left({ \theta^{t,\xi}_w}\right),u^{t,\xi}_w\right)\\
    &\quad\qquad\qquad -\widetilde{\e} \bigg[D_{{z'}^j}\frac{dh}{d\nu} \left(w,\widetilde{\theta^{t,\xi}_w},\lr\left({ \theta^{t,\xi}_w}\right),\widetilde{u^{t,\xi}_w}\right)\left({ \theta^{t,\xi}_w}\right) \bigg]\bigg\} dB^j_w.
\end{align*}
On the other hand, by using It\^o's formula for $\varphi\left(Y^{t,\xi}_w\right)$, we can also write  
\begin{align}
    &\varphi_y \left(Y^{t,\xi}_r\right) - \varphi_y \left(Y^{t,\xi}_t\right) \notag\\
    =\ & \int_t^r \left\{ -\varphi_{yy} \left(Y^{t,\xi}_w\right) f\left(w,\theta^{t,\xi}_w,\lr\left(\theta^{t,\xi}_w\right), u^{t,\xi}_w \right) + \frac{1}{2}\text{Tr}\left[\varphi_{yyy} \left(Y^{t,\xi}_w\right)Z^{t,\xi}_w\left(Z^{t,\xi}_w\right)^\top \right] \right\} dw \notag \\
    &+ \sum_{j=1}^n \int_t^r \varphi_{yy} \left(Y^{t,\xi}_w\right) Z^{t,\xi,j}_w dB^j_w. \label{Ito_R}
\end{align}
Therefore, we obtain the equivalence of the validity of the flow property \eqref{flow_property} and that of Conditions \eqref{condition:1}-\eqref{condition:2}. Suppose that \eqref{condition:1}-\eqref{condition:2} hold, then, by using \eqref{Ito_R} and the SDE for $R^{t,\xi}$, we see that
\begin{align*}
    R^{t,\xi}_s= - \varphi_y \left(Y^{t,\xi}_s\right);
\end{align*}
also, from the flow property \eqref{flow_property} and Theorem~\ref{thm:value_regu_mu}, we know that 
\begin{align*}
    P^{t,\xi}_s=P^{s,X_s^{t,\xi}}_s=D_{x}\frac{dV}{d\nu}\left(s,\lr\left(X_s^{t,\xi}\right)\right)\left(X_s^{t,\xi}\right),
\end{align*}
from which we obtain \eqref{feedback_RP}.

\subsection{Proof of Proposition~\ref{prop:flow:cinti:t}}\label{App:prop:flow:cinti:t}
By using the SDE for $X^{t,\xi}_s$ and the first estimate in \eqref{lem:kappa_0}, we know that 
\begin{align}
    \e\left[\left|X^{t,\xi}_s-\xi\right|^2\right]\le\ &  2\e\int_t^s \left[(s-t)\left|b\left(r,X^{t,\xi}_r,\lr\left(X^{t,\xi}_r\right),u^{t,\xi}_r \right)\right|^2 + \left|\sigma\left(r,X^{t,\xi}_r,\lr\left(X^{t,\xi}_r\right)\right)\right|^2 \right]dr \notag \\
    \le\ & C(L,T)(s-t) \left(1 + \int_t^s\left\|u^{t,\xi}_r\right\|_2^2 dr +  \left\|X^{t,\xi}\right\|_{\sr_\f^2}^2\right) \notag \\
    \le\ & C(L,T,\lambda)\left(1+\|\xi\|_2^2\right) (s-t). \label{prop:conti_t_1}
\end{align}
From the flow property \eqref{flow_property}, we have the fact that
\begin{align}
    Y^{t,\xi}_s - Y^{t,\xi}_t %=\ & Y^{t,\xi}_s - Y^{s,\xi}_s + Y^{s,\xi}_s -  Y^{t,\xi}_t \notag \\
    =\ & Y^{t,\xi}_s - Y^{s,\xi}_s + \e\left[ Y^{s,\xi}_s -  Y^{t,\xi}_t \Big| \xi\right] \notag \\
    =\ & Y^{t,\xi}_s - Y^{s,\xi}_s + \e \left[ Y^{s,\xi}_s - Y^{t,\xi}_s +  Y^{t,\xi}_s  -  Y^{t,\xi}_t \Big| \xi\right] \notag \\
    =\ & Y^{s,X^{t,\xi}_s}_s - Y^{s,\xi}_s + \e \left[ Y^{s,\xi}_s - Y^{s,X^{t,\xi}_s}_s \Big| \xi\right] + \e\left[ Y^{t,\xi}_s  -  Y^{t,\xi}_t \Big| \xi\right]. \label{prop:conti_t_2}
\end{align}
By using the second estimate in \eqref{lem:kappa_0} (i.e. the Lipschitz-continuity of the process $Y^{s,\xi}$ in its initial condition) and also the estimate \eqref{prop:conti_t_1}, we have
\begin{align}
    &\e\left[\left|Y^{s,X^{t,\xi}_s}_s - Y^{s,\xi}_s\right|^2\right],\quad  \e\left[\left|\e \left[ Y^{s,\xi}_s - Y^{s,X^{t,\xi}_s}_s \Big| \xi\right]\right|^2 \right] \notag\\
    \le\ & C(L,T,\lambda) \left\|X^{t,\xi}_s -\xi \right\|_2^2 \le C(L,T,\lambda)\left(1+\|\xi\|_2^2\right) (s-t). \label{prop:conti_t_3} 
\end{align}
From the BSDE of $Y^{t,\xi}_s$, the growth condition for the coefficient $f$ and the boundedness estimate for the processes $\Pi^{t,\xi}$ in \eqref{lem:kappa_0}, we know that
\begin{align}
    \e\left[\left|\e\left[ Y^{t,\xi}_s  -  Y^{t,\xi}_t \Big| \xi\right]\right|^2\right] =\ & \e\left[\left|\e\left[ - \int_t^s f\left(r,\theta^{t,\xi}_r,\lr\left(\theta^{t,\xi}_r\right),u^{t,\xi}_r \right)dr + \int_t^s Z^{t,\xi}_r dB_r \Big| \xi\right]\right|^2 \right] \notag \\
    =\ & \e\left[\left|\e\left[ - \int_t^s f\left(r,\theta^{t,\xi}_r,\lr\left(\theta^{t,\xi}_r\right),u^{t,\xi}_r \right)dr\Big| \xi\right]\right|^2 \right] \notag \\
    \le \ & \e\left[ \left|\int_t^s f\left(r,\theta^{t,\xi}_r,\lr\left(\theta^{t,\xi}_r\right),u^{t,\xi}_r \right)dr\right|^2  \right] \notag \\    
    \le\ & (s-t)\e\left[ \int_t^s \left|f\left(r,\theta^{t,\xi}_r,\lr\left(\theta^{t,\xi}_r\right),u^{t,\xi}_r \right) \right|^2dr  \right] \notag \\
    \le\ & C(L,T,\lambda)\left(1+\|\xi\|_2^2\right)(s-t). \label{prop:conti_t_4}
\end{align}
Combining \eqref{prop:conti_t_2}-\eqref{prop:conti_t_4}, we obtain the following continuity of $Y^{t,\xi}_s$ in time:
\begin{align*}
    \e\left[\left|Y^{t,\xi}_s - Y^{t,\xi}_t \right|^2\right] \le C(L,T,\lambda)\left(1+\|\xi\|_2^2\right)(s-t). 
\end{align*}
In a similar way, we can also obtain the estimate for $P^{t,\xi}_s-P^{t,\xi}_t$; and by using the first relation of \eqref{feedback_RP} and the Lipschitz-continuity of $\varphi_y$, we obtain the estimate for $R^{t,\xi}_s-R^{t,\xi}_t$, which altogether imply \eqref{prop:conti_t_0}. From Theorem~\ref{thm:value_regu_mu}, we have the following relation
\begin{align*}
    D_{x}\frac{dV}{d\nu}(s,\lr(\xi))(\xi) - D_{x}\frac{dV}{d\nu}(t,\lr(\xi))(\xi)=P^{s,\xi}_s - P^{t,\xi}_t = \e \left[ P^{s,\xi}_s - P^{t,\xi}_s +  P^{t,\xi}_s  -  P^{t,\xi}_t \Big| \xi\right].
\end{align*}
Then, by using a similar approach as in \eqref{prop:conti_t_3}-\eqref{prop:conti_t_4}, we know that 
\begin{align*}
    \left\|D_{x}\frac{dV}{d\nu}(s,\lr(\xi))(\xi) - D_{x}\frac{dV}{d\nu}(t,\lr(\xi))(\xi) \right\|_2  \le C(L,T,\lambda)\left(1+\|\xi\|_2 \right)(s-t)^{\frac{1}{2}}.
\end{align*}
Similar as in \eqref{FBSDEs:x} and Lemma~\ref{lem:tmux}, we can also study the system for the processes $\Pi^{t,x,\mu}_s$ (whose initial condition is $x$ and the coefficients are dependent on the law $\lr\left(\theta^{t,\xi}_s\right)$) to obtain the continuity estimate \eqref{prop:conti_t_0'}. The boundedness and continuity estimates for this system follow from a similar method as that for the processes $\Pi^{t,\xi}$, except that we also use the convexity of the function $\frac{dh}{d\nu}(s,x,y,z,\varrho,v)(\cdot)$. Then, under Assumption (A2'), we can proceed to study the system for the processes $D_x\Pi^{t,x,\mu}$ (see \cite[Section 4]{AB} for the case with controlled McKean-Vlasov SDEs) to obtain \eqref{prop:conti_t_0''}, whose proof is omitted here. 

\subsection{Proof of Proposition~\ref{prop:regu_U}}\label{App:prop:regu_U}
Under Assumption (A2'), the G\^ateaux derivative of $\theta^{t,\xi}$ in $\xi$ along some direction $\eta$, denoted by $D_\eta \theta^{t,\xi}:=\left(D_\eta X^{t,\xi},D_\eta Y^{t,\xi},D_\eta Z^{t,\xi}\right)$, can be characterized as the solution of the following FBSDEs:
\begin{equation}\label{System:XYZ_D_xi}
    \left\{
        \begin{aligned}
            &D_\eta X^{t,\xi}_s= \eta +\int_t^s \Bigg\{ \left[\left(\begin{array}{ccc}
                 B_x & B_y & B_z
            \end{array}\right)\left(r, \theta^{t,\xi}_r,\lr\left(\theta^{t,\xi}_r\right)\right) \right] \left(\text{\small $\begin{array}{c}
                 D_\eta X^{t,\xi}_r  \\
                 D_\eta Y^{t,\xi}_r  \\
                 D_\eta Z^{t,\xi}_r
            \end{array}$}\right)  \\
            &\qquad\qquad\qquad + \widetilde{\e}\Bigg[ \left[\left(\begin{array}{ccc}
                 D_{x'}\frac{dB}{d\nu} & D_{y'}\frac{dB}{d\nu} & D_{z'}\frac{dB}{d\nu} 
            \end{array}\right)\left(r, \theta^{t,\xi}_r,\lr\left(\theta^{t,\xi}_r\right)\right)\left(\widetilde{\theta^{t,\xi}_r}\right) \right] \Bigg(\text{\footnotesize $\begin{array}{c}
                  \widetilde{D_\eta X^{t,\xi}_r}  \\
                  \widetilde{D_\eta Y^{t,\xi}_r}\\
                  \widetilde{D_\eta Z^{t,\xi}_r}
            \end{array}$}\Bigg) \Bigg] \Bigg\} dr  \\
            &\qquad +\int_t^s \bigg\{ \sigma_x\left(r, X^{t,\xi}_r, \lr\left(X^{t,\xi}_r\right) \right)D_\eta X^{t,\xi}_r + \widetilde{\e}\left[D_{x'}\frac{d\sigma}{d\nu}\left(r, X^{t,\xi}_r, \lr\left(X^{t,\xi}_r\right)\right)\left(\widetilde{X^{t,\xi}_r}\right) \widetilde{D_\eta X^{t,\xi}_r} \right]  \bigg\} dB_r,\\
            &D_\eta Y^{t,\xi}_s=-\int_s^T D_\eta Z^{t,\xi}_r dB_r + g_1 D_\eta X^{t,\xi}_T + g_2\e \left[D_\eta X^{t,\xi}_T\right] \\
            &\qquad + \int_s^T \Bigg\{ \left[\left(\begin{array}{ccc}
                 F_x & F_y & F_z
            \end{array}\right)\left(r, \theta^{t,\xi}_r,\lr\left(\theta^{t,\xi}_r\right)\right) \right] \left(\text{\small $\begin{array}{c}
                 D_\eta X^{t,\xi}_r  \\
                 D_\eta Y^{t,\xi}_r  \\
                 D_\eta Z^{t,\xi}_r
            \end{array}$}\right)  \\
            &\qquad\qquad\qquad + \widetilde{\e}\Bigg[ \left[\left(\begin{array}{ccc}
                 D_{x'}\frac{dF}{d\nu} & D_{y'}\frac{dF}{d\nu} & D_{z'}\frac{dF}{d\nu} 
            \end{array}\right)\left(r, \theta^{t,\xi}_r,\lr\left(\theta^{t,\xi}_r\right)\right)\left(\widetilde{\theta^{t,\xi}_r}\right) \right] \Bigg(\text{\footnotesize $\begin{array}{c}
                  \widetilde{D_\eta X^{t,\xi}_r}  \\
                  \widetilde{D_\eta Y^{t,\xi}_r}\\
                  \widetilde{D_\eta Z^{t,\xi}_r}
            \end{array}$}\Bigg) \Bigg] \Bigg\} dr.
        \end{aligned}
    \right.
\end{equation}
The proof consists of two steps: establishing the well-posedness of the FBSDEs \eqref{System:XYZ_D_xi}, and proving the $L^2$-convergence $\frac{1}{\epsilon}\left(\theta_s^{t,\xi+\epsilon\eta}-\theta_s^{t,\xi}\right) \longrightarrow  D_\eta \theta_s^{t,\xi}$. Both steps rely on It\^o's formula and standard SDE/BSDE estimates, as in the proof of Subsection~\ref{subsec:FBSDE}. We omit the details and refer the reader to \cite[Theorem 6.1]{AB10} and \cite[Theorem 3.2]{AB11} for similar arguments. Following the same argument as in Lemma~\ref{lem:D_x}, one obtains the FBSDE characterization of $D_x \theta^{t,x,\mu}$. Then, using $D_\eta \theta^{t,\xi}$ and $D_x \theta^{t,x,\mu}$, the equations for the linear functional derivative $D_{x'}\frac{d\theta_s^{t,x,\mu}}{d\nu}(x')$ are obtained in the same manner as in Lemma~\ref{Lem:D_xi}; see System \eqref{FBSDEs:Dmu} for the first-order case, and \cite[Subsection 4.1]{AB11} for the second-order MFG setting. As in Proposition~\ref{prop:U_x-mu}, these derivatives imply the differentiability of the map $(x,\mu)\mapsto U(t,x,\mu)$, with
\begin{align*}
    U_x(t,x,\mu)=D_x Y^{t,x,\mu}_t,\qquad D_{x'}\frac{dU}{d\nu}(t,x,\mu)(x')= D_{x'} \frac{dY^{t,x,\mu}_t}{d\nu}(x').
\end{align*}
Under the stronger Assumption (A2''), we can further establish the well-posedness of the second-order derivatives $D^2_x \theta^{t,x,\mu}_s$ and $D^2_{x'}\frac{d\theta_s^{t,x,\mu}}{d\nu}(x')$, which yield the second-order derivatives
\begin{align*}
    U_{xx}(t,x,\mu)=D^2_x Y^{t,x,\mu}_t,\qquad D^2_{x'}\frac{dU}{d\nu}(t,x,\mu)(x')= D^2_{x'} \frac{dY^{t,x,\mu}_t}{d\nu}(x');
\end{align*}
see, e.g., \cite[Subsection 4.2]{AB11}. Moreover, following the argument in the proof of Proposition~\ref{prop:flow:cinti:t}, we can show the continuity in $s$ of $D_x Y_s^{t,x,\mu}$, $D_{x'}\frac{dY_s^{t,x,\mu}}{d\nu}(x')$, $D^2_x Y^{t,x,\mu}_s$, and $D^2_{x'}\frac{dY_s^{t,x,\mu}}{d\nu}(x')$.

\subsection{Proof of Lemma~\ref{convex_h}}\label{App:convex_h}
From \eqref{Condition:h}, we can compute that 
\begin{align}
    \frac{dh}{d\nu}(s,x,y,z,\varrho,v)(x',y',z')=\ &  \frac{dh^0}{d\nu}(s,x,\varrho_x,v)(x')+y^\top \Psi f_4(s) x' \notag \\
    &- \bigg[\overline{\varrho_y}^\top \left(f_5(s)^\top \Psi + \Psi f_5(s) \right) + \sum_{j=1}^n (z^j)^\top f^j_6(s)^\top \Psi \bigg] y', \label{add_2}
\end{align}
which is convex in $x'$ and linear in $y'$ and independent of $z'$. Therefore, $\frac{dh}{d\nu}(s,x,y,z,\varrho,v)(\cdot)$ is convex. Moreover, for any $s\in[0,T]$, $x^1,x^2,y^1,y^2\in\brn$, $v^1,v^2\in\brd$, $z^1,z^2\in\brnn$ and $ \zeta^1,\zeta^2\in L^2(\Omega,\f,\mathbb{P};\brn\times\brn\times\brnn)$, 
\begin{align*}
    &h\left(s,x^2,y^2,z^2,{ \lr\left(\zeta^2\right)},v^2\right)-h\left(s,x^1,y^1,z^1,{ \lr\left(\zeta^1\right)},v^1\right)\\
    & - h_x\left(s,x^1,y^1,z^1,{ \lr\left(\zeta^1\right)},v^1\right)^\top \left(x^2-x^1\right)-h_v\left(s,x^1,y^1,z^1,{ \lr\left(\zeta^1\right)},v^1\right)^\top \left(v^2-v^1\right)\\
    &-h_y\left(s,x^1,y^1,z^1,{ \lr\left(\zeta^1\right)},v^1\right)^\top \left(y^2-y^1\right)-\sum_{j=1}^n h_{z^j}\left(s,x^1,y^1,z^1,{ \lr\left(\zeta^1\right)},v^1\right)^\top \left(z^{2,j}-z^{1,j}\right)\\
    &-\e\bigg[\left(D_{x'}\frac{dh}{d\nu}\left(s,x^1,y^1,z^1, \lr\left(\zeta^1\right),v^1\right)\left(\zeta^1\right)\right)^\top \left(\zeta_x^2-\zeta_x^1\right)\\
    &\qquad + \left(D_{y'}\frac{dh}{d\nu}\left(s,x^1,y^1,z^1, \lr\left(\zeta^1\right),v^1\right)\left(\zeta^1\right)\right)^\top \left(\zeta_y^2-\zeta_y^1\right)\\
    &\qquad +\sum_{j=1}^n \left(D_{{z'}^j}\frac{dh}{d\nu}\left(s,x^1,y^1,z^1, \lr\left(\zeta^1\right),v^1\right)\left(\zeta^1\right)\right)^\top \left(\zeta_{z^j}^2-\zeta_{z^j}^1\right) \bigg] \\
    \geq\ & (\lambda+\lambda_v)\left|v^2-v^1\right|^2 + \lambda_x \left|x^2-x^1\right|^2 + \lambda_m \left|\e\left[\zeta^2_x-\zeta^1_x\right]\right|^2 + \frac{1}{2}\sum_{j=1}^n \left(z^{j,2}-z^{j,1}\right)^\top \Psi \left(z^{j,2}-z^{j,1}\right)  \\
    &- \frac{1}{2}\left(y^2-y^1\right)^\top \left(f_2(s)^\top \Psi  + \Psi f_2(s) \right) \left(y^2-y^1\right) - \frac{1}{2}\e\left[\zeta^2_y-\zeta^1_y\right]^\top \left(f_5(s)^\top \Psi  + \Psi f_5(s) \right) \e\left[\zeta^2_y-\zeta^1_y\right] \\
    & - \sum_{j=1}^n  \left(y^2-y^1\right)^\top \Psi f^j_3(s) \left(z^{j,2}-z^{j,1}\right)  - \sum_{j=1}^n  \e\left[\zeta^2_y-\zeta^1_y\right]^\top \Psi f^j_6(s) \left(z^{j,2}-z^{j,1}\right) \\
    & - \left(y^2-y^1\right)^\top \Psi f_1(s) \left(x^2-x^1\right) - \left(y^2-y^1\right)^\top \Psi f_4(s)\e\left[\zeta^2_x-\zeta^1_x\right]  - \left(y^2-y^1\right)^\top \Psi f_7(s)\left(v^2-v^1\right) \\
    \geq\ & \lambda \left|v^2-v^1\right|^2,
\end{align*}
where the first inequality comes from the convexity of $h^0$ in \eqref{convex_h^0}, and the last inequality is a direct consequence of the fact that the matrix \eqref{matrix} is positive semi-definite. This verifies that Assumption (D) is sufficient for the convexity of $h$ required in (A3).

%%%%%%%%%%%%%%%%%%%%%%%%%%%%%%%%%%%%%%%%%%%%%%%%%%%%%%%%%%%%%%%%%%%%
\footnotesize
\bibliographystyle{acm}
\bibliography{references}
\phantomsection\addcontentsline{toc}{section}{References}\tolerance=500

\end{document}